\documentclass[11pt]{amsart}

 \usepackage{amssymb}
 \usepackage[colorlinks,linkcolor={blue}]{hyperref}
 \usepackage{graphicx}
 \usepackage{amscd}
 \usepackage{mathtools}
 \usepackage{enumerate}
  \usepackage{wasysym}
 \usepackage[dvipsnames]{xcolor}
 \usepackage{yhmath}
 \usepackage{psfrag}
 \usepackage{scalerel}

 \def\a{\alpha}
 \def\da{{\dot\alpha}}
 \def\be{\beta}
 \def\dbe{{\dot\beta}}
 \def\C{{\mathbb C}}
 \def\de{\delta}
 \def\De{\Delta}
 \def\e{\varepsilon}
 \def\deta{{\dot{\eta}}}

 \def\ga{\gamma}
 \def\dga{{\dot{\gamma}}}
 
 \def\tga{{\tilde{\gamma}}}
 \def\Ga{\Gamma}
 
 \def\vr{\varphi}
 \def\vrt{\vartheta}

 \def\la{\lambda}
 
 \def\La{\Lambda}
 \def\bLa{{\mathbf{\Lambda}}}
 
 \def\si{\sigma}
 \def\Si{\Sigma}
 \def\om{\omega}
 \def\Om{\Omega}

 \def\tt{\theta}

 \def\otau{{\ov\tau}}
 
 \def\dxi{{\dot{\xi}}}

 \def\re{{\mathbb R}}
 \def\na{{\mathbb N}}

 \def\then{\Longrightarrow}
 \def\ov{\overline}
 \def\Z{{\mathbb Z}}
 \def\Pii{{\pi}}

\def\Á{\textexclamdown}

 \def\A{{\mathbb A}}
 \def\cA{{\mathcal A}}

 \def\B{{\mathbb B}}
 \def\cB{{\mathcal B}}
 
 \def\cC{{\mathcal C}}

 \def\D{{\mathbb D}}

 \def\cE{{\mathcal E}}

 \def\E{{\mathbb E}}

 \def\cG{{\mathcal G}}
 
 \def\cH{{\mathcal H}}

 \def\I{{\mathbb I}}

 \def\K{{\mathbb K}}
 
 \def\L{{\mathbb L}}

 \def\cM{{\mathcal M}}

 \def\cN{{\mathcal N}}
 \def\mN{{\widetilde{\mathcal N}}}

 \def\cO{{\mathcal O}}
 \def\bO{{\mathbb O}}
 \def\cP{{\mathcal P}}

 \def\cR{{\mathcal R}}

 \def\cS{{\mathcal S}}
 \def\SS{{\mathbb S}}

 \def\cT{{\mathcal T}}
 \def\fT{{\mathfrak T}}

 \def\cU{{\mathcal U}}

 \def\ou{{\ov{u}}}

 \def\cV{{\mathcal V}}

 \def\cW{{\mathcal W}}
 \def\ow{{\ov{w}}}

 \def\tq{{\tilde{q}}}
 
 \def\tq1{{\tilde{q}_1}}

 \def\dx{{\dot{x}}}
 \def\ox{{\ov{x}}}

 \def\X{{\mathbb X}}
 \def\fX{{\mathfrak X}}

 \def\oy{{\ov{y}}}

 \def\dy{{\dot{y}}}
 \def\dz{{\dot{z}}}

 \def\oz{{\ov{z}}}

 \def\ogd{{\ov\ga_\de}}
 \def\gad{{\ga_\de}}
 
 \def\0{{\mathbf 0}}
 
 \def\vD{{\breve D}}

 \def \lv{\left\vert}
 \def \rv{\right\vert}
 \def \lV{\left\Vert}
 \def \rV{\right\Vert}
 \def \ov{\overline}
 
 \def \then{\Longrightarrow}

 \definecolor{dgreen}{rgb}{0,0.3,0}
 \definecolor{dred}{rgb}{0.8,0,0}

 \def\ee{\text{\rm\large  e}}

 \DeclareMathOperator*{\tsum}{{\textstyle \sum}}
 \DeclareMathOperator{\supp}{supp}
 
 \DeclareMathOperator{\diam}{diam}

 \DeclareMathOperator{\interior}{int}
 \DeclareMathOperator{\intt}{int}

 \DeclareMathOperator{\Lip}{Lip}

 \DeclareMathOperator{\pr}{pr}

  \renewcommand{\proofname}{{\bf Proof:}}

 \theoremstyle{plain}

 \newtheorem{MainThm}{Theorem}
 \newtheorem{MainCor}[MainThm]{Corollary}
 \newtheorem{MainProp}[MainThm]{Proposition}

 \newtheoremstyle{Cl}
  {5pt}
  {3pt}
  {\sl}
  {}
  {\it}
  {:}
  {.5em}
  {}

 \newtheoremstyle{St}
  {5pt}
  {3pt}
  {\sl}
  {}
  {\bf}
  {.}
 {\newline}
  {}

 \def\begincproof{
                  \renewcommand{\proofname}{\it Proof:}
                  \begin{proof}
                 }

 \def\endcproof{
                \renewcommand{\qedsymbol}{$\diamondsuit$}
                \end{proof}
                \renewcommand{\qedsymbol}{\openbox}
                \renewcommand{\proofname}{\bf Proof:}
               }

  \swapnumbers

 \newtheorem{Thm}{Theorem}[section]
 
 \newtheorem{Lemma}[Thm]{\bf Lemma}
 
 \newtheorem{Corollary}[Thm]{\bf Corollary}
 \newtheorem{Theorem}[Thm]{\bf Theorem}
 \newtheorem{Proposition}[Thm]{\bf Proposition}

 \newtheorem{claim}[Thm]{\it Claim}

\swapnumbers
 \theoremstyle{Cl}
 \newtheorem{Claim}{Claim} [Thm]

\swapnumbers

 \theoremstyle{St}
\newtheorem{Statement}[Thm]{\bf Statement}

 \theoremstyle{definition}
 \newtheorem{Definition}[Thm]{\bf Definition}

 \theoremstyle{remark}
 \newtheorem{Remark}[Thm]{\bf Remark}
 
 \newtheorem{mysec}[Thm]{}

\usepackage[colorinlistoftodos]{todonotes}
\makeatletter
\providecommand\@dotsep{5}
\makeatother
 \renewcommand{\proofname}{{\bf Proof:}}

 \title
 {Generic Ma\~n\'e sets.}

 \author[G. Contreras]{Gonzalo Contreras}

\address{CIMAT \\
          A.P. 402, 36.000 \\
          Guanajuato. GTO \\
          M\'exico.}
\email{gonzalo@cimat.mx}
\thanks{Partially supported by CONACYT, Mexico, grant 
A1-S-10145.}

\begin{document}

\makeatother

\parskip +5pt

\begin{abstract}
We prove that $C^2$ generic hyperbolic Ma\~n\'e sets contain a periodic orbit.
In dimension 2, adding a result with A. Figalli and L. Rifford, 
which states that $C^2$ generic Ma\~n\'e sets are hyperbolic
we obtain  Ma\~n\'e's Conjecture for surfaces in the $C^2$ topology: 
Given a Tonelli Lagrangian $L$ on a compact 
surface $M$ there is a $C^2$ open and dense set of functions $f:M\to\re$ such
that the Ma\~n\'e set of the Lagrangian $L+f$ is a hyperbolic periodic orbit.
\end{abstract}

\maketitle

\tableofcontents


Let $M$ be a closed riemannian manifold.
A Tonelli Lagrangian is a $C^2$ function
\linebreak
 $L:TM\to\re$ that  is 
\begin{enumerate}[(i)]
\openup 5pt
\item {\it Convex:} $\exists a>0$\;
$\forall (x,v), (x,w)\in TM$, \;
$w\cdot \partial^2_{vv} L(x,v)\cdot w \ge a |w|_x^2$.
\end{enumerate}
The uniform convexity assumption and the compactness of $M$ imply that $L$ is
\begin{enumerate}[(i)]
\setcounter{enumi}{1}
\item {\it Superlinear:}
$\forall A>0$ $\exists B>0$ such that
$\forall (x,v)\in TM$: $L(x,v)> A\,|v|_x-B$.
\end{enumerate}

Given $k\in\re$, the Ma\~n\'e {\it action potential} is defined as
$\Phi_k:M\times M\to\re\cup\{-\infty\}$, 
\begin{equation}\label{actionpotential}
\Phi_k(x,y):=\inf_{\ga\in \cC(x,y)}
\int k+L(\ga,\dga),
\end{equation}
where 
\begin{equation}\label{defcxy}
\cC(x,y):=\{\ga:[0,T]\to M\; {\text{absolutely continuous }}\vert\;
T>0,\; \ga(0)=x,\;\ga(T)=y\;\}.
\end{equation}
The  Ma\~n\'e {\it critical value} is
\begin{equation}\label{defcL}
c(L) := \sup\{\, k\in \re \;|\; \exists\, x,y\in M :\;\Phi_k(x,y)=-\infty\;\}.
\end{equation}
See  \cite{CILib} for several characterizations of $c(L)$.

A curve $\ga:\re\to M$ is {\it semi-static} if 
$$
\forall s<t  \qquad
\int_s^t c(L)+L(\ga,\dga) = \Phi_{c(L)}(\ga(s),\ga(t)).
$$
Also $\ga:\re\to M$ is {\it static} if
$$
\forall s<t  \qquad
\int_s^t c(L)+L(\ga,\dga) = -\Phi_{c(L)}(\ga(t),\ga(s)).
$$
The {\it Ma\~n\'e set} of $L$ is
$$
\mN(L):=\{(\ga(t),\dga(t))\in TM \;|\; t\in\re,\; \ga:\re\to M\text{ is semi-static }\},
$$
and the {\it Aubry set } is
$$
\cA(L):=\{(\ga(t),\dga(t))\in TM \;|\; t\in\re,\; \ga:\re\to M\text{ is static }\}.
$$

The Euler-Lagrange equation
$$
\tfrac d{dt}\, \partial_vL = \partial_x L
$$
defines the Lagrangian flow $\vr_t$ on $TM$.
The {\it energy function} $E:TM\to\re$,
$$
E(x,v):= \partial_v L(x,v)\cdot v -L(x,v),
$$
is invariant under the Lagrangian flow.
The Ma\~n\'e set $\mN(L)$ is invariant under the Lagrangian flow and
it is contained in the energy level $\cE:=[E=c(L)]$ (see e.g. Ma\~n\'e \cite[p.~146]{Ma7} or \cite{CILib}).

Let $\cM_{\text{inv}}(L)$ be the set of Borel probabilities in $TM$ which are invariant under the Lagrangian flow.
Define the {\it action } functional   $A_L:\cM_{\text{inv}}(L)\to\re\cup\{+\infty\}$ as 
$$
A_L(\mu):=\int L\,d\mu.
$$
The set of {\it minimizing measures} is
$$
\cM_{\min}(L):=\arg\min_{\cM_{\text{inv}}(L)} A_L,
$$
and the {\it Mather set } $\cM(L)$ is the union of the support of minimizing measures:
$$
\cM(L):=\bigcup_{\mu\in\cM_{\min}(L)} \supp(\mu).
$$
Ma\~n\'e proves (cf. Ma\~n\'e~\cite[Thm. IV]{Ma7} also \cite[p. 165]{CDI}) that an invariant measure is minimizing if and only if it is supported in the Aubry set. Therefore we get the set of inclusions
\begin{equation}\label{mane}
\cM\subseteq\cA\subseteq\mN\subseteq\cE.
\end{equation}

\begin{Definition}\label{defhipL}\quad

We say that $\mN(L)$ is {\it hyperbolic} if there are sub-bundles
$E^s$, $E^u$ of $T\cE|_{\mN(L)}$ and $T_0>0$ such that
\begin{enumerate}[(i)]
\item $T\cE|_{\mN(L)} = E^s \oplus\langle \frac d{dt}\vr_t\rangle \oplus E^u$.
\item $\lV D\vr_{T_0}|_{E^s}\rV < 1$, $\lV D\vr_{-T_0}|_{E^u}\rV <1$.
\item $\forall t\in\re$ \quad $(D\vr_t)^*(E^s)=E^s$, $(D\vr_t)^*(E^u)=E^u$.
\end{enumerate}
\end{Definition}

 Hyperbolicity for {\sl autonomous}
lagrangian or hamiltonian flows is  always understood as hyperbolicity
for the flow restricted to the energy level.

Fix a Tonelli Lagrangian $L_0$. Let
$$
\cH^k(L_0):=\{\,\phi\in C^k(M,\re)\;|\; \mN(L_0+\phi) \text{ is hyperbolic }\},
$$
endowed with the $C^k$ topology.
By \cite[Lemma 5.2, p. 661]{CP} the map $\phi\mapsto \mN(L_0+\phi)$ is upper
semi-continuous. Therefore $\cH^k(L_0)$ is an open set for any $k\ge 2$.

Let 
$$
\cP^2(L_0) := \{\, \phi \in C^2(M,\re) \;|\;
\mN(L_0+\phi)\text{ contains a periodic orbit  or a singularity}\},
$$
and let $\ov{\cP^2(L_0)}$ be its closure in $C^2(M,\re)$.
We will prove
\begin{MainThm}\label{HYP}
$\cH^2(L_0)\subset \ov{\cP^2(L_0)}$.
\end{MainThm}

In \cite{ham} we proved that if $\Ga\subset \mN(L)$ is a periodic orbit, 
adding a potential $\phi_0\ge 0$ which is locally of the form $\phi_0(x) = \e\, d(x,\pi(\Ga))^2$
makes $\Ga$ a hyperbolic periodic orbit (or hyperbolic singularity) for the Lagrangian flow of $L+\phi_0$ and
also $\mN(L_0+\phi_0)=\Ga$. Therefore defining
$$
\cH\cP^2(L_0):= \{\, \phi \in C^2(M,\re) \;|\;
\mN(L_0+\phi)\text{ is a hyperbolic periodic orbit or singularity}\}
$$
and using the semicontinuity of $\mN(L)$ and the expansivity of $\Ga$, 
we get
\begin{MainCor}
The set $\cH\cP^2(L_0)$ contains an open and dense set in $\cH^2(L_0)$.
\end{MainCor}

With A. Figalli and L. Rifford in \cite{CFR} we prove

\begin{MainThm}\label{CFR}
If $\dim M=2$ then
$\cH^2(L_0)$ is open and dense.
\end{MainThm}

Thus for surfaces in the $C^2$ topology we obtain Ma\~n\'e's Conjecture
\cite[p. 143]{Ma7}:

\begin{MainCor}
If $\dim M =2$ then $\cH\cP^2(L_0)$ contains an  open and dense set in $C^2(M,\re)$.
\end{MainCor}

Observe that from the inclusions in \eqref{mane}, for potentials $\phi\in \cH\cP^2(L_0)$
the lagrangian $L+\phi$ has a unique minimizing measure and it is supported on a
hyperbolic periodic orbit or a hyperbolic singularity. The set $\cH\cP^2(L_0)$
is open in the $C^2$ topology, so we can approximate the lagrangian $L_0$ with a
$C^\infty$ potential $\phi$ to obtain a periodic minimizing measure, but 
the approximation is only proved to be $C^2$ small.

\bigskip
\medskip

The proof of Theorem~\ref{HYP} follows the lines of our proof \cite{ground} of the
corresponding conjecture in Ergodic Optimization. As such it is supported on the
work by G.~C. Yuan and B.~R. Hunt \cite{YH}, X. Bressaud and A. Quas \cite{BQ},
I. Morris \cite{Morris} and A. Quas and J. Siefken \cite{QS}.

For possible applications we want to remark that all the perturbing 
potentials in this paper are locally of the
form 
\begin{equation}\label{canal}
\phi(x) = \e\, d(x,\pi(\Ga))^k, 
\quad k\ge 2,
\end{equation}
 where $\Ga$ is a 
suitably chosen periodic orbit of the Lagrangian flow nearby the Ma\~n\'e set.

The proof of Theorem~\ref{HYP} has two main steps corresponding to sections
 \ref{Szeroentropy} and \ref{Sclosing}.
 In section~\ref{Szeroentropy} we prove that for $k\ge 2$, 
 $C^k$ generic hyperbolic Ma\~n\'e sets have zero topological entropy.
 Namely,
 
  \begin{MainThm}\label{Tzeroentropy}
If $L_0$ is a Tonelli lagrangian and $k\ge 2$, 
the set
$$
\cE_0(L_0) =\{\,\phi\in\cH^k(L_0)\;|\;
\mN(L_0+\phi)\text{ has zero topological entropy }\}
$$
contains a residual subset of $\cH^k(L_0)$.
\end{MainThm}

On surfaces Theorem~\ref{Tzeroentropy} follows from

\begin{MainProp}
If $L$ is a Tonelli Lagrangian, $k\ge 2$ and $\dim M=2$, \\
then
$\mN(L)$ has zero topological entropy.
\end{MainProp}
\begin{proof}
The topological entropy of $\mN(L)$ is the supremum of the metric entropies of 
the invariant measures supported on $\mN(L)$.  
By \cite[Theorem V(c)]{Ma7} the $\om$-limit of any orbit in $\mN(L)$ is in the Aubry set $\cA(L)$.
Thus, by the Poincar\'e Recurrence Theorem, any invariant measure on $\mN(L)$ is supported on $\cA(L)$.
Therefore it is enough to
prove that the Aubry set has zero topological entropy. 
By the Graph Property \cite[Theorem VI(b)]{Ma7}, the flow on $\cA(L)$ is the image   
under a Lipschitz (conjugacy) map $(\pi|_{\cA(L)})^{-1}$
of a flow on a (Lipschitz) continuous  lamination on the surface $M$.
By Fathi~\cite[Lemma~3.3]{Fa7} and Young~\cite{lsy1} the flow on the 
projected Aubry set $\pi(\cA(L))$ has zero entropy and then 
(Walters~\cite[Theorem~7.2]{Walters}) $\cA(L)$ has zero topological entropy.
\end{proof}

The proof of Theorem~\ref{Tzeroentropy} is done in two steps. 
 The first step in subsection~\ref{ssappperorb} is the study of 
 how well hyperbolic minimizing measures can be approximated 
 by closed  orbits of a given period. For approximation with large periods
 on generic lagrangians see Ma\~n\'e~\cite[Theorem~F]{Ma6}.
 We found that the arguments of this proof follow elegantly using
 symbolic dynamics on the Ma\~n\'e set.
 
 Symbolic dynamics are usually constructed for locally maximal 
 hyperbolic sets.   In order to use Theorem~\ref{CFR} we only
 assume hyperbolicity of the Ma\~n\'e set, not local maximality.
 See Crovisier~\cite{crovisier} and Fisher~\cite{Fisher} for examples 
 of diffeomorphisms with hyperbolic sets which are not contained 
 in nearby locally maximal hyperbolic sets. Fisher~\cite{Fisher}
 constructs Markov partitions for general hyperbolic sets but this 
 construction has not been done for hyperbolic flows.  
 In Appendix~\ref{AMP} we define Markov partitions for hyperbolic flows.
 In Appendix~\ref{ALM} we construct Markov partitions for 
 non locally maximal hyperbolic sets for flows. In Appendix~\ref{asst} 
 we recall from de la Llave, Marco, Moriy\'on~\cite{LlMM} a useful
 version of the Structural Stability for flows and extend the symbolic
 dynamics to an invariant set containing the hyperbolic set.
 In Appendix~\ref{ashms} we apply all this to obtain a single symbolic
 dynamics for a neighbourhood of hyperbolic Ma\~n\'e sets as
 used in subsection~\ref{ssappperorb}.
 
 For the second part of the proof of Theorem~\ref{Tzeroentropy}
 we show that for $\ga>0$ the set 
 $$
 \cT_\ga:=\{\,\phi\in C^k(M,\re)\;|\;
 h_{top}(\mN(L+\phi))\le \ga\,\}
 $$
 contains an open and dense set in $C^k(M,\re)$.
 For the open part we use the upper semicontinuity of the 
 Ma\~n\'e set and the uniform upper semicontinuity of the 
 metric entropy for h-expansive maps, which we prove in
 Appendix~\ref{AE}. The uniform h-expansivity needed is
 proved in Appendix~\ref{asha}. For the density we use 
 a short closed orbit with small action obtained in the first 
 step, perturb the lagrangian with a canal as in~\eqref{canal}
 and show that the new minimizing measures have to accumulate
 nearby the periodic orbit. Then we show that the entropy nearby
 a short periodic orbit must be small.
 
  We also include a fundamental Appendix~\ref{asha} on shadowing
  which is used for approximating lagrangian actions, 
  to define approximating segments in section~\ref{Sclosing}, for
  the construction of Markov partitions and for the uniform h-expansivity.
 
   The second step of the proof of Theorem~\ref{HYP} is also split in two parts.
   The first part in subsection~\ref{sscsr} is the observation that the existence of
   a special ``solitary'' returns and the shadowing lemma 
   allow that a perturbation by a canal as in~\eqref{canal} with $k=2$
   includes a periodic orbit in the Aubry set. The second part shows that 
   the support of a minimizing measure with zero entropy\footnote{The
   argument actually uses the zero dimension of the measure.}  
   always contain
   those solitary returns.

    Since we assume hyperbolicity of the Aubry set, which is chain recurrent
    (cf. Ma\~n\'e~\cite[Theorem V]{Ma7}), the shadowing lemma implies the 
    existence of many periodic orbits nearby. Along this paper we do not perturb
    recurrent orbits to close them. We just choose carefully an existing periodic
    orbit and perturb the lagrangian by a canal as in~\eqref{canal}.
    
    \bigskip

\section[Zero entropy]{Generic hyperbolic Ma\~n\'e sets have zero entropy.}
\label{Szeroentropy}

We begin by proving the following analogous of a Theorem by I. Morris \cite{Morris}:

\noindent{\bf Theorem~\ref{Tzeroentropy}:}
{\it If $L_0$ is a Tonelli lagrangian and $k\ge 2$, 
the set
$$
\cE_0(L_0) =\{\,\phi\in\cH^k(L_0)\;|\;
\mN(L_0+\phi)\text{ has zero topological entropy }\}
$$
contains a residual subset of $\cH^k(L_0)$.
}

\subsection{The Aubry set.}\quad

We say that a curve $\ga:\re\to M$ is {\it static} for a Tonelli Lagrangian $L$
if 
$$
s<t \quad \then\quad
\int_s^t L(\ga,\dga) = -\Phi_{c(L)}(\ga(t),\ga(s));
$$
equivalently (cf. Ma\~n\'e \cite[pp. 142--143]{Ma7}), 
if $\ga$ is semi-static and 
\begin{equation}\label{aubry}
 s<t\quad\then\quad \Phi_{c(L)}(\ga(s),\ga(t))+\Phi_{c(L)}(\ga(t),\ga(s))=0.
\end{equation}
The {\it Aubry set} is defined as
$$
\cA(L):=\{\,(\ga(t),\dga(t))\;|\; t\in\re, \;\ga\text{ is static}\;\},
$$
its elements are called {\it static vectors}.

\begin{Lemma}[A priori bound]\label{priori}\quad

For $C>0$ there exists $A_0=A_0(C)>0$ such that 
if $\ga:[0,T]\to M$
is a solution of the Euler-Lagrange equation with $A_L(\ga)<C\, T$, then 
$$
\lv\dga(t)\rv < A_0\qquad\text{ for all }t\in[0,T].
$$
\end{Lemma}

\begin{proof}
The Euler-Lagrange flow preserves the {\it energy function} 
\begin{equation}\label{EL}
E_L:= v \cdot \partial_v L -L.
\end{equation}
We have that
\begin{align}
\hskip -0.7cm\forall s\ge0\qquad
\tfrac{d\,}{ds}E_L(x,sv)\big|_s&=s\, v\cdot \partial_{vv}L(x,v)\cdot v \ge s\, a |v|_x^2.
\notag\\
E_L(x,v) &= E_L(x,0)+\int_0^{1} \tfrac{d\,}{ds}E_L(x,s v) \, ds
\notag\\
&\ge \min_{x\in M}E_L(x,0)+ \tfrac 12 a |v|_x^2.
\label{lbel}
\end{align}
Let
$$
g(r):= \sup\big\{w\cdot \partial_{vv}L(x,v)\cdot w \,:\, |v|_x\le r,\,|w|_x= 1\big\}.
$$
Then $g(r)\ge a$ and 
\begin{equation}\label{ubel}
E_L(x,v)\le \max_{x\in M}E_L(x,0) + \tfrac 12\, g(|v|_x)\, |v|_x^2.
\end{equation}

By the superlinearity there is $B>0$ such that $L(x,v)>|v|_x-B$ for all $(x,v)\in TM$.
Since $A_L(\ga)<C\,T$, the mean value theorem implies that there is $t_0\in]0,T[$ 
such that $|\dga(t_0)|<B+C$. Then~\eqref{ubel} gives an upper bound 
on the energy of $\ga$ and~\eqref{lbel} bounds the speed of $\ga$.

\end{proof}

 For $x,\,y\in M$ and $T>0$ define
 \begin{equation*}\label{defctxy}
 \cC_T(x,y):=
 \{\,\ga:[0,T]\to M\,|\, \ga \text{ is absolutely continuous}, \ga(0)=x,\,\ga(T)=y\,\}.
 \end{equation*}

 \begin{Corollary}\label{CAPB}\quad
 
 There exists $A_1>0$ such that if  $x,\,y\in M$ and $\ga\in\cC_T(x,y)$ is
 a solution of the Euler-Lagrange equation with 
 $$
 A_{L+c}(\ga) \le \Phi_c(x,y) + \max\{T,d(x,y)\},
 $$
 where $c=c(L)$,
 then
 \begin{enumerate}[(a)]
 \item $T\,\ge\,\tfrac 1A_1\; d(x,y)$.
 \label{CAPB1}
 \item  \;$|\dga(t)|\,\le\,A_1$ \; for all $t\in[0,T]$.
 \label{CAPB2}
 \end{enumerate}
 \end{Corollary}
 
 \begin{proof}
 First suppose that  $d(x,y)\le T$. Then item~\eqref{CAPB1} holds with $A_1=1$.
 Let
 \begin{equation}\label{llr}
 \ell(r):=|c|+\sup\{\,L(x,v)\,|\, (x,v)\in TM,\, |v|\le r\,\}.
 \end{equation}
 Since $d(x,y)\le T$, there exists a $C^1$ curve $\eta:[0,T]\to M$  joining $x$ to $y$ with $|\deta|\le 1$.
 We have that 
 $$
 A_{L+c}(\ga)\le \Phi_c(x,y)+T\le A_{L+c}(\eta)+T
 \le \big(\ell(1)+c\big)\,T+ T.
 $$
 Then item~\eqref{CAPB2} holds for $A=A_0(\ell(1)+c+1)$ where $A_0$ is from Lemma~\ref{priori}.
 
 Now suppose that $d(x,y)\ge T$.
 Let $\eta:[0,d(x,y)]\to M$ be a minimal geodesic with 
 $|\deta|\equiv 1$ joining $x$ to $y$. Let 
 $D:=\ell(1)+c +2>0$. From the superlinearity property there is 
 $B>1$ such that
 $$
 L(x,v) + c > D\, |v| - B, \qquad \forall (x,v)\in TM.
 $$
 Then
 \begin{align}
 [\ell(1)+c]\;d(x,y) 
    &\ge A_{L+c}(\eta) \ge \Phi_c(x,y)         
    \label{ap-1}\\
    &\ge A_{L+c}(\ga)-d(x,y)
    \label{ap-2}\\
    &\ge\int_0^T\bigl( D\;|\dga|-B\,\bigr)\,dt - d(x,y)
    \notag\\
    &\ge D\;d(x,y) - B\,T - d(x,y).
    \notag
 \end{align}
 Hence
 $$
 T \ge \tfrac{D-\ell(1)-c-1}{B}\; d(x,y) = \tfrac 1B \; d(x,y).
 $$
 From~\eqref{ap-1} and~\eqref{ap-2}, we get that
 \begin{align*}
 A_L(\ga)&\le \bigl[\,\ell(1)+c+1\,\bigr]\;d(x,y)-c\,T ,
        \\
         &\le \bigl\{\,B\,[\,\ell(1)+c+1\,]-c\,\bigr\}\;T.
 \end{align*}
 Then Lemma~\ref{priori} completes the proof.
 
 \end{proof}
 
 We say that a curve $\ga:[0,T]\to M$ is a {\it Tonelli minimizer} if it minimizes the action functional
 on $\cC_T(\ga(0),\ga(T))$, i.e. if it is a minimizer with fixed endpoints and fixed time interval.

 \begin{Corollary}\label{Cbtm}
 There is $A>0$ such that if $x,\,y\in M$ and $\eta_n\in\cC_{T_n}(x,y)$, $n\in\na^+$ is a Tonelli minimizer
 with
 $$
 A_{L+c}(\eta_n)\le \Phi_c(x,y)+\tfrac1n,
 $$
 then there is $N_0>0$ such that $\forall n>N_0$, $\forall t\in[0,T_n]$, $|\deta_n(t)|<A$.
 \end{Corollary}

 \begin{proof}
 If $d(x,y)>0$ then for $n$ large enough $d(x,y)>\tfrac 1n$.
 In this case Corollary~\ref{CAPB} implies the result with the constant $A_1$.
 If $d(x,y)=0$ let $\xi_n:[0,T_n]\to\{x\}$ be the constant curve.
 Since $\eta_n$ is a Tonelli minimizer, we have that
 $$
 A_L(\eta_n)\le A_L(\xi_n) =\int_0^{T_n}L(x,0)\,dt\le |L(x,0)|\, T_n.
 $$
 Lemma~\ref{priori} implies that $|\deta_n|\le A_0(C)$ with $C=\sup_{x\in M}|L(x,0)|$.
 Now take $A=\max\{A_0(C),A_1\}$.
 \end{proof}

\begin{Lemma}\label{ALinv}\quad

If $(x,v)$ is a static vector then $\ga:\re\to M$, $\ga(t)=\pi\vr_t(x,v)$ 
is a static curve, i.e. the Aubry set $\cA(L)$ is invariant.
\end{Lemma}

 \begin{proof}\quad
 
 Let $\ga(t)=\pi\,\vr_t(x,v)$ and suppose that $\ga|_{[a,b]}$ is static.
 We have to prove that all $\ga|_\re$ is static. 
 Let
 $\eta_n\in\cC_{T_n}(\ga(b),\ga(a))$
 be a Tonelli minimizer  with
  $$
 A_{L+c}(\eta_n)<\Phi_c(\ga(b),\ga(a))+\tfrac 1n.
 $$
 
 By Corollary~\ref{Cbtm}, for $n$ large enough,
 $|\deta_n|<A$. We can assume that 
 $\deta_n(0)\to w$. Let $\xi(s)=\pi\,\vr_s(w)$. If $w\ne \dga(b)$
 then for some $\e>0$ the curve $\ga|_{[b-\e,b]}*\xi|_{[0,\e]}$ is not $C^1$, and
 hence  it
 can not be a (Tonelli) minimizer of $A_{L+c}$ in 
 $\cC_{2\e}\big(\ga(b-\e),\xi(\e)\big)$. Thus
 $$
 \Phi_c(\ga(b-\e),\xi(\e))
        < A_{L+c}(\ga|_{[b-\e,b]})
        + A_{L+c}(\xi|_{[0,\e]}).
 $$
 \begin{align*}
 &\Phi_c(\ga(a),\ga(a))
    \le \Phi_c(\ga(a),\ga(b-\e))
         +\Phi_c(\ga(b-\e),\xi(\e))
         +\Phi_c(\xi(\e),\ga(a))
    \\
    &\;< A_{L+c}(\ga_{[a,b-\e]})
       +A_{L+c}(\ga|_{[b-\e,b]})
       +A_{L+c}(\xi|_{[0,\e]})
       +\liminf_n A_{L+c}(\eta_n|_{[\e,T_n]})
    \\
    &\;\le A_{L+c}(\ga|_{[a,b]})
      +\lim_nA_{L+c}\bigl(\,\eta_n|_{[0,\e]}*\eta_n|_{[\e,T_n]}\bigr)       
    \\
    &\;= -\Phi_c(\ga(b),\ga(a))+\Phi_c(\ga(b),\ga(a))
    =0.
 \end{align*}
 Thus there is a closed curve, from $\ga(a)$ to itself, with negative
 $L+c$ action, and also negative $L+k$ action for some $k>c(L)$.  
 Concatenating the curve with itself many times  
 shows that $\Phi_k(\ga(a),\ga(a))=-\infty$.
 By~\eqref{defcL} this implies that $k\le c(L)$, which is 
 a contradiction.
 Thus $w=\dga(b)$ and similarly $\lim_n\deta_n(T_n)=\dga(a)$.
 
 If $\limsup T_n<+\infty$, we can assume that $\tau=\lim_n T_n>0$
 exists. In this case $\ga$ is a semi-static
 periodic orbit of period $\tau+b-a$ and then $\ga|_\re$  is static.
 
 Now suppose that $\lim_n T_n=+\infty$.
 If $s>0$, we have that
 \begin{align*}
 A_{L+c}&(\ga|_{[a-s,b+s]})
   +\Phi_c(\ga(b+s),\ga(a-s))\le
   \\
   &\le\lim_n\big\{\,A_{L+c}(\eta_n|_{[T_n-s,T_n]})
       +A_{L+c}(\ga|_{[a,b]})
       \begin{aligned}[t]
       &+A_{L+c}(\eta_n|_{[0,s]})\,\big\}\\
       &+\Phi_c(\ga(b+s),\ga(a-s))
       \end{aligned}
    \\
    &\le\begin{aligned}[t]
    &-\Phi_c(\ga(b),\ga(a)) \\
    &+\lim_n \big\{\,A_{L+c}(\eta_n|_{[0,s]})
                    +A_{L+c}(\eta_n|_{[s,T_n-s]})
                    +A_{L+c}(\eta_n|_{[T_n-s,T_n]})
             \,\big\}
    \end{aligned}         
    \\
    &\le -\Phi_c(\ga(b),\ga(a))+\Phi_c(\ga(b),\ga(a))              
    =0.
 \end{align*}
 Thus $\ga_{[a-s,b+s]}$ is static for all $s>0$.

\end{proof}

Let  $\cM_{\text{inv}}(L)$ be the set of Borel probabilities in $TM$ invariant under the 
 Lagrangian flow.
Denote by $\cM_{\min}(L)$ the set of minimizing measures for the Lagrangian $L$, i.e.
\begin{equation}\label{Mmin}
\cM_{\min}(L):=\Big\{\,\mu\in\cM_{\text{inv}}(L)\;\Big|\;
\int_{TM}L\,d\mu = -c(L)\;\Big\}.
\end{equation}
Their name is justified (cf. Ma\~n\'e \cite[Theorem II]{Ma7}) by 
\begin{equation}\label{minmeas}
-c(L) = \min_{\mu\in\cM_{\text{inv}}(L)}\int_{TM} L \; d\mu 
= \min_{\mu\in\cC(TM)}\int_{TM} L \; d\mu.
\end{equation}
 Fathi and Siconolfi \cite[Theorem 1.6]{FaSi} prove the second equality in \eqref{minmeas}
 where the set of  {\it closed measures} is defined by 
$$
{\mathcal C}(TM):=\Big\{ \, \mu \text{ Borel probability on }TM\; \Big\vert\;
\forall \phi\in C^1(M,\re) \; \int_{TM} d\phi\; d\mu =0\,\Big\}.
$$

Recall
\begin{Theorem}[Theorem~IV in \cite{Ma7} or~\cite{CDI}]\label{ecmm}\quad

A probability $\mu\in\cM_{\text{inv}}(L)$ is minimizing if and only if 
$\supp\mu\subset\cA(L)$.
\end{Theorem}

\begin{Corollary}\label{cecmm}\quad

A probability $\mu\in\cM_{\text{inv}}(L)$ is minimizing if and only if 
$\supp\mu\subset\mN(L)$.
\end{Corollary}

\begin{proof}
It is enough to prove that invariant probabilities supported in $\mN(L)$ are 
actually supported in $\cA(L)$.
Denote by $\vr_t=\vr^L_t$  the Lagrangian flow.
We first prove that the non-wandering set of the restriction $\vr_t|_{\mN(L)}$
satisfies $\Om(\vr_t|_{\mN(L)})\subset\cA(L)$. If $\vrt\in\Om(\vr^L_t|_{\mN(L)})$ then there is a sequence $\tt_n\in\mN(L)$ and $t_n\ge 2$ such that $\lim_n\tt_n=\vrt=\lim_n\vr_{t_n}(\tt_n)$. 
The action potential $\Phi_k$ in~\eqref{actionpotential} is Lipschitz by Theorem~I in Ma\~n\'e~\cite{Ma7} or~\cite{CDI}.
Then
\begin{align*}
A_{L+c}(\vr_{[0,1]}(\vrt))+\Phi_{c(L)}\big(\pi\vr_1(\vrt),\pi(\vrt)\big)
&\le \lim_n A_{L+c}(\vr_{[0,1]}(\tt_n))+ \lim_n A_{L+c}(\vr_{[1,t_n]}(\tt_n))
\\
&\le \lim_n A_{L+c}(\vr_{[0,t_n]}(\tt_n)) =\lim_n\Phi_{c(L)}(\pi\tt_n,\pi\vr_{t_n}(\tt_n))
\\
&\le \Phi_{c(L)}(\pi\vrt,\pi\vrt)=0.
\end{align*}
And hence $\vrt\in\cA(L)$.
Now, if $\mu\in\cM_{\text{inv}}(L)$ has $\supp\mu\subset \mN(L)$, by Poincar\'e recurrence theorem
$\supp\mu\subset \Om(\vr|_{\mN(L)})\subset\cA(L)$.

\end{proof}

\subsection{Symbolic Dynamics for the Aubry set.}\quad

Throughout the rest of the section we will identify a periodic orbit with the 
invariant probability supported on
the periodic orbit. 

The first two results, Lemma~\ref{Lper} and Proposition~\ref{Pper} follow
arguments by X. Bressaud and A. Quas~\cite{BQ}.

Let $A\in\{0,1\}^{M\times M}$ be a $M\times M$ matrix  with entries in $\{0,1\}$.
The subshift  of finite type $\Si_A$ associated to $A$ is the set
$$
\Si_A=\big\{\;\ox=(x_i)_{i\in\Z}\in \{1,\ldots,M\}^\Z \;\big|\quad \forall i\in\Z\quad A(x_i,x_{i+1})=1\;\big\},
$$
endowed with the metric
$$
d_a(\ox,\oy) = a^{-i}, \qquad i=\max\{\;k\in\na\;|\; x_i=y_i\;\;\forall |i|\le k\;\}
$$
for some $a>1$ and the {\it shift transformation} is
$$
\si:\Si_A\to\Si_A, \qquad \forall i\in\Z\quad \si(\ox)_i = x_{i+1}.
$$

\bigskip

\begin{Lemma}\label{Lper}
Let $\Si_A$ be a shift of finite type  with $M$ symbols and topological
entropy $h$. 
Then $\Si_A$ contains a periodic orbit of period at most $1+M \ee^{1-h}$.
\end{Lemma}

\begin{proof}
Let $k+1$ be the period of the shortest periodic orbit in $\Si_A$.
We claim that a word of length $k$ in $\Si_A$ is determined by the set
of symbols that it contains.
First note that since there are no periodic orbits of period $k$ or less,
any allowed $k$-word must contain $k$ distinct symbols.
Now suppose that $u$ and $v$ are two distinct words of length $k$ 
in $\Si_A$ containing the same symbols. Then, since the words are different, 
there is a consecutive pair of   symbols, say $a$ and $b$, in $v$ which 
occur in the opposite order (not necessarily consecutively) in $u$.
Then the infinite concatenation of the segment of $u$ starting at $b$ 
and ending at $a$  gives a word in $\Si_A$ of period at most $k$, which 
contradicts the choice of $k$.

It follows that there are at most $\binom Mk$ words of length $k$.
Using the basic properties of topological entropy
\begin{align*}
\ee^{hk}\le \binom Mk\le \frac {M^k}{k!}
\le \left(\frac {M \ee} k\right)^k.
\end{align*}
Taking $k$th roots, we see that $k\le M\ee^{1-h}$.

\end{proof}

In Appendix~\ref{ALM} we prove that if $\cA(L)$ is hyperbolic then there is a
hyperbolic set $\La$ in the energy level $c(L)$ which contains $\cA(L)$,
$\cA(L)\subset \La\subset E^{-1}\{c(L)\}$, and which has a Markov partition,
as defined in Appendix~\ref{AMP}. The Markov partition consists of a finite 
set of rectangles $\fT=\{T_i\}_{i=1}^r$ included in mutually disjoint transversal disks 
$D_i\supset T_i$. There is a Lipschitz first return time 
$\tau:\cup\fT\to]0,\a]$, $\cup\fT:=\cup_{T\in\fT}\, T$, 
$$
\tau(\tt):=\min\{\,t>0\;|\; \vr_t(\tt)\in\cup\fT\,\},
$$
and first return map (also called Poincar\'e map) $F:\cup\fT \to\cup\fT$,
$$
F(\tt):=\vr_{\tau(\tt)}(\tt).
$$
By Lemma~\ref{LD4} and Theorem~\ref{TB22}
there is a subshift of finite type $\Om$ and a Lipschitz map
$\Pi:\Om\to \cup\fT$ which is a semiconjugacy between
the shift map $\si$ and the Poincar\'e map $F$.
Also $\Pi$ extends to a time preserving Lipschitz semiconjugacy
$\Pi:S(\Om,\otau)\to \La$ from the suspension $S(\Om,\otau)$ 
of the shift with return time $\otau:=\tau\circ\Pi:\Om\to]0,\a]$, 
as defined in~\eqref{som} in Appendix~\ref{ASD}, with suspended flow $S_t$
defined in~\eqref{sust}, to the lagrangian flow $\vr_t|_\La$ on $\La$.
In other words, the following diagram commutes for every $t\in\re$ and
$\Pi$ is Lipschitz.
$$
\begin{CD}
S(\Om,\otau) @>S_t>> S(\Om,\otau)
\\
@V\Pi VV   @VV \Pi V
\\
\La @> \vr_t >> \La
\end{CD}
$$

A $\si$-invariant measure $\nu$ on $\Om$ induces a $\vr_t$-invariant
measure $\mu_\nu$ on $\La$ by
$$
\int_\La f\,d\mu_\nu :=
\int_\Om\left[\int_0^{\otau(\ow)} f\big(\vr_t(\Pi(\ow))\big)\,dt\right]d\nu(\ow).
$$
Define $B:\Om\to\re$ by 
$$
B(\ow):=\int_0^{\otau(\ow)} \Big[L\big(\vr_t(\Pi(\ow))\big)+c(L)\Big]\,dt.
$$
Then we have that
$$
A_{L+c(L)}(\mu_\nu)=\int_\La (L+c(L))\,d\mu_\nu =\int_\Om B\, d\nu.
$$
To fix notation we use the metric $d_a$ from~\eqref{da} on $\Om$, namely
\begin{equation}\label{da1}
d_a(\ou,\ow):=a^{-n},
\qquad  
n:=\max\{\,k\in\na\;:\;\forall|i|\le k,\quad u_i=w_i\,\},
\end{equation}
and $a>1$ is chosen as in Lemma~\ref{LD4}\eqref{LD4i} such that $\Pi:\Om\to \cup\fT$ 
is Lipschitz.

As constructed before Theorem~\ref{TB22}, the subshift $\Om$ has
symbols in $\fT$ and a transition matrix $A:\fT\times\fT\to\{0,1\}$ such that
$$
\Om =\Si(A)=\{\,\ow\in \fT^\Z\;:\; \forall i\in\Z\quad A(w_i,w_{i+1})=1\,\}.
$$
We say that $(w_1,\ldots,w_n)\in\fT^n$ is a {\it legal word in $\Om$} iff
$A(w_i,w_{i+1})=1$ for all $i=1,\ldots,n-1$.
A {\it cylinder} in $\Om$ is a set of the form
$$
C(n,m,\ow):=\{\,\oz\in\Om : z_i=w_i \quad \forall i=n,\ldots,m\,\},
$$
where $\ow\in\Om$ or $\ow=(w_n,\ldots,w_m)$ is a legal word in $\Om$.
A {\it subshift} of $\Om$ is  a closed $\si$-invariant subset of $\Om$.
If $Y\subset\Om$ is a subshift of $\Om$ we say that  
$\ow=(w_n,\ldots,w_m)$ is a legal word in $Y$ iff
 $C(n,m,\ow)\cap Y\ne\emptyset$. 
  
 Since the diagram
 $$
 \begin{CD}
 \Om @> \si >> \Om
 \\
 @V \Pi VV @VV \Pi V
 \\
 \cup\fT @> F >> \cup\fT
 \end{CD}
 $$
 commutes and $\Pi$ is continuous, the set 
 \begin{equation}
\label{subshY}
 Y:=\Pi^{-1}\big((\cup\fT)\cap\cA(L)\big)
  \end{equation}
  is a subshift of $\Om$ and
  \begin{equation}\label{YAL}
 \begin{CD}
 Y @> \si >> Y
 \\
 @V \Pi VV @VV \Pi V
 \\
 (\cup\fT)\cap\cA(L) @> F >> (\cup\fT)\cap \cA(L)
 \end{CD}  
  \end{equation}
 commutes.

\subsection{Approximation by periodic orbits.}\quad
\label{ssappperorb}

Let $\cP_L(T)$ be the set of  Borel invariant probabilities for $L$ which are
supported on a periodic orbit with period $\le T$. For $\mu\in\cP_L(T)$ 
write
$$
c(\mu,\cA(L)) :=\sup_{\theta\in\supp(\mu)} d(\theta,\cA(L)).
$$

\medskip    
    
\begin{Proposition}\label{Pper}
Suppose that the Aubry set $\cA(L)$ is hyperbolic.
Then for all $\ell\in \na^+$ 
$$
\liminf_{T\to+\infty} \;T^\ell \left(\inf_{\mu\in\cP_L(T)} c(\mu,\cA(L))\right) = 0.
$$
\end{Proposition}    

\begin{proof}\quad

For $n\in\na^+$ let $Z^{(n)}$ be the 1-step subshift of finite type
whose symbols are the legal words of size $n$ in $Y:=\Pi^{-1}(\cA(L))$.
A transition from the word 
\begin{equation}\label{concate}
\text{$u$ to $v$ is allowed in $Z^{(n)}$  iff the word $uv$
of size $2n$ is a legal word in $Y$.}
\end{equation}
 (Observe that this is not the standard
transition matrix for the $n$-word recoding of a subshift of finite type.) 
There is a natural semiconjugacy $Z^{(n)}\to Y$ from the shift of finite type 
$Z^{(n)}$ to $\si^n$ on $Y$, and hence the topological entropy of $Z^{(n)}$ 
exceeds $n \,h_{\text{top}}(Y)$. 
Let $h=\tfrac 1n h_{top}(Z^{(n)})\ge h_{top}(Y)$.
We have that 
\begin{align*}
\#\text{ symbols of $Z^{(n)}$}
= \#\text{ $n$-words in $Y$}
=\#\text{ $n$-cylinders in $Y$}
= K_n\, \ee^{n h_{top}(Y)}
\le K_n \,\ee^{nh},
\end{align*}
where $K_n$ has subexponential growth.
 By Lemma~\ref{Lper} there is a periodic orbit $\Ga_n$
in $Z^{(n)}$ with period at most $1+ K_n \ee^{nh} \ee^{1-nh}= 1+\ee\,K_n$.
Since each symbol in $Z^{(n)}$ corresponds to a word of length $n$ in the original
shift space $\Om=\Si(A)$, the periodic orbit $\Ga_n$ corresponds to a periodic orbit in $\Si(A)$
of period at most 
$$
P(n):=n\, (1+\ee\, K_n),
$$
which has subexponential growth.

We claim that any $n$-word in the projection of $\Ga_n$
 is a legal $n$-word in $Y$.
For this, observe that if the word is a symbol of $Z^{(n)}$ then it is a
legal $n$-word in $Y$ by the definition of $Z^{(n)}$. If the $n$-word is 
inside a concatenation of two symbols of $Z^{(n)}$, the transition 
rule~\eqref{concate} defining $Z^{(n)}$ implies that it is a legal 
$n$-word in $Y$. It follows, using the metric~\eqref{da1}, and identifying $\Ga_n$
with its projection to $\Si(A)$, that
\begin{equation}\label{dgay}
c(\Ga_n,Y):=\sup_{\ow\in\Ga_n}d_a(\ow,Y) \le a^{-n}.
\end{equation}

Recall that the return time $\tau$ to $\cup\fT$ and the ceiling function
in $S(\Om,\otau)$ are bounded above by $\a$.
Let $B>0$ be such that 
\begin{equation}\label{blipphi}
\tt_1,\tt_2\in\La, \quad  |t|\le \a 
\quad\then \quad d(\vr_t(\tt_1),\vr_t(\tt_2))< B\, d(\tt_1,\tt_2).
\end{equation}
Let $\De_n:=\vr_{\re}(\Pi(\Ga_n))=\vr_{[0,\a]}(\Pi(\Ga_n))
\subset \La=\Pi(S(\Om,\otau))$ 
be the periodic orbit in $\La$ which corresponds to $\Ga_n$.
Recall that $\Pi$ is Lipschitz with the metric~\eqref{da1}.
Since $\cA(L)=\vr_{[0,\a]}(\Pi(Y))$, from~\eqref{dgay} and~\eqref{blipphi}
we get that 
\begin{equation}\label{deDen}
c(\De_n,\cA(L))=\sup_{\tt\in\De_n}d(\tt,\cA(L))
\le B\, \Lip(\Pi)\, a^{-n}.
\end{equation}
The period $T(\De_n)$ of $\De_n$ is bounded by
\begin{equation}\label{TDen}
T(\De_n)\le \a\, P(n).
\end{equation}
Since $P(n)$ has subexponential growth from 
\eqref{TDen} and \eqref{deDen} we have that
\begin{align*}
\liminf_{T\to+\infty} \;T^\ell \Big[\inf_{\mu\in\cP_L(T)} c(\mu,\cA(L))\Big] 
\le \liminf_n T(\De_n)^\ell \, c(\De_n,\cA(L))
=0.
\end{align*}
\end{proof}

\medskip

\begin{Corollary}\label{cmn}
Suppose that the Aubry set $\cA(L)$ is hyperbolic. 
There is   a sequence
of periodic orbits $\mu_n$ with periods $T_n$ and $m_n\in\na^+$, $m_n>n$,
  such that 
for any $0<\be<1$  and any
$k\in\na^+$
\begin{equation}\label{ecmn}
\int d\big(\tt,\cA(L)\big)\,d\mu_n(\tt) = o(\be^{k\, m_n}) 
\qquad {\rm and} \qquad
\lim_n \frac{\log T_n}{m_n}=0.
\end{equation}
\end{Corollary}

\begin{proof}
Observe that it is enough to prove the Lemma for $k=1$.
By Proposition~\ref{Pper} there is a sequence of periodic orbits $\mu_n$ 
with periods $T_n\to\infty$ such that for any $\ell\in\na^+$
$$
\lim_n T_n^\ell \left(\int d(\tt,\cA(L))\; d\mu_n(\tt)\right) =0.
$$
Let 
$$
r_n:= \log_\be\left( \int d(\tt,\cA(L))\, d\mu_n(\tt)\right).
$$
Since 
$$
\be^{r_n}\le T^\ell_n \,\be^{r_n}\le 1
\qquad \Longleftrightarrow \qquad
-\frac 1{\ell}\le \frac{\log_\be T_n}{r_n} \le 0
$$
we have that $r_n^{-1}\log T_n \to0$.
Define $m_n:=\lfloor \tfrac 12 r_n\rfloor$, then $m_n^{-1}\log T_n \to0$ and
$$
\int d(\tt,\cA(L))\, d\mu_n(\tt) =\be^{r_n}\le \be^{m_n+\frac 12 r_n}=o(\be^{m_n}).
$$
\end{proof}

\medskip

 \begin{Lemma}\label{A.1}
  Let $a_1,\ldots, a_n$ be non-negative real numbers, and let $A=\sum_{i=1}^n a_i\ge 0$.
  Then
  $$
  \sum_{i=1}^n -a_i\log a_i \le 1 + A\,\log n,
  $$
  where we use the convention $0\, \log 0 = 0$.
  Moreover,
    $$
  \text{if  $A=1$ then \qquad}\sum_{i=1}^n -a_i\,\log a_i \le \log n.
  $$
  \end{Lemma}
  
  \begin{proof}
  Applying Jensen's inequality to the concave function $x\mapsto -x \log x$ yields
  $$
  \frac 1n \sum_{i=1}^n-a_i\log a_i 
  \le -\left(\frac 1n \sum_{i=1}^n a_i\right)
  \log\left(\frac 1n \sum_{i=1}^n a_i\right)
  =-\frac An\,\log A + \frac An\, \log n
  $$
  from which the result follows.
  When $A=1$ use that $1\cdot\log 1=0$ in the previous inequality.
 \end{proof}

 \medskip

 Recall that  $\cM_{\text{inv}}(L)$ is the set of Borel probabilities in $TM$ invariant under the 
 Lagrangian flow.

 \medskip                     
 
 \begin{Lemma}\label{LA2}\quad
 
 Let  $L$ be a Tonelli Lagrangian and $e>c(L_0)$.
  There is  $B=B(L,e)>0$ 
 such that for every 
 $\nu\in \cM_{\text{\rm inv}}(L)$ with $\supp(\nu)\subset [\text{\sl Energy}(L)<e]$,
 $$
 \int L\, d\nu \le -c(L) + B(L,e)\int d\big(\tt,\cA(L)\big)\; d\nu(\tt).
 $$
 \end{Lemma}
 
 \begin{proof}
 From Bernard~\cite{Be3} after Fathi and Siconolfi~\cite{FaSi} 
 we know that there is a critical
 subsolution $u$ of the Hamilton-Jacobi equation  for the Hamiltonian $H$ of $L$, i.e.
  \begin{equation}\label{sHJ}
 H(x,d_xu)\le c(L),
 \end{equation}
  which is
 $C^{1}$ with Lipschitz derivatives. Let 
 $$
 \L(x,v):= L(x,v)+c(L)-d_xu(v).
 $$
 Inequality~\eqref{sHJ} implies that $\L\ge 0$. 
 Also  from \eqref{aubry}, $\L|_{\cA(L)}\equiv 0$. 
 The Aubry set $\cA(L)$ is included in the 
 energy level $c(L)$ (e.g. Ma\~n\'e~\cite[p. 146]{Ma7}), hence it is compact.
There is a Lipschitz constant $B$ for the function $\L$ on the convex $[E(L)<e]$:
 $$
 \forall\tt\in [E(L)<e]\qquad \L(\tt) \le 0 + B\, d(\tt,\cA(L)) .
 $$ 
 By Birkhoff ergodic theorem every invariant probability is closed:
 \begin{align*}
 \int du(x,v) \;d\nu &
 = \int\left[ \lim_{T\to\infty}\frac 1T\int_0^T du(\vr_t(\tt))\; dt\right] d\nu(\tt)
 \\
 &=\int \lim_{T\to\infty}\left[\frac{u(\pi(\vr_T(\tt)))-u(\pi(\tt))}T\right] d\nu(\tt) =0.
 \end{align*}
 Therefore
 $$
 \int \big(L+c(L)\big) \;d\nu =\int\L\,d\nu \le B \int d(\tt,\cA(L))\;d\nu.
 $$
 \end{proof}

\bigskip

\medskip
\begin{Lemma}\label{partAB}\quad

Let $N$ be a compact riemannian manifold and $\mu$ a 
Borel probability on $N$.

Given $h>0$ there exists a finite Borel partition 
$\A=\{A_1,\ldots, A_r\}$ of $N$ with the following properties:
\begin{enumerate}
\item\label{partAB1} $ \diam\A< h$,
\item\label{partAB2} $\forall A\in\A\quad\mu(\partial A)=0$,
\item\label{partAB3} $\forall\e>0$ \; $\forall A_i\in\A$ \; $\exists B_i\subset A_i$ such that 
$B_i$ is compact, $\mu(\partial B_i)=0$, $\mu(A_i\setminus B_i)<\e$. 
\end{enumerate}

\begin{proof}
We first show that
for any $x\in N$ there is a ball $B(x,r)$, $r<\tfrac 12 h$ such that $\mu(\partial B(x,r))=0$.
Indeed if $h$ is small the sets $F_r:=\partial B(x,r)$, $0<r<h$ are disjoint. Since $\mu$ is
finite, at most a countable number of the sets $F_r$ can have positive measure.   
Let $\cO=\{U_i\}_{i=1}^m$ be a finite cover of $N$ by open balls with $\mu(\partial U_i)=0$, $\diam U_i<h$ and such that $U_i\setminus\cup_{j\ne i}\ov{U_j}\ne \emptyset$. Define inductively $A_1:=U_1$, $A_{i+1}:=U_{i+1}\setminus\cup_{j\le i}U_j$. Then 
$\ov{A_i}=\ov{\intt A_i}$ and 
$\A:=\{A_i\}_{i=1}^m$ is a Borel partition of $N$ 
satisfying~\eqref{partAB1} and~\eqref{partAB2}.

For $r>0$ small let 
$$
B_i(r):=\{\,x\in A_i\;:\; d(x,A_i^c)\ge r\,\}.
$$
We have that $B_i(r)$ is compact and $B_i(r)\uparrow \intt A_i$. Thus $\lim_{r\to 0}\mu(B_i(r))=\mu(\intt A_i)=\mu(A_i)$. Also $\partial B_i(r_1)\cap\partial B_i(r_2)
=\emptyset$ if $r_1\ne r_2$ because
$$
\partial B_i(r)=\{x\in A_i\;:\;d(x,A^c)=r\}.
$$ 
Therefore there is $r_i>0$ such that $B_i:=B_i(r_i)$ satisfies
$\mu(A_i\setminus B_i)<\e$ and $\mu(\partial B_i)=0$.

\end{proof}

\end{Lemma}

\bigskip
\subsection{Small entropy nearby closed orbits.}\quad
\label{sssenco}

Recall that $\cM_{\text{inv}}(L)$ is the set of  Borel probabilities in $TM$ which 
are invariant under the Lagrangian flow and $\cM_{\min}(L)$ is the set 
of minimizing measures~\eqref{Mmin}.
For $\mu\in\cM_{\text{inv}}(L)$ denote by $h(L,\mu)$ the entropy of $\mu$ under the Lagrangian flow $\vr_t$.

\medskip

\noindent{\bf Proof of Theorem~\ref{Tzeroentropy}:}

For $\ga>0$ write
\begin{align*}
\cH^k(L_0):&=\{\,\phi\in C^k(M,\re)\;|\; \mN(L_0+\phi) \text{ is hyperbolic }\},
\\
\cE_\ga:&=\{\,\phi\in \cH^k(L_0)\;|\; 
 \forall \mu\in\cM_{\min}(L_0+\phi)
\quad  h(L_0+\phi,\mu) < \ga  \;\}.
\end{align*}

The upper semicontinuity of the Ma\~n\'e set (cf. \cite[Lemma~5.2]{CP}) 
and properties of hyperbolic sets 
(see e.g. Theorem~\ref{SDL} in Appendix~\ref{ashms})
imply that the set $\cH^k(L_0)$ is open.
As an open subset $\cH^k(L_0)\subset C^k(M,\re)$ of a complete metric
space we have that $\cH^k(L_0)$ is a Baire space.

It is enough to prove that for every $\ga>0$ the set $\cE_\ga$ is  open and dense in $\cH^k(L_0)$ because in that case using  Corollary~\ref{cecmm} and the variational principle (Theorem~\ref{VarPrin}),
we have that
\begin{align*}
\bigcap\limits_{n\in\na^+}\cE_{\frac 1n}
&=\big\{\,\phi\in\cH^k(L_0)\;\big|\;\forall \mu\in
\cM_{\min}(L_0+\phi)
\quad 
h(L_0+\phi,\mu)=0\,\big\}
\\ 
&=\big\{\, \phi\in\cH^k(L_0)\;\big|\; \mN(L_0+\phi) \text{ has zero topological entropy}\,\big\}
\end{align*}
is a residual set.

\medskip

{\it Step 1.} {\sl $\cE_\ga$ is $C^k$ open.}

The map $C^k(M,\re)\ni\phi\mapsto\mN(L_0+\phi)$ is upper semicontinuous 
(cf. \cite[Lemma~5.2]{CP}).
Therefore, given $\phi_0\in\cH^k(L_0)$, there are neighbourhoods 
$U$ of $\mN(L_0+\phi_0)$ in $TM$
and $\cU$ of $\phi_0$ in $C^k(M,\re)$ 
such that for any $\phi\in\cU$, $\mN(L_0+\phi)\subset U$ and
$\mN(L_0+\phi)$ is a hyperbolic set for $\vr^{L_0+\phi}_t$ 
restricted to
$\E_\phi:=E^{-1}_{L+\phi}\{c(L+\phi)\}$.
Moreover, by Theorem~\ref{SDL} and Remark~\ref{rue}
we can choose $\cU$ and $U$ such that for every $\phi\in\cU$
the lagrangian flow $\vr^{L_0+\phi}_t$ of $L_0+\phi$ is hyperbolic in the maximal invariant
subset of $\ov U$ and by defiition \ref{hexpan}, the set of maps   
$\{\,\vr^{L_0+\phi}_1|_{\E_\phi\cap\ov U}\,:\,\phi\in\cU\,\}$ 
is a
uniformly h-expansive family on $\ov U$.
In particular $\cH^k(L_0)$ is open in $C^k(M,\re)$.

Consider the subset $\cS\subset C^k(M,\re)$ of potentials $\phi$
for which $\mN(L_0+\phi)$
contains a singular point for the Lagrangian flow. By \cite[Theorem~C]{CP}
there is an open and dense subset $\cO_1\subset \cS$ such that 
if $\phi\in\cO_1$ then $\mN(L_0+\phi)$ is a single singularity of $\vr^{L_0+\phi}$.
Then by \cite[Theorem~D]{ham} generically this singularity is hyperbolic,
i.e. $\cO_1\cap\cH^k(L_0)$ is open and dense in $\cO_1$.

So we can restrict our arguments to Ma\~n\'e sets without singularities.
In this case by Corollary~\ref{CSH1} we can identify the energy level
$\E_\phi\supset\mN(L_0+\phi)$ with the unit tangent bundle $SM$ by the
radial projection.

Suppose that $\phi_n \in \cH^k(L_0)\setminus \cE_\ga$, $\phi_0\in \cH^k(L_0)$
and $\lim_n\phi_n=\phi_0$ in $C^k(M,\re)$.  Then there are 
$\nu_n\in\cM_{\min}(L_0+\phi_n)$ with $h(L_0+\phi_n,\nu_n)\ge \ga$.
The map $\phi\mapsto c(L_0+\phi)$ is continuous (cf. \cite[Lemma~5.1]{CP})
and $\supp\nu_n$ is in the energy level 
$\E_{\phi_n}=E_{L_0+\phi_n}^{-1}\{c(L_0+\phi_n)\}$
(cf. Carneiro~\cite{Carneiro}). Thus we can assume that all the probabilities
$\nu_n$ are supported on a fixed compact subset $\K$ of $TM$.
 Taking a subsequence if necessary we can assume that 
 $\nu_n\to\nu\in \cM(L_0+\phi_0)$ in the weak* topology.

 The map $(\mu,\phi)\mapsto \int (L_0+\phi)\,d\mu$ is continuous 
 with respect to  $\lV \phi\rV_{\sup}$ and to the weak*
 topology on the set of Borel probabilities on $\K$.
 Also the map $\phi\mapsto c(L_0+\phi)$ is 
 continuous with respect to $\lV \phi\rV_{\sup}$ 
(cf. \cite[Lemma~5.1]{CP}). Using that [see eq.~\eqref{minmeas}]
$$
c(L_0+\phi) = -\min_{\mu\in\cM(L_0+\phi)}\int \bigl(L_0+\phi\big)\,d\mu,
$$
we obtain that the limit $\nu\in\cM_{\min}(L_0+\phi_0)$.

  We can identify the energy levels $\E_{\phi_n}$ with the unit 
tangent bundle $SM$ under the radial projection 
$R(\phi_n):\E_{\phi_n}\to SM$. 
Since $\phi\mapsto c(L_0+\phi)$ is continuous, we have that 
the projected lagrangian vector fields 
$R(\phi_n)_*X(L_0+\phi_n)|_{\E_{\phi_n}}\to R(\phi_0)_*X(L_0+\phi_0)\vert_{\E_{\phi_0}}$
converge in the $C^{k-1}$ topology on $SM$ (see Remark~\ref{remell}).

By Corollary~\ref{cecmm} we have that $\supp \nu_n\subset\ov U$.
By Remark~\ref{rue} the family of conjugated  lagrangian flows 
$\psi^n_t := R(\phi_n)\circ\vr^{L+\phi_n}_t\circ R(\phi_n)^{-1}$, $n\ge 0$
in $SM$ is uniformly h-expansive on their maximal invariant sets of
$R(\ov U)$.
Applying Theorem~\ref{ESS} we get that
$$
\ga\le \limsup_n h({L_0+\phi_n},\nu_n)\le h({L_0+\phi_0},\nu).
$$
Therefore $\cH^k(L_0)\setminus\cE_\ga$ is relatively closed in 
(the open set) $\cH^k(L_0)$
and hence $\cE_\ga$ is open in $C^k(M,\re)$.

\bigskip

{\it Step 2.} {\sl Density.}

We have to prove that $\cE_\ga$ intersects every non-empty
open subset of $\cH^k(L)$.
Let 
\linebreak
$\cU_1\subset \cH^k(L)$ be open and non-empty. 
 By Ma\~n\'e~\cite[Thm.~C.(a)]{Ma6} there is $\phi_0\in\cU_1$ such that  
 $\#\cM_{\min}(L_0+\phi_0)=1$.
 Write 
 \begin{equation}\label{LL0p0}
 L:=L_0+\phi_0.
 \end{equation}

Let $m_n$, $T_n$, $\mu_n$ be given by Corollary~\ref{cmn} for $L$.
Let $\Ga_n:=\supp\mu_n$ and let $\Ga_n(t)$ be the associated periodic orbit
with its parametrization. Corollary~\ref{cmn} and Lemma~\ref{LA2} are proven
for a single lagrangian. The reader can check that in the following proof they
are only applied to the lagrangian $L$ in~\eqref{LL0p0}.

 For $\phi\in C^k(M,\re)$ near $0$ write
$\E_\phi:=E_{L+\phi}^{-1}\{c(L+\phi)\}$.

\refstepcounter{Thm}
\begin{Claim}\label{c1}\quad

There are $0<\be<1$, 
$N^1_\ga>0$, $C(\cU_1,\be)>0$
 and a neighbourhood 
$0\in\cU_2\subset\cU_1-\phi_0$ such that if 
$n>N^1_\ga$, $\phi\in\cU_2$, 
$\mu\in\cM_{\min}(L+\phi)$,
$\Ga_n(t)$ is also a periodic orbit for $L+\phi$ and
$h(\mu)> 3 \ga \;C(\cU_1,\be)$ then 
\begin{equation*}\label{mu>ga}
\mu(\{\,\tt\in \E_\phi \;|\; d(\tt,\Ga_n)\ge \be^{m_n}\,\})>\ga.
\end{equation*}
\end{Claim}

\noindent{\it Proof  Claim~\ref{c1}:}

 The hyperbolicity of $\mN(L)$ and the upper semicontinuity of 
 the Ma\~n\'e set implies that there is $h>0$ and an open subset
  $0\in\cU_{12}\subset\cU_1-\phi_0$ such that $h$ is a uniform 
 h-expansivity constant (cf. Definition~\ref{duhe}, Remark~\ref{rue}) 
 for all $\vr^{L+\phi}|_{\mN(L+\phi)}$, with $\phi\in\cU_{12}$.
 
  Using Corollary~\ref{CSH1} identify the energy levels $\E_\phi:=E_{L+\phi}^{-1}\{c(L+\phi)\}$ 
  with the unit tangent bundle $SM$ using the radial projection $R(v) = \tfrac{v}{|v|}$.
 Let $\mu_0$ be the minimizing measure for $L$:
  \begin{equation}\label{uniqmu0}
  \cM_{\min}(L)=\{\mu_0\}.
  \end{equation}
  Let $\A=\{A_1,\ldots,A_r\}$ be a finite Borel partition of the energy level
  $\E_\phi$
  with 
  \begin{equation}\label{diamah}
  \diam\A<h
  \end{equation}
  and $\mu_0(\partial A)=0$ for all $A\in\A$.
  Let 
  \begin{equation}\label{defer}
  0<\e<  \tfrac 12 (r\,\log r)^{-1},
  \end{equation}
  where $r=\#\A$.
  Using Lemma~\ref{partAB}, 
 for each $A_i\in \A$ let $B_i\subset A_i$ be a compact set such that 
 $\mu_0(A_i\setminus B_i)<\e$ and $\mu_0(\partial B_i)=0$. 
 Define the partition $\B:=\{B_0, B_1,\cdots,B_r\}$,
 where $B_0:=\E_\phi\setminus \bigcup_{i=1}^r B_i$.

  By the continuity of the critical value map $\phi\mapsto c(L+\phi)$ and 
  the ergodic characterization~\eqref{minmeas} we have that any weak* limit
  of minimizing measures is a minimizing measure.
  By the uniqueness~\eqref{uniqmu0} and the compactness of $\E_\phi$,
  if $\lim_k\phi_k= 0$ and $\nu_k\in\cM_{\min}(L+\phi_k)$ then 
  $\lim_k\nu_k=\mu_0$. Since $\mu_0(\partial A_i)=0=\mu_0(\partial B_i)$
  and $\partial(A_i\setminus B_i)\subset \partial A_i\cup \partial B_i$ 
  we have that
  $$
  \lim_k\phi_k=0,\; \nu_k\in\cM_{\min}(L+\phi_k) 
  \qquad \then \qquad\forall i \quad 
   \lim_k\nu_k(A_i\setminus B_i) =\mu_0(A_i\setminus B_i).
  $$
  Then there is an open set $0\in\cU_{13}\subset \cU_{12}$ such that 
  \begin{equation}\label{cU13}
  \forall \phi\in\cU_{13} \quad\forall \mu\in\cM_{\min}(L+\phi)\quad
  \forall i\le r\qquad
   \mu(A_i\setminus B_i)<\e.
  \end{equation}

  Let $\bO:=\{B_0\cup B_1,\ldots, B_0\cup B_r\}$. Observe that $\bO$ 
is an open cover because 
\linebreak
$B_0\cup B_i=(\cup_{j\ne i} B_i)^c$.
\begin{equation}\label{defde}
\text{
Let $\de>0$ be a Lebesgue number for $\bO$.
}
\end{equation}
Let $\be>0$ be such that 
\begin{equation}\label{defbe1}
0<\be<\min\Big\{\,\frac \de 2,\;
\inf_{\phi\in\cU_1}\Lip(\vr^{L+\phi}_1|_{\E_\phi})^{-1}\Big\}
\end{equation}
and such that 
\begin{equation}\label{defbe2}
\sup_{\phi\in\cU_1}\sup_{|\tau|\le\be}
\sup_{\tt\in\E_\phi}d\big(\vr^{L+\phi}_\tau(\tt),\tt\big)
<\frac \de 2.
\end{equation}

Given $n\in\na$ and $\a>0$ let $G(n,\a)$ be a cover of 
$\E_\phi$ of minimal cardinality  by balls of radius $\a^n$. Let
\begin{equation}\label{coverG}
C(\cU_1,\a):=\limsup_n \tfrac 1n\log\# G(n,\a).
\end{equation}
Then (cf. Falconer~\cite[Prop.~3.2]{Falconer0})
$$
1\le \lim_{\a\to 0} C(\cU_1,\a) = \dim \E_\phi<\infty.
$$
Shrink $\be$ if necessary so that $\be$ satisfies \eqref{defbe1}, \eqref{defbe2} and
\begin{equation*}\label{defbe3}
\tfrac 12\le C(\cU_1,\be) < \infty.
\end{equation*}

Let $Q\in\na^+$ be such that 
\begin{equation}\label{q2l2}
\frac{3+2\log 2}Q < \tfrac 14\,\ga\, C(\cU_1,\be).
\end{equation}
Let 
\begin{equation}\label{defnga}
N_\ga^1> 3Q>0
\end{equation}
 be such that 
\begin{align}
\forall n>N_\ga^1\qquad
\tfrac 1n \log\#G(n,\be) &\le 2 \, C(\cU_1,\be)
\label{ngagn}
\\
\intertext{and, using \eqref{ecmn},}
\forall n> N_\ga^1\qquad
\tfrac1{m_n}\log T_n -\tfrac 1{m_n}\log\be 
&< 
\tfrac 14\,\ga\,C(\cU_1,\be).
\label{ngatn}
\end{align}
Suppose that  $\phi\in\cU_{13}$, 
$\mu\in\cM_{\min}(L+\phi)$, $\ga>0$ and $n>N_\ga^1$ satisfy
\begin{equation}\label{mulega1}
\mu(\{\,\tt\in \E_\phi \;|\; d(\tt,\Ga_n) \ge \be^{m_n}\,\})\le\ga,
\end{equation}
and that $\Ga_n$ is also a periodic orbit for $\vr^{L+\phi}$,
we shall prove that then 
\begin{equation}\label{hmulphi3gac}
h_\mu(\vr^{L+\phi})\le3\,\ga\, C(\cU_1,\be).
\end{equation}
This  implies Claim~\ref{c1}.

 Observe that 
 $$
 A_i\cap B_j=
 \begin{cases}
 \emptyset &\text{if }\quad  i\ne j\ne 0,\\
 B_j &\text{if } \quad i=j\ne 0,\\
 A_i\setminus B_i &\text{if }\quad j=0.
 \end{cases}
 $$
 Let $\rho(x):=-x\log x$, $x\in[0,1]$, with $\rho(0):=0$. Then
 $$
 \rho\left(\frac{\mu(A_i\cap B_j)}{\mu(B_j)}\right)
 =\begin{cases}
 0 &\text{ if }\quad j\ne 0, \\
 \rho\left(\frac{\mu(A_i\setminus B_i)}{\mu(B_0)}\right)
 &\text{ if }\quad j=0.
 \end{cases}
 $$
 Observe that $B_0=\cup_{i=1}^r(A_i\setminus B_i)$, then
 from~\eqref{cU13}, 
 \begin{equation}\label{mub0e}
 \mu(B_0)< r\,\e.
 \end{equation}

 We have that the relative entropy satisfies
 \begin{align*}
 H_\mu(\A|\B):&=-\sum_{i=1}^r\sum_{j=0}^r \mu(A_i\cap B_j)
  \;\log\frac{\mu(A_i\cap B_j)}{\mu(B_j)}
  \\
  &= \sum_{i=1}^r\sum_{j=0}^r \mu(B_j) \,
  \rho\left(\frac{\mu(A_i\cap B_j)}{\mu(B_j)}\right)
  \\
  &=\sum_{i=1}^r \mu(B_0)\;\rho\left(\frac{\mu(A_i\cap B_0)}{\mu(B_0)}\right)
  \\
  &\le \mu(B_0) \log r \qquad \text{by Lemma~\ref{A.1} with A=1},
  \\
  &\le \e\, r\log r  < 1     \qquad\text{ using \eqref{mub0e} and \eqref{defer}.}
 \end{align*}
  From Walters \cite[Theorem~4.12(iv)]{Walters} for all 
  $f:\E_\phi\to\E_\phi$ continuous and $Q\in\na$ we have that
  \begin{align}\label{h+1}
  h_\mu(f^Q,\A)\le h_\mu(f^Q,\B)+H_\mu(\A|\B)
  \le  h_\mu(f^Q,\B) +1.
  \end{align}

 Define
 $$
 B(\Ga_n,\be^{m_n}):=\{\,\tt\in\E_\phi\;|\;d(\tt,\Ga_n)<\be^{m_n}\,\}.
 $$
 Let $f$ be the time 1 map  $f:=\vr^{L+\phi}_1$.
 Fix $\tt_0\in\Ga_n$ and let
 $$
 R_n:=\big\{\,\vr_{i\be}(\tt_0)\;\big|\; i=0, 1,\ldots,\lfloor\tfrac{T_n}\be\rfloor\,\big\}.
 $$
 We claim that $R_n$ is an $(m_n,\de,f)$-generating set for $B(\Ga_n,\be^{m_n})$,
 i.e.
 \begin{equation}\label{rngenerating}
 B(\Ga_n,\be^{m_n})\subset \bigcup_{p\in R_n}V(p,m_n,\de,f),
 \end{equation}
 where
 $$
 V(p,m_n,\de,f):=\{\,\tt\in\E_\phi\;|\; d(f^i(\tt),f^i(p))<\de, 
 \quad\forall i=0,\ldots, m_n-1\,\}.
 $$
 Indeed if $q\in B(\Ga_n,\be^{m_n})$ there is $\tt\in\Ga_n$ such that 
 $d(q,\tt)<\be^{m_n}$. 
 Recall that by hypothesis $\Ga_n(t)$ is a periodic orbit
 of $\vr^{L+\phi}$. There is $p\in R_n$ such that $p=\vr^{L+\phi}_\tau(\tt)$
 with $|\tau|\le \be$. 
  In particular
 $$
 f^j(p)=\vr^{L+\phi}_j(\vr_\tau^{L+\phi}(\tt))=\vr_\tau^{L+\phi}(f^j(\tt)).
 $$
 If $0\le j\le m_n-1$, by~\eqref{defbe1} and~\eqref{defbe2}
  we have that 
\begin{align*}
d\big(f^j(q),f^j(\tt)\big) 
&\le \Lip(f)^j\, d(q,\tt)\le \Lip(f)^j\,\be^{m_n}\le \be <\de/2,
\\
d\big(f^j(\tt),f^j(p)\big)
&= d\big(f^j(\tt),\vr_\tau^{L+\phi}(f^j(\tt))\big)
\le d_{C^0}(\vr^{L+\phi}_\tau,id)< \de/2,
\\
d\big(f^j(q),f^j(p)\big)
&\le d\big(f^j(q),f^j(\tt)\big) +d\big(f^j(\tt),f^j(p)\big)
< \de/2+\de/2=\de.
\end{align*}
Therefore $q\in V(p,m_n,\de,f)$ and $p\in R_n$.

  Recall that $Q\in\na^+$ is from~\eqref{q2l2}. Write
 \begin{align*}
 \B_{f^Q}^{(k)}:&=\textstyle\bigvee\limits_{i=0}^{k-1} f^{-iQ}(\B),
 \\
 \bO_{f^Q}^{(k)}:
 &=\big\{\textstyle\bigcap_{i=0}^{k-1}f^{-iQ}(U_i)\;
 \big|\; U_i\in \bO \quad \forall i=0,\ldots,k-1\,\big\}.
 \end{align*}
 Let 
 \begin{equation}\label{defW}
 W(n,k,Q):=\big\{\, \cB\in \B^{(k)}_{f^Q}
 \;\big |\; \cB\cap B(\Ga_n,\be^{m_n})\ne\emptyset\,\big\}.
 \end{equation}
 Let $k_n=k(n,Q)\in\na$ be such that
 \begin{equation}\label{kn1}
 (k_n-1) \, Q < m_n \le k_n \,Q.
 \end{equation}
 Since in~\eqref{ecmn} $m_n>n$, 
 by~\eqref{defnga} we have that $k_n\ge 2$ whenever $n>N_\ga^1$.
 We want to estimate $\#W(n,k_n,Q)$.
 Given $\cB\in W(n,k_n,Q)$, choose $q(\cB)\in \cB\cap B(\Ga_n,\be^{m_n})$.
 By the inclusion~\eqref{rngenerating} there is $p(\cB)\in R_n$ such that 
 $$
 q(\cB)\in V(p(\cB),m_n,\de,f)\subset V(p(\cB),k_n,\de,f^Q).
 $$
 In particular $d\big(f^{Qi}(q(\cB)),f^{Qi}(p(\cB))\big)<\de$ for $0\le i<k_n$.
 The choice of $\de$ in \eqref{defde}
  implies that there is $j_i=j_i(p)\in\{1,\ldots,r\}$, depending only on $(i,p)$,
   such that 
 $$
 f^{Qi}(q(\cB))\in B\big(f^{Qi}(p(\cB)),\de\big)
 \subset B_0\cup B_{j_i},
 \qquad i=0,\ldots,k_n-1.
 $$
 Therefore 
 \begin{align*}
 q(\cB)&\in \bigcap_{i=0}^{k_n-1}f^{-Qi}B\big(f^{Qi}(p(\cB)),\de\big)\subset
 \\
 &\qquad\subset \bigcap_{i=0}^{k_n-1}f^{-Qi}(B_0\cup B_{j_i})
 =\bigcap_{i=0}^{k_n-1}\Big(f^{-Qi}(B_0)\cup f^{-Qi}(B_{j_i})\Big).
 \end{align*}
 Since $\B$ is a partition, $\cB\in \B^{(k_n)}_{f^Q}$ 
 and $q(\cB)\in\cB$
 we have that 
 \begin{equation}\label{cbfq}
 \cB=\bigcap_{i=0}^{k_n-1} f^{-Qi}(B_{\ell_i}),
 \qquad \text{ where } \quad \ell_i\in\{0,j_i(p)\}.
 \end{equation}
 Observe that $j_i(p)$ depends only on $i$ and $p$.
Thus for $p\in R_n$  we have that 
 $$
 \#\big\{\,\cB\in W(n,k_n,Q)\subset \B^{(k_n)}_{f^Q}
 \;\big|\; p(\cB) =p\big\}\le 2^{k_n}.
 $$
 Therefore
 \begin{equation}\label{Wnknq}
  \# W(n,k_n,Q)\le 2^{k_n} \cdot \# R_n\le 2^{k_n}\frac{T_n}\be.
 \end{equation}

 By~\eqref{defbe1} we have that 
$B(\tt,\be^n)\subset V(\tt,n,\de,f)$. Therefore
\begin{equation}\label{BVVt}
\forall\tt\in\E_\phi\qquad 
B(\tt,\be^{m_n})\subset V(\tt,m_n,\de,f)\subset V(\tt,k_n,\de,f^Q).
\end{equation}
Identify the covering $G(m_n,\be)$ from~\eqref{coverG}
by balls of radius $\be^{m_n}$
with the set of their centers.
Then  by~\eqref{BVVt}, the set $G(m_n,\be)$ is a $(k_n,\de,f^Q)$-generating set, i.e.
$$
\E_\phi\subset \bigcup_{\tt\in G(m_n,\be)}V(\tt,k_n,\de,f^Q).
$$

We show that the same argument as in~\eqref{Wnknq} gives
for $n>N^1_\ga$ that
 \begin{equation}\label{B2kG}
  \#\B^{(k_n)}_{f^Q}\le 2^{k_n} \,\#G(m_n,\be).
 \end{equation}
 Namely, we construct a map $p: \#\B^{(k_n)}_{f^Q}\to G(m_n,\be)$
 which is at most $2^{k_n}$ to 1 as follows.
 Given $\cB\in  \#\B^{(k_n)}_{f^Q}$ choose $q(\cB)\in \cB$.
 Now choose $p(\cB)\in G(m_n,\be)$ such that 
 $$
 q(\cB)\in V(p(\cB),k_n,\de,f^Q).
 $$
 From~\eqref{defde} for all $i=0,\ldots, k_n-1$
  there is $j_i(p)$ such that $B\big(f^{iQ}(p(\cB)),\de\big)\subset B_0\cup B_{j_i(p)}$
  and therefore
  $$
  q(\cB)\in
   \bigcap\limits_{i=0}^{k_n-1}
   f^{-iQ}\big(B(f^{iQ} (p(\cB)),\de)\big)
   \subset\bigcap\limits_{i=0}^{k_n-1}
   f^{-iQ}(B_0\cup B_{j_i(p)}).
  $$
  This implies that 
  $$
  \cB=\bigcap\limits_{i=0}^{k_n-1}
   f^{-iQ}(B_{\ell_i})
   \quad\text{ with }\quad
   \ell_i\in\{0,\,j_i(p(\cB))\},
  $$
  which in turn implies~\eqref{B2kG}.
 
 By the hypothesis~\eqref{mulega1} and~\eqref{defW}
 we have that
 $$
 \tga_n:=\sum_{\cB\in\B^{(k_n)}_{fQ}\setminus W(n,k_n,Q)}\mu(\cB)\le \ga.
 $$
 Using Lemma~\ref{A.1} and inequalities
 \eqref{Wnknq} and \eqref{B2kG} we have that
 \begin{align*}
 h_\mu(f^Q,\B)
 &\le\frac 1{k_n}H_\mu(\B^{k_n}_{f^Q})
 \qquad\quad \text{ by~\eqref{HfNdec}  (cf.  Walters~\cite[Theorem~4.10]{Walters})},
 \\
 &\le \frac 1{k_n}\sum_{\cB\in W(n,k_n,Q)}\hskip -.5cm-\mu(\cB)\,\log\mu(\cB)
 +  \frac 1{k_n}\sum_{\cB\in \B^{(k_n)}_{f^Q}\setminus W(n,k_n,Q)}
 \hskip-.9cm -\mu(\cB)\,\log\mu(\cB)
 \\
 &\le \frac 1{k_n}\big(1+(1-\tga_n) \log\# W(n,k_n,Q)\big)
 +\frac 1{k_n} \big( 1 +\ga \, \log\#\B^{(k_n)}_{f^Q}\big)
 \\
 &\le(1+\ga) \log 2 + \tfrac 1{k_n}(2+\log T_n-\log\be)
+\tfrac 1{k_n} \,\ga\, \log\# G(m_n,\be).
 \end{align*}
 Thus, using~\eqref{h+1}, and then~\eqref{ngagn} and that $m_n>n$
 we have that for $n>N_\ga^1$,
 \begin{align}
 h_\mu(f^Q,\A) &\le 1+h_\mu(f^Q,\B) 
 \notag \\
 &\le 1+2\log 2+\tfrac 1{k_n}(2+\log T_n-\log\be)
 +\tfrac {m_n}{k_n}\, 2\,\ga\, C(\cU_1,\be).
 \label{hfqa}
 \end{align}
 The hyperbolicity of $\mN(L+\phi)$ implies that
  $f^Q$ has also h-expansivity constant $h$ on $\mN(L+\phi)$ 
  because for $\tt\in\mN(L+\phi)$ there is $\tau=\tau(\tt)>0$ such that
  $$
  \Ga_h(\tt,f)=\Ga_h(\tt,f^Q)=\vr_{[-\tau,\tau]}(\tt),
  $$
  where
 $$
 \Ga_h(\tt,f^Q):=\{\vrt\in \E_\phi\;|\;\forall n\in\Z\quad d(f^{nQ}(\tt),f^{nQ}(\vrt))<h\}.
 $$
 In particular $h_{top}(\Ga_h(\tt,f^Q))=h_{top}(\Ga_h(\tt,f))=0$.
 By Corollary~\ref{cecmm},
  $\supp(\mu)\subset\mN(L+\phi)$. Therefore $h$ is an h-expansivity 
  constant on $\supp(\mu)$. Since by~\eqref{diamah}
  $\diam \A<h$, by
 Theorem~\ref{TBhE} (cf. Bowen~{\cite[Theorem~3.5]{Bowen9}})
 we have that
 \begin{align*}
 h_\mu(f^Q)= h_\mu(f^Q,\A).
 \end{align*}
 By~\eqref{kn1}, $k_n\ge\tfrac{m_n}Q$,
  \begin{align*}
 h_\mu(f)&=\tfrac 1Qh_\mu(f^Q)=\tfrac 1Q h_\mu(f^Q,\A)
 \\
 &\le \tfrac1Q(3+2\log 2) +\tfrac 1{m_n}{\log T_n}-\tfrac 1{m_n}\log\be
 +2\,\ga\,C(\cU_1,\be).
 \end{align*}
 Using~\eqref{q2l2} and \eqref{ngatn} we obtain
 $$
 h_\mu(f)\le 3\,\ga\, C(\cU_1,\be).
 $$
 This proves inequality~\eqref{hmulphi3gac}, and also
 Claim~\ref{c1}.

 \hfill $\triangle$

\pagebreak

\bigskip

\begin{Claim}\label{c2}\quad

There are $C>0$, $N^2_\ga>N^1_\ga>0$ 
such that if $n>N^2_\ga$,
$\phi\in\cU_2$ are such that
$$
\phi \ge 0 \qquad\text{ and } \qquad \phi|_{\pi\Ga_n}\equiv 0,
$$
and $\nu\in\cM_{\min}(L+\phi)$,
$\tt\in \Ga_n$, $\vrt\in\supp(\nu)$,
$d(\tt,\vrt) \ge\be^{m_n}$, then
$$
d\big(\pi(\tt),\pi(\vrt)\big) \ge \tfrac 1C\, \be^{m_n}.
$$
\end{Claim}

Denote the action of a $C^1$ curve $\a:[S,T]\to\re$ by
$$
A_L(\a):=\int_S^T L(\a(t),\dot\a(t))\; dt.
$$
The following Crossing Lemma is extracted for Mather~\cite{Mat5} with 
the observation that the estimates can be taken uniformly on $\cU_2$.

\begin{Lemma}[Mather~{\cite[p. 186]{Mat5}}]\label{CL}\quad

If $K>0$, then there exist $\e$, $\de$, $\eta>0$ and 
\begin{equation}\label{CL>1}
C> 1,
\end{equation}
such that 
if $\phi\in\cU_2$, and 
$\a, \ga:[t_0-\e,t_0+\e]\to M$
are solutions of the Euler-Lagrange equation
with $\lV d\a(t_0)\rV$, $\lV d\ga(t_0)\rV\le  K$,
$d\big(\a(t_0),\ga(t_0)\big)\le\de$, and
$$d\big(d\a(t_0),d\ga(t_0)\big)\ge C\; d\big(\a(t_0),\ga(t_0)\big),$$
then there exist $C^1$ curves $a,c:[t_0-\e,t_0+\e]\to M$
such that $a(t_0-\e)=\a(t_0-\e)$, $a(t_0+\e)=\ga(t_0+\e)$,
$c(t_0-\e)=\ga(t_0-\e)$, $c(t_0+\e)=\a(t_0+\e)$, and
\begin{equation}\label{ecl}
A_{L+\phi}(\a)+A_{L+\phi}(\ga)
-A_{L+\phi}(a)-A_{L+\phi}(c)
\ge \eta\; d\big(d\a(t_0),d\ga(t_0)\big)^2.
\end{equation}

\end{Lemma}

\bigskip

\noindent{\it Proof of Claim~\ref{c2}:}
Let $K>0$ be such that
$$
\forall \phi\in\cU_2\qquad
E_{L+\phi}^{-1}\{c(L+\phi)\} \subset \{\,\xi\in TM\;|\; \lV \xi\rV<K\,\}.
$$
Let $\e$, $\de$, $\eta$, $C>0$ be from Lemma~\ref{CL}.
Choose $M_\ga>N^1_\ga$ such that if $n>M_\ga$ then
\begin{equation}\label{1Cd}
\tfrac 1C\be^{m_n} <\de.
\end{equation}
Suppose by contradiction that there are $\phi\in\cU_2$, 
$\nu\in\cM_{\min}(L+\phi)$,
$\tt\in\Ga_n$, $\vrt\in\supp(\nu)$
such that 
$\phi\ge 0$, 
$\phi|_{\Ga_n}\equiv 0$,
 $$
 d(\tt,\vrt)\ge \be^{m_n}
 \qquad\text{ and }\qquad
 d\big(\pi(\tt),\pi(\vrt)\big) < \tfrac 1C\, \be^{m_n}.
$$
Since $\phi$ is $C^2$, $\phi\ge 0$ and $\phi|_{\pi\Ga_n}\equiv 0$
we have that $d\phi|_{\pi\Ga_n}\equiv 0$ and then $\Ga_n(t)$
is also a periodic orbit for $L+\phi$, with the same parametrization.
Write  $\a(t)=\pi\vr^{L+\phi}_t(\vrt)$, $\ga(t)=\pi\vr^{L+\phi}_t(\tt)$, $t\in[-\e,\e]$.
By~\eqref{1Cd} we can apply 
Lemma~\ref{CL} 
and obtain $a,\,c:[-\e,\e]\to M$ satisfying \eqref{ecl}.
Since $\vrt\in\supp(\nu)\subset\cA(L+\phi)$ the segment $\a$ is semi-static
for $L+\phi$:
\begin{align}
&\hskip -.7cmA_{L+\phi+c(L+\phi)}(\a) = \Phi^{L+\phi}_{c(L+\phi)}(\a(-\e),\a(\e))
\notag \\
&=\inf\{\,A_{L+\phi+c(L+\phi)}(x)\;|\;x\in C^1([0,T], M),\;
 T>0, \;
x(0)=\a(-\e), \; x(T)=\a(\e)\;\}.
\label{1semi} 
\end{align}
Consider the curve 
$x = a *\pi \vr^{L+\phi}_{[\e,T_n-\e]}(\tt) *c$ joining $\a(-\e)$ to $\a(\e)$.
Writing
$$
L_\phi:=L+\phi,
$$
we have that
\begin{align}
A_{L_\phi+c(L_\phi)}(x) 
&= A_{L_\phi+c(L_\phi)}(a) + A_{L_\phi+c(L_\phi)}(\Ga_n)-A_{L_\phi+c(L_\phi)}(\ga)
+A_{L_\phi+c(L_\phi)}(c)
\notag\\
& \le A_{L_\phi+c(L_\phi)}(\a) -\eta\, d(\tt,\vrt)^2 + A_{L_\phi+c(L_\phi)}(\Ga_n)
\notag\\
&\le A_{L_\phi+c(L_\phi)}(\a) -\eta\, \be^{2 \, m_n} + A_{L_\phi+c(L_\phi)}(\Ga_n).
\label{Ax}
\end{align}

Now we estimate $A_{L_\phi+c(L_\phi)}(\Ga_n)$. 
Observe that since $L_\phi=L+\phi \ge L$ we have that
$$
c(L+\phi)\le c(L).
$$
Since $\phi\vert_{\Ga_n}\equiv 0$, we have that
$\phi |_{\Ga_n}+ c(L+\phi) \vert_{\Ga_n} \le c(L)$.
Observe that $\mu_n$ is an invariant measure for the
flows of both $L$ and $L+\phi$ with $\supp\mu_n=\Ga_n$,
thus 
\begin{align}
A_{L_\phi+c(L_\phi)}(\Ga_n)
&= T_n\int \big[ L+\phi+c(L+\phi)\big]\; d\mu_n
\notag\\
&\le T_n\int \big[L+c(L)\big]\, d\mu_n
=A_{L+c(L)}(\Ga_n).
\label{LphiL}
\end{align}
By the choice of $m_n$, $T_n$, $\mu_n$ from Corollary~\ref{cmn} 
for $L$ 
and using Lemma~\ref{LA2}  applied to $L$, 
we have that for any $\rho>0$, setting $e:=c(L)+1$,
\begin{align}
A_{L+c(L)}(\Ga_n) &= T_n \int \big[L+c(L)\big]\, d\mu_n
\le B(L,e)\; T_n \int d(\xi,\cA(L)) \,d\mu_n(\xi)
\notag\\
&\le  B\,T_n\cdot o(\be^{k\, m_n}) 
=o\big(\be^{(k-\rho)\,m_n}\big).
\label{AGa}
\end{align}
In Corollary~\ref{cmn} the estimate holds with the same sequence $\mu_n$ for any $k\in\na^+$.
Thus, choosing $k$, $\rho$ such that  $k-\rho\ge2$, if $n$ is large enough, 
from \eqref{Ax}, \eqref{LphiL} and \eqref{AGa}, we get
\begin{align*}
A_{L_\phi+c(L_\phi)}(x) &\le
A_{L_\phi+c(L_\phi)}(\a) -\eta\, \be^{2 m_n} + o(\be^{2 m_n})
\\
&< A_{L_\phi+c(L_\phi)}(\a).
\end{align*}
This contradicts \eqref{1semi} and  proves Claim \ref{c2}. 

\hfill $\triangle$

\pagebreak

\bigskip

We use $L$ from~\eqref{LL0p0} and 
$m_n$, $T_n$, $\mu_n$, $\Ga_n$, $\be$ and $\cU_2$ from Claim~\ref{c1} 
and $C$, $N^2_\ga$ from Claim~\ref{c2}.
By Whitney Extension Theorem~\cite[p. 176 ch. VI \S 2.3]{stein} there is $A>0$ 
and $C^k$ functions  $f_n\in C^k(M,[0,1])$ such that
\begin{equation*}
0\le f_n(x) = \begin{cases}
0 & \text{in a small neighbourhood of }\pi(\Ga_n),
\\
\be^{(k+1) m_n} &\text{if } d(x,\pi(\Ga_n)) \ge\tfrac 1C\, \be^{m_n},
\end{cases}
\end{equation*} 
and $\lV f_n\rV_{C^k}\le A$.
Take $\e>0$ such that $\forall n>N^2_\ga$, $\phi_n:=\e f_n\in \cU_2$.

Write $L_n:=L+\phi_n$. 
Observe that $\Ga_n$ is also a periodic orbit for $L_n$.
In particular Claim~\ref{c1} and Claim~\ref{c2} hold for measures
in $\cM_{\min}(L_n)$.
Suppose by contradiction
 that there is a sequence $n=n_i\to+\infty$ 
 such that 
 $\phi_n\notin\cE_{\ov\ga}$ with
$\ov{\ga}= 4\ga\, C(\cU_1,\be)$.
Then there   is a minimizing measure $\nu_n\in \cM_{\min}(L_n)$ with
$h(\nu_n)\ge 4 \ga \,C(\cU_1,\be)>3\ga\, C(\cU_1,\be)$.
By Claim~\ref{c1} and Claim~\ref{c2} we have that
$$
\nu_n\big(\big\{\,\vrt\in TM \;\big|\; d(\pi(\vrt),
 \pi(\Ga_n))\ge\tfrac 1C\,\be^{m_n}\;\big\}\big)>\ga.
$$
Since $\nu_n\in \cM(L_n)\subset \cC(TM)$ is a  closed measure, 
by \eqref{minmeas}, 
$
\int \big[L +c(L)\big]\, d\nu_n \ge 0.
$
Then
\begin{align}
\int \big( L+\phi_n\big) \,d\nu_n &\ge -c(L) +\int\phi_n\,d\nu_n
\notag
\\
&\ge -c(L) + \e \ga\, \be^{(k+1) m_n}.
\label{AM}
\end{align}

Observe that $\mu_n$ is also an invariant probability for $L_n=L+\phi_n$.
From Lemma~\ref{LA2} and Corollary~\ref{cmn}
applied to $L$ and $e=c(L)+1$,  we have that
\begin{align}
\int \big(L+\phi_n\big) \,d\mu_n &
= \int L\;d\mu_n
\le -c(L) + B(L,e) \int d(\theta,\cA(L)) \; d\mu_n(\theta)
\notag
\\
&\le -c(L)+ o\big(\be^{(k+1)m_n} \big).
\label{AN}
\end{align}
Inequalities \eqref{AN} and \eqref{AM} imply
that for $n=n_i$ large enough $\nu_{n_i}$ is not minimizing, 
contradicting the choice of $\nu_{n_i}$.
Therefore $\phi_n\in \cE_{\ov\ga}\cap\cU_1$ for $n$ large enough.

\qed

\color{black}

\section{Hyperbolic Aubry sets can be closed.}\label{Sclosing}

In this section we prove Theorem~\ref{HYP}. Throughout the section we assume that
$L$ is a Tonelli lagrangian with Aubry set $\cA(L)$ hyperbolic.

\subsection{The action of a periodic specification.}\quad

A {\it dominated function} for $L$ is a function $u:M\to\re$ such that 
for any $\ga:[0,T]\to M$ absolutely continuous and $0\le s<t\le T$ we have
\begin{equation}\label{domi}
 u(\ga(t))-u(\ga(s))  \le \int_s^t \big[c(L)+L(\ga,\dga)\big]. 
\end{equation}
We say that the curve $\ga$ {\it calibrates} $u$ if the equality holds in \eqref{domi}
for every $0\le s<t\le T$. 
Dominated functions always exist, for example, by the triangle inequality for Ma\~n\'e's potential $\Phi_c$, the functions $u_p(x):=\Phi_c(p,x)$
are dominated for every $p\in M$. 
The definition of the Hamiltonian $H$ associated to $L$ implies 
that any $C^1$ function  
$u:M\to\re$ which satisfies
$$
\forall x\in M, \quad H(x,d_xu)\le c(L)
$$
is dominated.

\begin{Lemma}\label{statcal}
If $u$ is a dominated function and $\ga$ is a static curve
then $\ga$ calibrates $u$.
\end{Lemma}

\begin{proof}
Recall that $\ga$ is static iff for all $s<t$ we have
\begin{equation}\label{staticd2}
\int_s^t\big[c(L)+L(\ga,\dga)\big] = -\phi_{c(L)}(\ga(t),\ga(s))=\phi_{c(L)}(\ga(s),\ga(t)).
\end{equation}
If $u$ is dominated, $\ga$ is static and $s<t$ we have that
\begin{align*}
u(\ga(t))\le u(\ga(s))+\phi_{c(L)}(\ga(s),\ga(t))
= u(\ga(s))-\phi_{c(L)}(\ga(t),\ga(s)).
\end{align*}
Using again the domination of $u$ and then the previous inequality we get
$$
u(\ga(s))\le u(\ga(t))+\phi_{c(L)}(\ga(t),\ga(s))\le u(\ga(s)).
$$
Therefore, using~\eqref{staticd2}, 
$$
u(\ga(t))=u(\ga(s))-\phi_{c(L)}(\ga(t),\ga(s))=
u(\ga(s))+\int_s^t\big[c(L)+L(\ga,\dga)\big].
$$
\end{proof}

\begin{Lemma}\label{domifathi}\quad

There are $K>0$ and $\de_0>0$ such that if $(z,\dz)\in\cA(L)$
 is a static vector, $u$ is a dominated function and $d(z,y)<\de_0$,
 then
 \begin{equation}\label{upvl}
 \big| u(y)-u(z)-\partial_vL(z,\dz)(y-z)\big|
 \le K\,|y-z|^2,
 \end{equation}
 where $y-z:=(\exp_z)^{-1}(y)$.

\end{Lemma}

\begin{proof}
Let $\E\subset TM$ be a compact subset such that 
$E^{-1}_L\{c(L)\}\subset\intt \E$. Cover $M$ by a finite 
set $\cO$ of charts. Fix $0<\e<1$  such that   if $\ga:[-\e,\e]\to M$
has velocity $(\ga,\dga)\in\E$ then $\ga([-\e,\e])$ lies inside the
domain of a chart in $\cO$.
There are $\de_1>0$ smaller than the Lebesgue
number of the covering $\cO$ and $A>0$ such that if $(x,v)\in\E$
and $\max\{|h|,\,|k|\}\le \de_1$ then in the charts
\begin{equation}\label{ldla}
\big| L(x+h,v+k)-L(x,v)-DL(x,v)(h,k)\big|\le A(|h|^2+|k|^2).
\end{equation}
Let $u:M\to\re$ be dominated and  $(z,\dz)\in\cA(L)$. Recall that 
$\cA(L)\subset E_L^{-1}\{c(L)\}\subset \E$. Write $\ga(t):=\pi\vr^L_t(z,\dz)$.
By Lemma~\ref{ALinv} the complete curve $\ga:\re\to M$ is static.
By Lemma~\ref{statcal}, $\ga$ calibrates $u$. Let $\de_0:=\e\,\de_1$.
Let $y\in M$ with $|y-z|<\de_0$ in a local chart. Define $\be:]-\e,0]\to M$ by
$$
\be(t):=\ga(t)+\left(\tfrac{t+\e}\e\right) (y-z).
$$
Then $\be(-\e)=\ga(-\e)$, $\be(0)=y$, $\dbe=\dga+\tfrac 1\e(y-z)$.
In particular $|\dbe-\dga|\le \tfrac 1\e|y-z|\le \de_1$ and we can apply~\eqref{ldla}.
\begin{align*}
\int_{-\e}^0 L(\be,\dbe)\le
\int_{-\e}^0L(\ga,\dga) 
+\int_{-\e}^0 
\Big\{L_x(\ga,\dga)(\be-\ga)
+L_v(\ga,\dga)(\dbe-\dga)\Big\}
+ A \e\big(1+\tfrac 1{\e^2}\big)\,|y-z|^2.
\end{align*}
Using that $\ga$ is a solution of the Euler-Lagrange equation
$\tfrac {d\,}{dt}L_v=L_x$ and integrating by parts, we get that
\begin{align}
\int_{-\e}^0L(\be,\dbe)&\le \int_{-\e}^0 L(\ga,\dga)\, dt+
L_v(\ga,\dga)(\be-\ga)\Big\vert_{-\e}^0 +  \tfrac {2A}{\e}\,|y-z|^2,
\notag\\
&\le \int_{-\e}^0 L(\ga,\dga)\, dt+ L_v(z,\dz)(y-z) + \tfrac{2A}{\e}\,|y-z|^2.
\label{lbdb}
\end{align}
Since $u$ is dominated and calibrated by $\ga|_{[-\e,0]}$ we obtain 
one of the inequalities in~\eqref{upvl}:
\begin{align*}
u(y) &\le u(\ga(-\e))+\int_{-\e}^0 c(L)+ L(\be,\dbe) 
\\
&\le u(\ga(-\e))+\int_{-\e}^0\Big\{L(\ga,\dga)+c(L)\Big\} dt+ L_v(z,\dz)(y-z) + \tfrac{2A}{\e}\,|y-z|^2
\\
&\le u(z) + L_v(z,\dz)(y-z) + \tfrac{2A}{\e}\, |y-z|^2.
\end{align*}

Now define $\a:[0,\e]\to M$ by
$$
\a(t):=\ga(t)+\left(\tfrac {\e-t}\e\right) (y-z).
$$
A similar argument to~\eqref{lbdb} gives
$$
\int_0^\e L(\a,\da)\, dt
\le \int_0^\e L(\ga,\dga) \,dt
-L_v(z,\dz)(y-z) + \tfrac{2A}{\e}\,|y-z|^2.
$$
Since $u$ is dominated we have that
\begin{align*}
u(\ga(\e))
&\le u(y)+\int_0^\e \Big\{L(\a,\da)+c(L)\Big\}
\\
&\le u(y) +\int_0^\e \Big\{L(\ga,\dga)+c(L)\Big\}\, dt
-L_v(z,\dz)(y-z) + \tfrac{2A}{\e}\,|y-z|^2.
\end{align*}
Since $u$ is calibrated by $\ga|_{[0,\e]}$ we have that
$$
u(\ga(\e))-\int_0^\e \Big\{L(\ga,\dga)+c(L)\Big\} = u(z).
$$
Thus we get the remaining inequality
$$
u(z)\le u(y)-L_v(z,\dz)(y-z)+ \tfrac{2A}{\e}\,|y-z|^2.
$$

\end{proof}

Recall that the Aubry set $\cA(L)$ is inside the energy level $c(L)$, 
 $\cA(L)\subset E^{-1}_L\{c(L)\}$. Recall also that we are assuming that 
 $\cA(L)$ is a hyperbolic set for the lagrangian flow restricted to 
 $E^{-1}_L\{c(L)\}$.  In Corollary~\ref{CSH} in Appendix~\ref{asha}
 we state the version of the Shadowing Lemma used in the following 
 proposition. From Definition~\ref{B8}, a $\de$-possible $\ell$-specification
 is a collection of orbit segments $\{\,\vr_{[t_i,t_{i+1}[}(\tt_i)\,\}_{i\in\Z}$ such that
 for all $i\in\Z$, $t_{i+1}-t_i\ge \ell$ and $d(\vr_{t_{i+1}}(\tt_i),\tt_{i+1})<\de$.
 We say that a specification $\{\,\vr_{[t_i,t_{i+1}]}(\tt_i)\,\}_{i\in\Z}$ in $\cA(L)$
 is $\e$-shadowed by $\xi\in E^{-1}_L\{c(L)\}$ if there is a increasing homeomorphism
 $s:\re\to\re$,  such that $s(t_0)=t_0$ and
 $$
 \forall i\in\Z,\quad \forall t\in]t_i,t_{i+1}[\qquad 
 |s(t)-t|\le \e \quad\text{and}\quad d(\vr_{s(t)}(\xi),\vr_t(\tt_i)) < \e.
 $$
 Corollary~\ref{CSH} states that there are $\de_0(\ell)>0$ and $Q(\ell)>0$
 such that every $\de$-possible $\ell$-specification in $\cA(L)$ with $0<\de<\de(\ell)$
 is $\e$-shadowed by a point in the energy level $E^{-1}\{c(L)\}$ with 
 $\e=Q(\ell)\,\de$.
 
 We say that a specification $\{\, \vr_{[t_i,t_{i+1}[}(\tt_i)\,\}_{i\in\Z}$ is
 {\it periodic with $J\in\na^+$ jumps} if  
 $$
 \forall i\in\Z \quad (t_{i+J},\tt_{i+J})=(t_i,\tt_i).
 $$
 
\bigskip

\begin{Lemma}\label{ALc}\quad
 
 Let $Q=Q(1)>0$, $\de_0=\de_0(1)>0$ be from Corollary~\ref{CSH} applied 
 to the hyperbolic set $\cA(L)$ in $E^{-1}\{c(L)\}$.

There is $E>0$ such that if $0<\de<\de_0$,
 $\{(x_k,\dx_k)\vert_{[T_k,T_{k+1}]}\}$ is a periodic 
 $\de$-possible $1$-specification  in $\cA(L)$
with $J\in\na^+$ jumps
and  $(y,\dy)$ is the periodic orbit with energy $c(L)$ which
 $\e$-shadows $\{(x_k,\dx_k)\}$
with $\e= Q\,\de$ then
$$
A_{L+c(L)}(y)\le J E\,\de^2.
$$
\end{Lemma}

\begin{proof}
The set $\cA(L)$ is hyperbolic for the Euler-Lagrange flow restricted 
to the energy level $[E=c(L)]$.
By the Shadowing Corollary~\ref{CSH}, there is an Euler-Lagrange solution
$(y,\dy)$ with energy $c(L)$ and a  continuous reparametrization 
$\si(t)$, with 
$| \si(t)-t|\le \e$ such that
$$
\forall t \qquad
d\big([x_k(t),\dx_k(t)],[y(\si(t)),\dy(\si(t))]\big) <\e.
$$

Then $Y(s):=(y(s),\dy(s))$ is a periodic orbit with a period near $\si(T_J-T_0)$.
We want a sequence of times $S_k$ nearby $\si(T_k)$ such that $S_J-S_0$ is
a period for $Y(s)$. Write $X(t):=(x(t),\dx(t))$. Using canonical coordinates~\ref{caco}
 define $w_k\in \re$ small by
\begin{align*}
\langle Y(\si(T_k)),X(T_k)\rangle 
&=W^{s}_\ga(Y(\si(T_k)))\cap W^{uu}_\ga(X(T_k))
\\
&=W^{ss}_\ga\big(\vr_{w_k}(Y(\si(T_k)))\big)\cap W^{uu}(X(T_k))
\ne\emptyset.
\end{align*}
Now let $S_k:=w_k+\si(T_k)$.
Observe that the time shift $w_k$ is determined by the sequence $X(T_k)$ which 
is periodic. Then the sequence $S_k$ is periodic modulo the period $S_J-S_0$ of $Y$.

By Proposition~\ref{B71} there are $D>0$ and $0<\la<1$ such that  
if $\e$ is small enough  
 there is $|v_k|<D \e$,  such that
$$
\forall s\in[S_k,S_{k+1}]
\qquad
d\big([x_k(s+v_k),\dx_k(s+v_k)],[y(s),\dy(s)]\big)
\le D\,\e\, \la^{\min\{s-S_k,\,S_{k+1}-s\}}.
$$
Let $z_k(s):=x_k(s+v_k)$.
Since $\cA(L)$ is invariant we also have that $(z_k,\dz_k)\in\cA(L)$.

By adding a constant we can assume that $c(L)=0$.
On local charts we have that
\begin{align*}
L(y,\dy) &\le L(z_k,\dz_k) 
+\partial_xL(z_k,\dz_k) (y-z_k) 
+\partial_vL(z_k,\dz_k) (\dy-\dz_k)
+ K_1\,\e^2\, \la^{2\min\{s-S_k,\,S_{k+1}-s\}}.
\end{align*}
Using that $z_k$ is an Euler-Lagrange solution we obtain
\begin{align*}
\int_{S_k}^{S_{k+1}}L(y,\dy) 
&\le 
\left[\int_{S_k}^{S_{k+1}}L(z_k,\dz_k) \right]
+ \partial_vL(z_k,\dz_k)(y-z_k)\Big\vert_{S_k}^{S_{k+1}}
+ K_2 \,\e^2.
\end{align*}
\begin{align}\label{ALz}
A_L(y) \le &\sum_{k=1}^J A_L(z_k)\;+
\notag\\
+&\sum_{k=1}^J\Big\{ \partial_v L(z_k,\dz_k)(y-z_k)\Big\vert_{S_{k+1}}
-\partial_v L(z_{k+1},\dz_{k+1})(y-z_{k+1})\Big\vert_{S_{k+1}}
\Big\}
+J K_2\,\e^2.
\end{align}

Let $u$ be a dominated function. By Lemma~\ref{domifathi}
if $(z,\dz)\in\cA(L)$ is a static vector, then 
$$
\big| u(y)-u(z) - \partial_vL(z,\dz)(y-z) \big|\le K_3\, |y-z|^2.
$$
By Lemma~\ref{statcal}, $u$ is necessarily calibrated on static curves.
Therefore
\begin{align}
\sum_{k=1}^JA_L(z_k) &=\sum_k u(z_k(S_{k+1}))-u(z_k(S_k))
\notag \\
&=
\sum_k u(z_k(S_{k+1}))-u(z_{k+1}(S_{k+1}))
\notag \\
&= 
\sum_k
\big\{
u(z_k) -u(y) +u(y) -u(z_{k+1})\big\}\Big\vert_{S_{k+1}}
\notag \\
&\le \sum_k
\Big\{
\partial_v L(z_k,\dz_k)(z_k-y)
+
\partial_v L(z_{k+1},\dz_{k+1})(y-z_{k+1})
+ 2K_3 D^2  \e^2\Big\}
\Big\vert_{S_{k+1}}.
\label{ALxk}
\end{align}
Replacing  estimate  \eqref{ALxk} for $\sum_k A_L(z_k)$ 
in inequality \eqref{ALz}
we obtain
$$
A_L(y) \le 2J K_3 D^2\,\e^2+ J K_2\,\e^2 =: J K_4\,\e^2.
$$
Since $\e=Q\,\de$, we obtain Lemma~\ref{ALc} with $E=K_4 Q^2$.
\newline
\end{proof}

\medskip

\begin{Lemma}\label{LKTr}
Given a Tonelli lagrangian  $L_0$ a compact subset $\De\subset TM$  and $\e>0$ 
there are $K>0$ and $\de_1>0$ 
such that for any Tonelli lagrangian $L$ with $\lV (L-L_0)|_E\rV_{C^2}<\e$,
and any $T>0$:
\begin{enumerate}[(a)]
\item\label{KTra}
If $x\in C^1([0,T],M)$  is a solution of the Euler-Lagrange equation for $L$
with $(x,\dx)\in \De$ and 
$z\in C^1([0,T],M)$ satisfies
$$
d\big([z(t),\dz(t)],[x(t),\dx(t)]\big) \le 2 \rho\le\de_1 \qquad \forall t\in[0,T],
$$
then 
\begin{equation}\label{KTr}
\lv \int_0^T L(z,\dz)\,dt -\int_0^T L(x,\dx)\, dt -\partial_v L(x,\dx)\cdot(z-x)\Big|_0^T\rv
\le K \,(1+T)\,  \rho^2,
\end{equation}
where $z-x:=(\exp_x)^{-1}(z)$.

\item\label{KTrb} 
If $x\in C^1([0,T],M)$ is a solution of the Euler-Lagrange equation for $L$ with $(x,\dx)\in \De$
 and
the curves 
$w_1,\,w_2,\,z\in C^1([0,T],M)$ satisfy
$w_1(0)=x(0)$, $w_1(T)=z(T)$, $w_2(0)=z(0)$, $w_2(T)=x(T)$, and 
for $\xi=z,\,w_1,\,w_2$ we have
$$
d\big([\xi(t),\dxi(t)],[x(t),\dx(t)]\big) \le 2 \rho\le \de_1 \qquad \forall t\in[0,T], 
$$
then
$$
\lv A_L(x)+A_L(z)-A_L(w_1)-A_L(w_2)\rv \le 3 K  \rho^2  (1+T).
$$
\end{enumerate}
\end{Lemma}

\begin{proof}\quad
\begin{enumerate}[(a)]
\item
We use a coordinate system on a tubular neighbourhood of $x([0,T])$ with 
a bound in the  $C^2$ norm independent of $T$ and of $\dx(0)$. 
In case $x$ has self-intersections
or short returns the coordinate system is an immersion. 

We have that
\begin{align*}
L(z,\dz)-L(x,\dx) = \partial_x L(x,\dx)(z-x) +\partial_v L(x,\dx)(\dz-\dx) +O(\rho^2),
\end{align*}
here $O(\rho^2)\le K\, \rho^2$
where $K$ depends on the second derivatives of $L$ on a small neighbourhood of 
 the compact $\De$ and hence
it can be taken uniform on a $C^2$ neighbourhood of $L$. Since $x$ satisfies the 
Euler-Lagrange equation for $L$,
\begin{align*}
L(z,\dz)-L(x,\dx) = \tfrac d{dt}[\partial_v L(x,\dx)(z-x)] +O(\rho^2).
\end{align*}
This implies \eqref{KTr}.

\item By item~\eqref{KTra}
\begin{align*}
A_L(w_1) -A_L(x) &\le \partial_vL(x,\dx) (w_1-x)\Big|_0^T + K \rho^2 (1+T)
\\
&\le   \partial_vL(x(T),\dx(T)) (z(T)-x(T)) + K \rho^2 (1+T) .
\\
A_L(w_2)-A_L(x) 
&\le -\partial_vL(x(0),\dx(0)) (z(0)-x(0)) + K \rho^2 (1+T) .
\\
A_L(x)-A_L(z) 
&\le  -\partial_vL(x,\dx) (z-x)\Big|_0^T + K \rho^2 (1+T)
\\
&\le -\partial_vL(x(T),\dx(T)) (z(T)-x(T)) 
\\
&\hskip 10pt 
+\partial_vL(x(0),\dx(0)) (z(0)-x(0))+ K \rho^2 (1+T) .
\end{align*}
Adding these inequalities we get
$$
A_L(w_1)+A_L(w_2)-A_L(x)-A_L(z)
\le 3 K \rho^2 (1+T).
$$
The remaining inequality is obtained similarly.
\end{enumerate}
\end{proof}

\subsection{Closing solitary returns.}\quad
\label{sscsr}

The following proposition has its origin in Yuan and Hunt \cite{YH}, 
the present proof uses some arguments by Quas and Siefken \cite{QS}.

\begin{Proposition}\label{Ppert}\quad

Let $J\in\na^+$. Suppose that for any $\de>0$ there is a periodic 
$\de$-possible 1-specification  $\{\vr_{[T_k,T_{k+1}]}(\tt_k)\}_{k=1}^\ell$ in $\cA(L)$
with at most $J$ jumps ($\ell\le J$) such that the smallest approach
$$
\ga_\de:=\min \big\{\,d(\vr_{s_i}(\tt_i),\vr_{t_j}(\tt_j))
\;\big|\;  s_i\in[T_i,T_{i+1}], t_j\in[T_j,T_{j+1}];\;  |s_i-t_j|_{\text{\rm  mod }(T_{\ell+1}-T_1)}\ge 1\;\big\}
$$
satisfies
$$
\liminf\limits_{\de\to 0}\frac\de{\ga_\de}=0.
$$
Then for any $\e>0$ there is $\phi\in C^2(M,\re)$ with $\lV \phi\rV_{C^2}<\e$
such that $\cA(L+\phi)$ contains a periodic orbit.
\end{Proposition}

\begin{Remark}
The function $\phi$ used in proposition~\ref{Ppert} does not require
a special technique to perturb the flow, we describe it here explicitly.
 For $\de$ and $\frac\de{\ga_\de}$ 
sufficiently small let $\Ga$ be the periodic orbit with energy $c(L)$ 
which shadows the pseudo-orbit.
The function $\phi$ is a canal about $\pi(\Ga)$ 
defined in \eqref{defphi}. In particular $\Ga$ is a common 
periodic orbit for the flows of $L$ and $L+\phi$. Proposition~\ref{Ppert} proves that 
$\Ga$ is a periodic orbit in the Aubry set $\cA(L+\phi)$.
\end{Remark}

\begin{Lemma}\label{Lgade}
If $\cA(L)$ has no periodic orbits and $\{\vr^\a_{[T^\a_k,T^\a_{k+1}]}(\tt^\a_k)\}_{k=1}^\ell$, $\a\in\na$
is a sequence of $\de_\a$-possible 1-specifications with $\ell\le J$ jumps with $\lim_\a\de_\a=0$,
then $\lim_\a \ga_{\de_\a}=0$.
\end{Lemma}

\begin{proof}
We can assume that $T^\a_1=0$ for all $\a$. 
We first prove that $\lim_\a (T^\a_{\ell+1}-T^\a_1)=+\infty$. If not, we can extract a subsequence $\a_n$
such that  all
$\{\tt^{\a_n}_k\}_{k=1}^\ell$, $\{T^{\a_n}_k\}_{k=1}^{\ell+1}$ converge in $n$. Define
$\tt_k:=\lim_n \tt^{\a_n}_k\in\cA(L)$ and $T_k:=\lim_n T^{\a_n}_k$. Since $\lim_n\de_{\a_n}=0$, we have that $\forall k$
$\vr_{T_{k+1}}(\tt_k)=\vr_{T_{k+1}}(\tt_{k+1})$ and hence $\tt_1$ is a periodic point in $\cA(L)$ with period $T_{\ell+1}-T_1$.
This contradicts the hypothesis.

Consider the points $\xi^\a_{4m}:=\vr_{4m}(\tt^\a_i)$, where $m\in\na$ and  $i$ is such that $T_i\le 4m < T_{i+1}$
and $1\le m <M_\a:= \big[\frac 14T_{\ell+1}\big]$.
Recall that $T^\a_1\equiv 0$.
Since $T^\a_{\ell+1}\to\infty$, when $\a\to\infty$ the quantity $M_\a$ 
of these points $m$ tends to infinity.
Therefore 
$$
\ga_{\de_\a}\le \min_{m_1\ne m_2} d(\xi^\a_{4m_1},\xi^\a_{4m_2}) \overset{\a}\longrightarrow 0
$$

\end{proof}

{\bf Idea of the Proof:}

We first close the specification using the Shadowing Corollary~\ref{CSH} and obtain
a periodic orbit $\Ga$.
Then perturb the Lagrangian by a potential $\phi$ which is a 
non-negative channel centred at $\pi(\Ga)$
defined in~\eqref{defphi}. 
The curve $\Ga$ is a periodic orbit for the flows of $L$ and of $L+\phi$.
We show that $\Ga$ is contained in the Aubry set $\cA(L+\phi)$ by proving that
any semi-static curve $x:]-\infty,0]\to M$ for $L+\phi$ has 
$$
\a\text{-limit of }(x,\dx)=\Ga;
$$
because by Ma\~n\'e \cite[Theorem V.(c)]{Ma7}, $\a$-limits of semi-static orbits are static.
This is done by calculating the action of each segment of the semi-static which is spent 
outside of a small neighbourhood of $\pi(\Ga)$, and proving that it has a uniform positive
lower bound. Since the total action of a semi-static is finite, the quantity of those segments 
is finite. Thus the semi-static eventually stays forever in a small neighbourhood of $\Ga$.
The expansivity of $\cA(L+\phi)$ implies that the $\a$-limit of the semi-static is $\Ga$.

\bigskip

{\bf Proof of Proposition~\ref{Ppert}:}

By adding a constant we can assume that 
\begin{equation}\label{clo}
c(L)=0.
\end{equation}
Let $u$ be a $C^1$ critical subsolution of the Hamilton-Jacobi equation 
for $L$, $H(x, d_xu)\le c(L)$. Thus
\begin{equation}\label{Ldu}
L - du \ge 0.
\end{equation}

By Gronwall's inequality and the continuity of Ma\~n\'e's critical value $c(L)$
(see 
\cite[Lemma~5.1]{CP})  there is $\a>0$ and $\ga_0$  such that if 
$\lV\phi\rV_{C^2}\le 1$,  $0<\ga<\ga_0$ and
$\Ga$ is a periodic orbit for $L+\phi$ 
with energy smaller than $c(L+\phi)+1$ 
then
\begin{equation}\label{tima}
d(\vr^{L+\phi}_s(\vrt),\Ga)\le\frac\ga 4 
\quad \text{ and }\quad
d(\vr^{L+\phi}_t(\vrt),\Ga)\ge\frac\ga3
\quad\then \quad 
|t-s|> \a.
\end{equation}

The graph property states that the projection $\pi:\cA(L)\to M$ has a Lipschitz
inverse (see Ma\~n\'e \cite{Ma7}). The Lipschitz constant is the same as $C$ in
Mather's Crossing Lemma~\ref{CL}.
The Aubry set has energy $c(L)$ and $c(L+\phi)$ is continuous on $\phi$
(see \cite{CP} Lemma~5.1, p.~660).
Then one can choose $\e_1$ and $K$, $C>1$ in Lemma~\ref{CL}  such that if
$\lV \phi\rV_{C^2}<\e_1$ then $\cA(L+\phi)$ is a graph with Lipschitz constant $C$.

By the upper semicontinuity of the Ma\~n\'e set 
(see \cite{CP} Lemma~5.2, p.~661),  we can 
choose a neighbourhood $U$ of $\mN(L)$
and $0<\e_2<\e_1$  such that if $\lV \phi\rV_{C^2}<\e_2$
then the set 
$$
\La(\phi) := \bigcap_{t\in\re}\vr_{-t}^{L+\phi}(\ov{U})
$$
is hyperbolic and contains $\mN(L+\phi)$.
We can assume that in the statement of 
Proposition~\ref{Ppert}, $\e<\e_3$
from~\eqref{gadbe0}.
Let $\epsilon_0>0$ be a flow expansivity constant for $\mN(L+\phi)$
as in Definition~\ref{dfe} and Remark~\ref{rue}.

Fix $K_1>0$ such that
\begin{equation}\label{K1}
[E_L\le c(L)+1]\subset [|v|\le K_1].
\end{equation}

We can assume that $\cA(L)$ has no periodic points.
By Lemma~\ref{Lgade}, $\ga_\de$ is small when $\de$ is small.
Choose $\de$ and a $\de$-possible 1-specification 
with $\ga_\de$
and $\frac\de{\ga_\de}$  so small that
\begin{align}
&\;\de < \de_0(1) \quad\text{ from Corollary~\ref{CSH}},
\notag
\\
&\gad < \epsilon_0 \qquad \text{ a flow expansivity constant for $\mN(L+\phi)$, }
\label{gadeo}
\\
&\ga_\de <\min\{\be_0,\be_1\} \quad 
\begin{matrix}
\text{where  $\be_0$  and $\be_1$  are from 
Propositions~\ref{B71} and~\ref{B16} }
\\
\text{ for $\mN(L+\phi)$, for all  $\lV\phi\rV_{C^2}<\e_3<\e_2$.}
\label{gadbe0}
\end{matrix}
\\
2 &\ga_\de < \de_1 \qquad \de_1:=\de[K_1]
\text{ from Lemma~\ref{CL}, where  $K_1$ is from~\eqref{K1},
}
\label{dfde1}
\\
&\ga_\de < \eta 
\qquad \text{ where $\eta$ is from the canonical coordinates \ref{caco} for $\mN(L)$ as in \eqref{cacoN},}
\label{gadeta}
\end{align}
and such that writing 
\begin{equation}\label{gabarde}
\ogd:=\frac{\ga_\de}{C  (B+1)}
<\tfrac 12\ga_\de
\end{equation}
we have that
\begin{align}
&\ogd<\ga_0\hskip 2cm\text{ where $\ga_0$ is from~\eqref{tima},}
\label{gatima}
\\
&\ogd - 2 Q\, \de > \tfrac 3 4 \ogd
\qquad Q:=Q(1) \text{ from Corollary~\ref{CSH} with} 
\label{gaqd}
\\
&  Q>1,
\label{Q>1}
\end{align}
 and there is 
\begin{equation}\label{rho14ga}
\de<\rho<\tfrac 14\ogd
\end{equation} 
such that
\begin{gather}
\tfrac 14\e\, \rho^2 > JE\, \de^2,
\label{rho1}
\\
 C\rho > \de\; \sqrt{\tfrac{JE}\eta},
 \label{rho15}\\
 \left(  \tfrac 1{32}\, \e\, (\ogd)^2 -JE \de^2\right) \a
  -96 KD^2 C^2 (B+1)^2\rho^2 - 3JE\,\de^2
  > 0.
  \label{rho2}
\end{gather}
where $\a$ is from~\eqref{tima}, 
$E$ is from Lemma~\ref{ALc}, $B$ is from Lemma~\ref{B4},
 $C$ and $\eta$ are from Lemma~\ref{CL} and \eqref{CL>1}, with
 \begin{equation}\label{C>1}
 C>1,
 \end{equation}
 $D$ is from Proposition~\ref{B71} 
and $K$ is from Lemma~\ref{LKTr}
applied to the compact $\Delta=$ \linebreak
$[E_L\le c(L)+5]$.

 Let $Q=Q(1)>1$ be from the Shadowing Corollary~\ref{CSH} applied 
 to the hyperbolic set $\cA(L)$ in $E^{-1}\{c(L)\}$.
 Let $\Ga$ be the periodic orbit with energy $c(L)$ which $Q\de$-shadows the 
 1-specification $\{\vr_{[T_k,T_{k+1}]}(\tt_k)\}$
and let $\mu_\Ga$ be the $L$-invariant probability measure supported on $\Ga$.
Let $T$ be the period of $\Ga$.

Let $\phi:M\to[0,1]$ be a $C^\infty$ function such that  $\lV \phi\rV_{C^2}<10\, \e$ and
\begin{equation}\label{defphi}
0\le \phi(x)=
\begin{cases}
0 &\text{if } x\in\pi(\Ga),
\\
\ge \tfrac 14 \e \, \rho^2 
&\text{if } d(x,\pi\Ga)\ge \rho,
\\
\tfrac 1{32}\e\, (\ogd)^2
&\text{if }d(x,\pi\Ga)\ge \tfrac 14 \ogd.
\end{cases}
\end{equation}

Write
\begin{equation*}\label{defLL}
\L:=L+\phi+c(L+\phi)-du.
\end{equation*}

\begin{Claim}
If $\frac\de{\ga_\de}$ is small enough then
\begin{enumerate}
\item \label{cl1}
We have that
$$
\inf\limits_{d(s,t)_{\text{mod }T}\ge 2} d\big(\pi\Ga(s),\pi\Ga(t)\big)
 >\tfrac 34 \,\ogd.
 $$
 In particular the neighbourhood $B(\pi\Ga,\tfrac 38\ogd)$ of $\pi\Ga$ of radius 
 $\frac 38 \ogd$ has no self intersections, i.e. it is homeomorphic to
 $S^1\times ]0,1[^{\dim M-1}$.
 
\item\label{cl2}
 If  $x:]-\infty,0]\to M$ is a semi-static orbit for $\L$ then for all $t\le -1$
\begin{align}\label{ecl2}
&\text{either }\quad  
d\big([x(t),\dx(t)],\Ga\big) \le \de \, \sqrt{\tfrac{ J E}{\eta_1}}
\quad\text{ or } \quad d\big([x(t),\dx(t)],\Ga\big) \le C\, d\big(x(t),\pi\Ga\big),
\\
&\text{or }\quad d(x(t),\pi\Ga)\ge \de_1,
\label{ecl3}
\end{align}
where  $E$ is from Lemma~\ref{ALc}; $\eta_1=\eta_1(K_1)$, $C=C(K_1)$
and $\de_1=\de_1(K_1)$ are from Lemma~\ref{CL} for $K=K_1$ from~\eqref{K1}.
\end{enumerate}
\end{Claim}
{\it Proof:}

\eqref{cl1}.
 If  $d(s,t)_{\text{mod T}}\ge 2$, 
 by Corollary~\ref{CSH} for $L$,
  there are points $\tt_i$, $\tt_j$
 and $s_i,\, t_j\in\re$  in the specification,
  with $\max\{|s-s_i|,\,|t-t_j|\}\le Q\de$ such that
\begin{gather*}
d\big(\pi\Ga(s),\pi\vr_{s_i}\tt_i\big)
\le d\big(\Ga(s),\vr_{s_i}(\tt_i)\big)
\le Q\de,
\\
d\big(\pi\Ga(t),\pi\vr_{t_j}\tt_j\big)
\le d\big(\Ga(t),\vr_{t_j}(\tt_j)\big)
\le Q\de.
\end{gather*}
If $\de$ is small we have that $|s_i-t_j|_{\text{mod }T}\ge 1$ and then
$d(\vr_{s_i}(\tt_i),\vr_{t_j}(\tt_j))\ge \ga_\de$.
Since $\tt_i,\,\tt_j\in\cA(L)$, by the graph property for $\cA(L)$ and \eqref{gabarde}
we have that
$$
d(\pi\vr_{s_i}\tt_i,\pi\vr_{t_j}\tt_j)
\ge \tfrac 1C\,
d(\vr_{s_i}\tt_i,\vr_{t_j}\tt_j)
\ge \tfrac 1C\ga_\de\ge \ogd.
$$
Then
\begin{align*}
d\big(\pi\Ga(s),\pi\Ga(t)\big)
&\ge d(\pi\vr_{s_i}\tt_i,\pi\vr_{t_j}\tt_j)
- d(\pi\Ga(s),\pi\vr_{s_i}\tt_i)
- d(\pi\Ga(t),\pi\vr_{t_j}\tt_j)
\\
&\ge \ogd-2 Q\de.
\end{align*}
Therefore by~\eqref{gaqd}
$$
\inf_{d(s,t)_{\text{mod $T$}}\ge 2} d\big(\pi\Ga(s),\pi\Ga(t)\big)
\ge \ogd - 2 Q\de >\tfrac 34 \,\ogd.
$$

\eqref{cl2}. Suppose by contradiction that there exists $t\le -1$ such that 
\begin{equation}\label{posso}
d(x(t),\pi\Ga)<\de_1 \qquad \text{ and }
\end{equation}
$$
 d\big([x(t),\dx(t)],\Ga\big)^2 >  \de^2 \, \tfrac{ J E}{\eta_1}
\quad\text{ and } \quad d\big([x(t),\dx(t)],\Ga\big) > C\, d(x(t),\pi\Ga).
$$

First we check that we can apply Lemma~\ref{CL} to $\L$.
For $\ga:[0,T]\to M$ we have that 
$$
\oint_\ga c(L+\phi)-du = T \, c(L+\phi) -u(\ga(T))+u(\ga(0))
$$
depends only on the time interval $T$ and the endpoints of $\ga$.
Thus instead of  $\L$ it is enough to apply Lemma~\ref{CL} to $L+\phi$, for whom 
it holds if $\phi\in\cU_2$ is small enough.

Now we check the speed hypothesis in Lemma~\ref{CL}.
Observe that 
$$
E_\L=v\,\L_v-\L=E_{L+\phi}-c(L+\phi)=E_L-\phi-c(L+\phi)
$$
 and that by~\eqref{minmeas}
 $$
 c(\L)=c\big(L+\phi+c(L+\phi)\big)=0.
 $$
 Therefore
 $$
 \cN(\L)\subset [E_\L= c(\L)]
 \subset [E_L=\phi+c(L+\phi)].
 $$
 If $\phi$ is small enough 
$$
\phi+c(L+\phi)<c(L)+1,
$$
and then $\dx(t)\in
\cN(\L)\subset [E_L\le c(L)+1]$. The Shadowing Corollary~\ref{CSH} is applied 
to the hyperbolic set $\cA(L)$ of the lagrangian flow $\vr^L$ restricted to the energy 
level $[E_L=c(L)]$. So the shadowing periodic orbit $\Ga$ is in the same manifold
$\Ga(t)\in [E_L=c(L)]$. Therefore
$$
\forall t \qquad
\dx(t),\; \Ga(t) \in [E_L\le c(L)+1].
$$
 
Finally we check the distance hypothesis in Lemma~\ref{CL}. 
Let $t_0$ be such that $d(x(t),\pi\Ga)=d(x(t),\pi(\Ga(t_0)))$.
By~\eqref{posso} and the definition of $\de_1$ in \eqref{dfde1} we can apply Lemma~\ref{CL} for $\L$ 
and  $K=K_1$ from~\eqref{K1},
to $x$ and $\pi\Ga$ at $x(t)$ and $\pi(\Ga(t_0))$.
Using $0<\e\le 1$ from Lemma~\ref{CL} we obtain $C^1$ curves
$w_1,\,w_2:[-\e,\e]\to M$ with $w_1(-\e)=x(t-\e)$, $w_1(\e)=\pi\Ga(t_0+\e)$,
$w_2(-\e)=\pi\Ga(t_0-\e)$, $w_2(\e)=x(t+\e)$ such that
$$
A_\L(w_1)+A_\L(w_2) 
\le A_\L(\pi\Ga|_{[t_0-\e,t_0+\e]}) + A_\L(x|_{[t-\e,t+\e]})
- \eta_1\, d([x(t),\dx(t)],\Ga(t_0))^2.
$$
Since $\phi\ge 0$ we have that 
\begin{equation}\label{clf}
c(L+\phi)\le c(L)=0.
\end{equation}
Using Lemma~\ref{ALc}, $\phi|_{\pi\Ga}\equiv 0$ and that $\Ga$ is a closed curve
we have that 
$$
A_\L(\pi\Ga) = A_{L+c(L+\phi)}(\pi\Ga) \le A_{L+c(L)}(\pi\Ga) \le 
J E \de^2.
$$
We compute the action of the curve $w_1*\pi\Ga|_{[t_0+\e,t_0+T-\e]}*w_2$
which joins $x(t-\e)$ to $x(t+\e)$.
\begin{align*}
A_\L(w_1)&+A_\L(\pi\Ga|_{[t_0+\e,t_0+T-\e]})+A_\L(w_2) <
\\
&< A_\L(x|_{[t-\e,t+\e]})
+ A_\L(\pi\Ga|_{[t_0-\e,t_0+\e]}) + A_\L(\pi\Ga|_{[t_0+\e,t_0+T-\e]})
- \eta_1\, d([x(t),\dx(t)],\Ga(t_0))^2
\\
&\le  A_\L(x|_{[t-\e,t+\e]})
+ J E \de^2 - \eta_1\, d([x(t),\dx(t)],\Ga)^2
\\
&< A_\L(x|_{[t-\e,t+\e]}), \qquad \text{using \eqref{posso}}.
\end{align*}
This contradicts the assumption that $x$ is semi-static for $\L$.

\hfill$\triangle$

\bigskip

Observe that $\Ga$ is also a periodic orbit for $L+\phi$. 
By Lemma~\ref{ALc} and \eqref{clo} we have that 
\begin{equation}\label{cLp}
c(L+\phi)\ge -\int (L+\phi) \;d\mu_\Ga
= -\int L \; d\mu_\Ga
\ge - \frac{JE}T \, \de^2,
\end{equation}
where $T$ is the period of $\Ga$.
Since we can assume that $\cA(L)$
has no periodic orbits, if $\de$ is small enough 
\begin{equation}\label{T>1}
T\ge 1.
\end{equation}

We will prove that  any semi-static curve $x:]-\infty,0]\to M$ for $L+\phi$
has 
$\a$-limit$\{(x,\dx)\}$ $= \Ga$.
Since $\a$-limits of semi-static orbits are static (Ma\~n\'e \cite[Theorem V.(c)]{Ma7}),
this implies that $\Ga\subset\cA(L+\phi)$.
Thus finishing the proof of Proposition~\ref{Ppert}.

Since by~\eqref{gadeo} the number
 $\ogd$ is smaller than the flow expansivity constant of $\mN(L+\phi)$,
it is enough to prove that  the tangent $(x,\dx)$ of 
any semi-static curve $x:]-\infty,0]\to M$ 
spends only  a bounded time outside the $\tfrac 38\ogd$-neighbourhood of $\Ga$.

Let $x:]-\infty,0]\to M$ be a semi-static curve for $L+\phi$.
Let $\tt:=(x(0),\dx(0))$.

By~\eqref{dfde1} and~\eqref{gabarde} we have that
\begin{equation}\label{dxpgd1}
d(x(t),\pi\Ga)\ge \de_1
\quad\then\quad
d(x(t),\pi\Ga)>\tfrac 14 \ogd.
\end{equation}
By \eqref{ecl2}-\eqref{ecl3} and \eqref{rho15} we have that
\begin{equation}\label{Drho}
d(\vr_t\tt,\Ga)> C\rho
\quad \&\quad 
d(x(t),\pi\Ga)<\de_1
 \quad \then \quad 
d(x(t),\pi\Ga) \ge \tfrac 1C \, d(\vr_t\tt,\Ga).
\end{equation}
By~\eqref{rho14ga}, \eqref{gabarde} we have that $\tfrac 14\ga_\de>C\rho$.
And then  from~\eqref{dxpgd1} and~\eqref{Drho} we get
\begin{equation}\label{Crho}
d(\vr_t\tt,\Ga)\ge\tfrac 14 \gad\left(> \tfrac 14 C\,\ogd\right)
 \quad \then \quad 
d(x(t),\pi\Ga) 
> \tfrac 14\,\ogd.
\end{equation}
Then by \eqref{CL>1},  \eqref{defphi},  \eqref{cLp},  \eqref{T>1} and \eqref{rho1}, 
 we have that 
\begin{equation}\label{phirho}
d(\vr_t\tt,\Ga)>C\rho \quad \then \quad
\phi(x(t))+c(L+\phi) \ge  \tfrac14 \e\rho^2-\tfrac{JE}T \de^2 =:a_0>0.
\end{equation}

For $\xi\in\mN(L+\phi)$ consider the local invariant manifolds
\begin{align*}
W^s_\eta(\xi)&:=\{\,\zeta\in E_L^{-1}\{c(L)\}\;:\; \forall t\ge0 \quad
d(\vr_t(\zeta),\vr_t(\xi))\le \eta\,\},
\\
W^{ss}_\eta(\xi) &:=\{\, \zeta\in W^s_\eta(\xi)\;:\; \lim_{t\to+\infty}
d(\vr_t(\zeta),\vr_t(\xi))=0\,\},
\\
W^u_\eta(\xi)&:=\{\,\zeta\in E_L^{-1}\{c(L)\}\;:\; \forall t\le0 \quad
d(\vr_t(\zeta),\vr_t(\xi))\le \eta\,\},
\\
W^{uu}_\eta(\xi) &:=\{\, \zeta\in W^s_\eta(\xi)\;:\; \lim_{t\to-\infty}
d(\vr_t(\zeta),\vr_t(\xi))=0\,\}.
\end{align*}

Also consider the canonical coordinates (cf.~\ref{caco}) on $\mN(L)$, i.e.
{\sl there are $\a_1,\,\eta>0$ such that if
 $\xi,\,\zeta\in\mN(L)$
and $d(\xi,\zeta)<\a_1$
then there is $v=v(\xi,\zeta)\in\re$, $|v|\le\eta$ such that
\begin{align}
\langle \xi, \zeta \rangle:=W^{ss}_\eta(\vr_v(\xi))\cap W^{uu}_\eta(\zeta)
\ne \emptyset.
\label{cacoN}
\end{align}
}
We use the canonical coordinates to parametrize the approaches of
$\vr_t(\tt)$ to $\Ga$ in the following way. We can assume that $\gad<\a_1$.
The local weak stable manifold of $\Ga$
$$
W^{ss}_\eta(\Ga):=\textstyle\bigcup_{\xi\in\Ga}W^{s}_\eta(\xi)
=\bigcup_{\xi\in\Ga}W^{ss}_\eta(\xi)
$$
forms a cylinder homeomorphic to $\Ga(\re)\times ]0,1[^{\dim M-1}$.
When $d(\vr_t(\tt),\Ga(\re))<\gad$ the strong local unstable manifold
$W^{uu}_\eta(\vr_t(\tt))$ intersects this cylinder transversely  and 
defines a unique time parameter $v(t)$ (mod $T$) such that 
\begin{equation}\label{vttt}
W^{ss}_\eta(\Ga(v(t)))\cap W^{uu}_\eta(\vr_t(\tt))\ne 0.
\end{equation}
Since the family of strong invariant manifolds is invariant 
under each iterate $\vr_t$
we have that if $d(\vr_t(\tt),\Ga(\re))<\gad$ for all $t\in[a,b]$
then
$$
\forall s\in[0,b-a]\qquad v(a+s)= v(a)+s.
$$

 Let  $B$ be from Lemma~\ref{B4}.
 Write
$\tt=(x(0),\dx(0))$ and
define $S_k(\tt)$, $T_k(\tt)$ recursively by
\begin{align*}
S_0(\tt)&:=0,\\
T_k(\tt)&:=\sup\,\big\{\,t< S_{k-1}(\tt)\;\big|\; d\big(\vr_t(\tt),\Ga(v(t))\big)
\le C(B+1)\rho\,\big\},
\\
C_k(\tt)&:=\sup \,\big\{\,t<T_k(\tt)\;\big|\; d(\vr_t(\tt),\Ga(\re))=\tfrac\gad3\,\big\},
\\
S_k(\tt)&:=\inf\big\{\,t>C_k(\tt)\;\big|\;d\big(\vr_t(\tt),\Ga(v(t))\big)\le C(B+1)\rho\,\big\}.
\end{align*}

\begin{figure}[h]
\resizebox*{13cm}{3.5cm}{\includegraphics{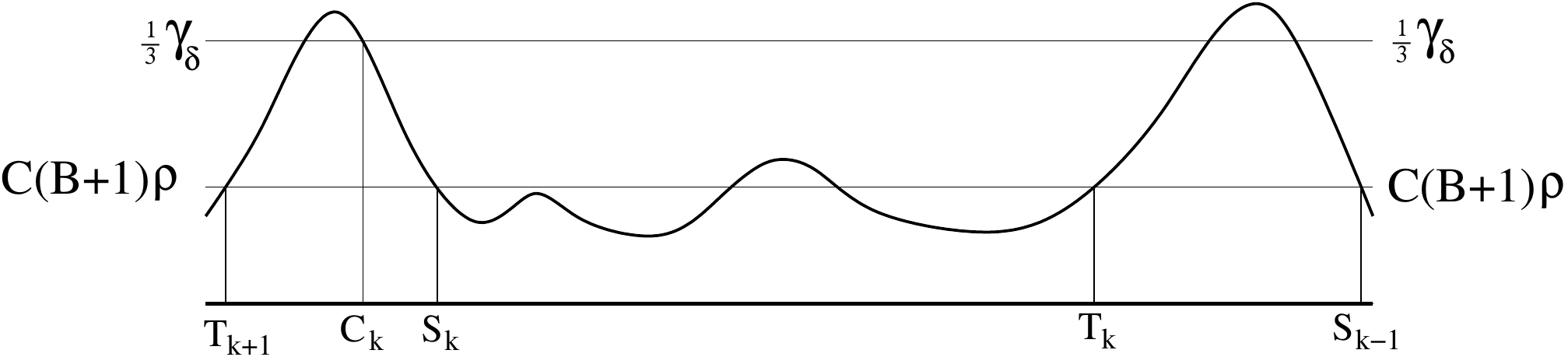}} 
\caption{This figure illustrates the distance of the orbit of $\tt$ to the
periodic orbit $\Ga$ and the choice of $C_k$, $S_k$ and $T_k$.
}\label{escapes}
\end{figure}

\begin{Claim}\label{stcs}\quad

\begin{enumerate}
\item\label{stcs1} If $S_{k-1}(\tt)>-\infty$ then $T_k(\tt)>-\infty$.
\item\label{stcs25} If $T_k(\tt)>-\infty$ then $T_{k+1}(\tt) \le C_k(\tt)$.
\item\label{stcs15} If $C_{k-1}(\tt)>-\infty$ then
$d\big[\vr_{T_k(\tt)}(\tt),\Ga(v(T_k(\tt)))\big]=C(B+1)\rho$.
\item\label{stcs2} If $C_k(\tt)>-\infty$ then $C_k(\tt)<S_k(\tt)\le T_k(\tt)$.
\item\label{stcs3} If the sequence $\{T_k\}$ is finite, then 
$\a\text{-limit}(x,\dx)=\Ga$.
\item\label{stcs6} If $t\in[S_k(\tt),T_k(\tt)]$ then $d(\vr_t (\tt),\Ga(\re))\le\tfrac 13\gad$.
\end{enumerate}

\end{Claim}

\noindent{\it Proof:}

\noindent\eqref{stcs1}. 
 Suppose by contradiction that $S_{k-1}(\tt)>-\infty$ 
 but $T_k(\tt)=-\infty$.
Let $\Phi^L_k$ be the action potential for $L$. 
Since $\Phi^L_{c(L)}$ is Lipschitz, it is bounded 
on $M\times M$.
\begin{align}
\int_{-t}^{S_{k-1}(\tt)}L(x,\dx)
=
\int_{-t}^{S_{k-1}(\tt)}\big\{c(L)+L(x,\dx)\big\} 
&\ge \Phi^L_{c(L)}\big(x(-t),x(S_{k-1}(\tt))\big)
\notag
\\
&\ge\inf_{y,z\in M}\Phi^L_{c(L)}(y,z)=:b_0 >-\infty.
\label{intb0}
\end{align}
Recall that $\eta$ is from the canonical coordinates~\ref{caco} for $\mN(L)$ as in~\eqref{cacoN}
and satisfies~\eqref{gadeta}.
Since $T_k(\tt)=-\infty$ we have that for all $t<S_{k-1}(\tt)$ either
\begin{equation}\label{caseta1}
d(\vr_t(\tt),\Ga(\re))>\eta>\gad
\qquad \text{ or }
\end{equation}  
\begin{equation}\label{caseta}
d(\vr_t(\tt),\Ga(\re))\le \eta\quad \text{ but }\quad
d(\vr_t(\tt),\Ga(v(t)))> C(B+1)\rho.
\end{equation}
In the case~\eqref{caseta}
let $s(t)$ be such that $d(\vr_t(\tt),\Ga(s(t)))=d(\vr_t(\tt),\Ga(\re))\le\eta$.
We have that
\begin{align}
\langle \Ga(s(t)),\vr_t(\tt)\rangle
&=W^{s}_{\text{loc}}(\Ga(s(t)))\cap W^{uu}_\eta(\vr_t(\tt))
\notag
\\
&=W^s_{\text{loc}}(\Ga(v(t)))\cap W^{uu}_\eta(\vr_t(\tt))
=\langle \Ga(v(t)),\vr_t(\tt)\rangle
\notag
\\
&=W^{ss}_\eta(\Ga(v(t)))\cap W^{uu}_\eta(\vr_t(\tt)).
\label{Gavt}
\end{align}
 We apply Lemma~\ref{B4} with $x:=\Ga(s(t))$ and $y:=\vr_t(\tt)$.
 Using~\eqref{Bdxy} we have that
 \begin{equation}\label{dyvrv}
 d(y,\vr_v(x))\le d(y,x)+ d(x,\vr_v(x))
\le (1+B)\,d(y,x).
\end{equation}
Observe that \eqref{Gavt} implies that $\vr_v(x)=\Ga(v(t))$.
Replacing $x$ and $y$ in~\eqref{dyvrv} 
and using~\eqref{caseta} we have that
\begin{align}
d(\vr_t(\tt),\Ga(\re))
&= d(\vr_t(\tt),\Ga(s(t)))
\ge \tfrac 1{1+B} \;d(\vr_t(\tt),\Ga(v(t)))
\notag
\\
&> C \rho.
\label{caseta3}
\end{align}
Observe that by~\eqref{rho14ga} and~\eqref{gabarde}, 
in case~\eqref{caseta1} inequality~\eqref{caseta3} 
also holds. Therefore
\begin{equation}\label{tsk}
\forall t<S_{k-1}(\tt) \qquad
d(\vr_t(\tt),\Ga(\re)) \ge C\,\rho.
\end{equation}

Since $x$ is semi-static for $L+\phi$ we have that 
\begin{align}
\infty>
\sup_{y,z}\Phi^{L+\phi}_{c(L+\phi)}(y,z)
&\ge
\Phi_{c(L+\phi)}^{L+\phi}\big(x(-t),x(S_{k-1}(\tt))\big)
\notag\\
&=\int_{-t}^{S_{k-1}(\tt)} \big[ L(x,\dx)+\phi(x)+c(L+\phi)\big]
\notag\\
&=\int_{-t}^{S_{k-1}(\tt)} L(x,\dx) 
+
\int_{-t}^{S_{k-1}(\tt)} 
\big[ \phi(x)+c(L+\phi) \big]
\notag\\
&\ge b_0 + a_0 \big(t+S_{k-1}(\tt)\big)
\qquad\text{ by \eqref{intb0} and \eqref{tsk}, \eqref{phirho}.}
\label{b+at}
\end{align}
By~\eqref{phirho} we have that $a_0>0$.
Letting $t\to+\infty$, inequality~\eqref{b+at} gives a contradiction.

\medskip
\noindent
\eqref{stcs25}.  
Let 
\begin{equation}\label{fgtv}
f(t):=d(\vr_t(\tt),\Ga(\re))
\qquad\text{  and }\qquad
g(t):=d(\vr_t(\tt),\Ga(v(t))),
\end{equation}
when $g$ is defined (in particular by~\eqref{gadeta} when $f(t)<\gad$).
Then $f(t)\le g(t)$.

Suppose first that $C_k(\tt)=-\infty$. 
Then $f(t)\ne \tfrac 13\gad$ for all $t<T_k(\tt)$.
By hypothesis  $T_k(\tt)>-\infty$, then 
$f(T_k(\tt))\le g(T_k(\tt))\le C(B+1)\rho$.
Therefore, by~\eqref{rho14ga}, \eqref{gabarde},
\linebreak
 $f(t)<\tfrac 13\gad$ for all $t<T_k(\tt)$.
By~\eqref{gadbe0} and  Proposition~\ref{B5}, $\vr_t(\tt)\in W^{uu}_\eta(\Ga(v(t)))$
and by Proposition~\ref{pHPS} 
$\lim_{t\to-\infty}g(t)=0$.
Then $S_k(\tt)=-\infty$ and also $T_{k+1}(\tt)=-\infty$.

Now suppose that $C_k(\tt)>-\infty$.
By the definition of $S_k(\tt)$ 
for all $t\in]C_k(\tt),S_k(\tt)[$
we have that 
$g(t)> C(B+1)\rho$.
 This implies that $T_{k+1}(\tt)\le C_k(\tt)$.

\medskip
\noindent\eqref{stcs15}.
Let $f,\,g$ be as in~\eqref{fgtv}.
By the hypothesis $C_{k-1}(\tt)>-\infty$ and by the  definition of $C_{k-1}(\tt)$,
 $C_{k-1}(\tt)\le T_{k-1}(\tt)$. Then $f(C_{k-1}(\tt))=\tfrac 13\gad$.
By~\eqref{rho14ga}, $C(B+1)\rho< \tfrac 13\gad$, and then
\begin{equation}\label{gcgfg}
C(B+1)\rho<\tfrac 13\gad=f(C_{k-1}(\tt))\le g(C_{k-1}(\tt)).
\end{equation}
By the definition of $S_{k-1}(\tt)$ we have that
$C_{k-1}(\tt)\le S_{k-1}(\tt)$.
But by~\eqref{gcgfg}, $g(C_{k-1}(\tt))\ge\tfrac 13\gad$
and by the definition of  $S_{k-1}(\tt)$, if $S_{k-1}(\tt)<+\infty$
then $g(S_{k-1}(\tt))\le C(B+1)\rho<\tfrac 13\gad$.
Therefore $C_{k-1}(\tt)\ne S_{k-1}(\tt)$ and then 
\begin{equation}\label{cksk}
C_{k-1}(\tt)<S_{k-1}(\tt)\le+\infty.
\end{equation}
By~\eqref{gcgfg}
and the definition of $S_{k-1}(\tt)$ we have that 
$$
\forall t\in  ]C_{k-1}(\tt),S_{k-1}(\tt)[
\qquad
g(t)> C(B+1)\rho.
$$
This implies that $T_k(\tt)< C_{k-1}(\tt)$,
with strict inequality by~\eqref{gcgfg}.
By~\eqref{cksk} and  item~\eqref{stcs1} we have that $C_{k-1}(\tt)>-\infty$ 
implies that $T_k(\tt)>-\infty$.
Therefore
\begin{equation}\label{tkck1}
-\infty<T_k(\tt)<C_{k-1}(\tt)<S_{k-1}(\tt).
\end{equation}
The definition of $T_k(\tt)$ and the continuity of $g(t)$
on its domain
 imply that 
 \begin{equation}\label{gcb1}
 g(T_k(\tt))\le C(B+1)\rho.
 \end{equation}
 The domain of definition and continuity of $g$ contains
 $f^{-1}(]0,\gad[)\supset g^{-1}(]0,\gad[)$.
By the intermediate value theorem for $g$ on connected components 
of $[g\le\gad]$
and~\eqref{tkck1}, \eqref{gcb1}, \eqref{gcgfg}, the image 
$g([T_k(\tt),C_{k-1}(\tt)])$, and hence also 
$g(] -\infty,S_{k-1}(\tt)[)$,
contain the closed interval 
$\big[C(B+1)\rho,\tfrac 13\gad\big]$.
Therefore, by the definition of $T_k(\tt)$, we have that 
$g(T_k(\tt))=C(B+1)\rho$.

\medskip
\noindent\eqref{stcs2}.
Let $f$, $g$ be from~\eqref{fgtv}.
 If $C_k(\tt)>-\infty$ then 
 by the definition of $C_k(\tt)$,
 \begin{equation}\label{ckletk}
 C_k(\tt)\le T_k(\tt).
 \end{equation}
 Therefore $T_k(\tt)>-\infty$.
 Then the definition of $T_k(\tt)$ implies that 
 \begin{equation}\label{gtkcbr}
 g(T_k(\tt))\le C(B+1)\rho.
 \end{equation}
 Since $f(t)$ is continuous, 
 \begin{equation}\label{fck3g}
 f(C_k(\tt))=\tfrac 13\gad.
 \end{equation}
By~\eqref{gtkcbr}, ~\eqref{rho14ga} and~\eqref{fck3g} we have that 
\begin{equation}\label{c43gf}
g(T_k(\tt))\le C(B+1)\rho<\tfrac 14\ga_\de<\tfrac13\gad=f(C_k(\tt))\le g(C_k(\tt)).
\end{equation}
This implies that $C_k(\tt)\ne T_k(\tt)$.
This together with \eqref{ckletk} imply that
\begin{equation}\label{ck<tk}
C_k(\tt)<T_k(\tt).
\end{equation}
By~\eqref{gtkcbr} and~\eqref{ck<tk} 
the value $S_k(\tt)$ is an infimum of a set which contains $T_{k}(\tt)$,
therefore 
\begin{equation}\label{skztk}
S_k(\tt)\le T_k(\tt).
\end{equation}
This proves the second inequality in item~\eqref{stcs2}.

The first of the following inequalities follows from the definition of $S_k(\tt)$.
The second inequality is~\eqref{skztk}. The third inequality follows from
the definition of $T_k(\tt)$.
\begin{equation}\label{csts}
C_k(\tt)\le S_k(\tt)\le T_k(\tt)\le S_{k-1}(\tt).
\end{equation}
We get that
$$
-\infty <C_k(\tt)\le S_k(\tt)\le S_{k-1}(\tt)\le\cdots \le S_0(\tt):=0<+\infty.
$$
From the definition of $S_k(\tt)$ and $S_k(\tt)<+\infty$, and then~\eqref{fck3g},
 we have that
$$
g(S_k(\tt))\le C(B+1)\rho<\tfrac 13\gad=f(C_k(\tt))\le g(C_k(\tt)).
$$
In particular $C_k(\tt)\ne S_k(\tt)$.
Thus from~\eqref{csts}, $C_k(\tt)<S_k(\tt)$.

\medskip
\noindent
\eqref{stcs3}. If the sequence $\{T_k\}$ is finite, there is $\ell\in\na$ such that 
$T_\ell>-\infty$ and $T_{\ell+1}=-\infty$. Let $f,\,g$ be as in~\eqref{fgtv}.
By item~\eqref{stcs25} we have that 
$-\infty< T_\ell(\tt)\le C_{\ell-1}(\tt)$.
Then we can apply item~\eqref{stcs15} and use~\eqref{rho14ga}
to obtain
\begin{equation}\label{ftltt}
f(T_\ell(\tt))\le g(T_\ell(\tt))=C(B+1)\rho<\tfrac 13\gad.
\end{equation}
Since $T_{\ell+1}(\tt)=-\infty$, by item~\eqref{stcs1}, $S_{\ell}(\tt)=-\infty$
and by item~\eqref{stcs2}, $C_{\ell}(\tt)=-\infty$. 
Since $C_{\ell}(\tt)=-\infty$  
we have that $f(t)\ne \tfrac 13\gad$ for all $t<T_\ell(\tt)$.
But by~\eqref{ftltt}, $f(T_\ell(\tt))<\tfrac 13\gad$.
Since $f(t)$ is continuous, using~\eqref{gadbe0}  we get that
$$
f(t)<\tfrac 13\gad<\be_0 \qquad\text{ for all } t<T_\ell(\tt). 
$$
This implies that there is a continuous function $s:]-\infty,T_\ell(\tt)]\to\re$ such that
$$
\forall t\le T_k(\tt) \qquad d\big(\vr_t(\tt),\Ga(s(t))\big)\le \be_0.
$$
By Proposition~\ref{B71} and Proposition~\ref{pHPS} there is $v\in\re$ and $\la>0$ such that 
$$
\forall t\le T_\ell(\tt)\qquad 
d(\vr_t(\tt),\Ga(t+v))\le D\,\be_0 \,\ee^{-\la (T_\ell(\tt)-t)}.
$$
This implies that $\lim\limits_{t\to +\infty}d(\vr_{-t}(\tt),\Ga)=0$ and that 
$\a\text{-limit}(\tt)=\Ga(\re)$.

\noindent\eqref{stcs6}.
We can assume that $T_k(\tt)>-\infty$, otherwise the statement is empty.
The definition of $T_k(\tt)$ implies that there is a sequence $t_n\uparrow T_k(\tt)$
such that $f(t_k)\le g(t_k)\le C(B+1)\rho$.
The continuity of $f(t)$ implies that $f(T_k(\tt))\le C(B+1)\rho<\tfrac13\gad$.
By the definition of $C_k(\tt)$ we have that $\forall t\in]C_k(\tt),T_k(\tt)]$ $f(t)\ne \tfrac 13\gad$.
Then by the continuity of $f(t)$, $\forall t\in]C_k(\tt),T_k(\tt)]$ $f(t)< \tfrac 13\gad$.
Now it is enough to see that by item~\eqref{stcs2}, $[S_k(\tt),T_k(\tt)]\subset ]C_k(\tt),T_k(\tt)]$.

 \hfill$\triangle$

Let
$$
B_k(\tt):=\sup\big\{\;t<C_k(\tt)\;\big|\; d(\vr_t(\theta),\Ga) \le \tfrac\gad 4\;\big\}.
$$ 
\begin{Claim}
$$
[B_k(\tt),C_k(\tt)]\subset [T_{k+1}(\tt),S_k(\tt)].
$$
\end{Claim}
\noindent{\it Proof:} \quad

Let $f,\,g$ be as in~\eqref{fgtv}.
By the definition of $S_k(\tt)$ we have that $S_k(\tt)\ge C_k(\tt)$.
By the definition of $B_k(\tt)$ and~\eqref{rho14ga}, we have that
\begin{equation}\label{gbkck}
g|_{]B_k,C_k[}\ge  f|_{]B_k,C_k[} > \tfrac 14 \gad > C(B+1)\rho.
\end{equation}
By the definition of $S_k(\tt)$ we have that 
\begin{equation}\label{gcksk}
g|_{]C_k,S_k[}> C(B+1)\rho.
\end{equation}
By the definition of $C_k(\tt)$ and the continuity of $f(t)$
we have that 
\begin{equation}\label{gck}
g(C_k(\tt))\ge f(C_k(\tt))=\tfrac 13\gad > C(B+1)\rho.
\end{equation}
Joining~\eqref{gbkck}, \eqref{gcksk} and \eqref{gck} we get that
$$
g|_{]B_k,S_k[}>C(B+1)\rho.
$$
By the definition of $T_{k+1}(\tt)$ this implies that $T_{k+1}(\tt)\le B_k(\tt)$.

\hfill$\triangle$

If $t\in[B_k(\tt),C_k(\tt)]$, by the definition of $B_k(\tt)$ we have that
$$
d(\vr_t(\tt),\Ga)\ge \tfrac 14 \gad .
$$
Then by \eqref{Crho},
\begin{equation}\label{tbkck}
t\in[B_k(\tt),C_k(\tt)] \quad\then\quad d(x(t),\pi\Ga) > \tfrac 14 \ogd.
\end{equation}
By the definition of $T_{k+1}(\tt)$ we have that 
\begin{equation}\label{dvgcb}
\forall t\in]T_{k+1}(\tt),S_k(\tt)[\qquad d\big(\vr_t(\tt),\Ga(v(t))\big)> C(B+1)\rho.
\end{equation}
The arguments in~\eqref{dyvrv}-\eqref{caseta3}
apply in the case~\eqref{dvgcb} to obtain
\begin{equation}\label{tk1sk}
t\in]T_{k+1}(\tt),S_k(\tt)[\quad\then\quad
d(\vr_t(\tt),\Ga)> C\rho.
\end{equation}
From \{\eqref{tbkck}, \eqref{defphi}\}, \eqref{cLp},
 \{\eqref{gatima}, \eqref{tima}\} and \{\eqref{tk1sk}, \eqref{phirho}\}, we have that
\begin{align}
\int_{T_{k+1}(\tt)}^{S_k(\tt)}\Big(\phi +c(L+\phi)\Big)
&\ge 
\int_{B_k(\tt)}^{C_k(\tt)}\Big( \tfrac 1{32}\, \e\, (\ogd)^2 -\tfrac{ JE}T \,\de^2\Big)
+ \int_{[T_{k+1},S_k]\setminus[B_k,C_k]} \Big( \phi +c(L+\phi) \Big)
\notag\\
&\ge\big( \tfrac 1{32}\, \e\, (\ogd)^2 -\tfrac{ JE}T \,\de^2\big)\, \a + 0.
\label{itk1sk}
\end{align}

Recall that
\begin{equation*}
\L:=L+\phi +c(L+\phi) -du,
\end{equation*}
where $u$ is from~\eqref{Ldu}.
Observe that the lagrangian flow for $\L$ is the same 
as the lagrangian flow $\vr_t$ for $L+\phi$.
Also $\mN(\L)=\mN(L+\phi)$ and $\cA(\L)=\cA(L+\phi)$.
Using \eqref{Ldu} and~\eqref{itk1sk},
\begin{align}
\int_{T_{k+1}(\tt)}^{S_k(\tt)} \L(\vr_t(\tt))\,dt
&= \int_{T_{k+1}(\tt)}^{S_k(\tt)}( L-du )
+\int_{T_{k+1}(\tt)}^{S_k(\tt)} \Big(\phi +c(L+\phi)\Big)
\notag
\\
&\ge 0+ \left(  \tfrac 1{32}\, \e\, (\ogd)^2 -\tfrac{JE}T \de^2\right) \a.
\label{farp}
\end{align}

\noindent
{\it Case 1:}
{\it Suppose that $T_k(\tt)-S_k(\tt)>T+2$.}

Let $m_k\in\na$ be such that
$$
S_k(\tt)+m_k T\le T_k(\tt)-1< S_k(\tt)+(m_k+1) T.
$$
Then $m_k\ge 1$.
Let $R_k(\tt):= S_k(\tt)+m_k T$.
Then $1\le T_k(\tt)-R_k(\tt) < T+1$.
By Claim~\ref{stcs}.\eqref{stcs6}
 $\Ga$ is $\tfrac\gad 3$-shadowed by $\vr_{[S_k,T_k]}(\tt)$. 
 Therefore by  inequality~\eqref{dextxy}  in Proposition~\ref{B71} there is $v\in\re$
  such that $ \forall t\in[S_k,T_k]$
 \begin{equation}\label{sktkd}
 d(\vr_t(\tt),\Ga(t+v))\le D\, \ee^{-\la\min\{t-S_k, T_k-t\}}
 [d(\vr_{S_k}(\tt),\Ga(S_k+v))
 +d(\vr_{T_k}(\tt),\Ga(T_k+v))].
 \end{equation}
 Also the choice of $v$ in Proposition~\ref{B71} is the same as in~\eqref{vttt}
 so that 
 \begin{equation}\label{tvvt}
 t+v = v(t) \qquad \forall t\in[S_k(\tt),T_k(\tt)].
 \end{equation}

\begin{figure}[h]
\resizebox*{8cm}{6cm}{\includegraphics{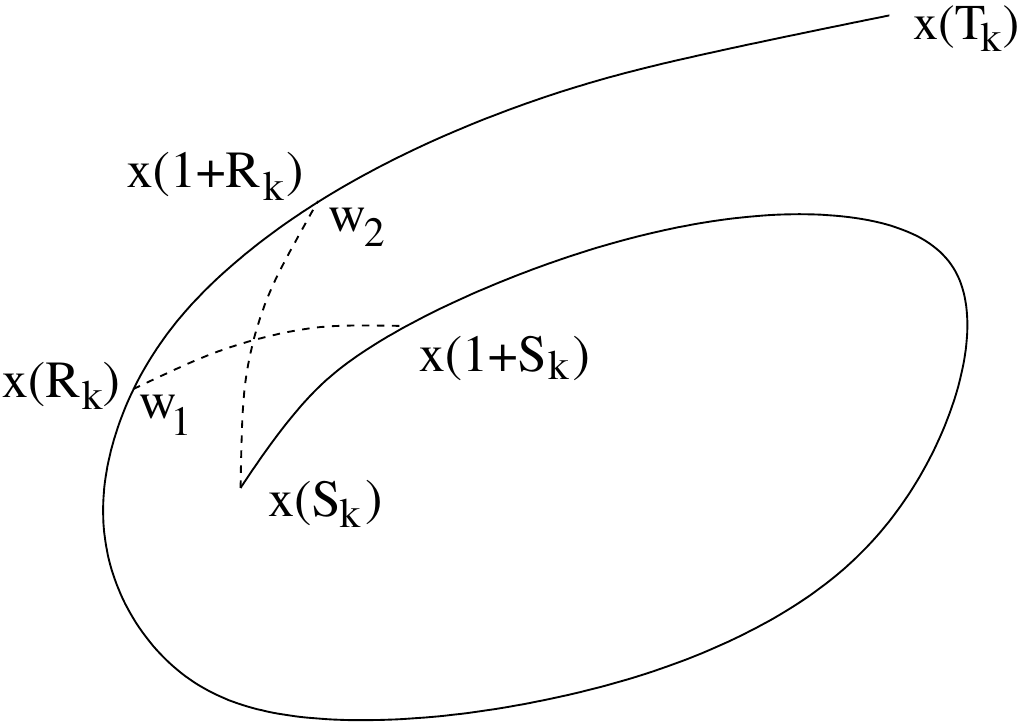}} 
\caption{The auxiliary segments $w_1$ an $w_2$.}
\label{near}
\end{figure}

By the definition of $S_k$ and $T_k$ and the continuity of $g(t)$ on its domain we have that
\begin{equation}\label{sktkm}
g(S_k)\le C(B+1)\rho, \qquad
g(T_k)\le C(B+1)\rho.
\end{equation}

By~\eqref{sktkd}, \eqref{tvvt} and \eqref{sktkm} we have 
for $s\in [0,1]$ that
\begin{align*}
d\big(\vr_{s+R_k}\tt &,\Ga(v(s+R_k))\big)\le
\\
&\le 
 D \ee^{-\la\min\{s+R_k-S_k,T_k-s-R_k\}}
\big[d(\vr_{S_k}\tt,\Ga(v(S_k))) + d(\vr_{T_k}\tt,\Ga(v(T_k)))\big] 
\notag \\
&\le D \,\ee^0\;[g(S_k)+g(T_k)]
\le 2 DC (B+1) \rho.
\notag\\
d(\Ga(v(s+&S_k)),\vr_{s+S_k}\tt) \le 2 D C(B+1)\rho. 
\end{align*}
From~\eqref{tvvt} we have that 
$$
v(s+R_k)=s+R_k+v=s+S_k+v+ m_k T =v(s+S_k) + m_k T.
$$
So that $\Ga(v(s+R_k))=\Ga(v(s+S_k))$.
Adding the inequalities above we get
\begin{equation}
\forall s\in[0,1]\qquad d(\vr_{s+R_k}\tt,\vr_{s+S_k}\tt) \le 4   DC(B+1)\rho.
\label{2cr}
\end{equation}

In local coordinates about $\pi(\Ga)$ define
\begin{alignat*}{6}
w_1(s+R_k) &= (1-&&s) \,&&x(s+R_k)  \;&+& &&s \, &&x(s+S_k), \qquad  s\in[0,1];
\\
w_2(s+S_k) &= &&s \, &&x(s+R_k) &+&\; (1-&&s)\, &&x(s+S_k),\qquad s\in [0,1].
\end{alignat*}
By Lemma~\ref{LKTr}\eqref{KTrb} and \eqref{2cr}
we have that
$$
A_{L+\phi}(x\vert_{[S_k,1+S_k]})+A_{L+\phi}(x|_{[R_k,1+R_k]})
\ge 
A_{L+\phi}(w_1)+A_{L+\phi}(w_2) -96 K D^2C^2 (B+1)^2\rho^2.
$$
Since the pairs of segments $\{\,x\vert_{[S_k,1+S_k]}, \, x|_{[R_k,1+R_k]}\,\}$
and $\{\, w_1,\,w_2\,\}$ have the same collections of endpoints
$$
\int_{S_k}^{1+S_k} du(\dx) + \int_{R_k}^{1+R_k} du(\dx)
= \oint_{w_1}du +\oint_{w_2}du.
$$
Therefore, since $c(L+\phi)$ is constant, 
\begin{equation}\label{crossingp}
A_\L(x\vert_{[S_k,1+S_k]})+A_\L(x|_{[R_k,1+R_k]})
\ge 
A_\L(w_1)+A_\L(w_2) -96 K  D^2 C^2 (B+1)^2 \rho^2.
\end{equation}

The integral of $d_xu$ on closed curves is zero. Therefore
\begin{equation}\label{cllo}
c(\L)=c(L+\phi+c(L+\phi))=0.
\end{equation}
Since $w_1*x|_{[1+S_k,R_k]}$ is a closed curve and $c(\L)=0$, 
\begin{equation}\label{closedp}
A_\L(w_1)+A_\L(x|_{[1+S_k,R_k]}) \ge 0.
\end{equation}
Using \eqref{Ldu} and \eqref{cLp},
\begin{equation}\label{LL>}
\L = (L-du) + \phi + c(L+\phi)\ge 0+0 -\frac{JE}T\,\de^2.
\end{equation}
Since $T_k(\tt)-R_k(\tt)\le T+2$, on the curve $w_2*x|_{[1+R_k,T_k]}$, 
using\eqref{T>1} we have that
\begin{equation}\label{restp}
A_\L(w_2)+A_\L(x|_{[1+R_k,T_k]}) \ge
 - \frac{JE}T\,\de^2\,(T+2) \ge - 3JE\,\de^2.
\end{equation}
From \eqref{crossingp}, \eqref{closedp} and \eqref{restp} we get that
\begin{alignat*}{3}
A_\L(x|_{[S_k,T_k]})
&\ge A_\L(w_1) 
&&+A_\L(w_2)-96KD^2 C^2(B+1)^2\rho^2
\\
&  &&+A_\L(x|_{[1+S_k,R_k]})+A_\L(x|_{[1+R_k,T_k]})
\\
&\ge -96 KD&&^2C^2(B+1)^2\rho^2 - 3JE\,\de^2.
\end{alignat*}
{\it Case 2:  If $T_k-S_k\le T+2$,} from~\eqref{LL>} we also have
\begin{align*}
\hskip -2cm
A_\L(x|_{[S_k,T_k]})&\ge - \frac{JE}T\,\de^2 (T+2) 
\ge - 3JE\,\de^2
\\
&\ge -96 KD^2C^2(B+1)^2\rho^2 - 3JE\,\de^2.
\end{align*}
Adding inequality \eqref{farp} and using \eqref{rho2} we obtain
a positive lower bound for the action independent of $k$:
$$
A_\L(x|_{[T_{k+1},T_k]}) 
\ge 
 \left(  \tfrac 1{32}\, \e\, (\ogd)^2 -\tfrac{JE}T \de^2\right) \a
  -96 K D^2C^2(B+1)^2\rho^2 - 3JE\,\de^2
  > 0.
$$

Since $x$ is also semi-static for $\L$ and by~\eqref{cllo}
$c(\L)=0$, the total action is finite:
 $$
 A_\L(x|_{]-\infty,0]})\le \max_{p,q\in M}\Phi^\L_{c(\L)}(p,q)<+\infty
 $$ 
 is finite. 
Therefore there must be at most finitely many $T_k$'s.

By item~\eqref{stcs3} in Claim~\ref{stcs},  we have that $\a$-limit$(x,\dx)=\Ga$.
Since $\a$-limits of semi-static orbits are static (Ma\~n\'e \cite[Theorem V.(c)]{Ma7}),
we obtain that $\Ga\subset\cA(L+\phi)$. This finishes the proof of Proposition~\ref{Ppert}.

\hfill\qed

\noindent{\bf Proof of Theorem~\ref{HYP}:}

By Theorem~C(a) in \cite{CP}
the set
$$
\cG^k_2:=\{\, \phi\in C^\infty(M,\re) \;|\;
\cM_{\min}(L_0+\phi)=\{\mu\} \text{ and } 
\cA(L_0+\phi) =\supp(\mu)\,\}
$$
is residual\footnote{The proof is the same for all $2\le k\le \infty$.} in $C^k(M,\re)$, $2\le k\le\infty$.
Since any invariant probability in the Aubry set is minimizing 
(c.f. Theorem~IV in Ma\~n\'e~\cite{Ma7},  and  in \cite{CDI}),
for $\phi\in\cG^k_2$ the set $\cA(L_0+\phi)=\supp(\mu)$ is uniquely ergodic.

Let $\cE_0$ be from Theorem~\ref{Tzeroentropy}.
We will prove that any $\phi\in \cE_0\cap\cG^2_2$ can be $C^2$ 
approximated by a potential
$\phi_1$ for which $\cA(L_0+\phi_1)$ contains a periodic orbit.

Fix $\phi\in\cE_0\cap\cG^2_2$. Suppose that $\phi$ can not be $C^2$ 
approximated by a potential
$\phi_1$ for which $\cA(L_0+\phi_1)$ contains a periodic orbit.
Write $L=L_0+\phi$ and let $\vr_t=\vr^L_t$ be the lagrangian flow of $L$.

Let 
$$
A:=1+\sup\nolimits_{|t|\le 1}\Lip\big(\vr_t|_{[E_L\le c(L)+1]}\big). 
$$
By Proposition~\ref{Ppert} with $J=2$ we have that

\medskip
\begin{Statement}\label{st1} \quad

\noindent
There are $\de_0>0$ and $Q>0$ such that for any periodic
$\de$-possible 1-specification in $\cA(L)$ with at most 2 jumps
$\{\, \vr_{[T_i,T_{i+1}]}(\tt_i)\,\}_{i=1,2}$ with $\de<\de_0$
there is an approach
$$
d(\vr_s(\tt_i),\vr_t(\tt_j)) < \tfrac 1{2A} Q\,\de,
$$
with $\{i,j\}\subset\{1,2\}$, $|s-t| (\text{mod }(T_3-T_1))\ge 1$ and 
$s\in[T_i,T_{i+1}]$, $t\in[T_j,T_{j+1}]$,
(if there is only one jump, $J=1$, we set $T_3=T_2$).
\end{Statement}

Let $\mu$ be the minimizing measure of $L$.
Fix a point $\vrt\in\cA(L)$ which is generic for $\mu$, 
i.e. for every continuous function
$F:TM\to\re$ 
$$
\lim_{T\to+\infty}\frac 1T \int_0^T F(\vr_t(\vrt))\,dt = \int F\,d\mu.
$$

Let $N_0\in\na$ be such that 
\begin{equation}\label{qno}
Q^{-N_0}<\de_0.
\end{equation}

Given $\tt\in\cA(L)$, let $\Si(\tt)$ 
be a small codimension 1 submanifold of $TM$
 transversal to $\vr$ containing $\tt$.
 Since  $\cA(L)$ has no periodic orbits we can choose
 $\Si(\tt)$ small enough such that its return time is larger than one, i.e.
 \begin{equation}\label{rtl1}
 \Si(\tt)\cap\vr_{]0,1]}(\Si(\tt))=\emptyset.
 \end{equation}
Given $N>N_0$ let $t_1^N(\tt)<t_2^N(\tt)<\cdots$ 
be all the $\tfrac 12 Q^{-N}$ returns near $\tt$ in $\Si$
of the orbit of $\vrt$, i.e.
\begin{equation}\label{12qn}
\{\,t_1^N(\tt),\,t_2^N(\tt),\ldots\,\}=
\big\{\,t>0\;\big|\; \vr_t(\vrt)\in\Si(\tt),\; d(\vr_t(\vrt),\tt)< \tfrac 12 Q^{-N}\,\big\}.
\end{equation}

We need the following result 
which will be proved in subsection~\ref{countingapprox}.

\begin{Proposition}\label{PertEnt}
For any $\tt\in\cA(L)$ and  $\ell\ge 1$, \quad
$t^N_{\ell+1}(\tt)-t^N_\ell(\tt) \ge \sqrt[3]{2}^{N-N_0-3}$.
\end{Proposition}

We continue the proof using Proposition~\ref{PertEnt}. Write
$$
B(\tt,r):=\{\,\om\in TM\;|\; d(\om,\tt)<r\,\}.
$$
In the definition of the return times $t^N_i(\tt)$ in~\eqref{12qn}
we shall use the following family of transversal sections.

 \begin{Claim}\quad
 
 If $\mN(L)$ has no periodic points then
 there is a family of local transversal sections $\{\Si(\tt)\}_{\tt\in\mN(L)}$
  and there are $N_3>N_2>N_0$ such that if $N>N_3$ then
 $$
 \forall\tt\in\cA(L),\qquad
 B(\tt,Q^{-N-N_2})\subset \vr_{[-.3,.3]}\big(\Si(\tt)\cap B(\tt,\tfrac 1{2A} Q^{-N})\big)
 $$
 and
 \begin{equation}\label{return>1}
 \forall \tt\in\mN(L), \qquad 
 \Si(\tt)\cap \textstyle \bigcup\limits_{0<|t|\le 1} \vr_t(\Si(\tt)) =\emptyset.
 \end{equation}

 \end{Claim}
 \noindent{\it Proof:}

 Since $\mN(L)$ has no periodic points, there is $r_0>0$ such that condition~\eqref{return>1} is 
 satisfied if $\diam \Si(\tt) < r_0$ for all $\tt\in \mN(L)$.
 Let $X(\tt)$ be the lagrangian vector field.
 There is $0<r_1<r_0$ such that for all
 $\tt\in\cA(L)$, setting
 $$
 \Si(\tt):=\exp_\tt\big(\{v\in T_\tt (E_L^{-1}\{c(L)\})\;:\;
 \langle v ,X(\tt)\rangle =0,\;|v|<r_1\}\big),
 $$
  we have that $\diam \Si(\tt)<r_0$, that  $\Si(\tt)$ is a transversal section to $\vr_t$ 
  in the energy level $E_L^{-1}\{c(L)\}$ and that
  $$
  W(\tt):=\vr_{[-.3,.3]}\big(B(\tt,r_1)\cap\Si(\tt)\big)
  $$
  is a tubular neighbourhood of  $\vr_{[-.3,.3]}(\tt)$. 
  There is $r_2>0$ such that for all $\tt\in\cA(L)$, $B(\tt,r_2)\subset W(\tt)$.
  Choose $N_4>N_0$ such that $Q^{-N_4}<r_2$ and hence
  $$
  \forall\tt\in\cA(L)\qquad B(\tt,Q^{-N_4})\subset W(\tt).
  $$
  
 The projection map 
 $P_\tt:W(\tt)\to \Si(\tt)\cap B(\tt,r_1)$, $P_\tt(\vr_\tau(\xi))=\xi$
 is smooth. We show that it has a uniform Lipschitz constant.
 Consider the map 
 $$F_\tt:[-.3,.3]\times(\Si(\tt)\cap B(\tt,r_1))\to W(\tt), \qquad
 F(t,\xi):=\vr_t(\xi).
 $$ 
 Then $\partial_t F_\tt(t,\xi)=X(\vr_t(\xi))$ and 
 $\partial_\xi F_\tt(t,\xi) =D\vr_t(\xi)$ is near the identity.
 The angle $\measuredangle(\Si(\tt),X(\xi))$, $\xi\in\Si(\tt)$, $\tt\in\cA(L)$
  is uniformly bounded away from zero. This implies that there 
  is a uniform bound  $a>0$ such that $|DF_\tt(t,\xi)\cdot\zeta|\ge a|\zeta|$
  for all $\xi\in\Si(\tt)$, $t\in[-.3,.3]$,  all $\tt\in\cA(L)$ and 
  $\zeta\in T_{(t,\xi)}\big(]-.3,.3[\times\Si(\tt)\big)$. By the inverse function theorem
  $\lV DF_\tt^{-1}(\vr_t(\xi))\rV\le a^{-1}$. Since $F_\tt^{-1}(\vr_t(\xi))
  =\big(*,P_\tt(\vr_t(\xi))\big)$ we have that $a^{-1}$ is a uniform 
  Lipschitz constant for all $P_\tt$, $\tt\in\cA(L)$.
 
 Choose $N_2>N_4$ such that $\tfrac 1{2A} Q^{N_2}> a^{-1}$ 
 so that $\tfrac 1{2A}Q^{N_2}$ is a Lipschitz 
 constant for $P_\tt$ on $W(\tt)$,  for all $\tt\in\cA(L)$. 
 Choose $N_3>N_2$. Then
 for $N>N_3$ we have that $N+N_2>N_4$ and hence 
 $B(\tt,Q^{-N-N_2})\subset W(\tt)$. Since $\Lip(P_\tt)\le \tfrac 1{2A} Q^{N_2}$ 
 we have that $P_\tt(B(\tt,Q^{-N-N_2}))\subset \Si(\tt)\cap B(\tt,\tfrac 1{2A} Q^{-N})$.
 Therefore
 $$
 B(\tt,Q^{-N-N_2})\subset (P_\tt|_{W(\tt)})^{-1}\big( \Si(\tt)\cap B(\tt,\tfrac 1{2A}Q^{-N})\big)
 =\vr_{[-.3,.3]}\big( \Si(\tt)\cap B(\tt,\tfrac 1{2A} Q^{-N})\big).
 $$
 \hfill$\triangle$
 
\medskip
 
 Thus if  $\xi\in B(\tt,Q^{-N-N_2})\cap \vr_{[.3,T-.3]}(\vrt)$ then there are 
 $\tau\in[-.3,.3]$
 and $s\in[0,T]$ such that $\vr_s(\vrt)\in \Si(\tt)\cap B(\tt,\tfrac 1{2A} Q^{-N})$
 and $\xi=\vr_\tau(\vr_s(\vrt))$.
 Observe that then $s=t_i^N(\tt)$ for some $t_i^N(\tt)\le T$.
 Therefore, denoting by $m$ the Lebesgue measure on $\re$, 
 \begin{align*}
 m\big\{\,t\in[{\scriptstyle .3,T-.3}]\,:\,\vr_t(\vrt)\in B(\tt,Q^{-N-N_2})\big\}
 &\le 0.6\,\#\{\,\ell\in\na \;:\;t_\ell^N(\tt)\le T\,\}.
 \end{align*}
 Similarly, using~\eqref{return>1},
   \begin{align*}
  m\Big\{t\in[{\scriptstyle.41,T-.41}]\,:\, \vr_t(\vrt)\in \vr_{[-.11,.11]}\big(B(\tt,Q^{-N-N_2})\big)\Big\}
 &\le 0.82\,\#\{\,\ell\in\na \;:\;t_\ell^N(\tt)\le T\,\}
 \\
 &\le \#\{\,\ell\in\na \;:\;t_\ell^N(\tt)\le T\,\}.
 \end{align*}
 
 Given $N\gg N_3$ and $\tt\in\cA(L)$, let $f_\tt:TM\to[0,1]$
  be a continuous function 
 such that
$f_\tt|_{\vr_{[-.1,.1]}\left(B(\tt,Q^{-N-N_2-1})\right)}\equiv 1$ and
 $\supp f_\tt\subseteq \vr_{[-.11,.11]}\left(B(\tt,Q^{-N-N_2})\right)$.
 Then
\begin{align*}
\int_{0.41}^{T-0.41} f_\tt(\vr_t(\vrt))\,dt
&\le m\Big\{t\in [{\scriptstyle.41,T-.41}]
\,:\, \vr_t(\vrt)\in\vr_{[-.11,.11]}\big(B(\tt,Q^{-N-N_2})\big)\Big\}
\\
 &\le \#\{\,\ell\in\na \;:\;t_\ell^N(\tt)\le T\,\}.
 \end{align*}

Using that $\vrt$ is a generic point for $\mu$ 
and Proposition~\ref{PertEnt}
we have that
\begin{align}
\mu\big(\vr_{[-.1,.1]}\big(B(\tt,Q^{-N-N_2-1})\big)\big) 
&\le \int f_\tt\, d\mu
\notag
=\lim_{T\to+\infty} \frac 1T \int_0^T f_\tt(\vr_t(\vrt))\, dt
\\
&=\lim_{T\to+\infty}\frac 1T\int_{0.41}^{T-0.41}f_\tt(\vr_t(\vrt))\,dt
\notag\\
&\le 
\limsup_{T\to+\infty}\frac 1T\,\#\big\{\,\ell\in\na \;|\; t^N_\ell(\tt) \le T \,\big\}
\notag\\
&\le 
\sqrt[3]{2}^{-N+N_0+3} \qquad
\forall \tt\in\cA(L).
\label{muB}
\end{align}

Fix a point $\om\in\supp(\mu)=\cA(L)$ 
for which Brin-Katok Theorem (see Theorem~\ref{BK} 
in  Appendix~\ref{AE} or \cite{BK})  
holds:
\begin{equation}\label{homp}
0=
h_\mu(\vr^L)=\lim_{\e\to 0}
\limsup_{T\to+\infty} -\frac 1T \, \log\big(\mu(V(\om,T,\e))\big),
\end{equation}
where $V(\om,T,\e)$ is the {\it dynamic ball}:
$$
V(\om,T,\e):= \{\, \xi\in TM\;|\; d(\vr_t(\xi),\vr_t(\om))\le \e, \quad \forall t\in[0,T]\;\}.
$$
Since $h_\mu(\vr^L)=0$, equation~\eqref{homp} is equivalent 
to 
$$
0=h_\mu(\vr^L)\ge\lim_{\e\to 0} \limsup\limits_{T\to+\infty} 
-\frac 1T \, \log\big(\mu(V(\om,T,\e))\big).
$$
Since the inner limit increases as $\e\downarrow 0$, the inequality holds
without taking $\lim_{\e\to 0}$:
\begin{equation}\label{hmet}
0=
h_\mu(\vr^L)=
\limsup_{T\to+\infty} -\frac 1T \, \log\big(\mu(V(\om,T,\e))\big).
\end{equation}
Fix $\e_0$ for which~\eqref{hmet} holds for every $\e<\e_0$.

By Proposition~\ref{B71} there are constants $D>0$ and $\la\in]0,1[$ such that if $\e$ is small
enough, $\tt\in\cA(L)$ and $d(\vr_s(\xi),\vr_s(\tt))\le \e$ for all $s\in[-T,T]$ then there is 
$|v|< D\e$ such that
\begin{equation}\label{delat}
\forall |s|\le T,\qquad
d(\psi_{s}(\xi),\vr_{s+v}(\tt))\le D \e\, \la^{(T-|s|)}.
\end{equation}
We can choose $\e<\e_0$ such that \eqref{delat} holds and $D\e<0.1$.

Since $\om\in\cA(L)$, from~\eqref{delat} we get that
\begin{equation}\label{vrtvw}
\vr_T\big(V(\om,2T,\e)\big)\subset \vr_{[-D\e,D\e]}\big(B(\vr_T(\om),D\e \la^T)\big).
\end{equation}
Take  $N=N(T)$ such that
\begin{equation}\label{eN}
Q^{-N-N_2-2}< D\e\,\la^T \le Q^{-N-N_2-1}.
\end{equation}
Then $B(\vr_T(w), D\e \la^T)\subseteq B(\vr_T(w),Q^{-N-N_2-1})$
and from~\eqref{vrtvw},
$$
\vr_T(V(w,2T,\e))\subset \vr_{[-.1,.1]}\big(B(\vr_T(w),Q^{-N-N_2-1})\big).
$$
Using~\eqref{muB} we have that
\begin{align*}
\frac{-1}{2T}\log\mu(V(w,2T,\e))
&=\frac{-1}{2T}\log\mu\big(\vr_T(V(w,2T,\e))\big)
\\
&\ge \frac{-1}{2T}\log\mu\big(\vr_{[-.1,.1]}(B(\vr_T(w),Q^{-N-N_2-1}))\big)
\\
&\ge \frac{-1}{2T}\log \sqrt[3]{2}^{-N+N_0+3}.
\end{align*}
Taking the limit when $T\to+\infty$ and using~\eqref{hmet} and  \eqref{eN},
we get
$$
h_\mu(\vr^L) \ge \frac{-\log\la \cdot\log\sqrt[3]{2}}{2\, \log Q}>0.
$$
This contradicts the choice of $\phi\in\cE_0$.

\qed

\bigskip

\color{black}

\subsection{Counting approximations.}
\label{countingapprox}\quad

Now we prove
\addtocounter{Thm}{-1}
\begin{Proposition}
For any $\tt\in\cA(L)$ and  $\ell\ge 1$, \quad
$t^N_{\ell+1}(\tt)-t^N_\ell(\tt) \ge \sqrt[3]{2}^{N-N_0-3}$.
\end{Proposition}

\begin{proof}\quad

  For $N\in \na$, let 
  \begin{align}\label{AAN}
  \A_N &:= \{(x,y)\in \cA(L)\times \cA(L)\;|\; d(x,y) \le Q^{-N}\}.
 \end{align}
Recall that
\begin{equation}\label{lipa}
A:=1+\sup\nolimits_{|t|\le 1}\Lip\big(\vr_t|_{[E_L\le c(L)+1]}\big). 
\end{equation}
From~\eqref{qno} and Statement~\ref{st1}, setting $\de=Q^{-N}$,  we get

\begin{Statement}\label{st2}\quad

If $N>N_0$ and  $\{\, \vr_{[T_i,T_{i+1}]}(\tt_i)\,\}_{i=1,2}$ is a periodic 
$Q^{-N}$-possible 1-specification in $\cA(L)$ with at most two jumps
then there is a $Q^{-N+1}$ approach
$$
d(\vr_s(\tt_i),\vr_t(\tt_j))< \tfrac 1{2A} Q^{-N+1}
$$
with $\{i,j\}\subset\{1,2\}$, $|s-t|(\text{mod}(T_3-T_1))\ge 1$ 
and $s\in[T_i,T_{i+1}]$, $t\in[T_j,T_{j+1}]$.
And with $T_3=T_2$, $i=j=1$  if there is only one jump.
\end{Statement}

\begin{Lemma}\label{applength}
If $\cA(L)$ has no periodic orbits and $N_0$ is large enough
then

\begin{equation}\label{dxiz}
\left\{
\begin{aligned}
&\xi,\zeta\in\cA(L),\;d(\xi,\zeta)<Q^{-N_0},
\\ &\si,\tau\in[0,1],\;\;2\ge \si+\tau\ge 1
\end{aligned}
\right\}
\quad\then\quad
d(\vr_{-\si}(\xi),\vr_\tau(\zeta))\ge Q^{-N_0},
\end{equation}
and
\begin{gather}
\text{if }\{\vr_{[a,b]}(\tt_1),\;\vr_{[b,c]}(\tt_2)\} 
\text{ is a periodic $Q^{-N_0}$-specification }
\notag\\
\text{ with $\tt_i\in\cA(L)$, $i=1,2$, \quad
and $|c-a|\ge 1$,}
\notag\\
\text{then }\max\{|b-a|,|c-b|\}>500.
\label{lpss}
\end{gather}

\end{Lemma}

\begin{proof}
If~\eqref{dxiz} does not hold there is a sequence $n\to+\infty$, points 
$\xi_n,\zeta_n\in\cA(L)$ and $\si_n,\tau_n\in[0,1]$ such that 
$d(\xi_n,\zeta_n)<Q^{-n}$, 
$2\ge \si_n+\tau_n\ge 1$ but $d(\vr_{-\si_n}(\xi_n),\vr_{\tau_n}(\zeta_n))< Q^{-n}$.
Since $\cA(L)$ is compact, extracting a subsequence we can assume that the limits
$\xi:=\lim_n \xi_n$, $\zeta:=\lim_n\zeta_n$, $\si:=\lim_n\si_n$, $\tau:=\lim_n\tau_n$ exist.
Then we have that $\xi=\zeta\in\cA(L)$,  $1\le \si+\tau\le 2$,
$\vr_{\si+\tau}(\zeta)=\xi=\zeta$. Therefore $\zeta$ is a periodic point in $\cA(L)$, 
which contradicts the hypothesis.

Now we prove~\eqref{lpss}.
If ~\eqref{lpss} does not hold then there is a sequence  $n\to+\infty$ and
periodic $Q^{-n}$-possible specifications $\{\vr_{[a_n,b_n]}(\tt^n_1),\,\vr_{[b_n,c_n]}(\tt^n_2)\}\subset\cA(L)$
such that $|c_n-a_n|\ge 1$, $|b_n-a_n|\le 500$ and $|c_n-b_n|\le 500$.
Translating the time intervals we can assume that $a_n=0$ for all $n$,
then $\forall n\;\;\{a_n,b_n,c_n\}\subset [0,500]$.
Since $\cA(L)$ is compact, extracting a  subsequence we can assume that 
the limits $\tt_1:=\lim_n \tt^n_1$, $\tt_2:=\lim_n \tt^n_2$,
$a:=\lim_n a_n$, $b:=\lim_n b_n$, $c:=\lim_n c_n$ exist
with $|c-a|\ge 1$ and $\tt_1\in\cA(L)$.       
Since $d(\vr_{b_n}(\tt^n_1),\vr_{b_n}(\tt^n_2))<Q^{-n}$
and $d(\vr_{c_n}(\tt^n_2),\vr_{a_n}(\tt^n_1))<Q^{-n}$,
we have that $\vr_b(\tt_1)=\vr_b(\tt_2)$, hence $\tt_1=\tt_2$, 
also $\vr_c(\tt_1)=\vr_c(\tt_2)=\vr_a(\tt_1)$ and $|c-a|\ge 1$.
Therefore $\tt_1$ is a periodic point in $\cA(L)$ which contradicts the hypothesis.

\end{proof}

Since $\cA(L)$ has no periodic orbits we can assume that $N_0$ is so large that
\eqref{dxiz} and \eqref{lpss} hold.

By~\eqref{12qn} and \eqref{rtl1} 
we have that $\vr_{[t^N_\ell(\tt),t^N_{\ell+1}(\tt)]}(\vrt)$ is
a periodic $Q^{-N}$-possible 
1-spec\-i\-fi\-ca\-tion 
in $\cA(L)$. 
Therefore there is a $Q^{-N+1}$ approach 
$d(\vr_s(\vrt),\vr_t(\vrt))< \tfrac 1{2A} Q^{-N+1}\le Q^{-N+1}$
with $1< (t-s)\mod(t^N_{\ell+1}-t^N_\ell)$, i.e.
$$
t^N_\ell(\tt)\le s<s+1<t<
\min\big\{t^N_{\ell+1}(\tt),\, (t^N_{\ell+1}(\tt)-1)+(s-t^N_\ell(\tt))\big\}.
$$ 
This implies that the length $t^N_{\ell+1}(\tt)-t^N_\ell(\tt)\ge 2$.
It also gives rise to two periodic $\frac 1{2A}Q^{-N+1}$-possible
specifications in $\cA(L)$ with at most 2 jumps. Namely 
$\vr_{[s,t]}(\vrt)$ and 
$\big\{\vr_{[t,t^N_{\ell+1}(\tt)]}(\vrt)$, $\vr_{[t^N_\ell(\tt),s]}(\vrt)\big\}$. 
At most one of $\{s, t\}$ is possibly at distance  $\le 1$  to the endpoints
$\{t^N_\ell(\tt)$, $t^N_{\ell+1}(\tt)\}$.
Say  $|t^N_{\ell+1}(\tt)-t|\ge |s-t^N_{\ell}(\tt)|$.
Suppose for example that also 
\linebreak
$|s-t^N_\ell(\tt)|>1$.
In this case both periodic $Q^{-N+1}$-specifications, 
$\vr_{[s,t]}(\vrt)$ and 
$\big\{\vr_{[t,t^N_{\ell+1}(\tt)]}(\vrt)$, $\vr_{[t^N_\ell(\tt),s]}(\vrt)\big\}$,
are 1-specifications.
 By  Statement~\ref{st2}, each of these specifications
  imply the existence of a new  $\frac 1{2A}Q^{-N+2}$ approach and 
we expect that each
approach adds a length of one to a lower bound for 
$t^N_{\ell+1}(\tt)-t^N_\ell(\tt)$ ...
This process will continue as long as the exponent 
$-N+k\le -N_0$. In this way we expect that the number of distinct
$Q^{-N+k}$ approaches in the segment $\vr_{[t^N_\ell(\tt),t^N_{\ell+1}(\tt)]}(\vrt)$  
grows exponentially with $k$ and then implying that the length of the segment
$t^N_{\ell+1}(\tt)-t^N_\ell(\tt)$ grows exponentially with $k$, giving the result
in Proposition~\ref{PertEnt}.

  \begin{figure}[h]
  \resizebox*{14cm}{4cm}{\includegraphics{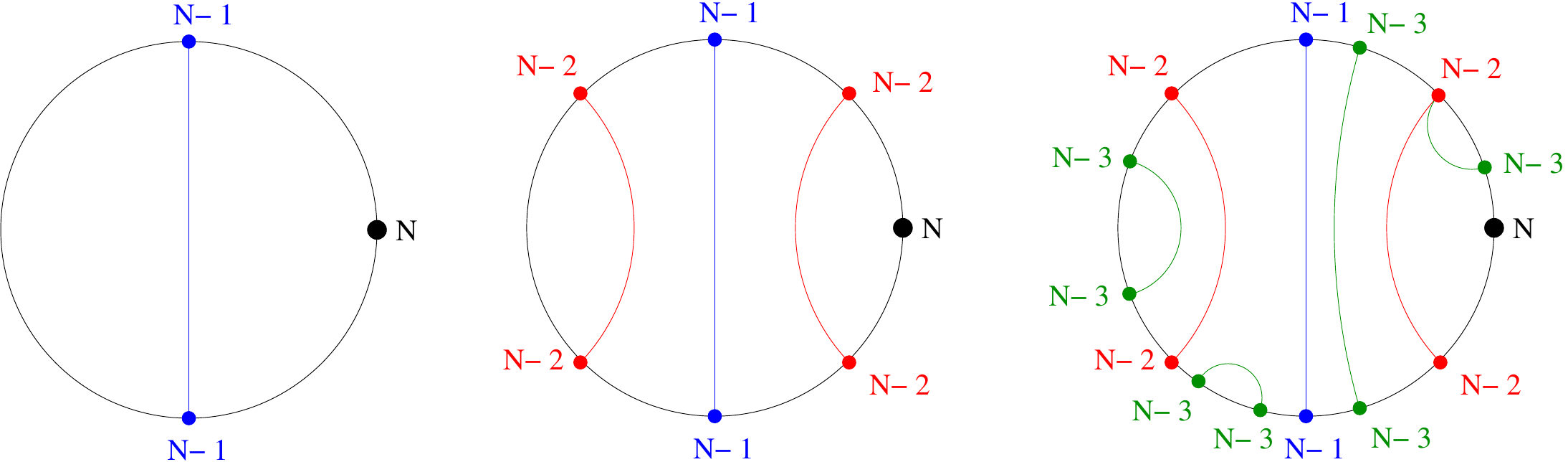}} 
   \caption{\small Example of a cascade of returns implied 
                  by the inductive process. }\label{ex0}
   \end{figure}

There are two possible problems we have to consider. The first one is when 
 some of the new approaches in this process
is near to the endpoints of the mother specification.
In this case  the process  will imply only one
or no new segment or length one in $[t^N_\ell(\tt),t^N_{\ell+1}(\tt)]$. 
When there is no new segment we will obtain one return 
giving rise to a specification with {\it only one jump}.
The other problem is when one of the
child specifications obtained in the process has more than two jumps.  
For such specification with three jumps we stop the process.
When this happens, we 
will see that the sister 1-specification has {\it only one jump}.
In both cases we get a 1-specification with only one jump,
this gives rise in the next generation 
to two specifications with at most two jumps, re-activating the duplication process.
In this way we can assure that in three generations this process gives a duplication 
of implied  segments of length one in the interval $[t^N_\ell(\tt),t^N_{\ell+1}(\tt)]$, obtaining
an exponential growth with rate $\sqrt[3]{2}$ as Proposition~\ref{PertEnt} claims.

      \begin{figure}[h]
     \resizebox*{6cm}{4cm}{\includegraphics{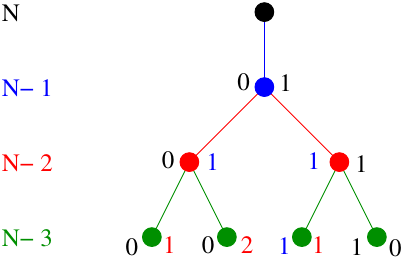}} 
     \hskip 1cm
 \resizebox*{6cm}{5cm}{\includegraphics{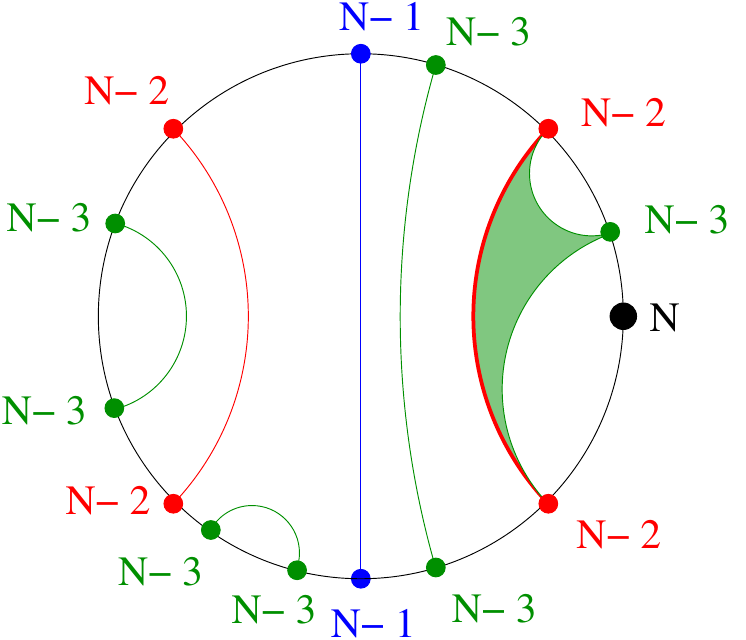}}
  \caption{\small An example of a distribution of approaches implied by Statement~\ref{st2}
   and the tree representing it. 
   The shadow is explained in subsection~\ref{0B1}.
   It corresponds to an approach with one of its points near 
   one endpoint of a previous specification.}\label{ex1}
\end{figure}

It is simpler to show the inductive process in a picture.
On a plane draw a circle $\SS$ representing the specification 
$\vr_{[t^N_\ell(\tt),t^N_{\ell+1}(\tt)]}(\vrt)$. 
On $\SS$ both points $\vr_{t^N_\ell}(\tt)$ and $t^N_{\ell+1}(\tt)$
are represented by a single black dot with level $N$.
Inside the  disk $\D$ with boundary $\SS$
draw a line from $\vr_s(\vrt)$ to $\vr_t(\vrt)$. 
It may be that the distance $d(s,\{t^N_\ell(\tt),t^N_{\ell+1}(\tt)\})\le 1$  but in that case 
$d(t,\{t^N_\ell(\tt),t^N_{\ell+1}(\tt)\})>1$ 
because by Proposition~\ref{Ppert},
$|s-t|{\mod(t^N_{\ell+1}-t^N_\ell)}\ge 1$
and then by~\eqref{lpss}
\begin{equation}\label{max5}
\max\big\{\quad d(s,\{t^N_\ell(\tt),t^N_{\ell+1}(\tt)\}),\quad d(t,\{t^N_\ell(\tt),t^N_{\ell+1}(\tt)\})\quad\big\}>5.
\end{equation}
The line
$\color{blue}\ov{\vr_s(\vrt)\vr_t(\vrt)}$ 
separates the disk in two 
components. 
If for example\footnote{The other case is treated in {\it Case $0\bullet 0$} in subsection~\ref{0B0}, but here we continue with
this example to illustrate the construction of the figures in the circle and the construction of the tree.}
 in \eqref{max5} also the minimum is larger than 1,
then each component in the disk $\D$ is a $Q^{-N+1}$-possible 
1-specification in $\cA(L)$ with at most 2 jumps
(one jump of size $\le Q^{-N+1}$ and possibly another
with size $\le Q^{-N}<Q^{-N+1}$).
By Statement~\ref{st2} each component has at least one $\frac 1{2A}Q^{-N+2}$ approach. The construction 
is iterated in this way.
The interior of the lines in this construction do not intersect.

   \begin{figure}[h]
     \resizebox*{6cm}{3.5cm}{\includegraphics{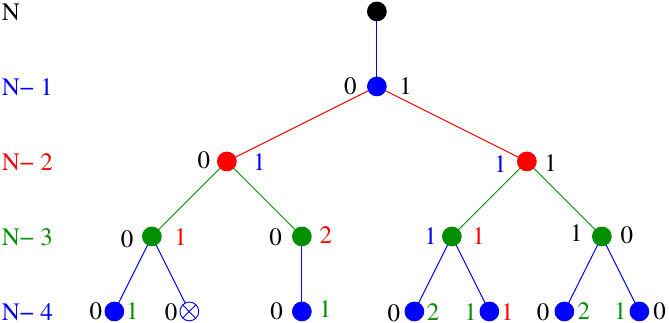}}
      \resizebox*{6cm}{4.9cm}{\includegraphics{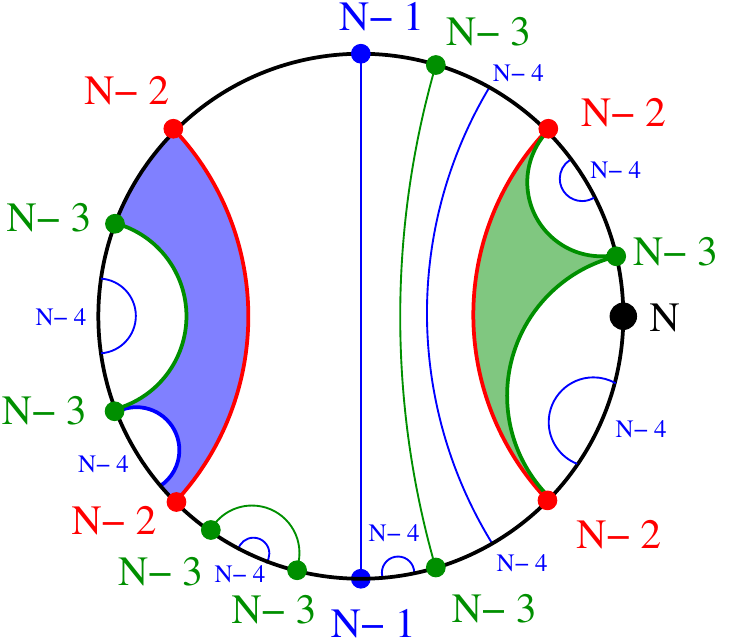}} 
      \caption{\small This is a possible next step from the example in Figure~\ref{ex1}.
      At level ${\color{dgreen}N-3}$ we had a node $0\color{dgreen}\bullet\color{red} 2$
       which only issues one branch  with label $0$. 
       At level ${\color{blue}N-4}$ the node {\color{blue}$0\otimes$ }corresponds to the shadowed region on 
      the left of the disc. This node comes from a branch with label $\color{red} 1$, i.e. a periodic specification with $2$ jumps. 
      In this case the implied return in $\color{blue}\A_{N-4}$ has both of its points at time distance $\le 1$ from the jumps of
      the specification. We put a white (or empty) node $\color{blue}\otimes$ in the tree, signifying that this node (approach) does not count as 
      another interval of length one in $[t^N_\ell(\tt),t^N_{\ell+1}(\tt)]$. We write a label $\color{blue}0$ in the node $\color{blue}0\otimes$
      meaning that the implied specification has only one jump. The node $\color{blue}0\otimes$ will issue only one branch 
      (with label $0$). We shadow the pentagonal region at the left  of the disc so that it is not considered  in the process
      any more.  After drawing the shadow it remains a white region
      with periodic specification with only one jump that will restore the duplication process.}
      \label{ex1a}
     \end{figure}

  We will also draw a tree with the approaches, in order to see 
   that their number grows exponentially.
   We put a black node $\bullet$ when we find a new approach which
   has at least one of its points  at distance at least 1 from
   all the points in the approaches considered earlier.
   In this way we shall obtain
   $$
    t^N_{\ell+1}-t^N_\ell \ge \#\{\text{black nodes}\}.
   $$
   It may be that the newly obtained approach
   has its two points at distance $\le 1$ from the previously
   obtained approaches. In this case we put a white (or empty) node
   $\otimes$ in the tree. The branches of the tree represent the new specifications
   which are implied by the approach in their upper node.

   \begin{Remark}[by Andrea Davini]\quad
   
   {\sl 
   Black and white nodes are used because in a black node we count a new end-interval of length one 
   in a new specification obtained in the process. White nodes appear when the points in a new implied
   approach are at small distance $(\le 1)$ to the endpoints of a previous approach.
   
   But we can  observe   that the specifications
   corresponding to the bottom level $N_0$ of the tree have disjoint interiors. 
   And then use Lemma~\ref{applength}, which says that any new specification in the inductive process
   must have length at least 500, to obtain
   $$
   t^N_{\ell+1}-t^N_\ell \ge 498\cdot\#\{\text{bottom level nodes}\};
   $$
   without caring about the end branches of the bottom specifications and neither
   if the nodes in the tree were black or white.
   
   }
    \end{Remark}

   An example of the inductive process appears in Figure~\ref{ex1}.
   The nodes of the tree are the approaches 
   implied by Statement~\ref{st2}.
   The height
    of the node bounds the size of the approach, namely $Q^{-N+k}$. 
    The numbers near a node
   are the quantity of approaches in upper levels of 
   the tree which are adjacent to the approach of the node, 
   either at its left or at its right. These numbers 
   are also equal to -1 + the quantity of jumps
   of the two new periodic specifications determined by the node. 
   Thus the numbers are associated to the branches issued from the node.
   The branches in the tree are the new periodic specifications of the next
   level $Q^{-N+k+1}$ obtained by cutting through the approach of the issuing node.
   Most of the nodes issue two branches but some of them issue only one branch. 
   These are the nodes $0\otimes$ which issue only one branch corresponding to the
   number $0$, i.e. a periodic specification with only one branch; and the node $0\bullet 2$
   which issues one branch corresponding to the number $0$ and no branch for the number $2$,
   because it is a periodic specification with $3$ jumps and we are stopping the process 
   at at most two jumps.

We now describe the possible nodes that end a branch with number $0$ in subsection~\ref{br0} and a branch 
with number $1$ in subsection~\ref{br1}

      \begin{figure}[h]
     \resizebox*{12cm}{6cm}{\includegraphics{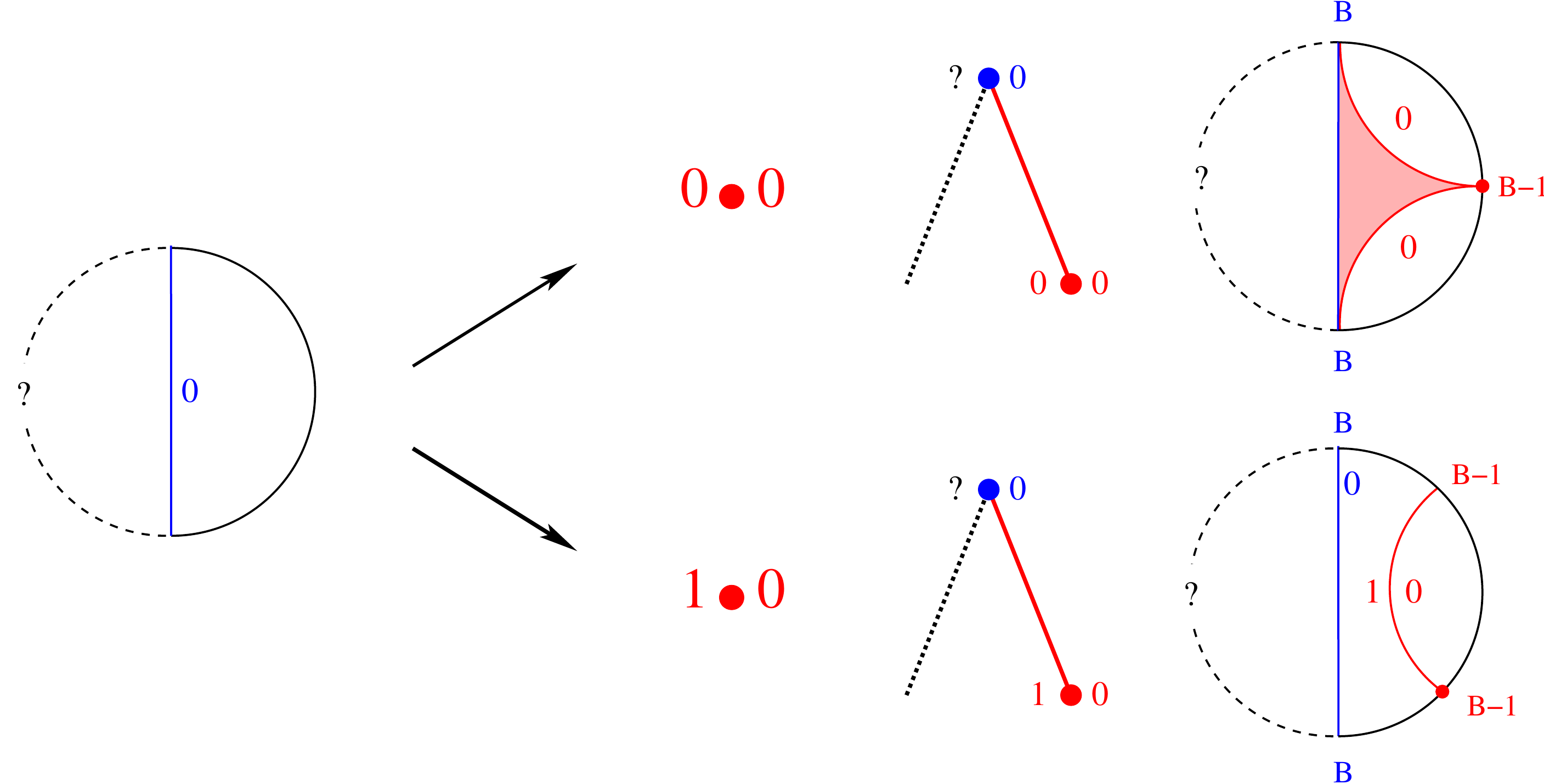}} 
  \caption{\small Possible nodes ending a branch with a label 0, i.e. child specifications
  of a periodic  1-specification with only one jump.}\label{case0}
\end{figure}

\subsection{Childs of a 1 jump periodic 1-specification.}\label{br0}\quad

This case appears in Figure~\ref{case0}.
Let $\vr_{[T_1,T_2]}(\vrt)$ be a $Q^{-B}$-possible 1-specification with only one jump.
Statement~\ref{st2} implies the existence of an approach $(\vr_a(\vrt),\vr_b(\vrt))\in\A_{B-1}$
where $T_1\le a<a+1<b\le T_2$ and
\begin{equation}\label{1ab}
d(\vr_a(\vrt),\vr_b(\vrt))< \tfrac1{2A}\, Q^{-B+1}.
\end{equation}

\bigskip

\subsubsection{Case $0\bullet 1$.} {\it When both points in the approach are at time interval $>1$ from the
endpoints of the specification, i.e. when $a-T_1>1$, $T_2-b>1$ (and $b-a>1$).}

In this case we obtain two child periodic $Q^{-B+1}$-possible 1-specifications: 
$\vr_{[a,b]}(\vrt)$ with 1-jump and $\{\vr_{[T_1,a]}(\vrt), \vr_{[b,T_2]}(\vrt)\}$
with 2-jumps. We write a node $0\bullet 1$ in the tree, representing a specification
with 1 jump and another with 2 jumps. The node is black $\bullet$ because at least
one (in this case two) of the points of the approach is at distance $\ge 1$ from the 
endpoints of the mother specification.

\subsubsection{Case $0\bullet 0$.} \label{0B0}
{\it When one of the points in the $\A_B$ approach\footnote{The set $\A_B$ is defined in \eqref{AAN}.}
 is at time interval
$\le 1$ from one of the endpoints of the mother specification.} 

Say $|a-T_1|\le 1$ and 
$$|b-a| >1.$$
We have that $\{\vr_{[T_1,a]}(\vrt),\vr_{[b,T_2]}(\vrt)\}$ is a periodic $Q^{-B+1}$-specification with
$$
|T_2-b|+ |a-T_1|\ge 1,
$$
then by \eqref{lpss}, 
\begin{equation*}\label{t2b10}
|T_2-b| >5.
\end{equation*}
 We will change the approach to $\ov{a}=T_1$, $\ov{b}=b-(a-T_1)=b+(T_1-a)$.
 Then
 \begin{gather}
 |\ov b -T_1|=|\ov{b}-\ov{a}|=|b-a|>1.
\label{b-a1} \\
|T_2-\ov b|\ge |T_2-b|-1 >1.
\label{t2b1}
\end{gather}
Using~\eqref{lipa} and~\eqref{1ab} we have that
\begin{align}
d\big(\vr_{\ov b}(\vrt),\vr_{T_1}(\vrt)\big) 
&\le \Lip(\vr_{T_1-a})\,d\big(\vr_b(\vrt),\vr_a(\vrt)\big) 
< A\cdot \tfrac 1{2A} Q^{-B+1} \le \tfrac 12 Q^{-B+1}.
\label{dobT1}\\
d\big(\vr_{\ov{b}}(\vrt),\vr_{T_2}(\vrt)\big)
&\le d\big(\vr_{\ov b}(\vrt),\vr_{T_1}(\vrt)\big) + d\big(\vr_{T_1}(\vrt),\vr_{T_2}(\vrt)\big)
\notag\\
&\le\tfrac 1{2} Q^{-B+1} + Q^{-B} < Q^{-B+1}.
\label{abB1}
\end{align}
By   \eqref{b-a1} and \eqref{dobT1} we have that $\vr_{[T_1,\ov{b}]}(\vrt)$
is a periodic $Q^{-B+1}$-possible 1-specification with only 1 jump. 
From~\eqref{t2b1} and~\eqref{abB1} we get that 
$\vr_{[\ov{b},T_2]}(\vrt)$ is another $Q^{-B+1}$-possible 1-specification with  only 1 jump. 
In the disc $\D$ we draw lines $\ov{T_1\ov{b}}$, $\ov{\ov{b} T_2}$ representing  approaches in $\A_{B-1}$.
Before we had drawn the line $\ov{T_1T_2}$ corresponding to a $Q^{-B}$ approach. 
These three lines bound a triangular region that we shadow as in Figure~\ref{case0}.
In the tree we add a node with numbers $0\bullet 0$ which represents two specifications with only 1 jump.
This node will issue two branches in the next step corresponding to these two child periodic specifications with
only one jump. The node is black $\bullet$ because there is a new point $\ov b$ in the approach which 
is at time interval at least 1 from the endpoints $T_1$ and $T_2$ of the mother specification.

     \begin{figure}[h]
     \resizebox*{12cm}{8.5cm}{\includegraphics{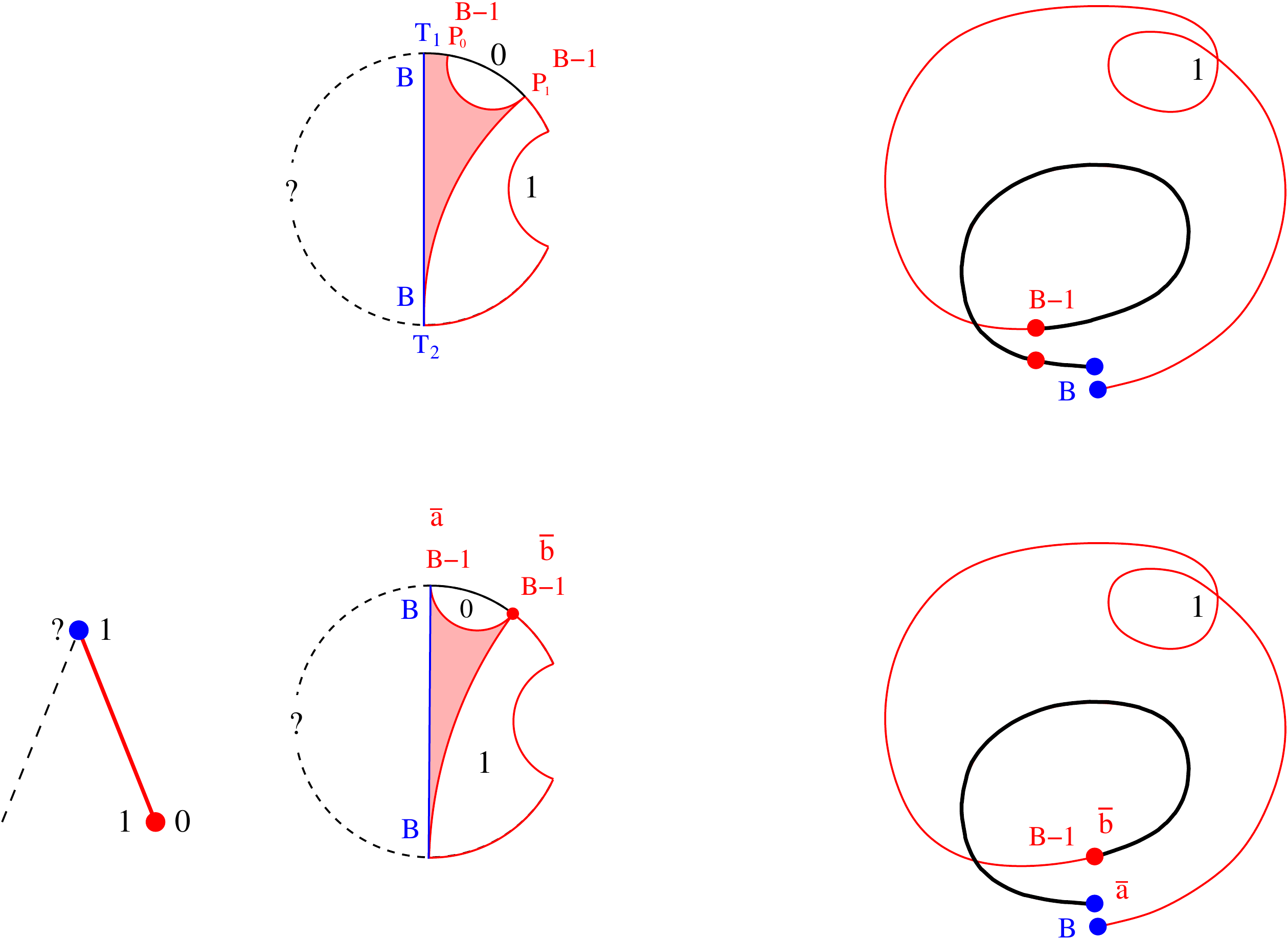}} 
     \caption{{\it Case $0\bullet 1$}. \small Here we study the case in which  the approach in $\A_{B-1}$ 
     implied by Statement~\ref{st2} contains a point {\color{red}$P_0$}
     at a time interval $0\le a\le1$ of one of the endpoints ${\color{blue}R_1}$ 
     of the mother periodic specification and the other point ${\color{red}P_1}$
     at a time interval $> 1+a$ from all the jumping points of the mother specification. 
     Let {\color{blue}$R_2$} be the other point in the approach to {\color{blue}$R_1$} in the mother specification, i.e.
     $d(R_1,R_2)<Q^{-B}$. If $P_0=\vr_a(R_1)$, we flow the two endpoints 
     to $\ov{P}_0:=\vr_{-a}(P_0)=R_1$ and $\ov P_1:=\vr_{-a}(P_1)$ and calculate that
     both $(\ov P_0,\ov P_1)$ and $(\ov P_1, R_2)$ are approaches in $\A_{B-1}$.
     They imply two child periodic $Q^{-B+1}$-possible 1-specifications. The approach $(\ov P_0,\ov P_1)$
     implies a specification with only one jump and hence we write the number $0$ in the next node.
     The approach  $(\ov P_1,R_2)$ implies a specification with two jumps, and hence we write 
     the number $1$ in the next node. Since the point $P_1$ is at distance $>1$ from all the previous endpoints
     (nodes) in the tree it counts as a new interval of length one in $[t^N_\ell(\tt),t^N_{\ell+1}(\tt)]$. So
     the next node in the tree is coded $1\bullet 0$. The number 0 issues a new branch corresponding 
     to a periodic specification with only $0+1$ jump and the number $1$ issues another new branch 
     corresponding to a periodic specification with $1+1$ jumps.
     }\label{shadow}
   \end{figure}

\subsection{Childs of a 2 jump periodic 1-specification.}\label{br1}\quad

Let $\{\vr_{[0,T_1]}(\vrt_1), \, \vr_{[T_1,T_2]}(\vrt_2)\}$ be a $Q^{-B}$-possible 1-specification with 2 jumps.
In particular
\begin{equation}\label{exqn}
d(\vrt_1,\vr_{T_2}(\vrt_2))<Q^{-B} \qquad\text{and}\qquad
d(\vr_{T_1}(\vrt_1),\vr_{T_1}(\vrt_2))<Q^{-B}.
\end{equation}
Statement~\ref{st2} implies the existence of an approach $(\vr_a(\vrt_i),\vr_b(\vrt_j))\in\A_{B-1}$ such that 
${|b-a|_{\text{\rm mod }T_2}> 1}$ and
\begin{equation}\label{1ab1}
d(\vr_a(\vrt_i),\vr_b(\vrt_j))< \tfrac1{2A}\, Q^{-B+1}.
\end{equation}
We can assume that
$\vrt_i=\vrt_1$ and $a\in [0,T_1]$.
The following cases appear in Figure~\ref{case1}.

The first two cases are when the distance $d(\{a,b\},\{0,T_1,T_2\})\ge 1$.
Interchanging $\vr_a(\vrt_1)$ and $\vr_b(\vrt_j)$ if necessary we can assume that
$0<a<a+1\le b< T_2$.
\subsubsection{Case $1\bullet 1$.}
{\it When $0< a < T_1< b < T_2$ and $d(\{a,b\},\{0,T_1,T_2\})\ge 1$.}

Since $|b-a|_{\text{\rm mod }T_2}\ge 1 $,
in this case we have two child $Q^{-B+1}$-possible 1-specifications with 2 jumps, namely
$\{\vr_{[0,a]}(\vrt_1), \vr_{[b,T_2]}(\vrt_2)\}$ and $\{\vr_{[a,T_1]}(\vrt_1),\vr_{[T_1,b]}(\vrt_2)\}$.
In the disc we draw a line $\ov{ab}$ 
and on the tree we add a node with
numbers $1\bullet 1$ and a black node $\bullet$.

      \begin{figure}[h]
     \resizebox*{14.8cm}{7.4cm}{\includegraphics{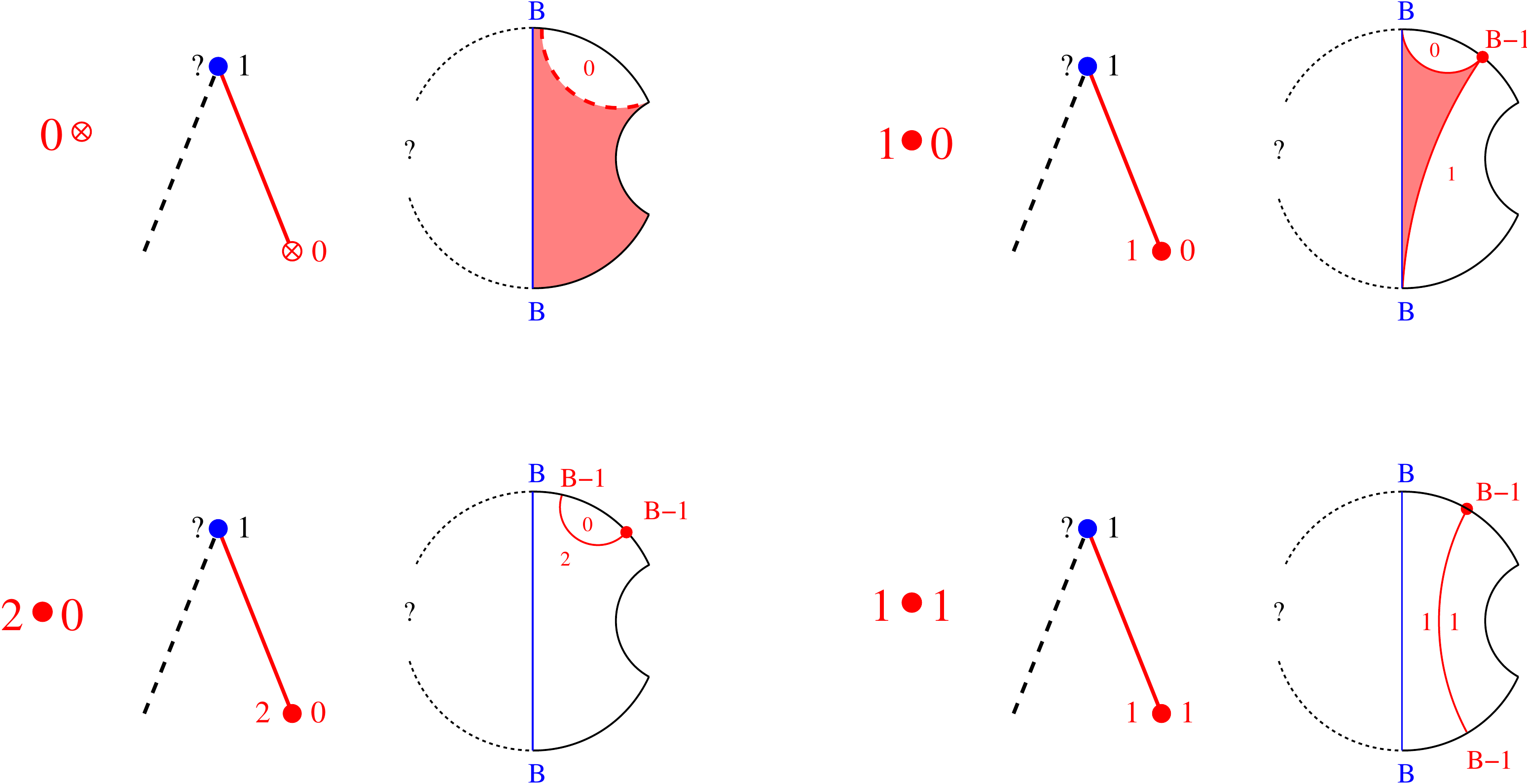}} 
  \caption{\small Possible nodes ending a branch with a label 1, i.e. child specifications
  of a periodic  1-specification with  two jumps.}\label{case1}
\end{figure}

\subsubsection{Case $0\bullet 2$.}
{\it When $0< a<a+1\le b < T_1< T_2$ and $d(\{a,b\},\{0,T_1,T_2\})\ge 1$.}

In this case $\vr_{[a,b]}(\vrt_1)$ is a $Q^{-B+1}$-possible periodic 1-specification
with only $1$ jump. Also 
\linebreak 
$\{\vr_{[b,T_1]}(\vrt_1), \vr_{[T_1,T_2]}(\vrt_2),\vr_{[0,a]}(\vrt_1)\}$ is a
$Q^{-B+1}$-possible 1-specification with 3 jumps. In the tree write a node with 
numbers $0\bullet 2$. We will not follow the the specification with 3 jumps.
From this node $0\bullet 2$ the tree will only have one branch corresponding to the number 0,
the specification with only 1 jump.

\bigskip

Now we study the cases with  $d(\{a,b\},\{0,T_1,T_2\})<1$.
Without loss of generality we can assume that $a\in[0,T_1]$,
$d(a,\{0,T_1\})<d(b,\{0,T_1,T_2\})$, $d(a,\{0,T_1\})<1$.
Now $a\in[0,T_1]$, $|a-b|_{\text{mod }T_2}\ge 1$ but $b\in[0,T_2]$.
In this case the approach point $\vr_a(\vrt_1)$ will not
count as a new interval of length one inside $[t^N_\ell(\tt),t^N_{\ell+1}(\tt)]$.
Define
\begin{align*}
\ov a &:= \begin{cases}
0 &\text{if }\quad  a < 1\le T_1-a,\\
T_1 &\text{if }\quad a\ge 1 > T_1-a.
\end{cases}
\\
\ov b&:=\begin{cases}
b-a &\text{if }\quad  a< 1\le T_1-a,\\
(b+T_1-a)_{\text{mod }T_2} &\text{if } \quad a\ge 1 > T_1-a.
\end{cases}
\end{align*}
We have that 
\begin{equation}\label{bage1}
|\ov a-\ov b|_{\text{mod }T_2}=|a-b|_{\text{mod }T_2}\ge 1.
\end{equation}
Since $d(a,\{0,T_1\})<d(b,\{0,T_1,T_2\})$ we have that  $\{b,\ov b\}\subset [0,T_1]$ or $\{b, \ov b\} \subset [T_1,T_2]$
and, using that $|\ov a-a|<1$, 
\begin{align}
d(\vr_{\ov a}(\vrt_1),\vr_{\ov b}(\vrt_j))&=
d(\vr_{\ov a}(\vrt_1),\vr_{b+\ov a-a}(\vrt_j))\notag \\
&\le  \Lip(\vr_{\ov a-a})\; d(\vr_a(\vrt_1),\vr_b(\vrt_j))
< A\cdot \tfrac1{2A} Q^{-B+1} \le \tfrac 12 Q^{-B+1}.
\label{doaob}
\end{align}

Define a new approach  $\vr_{c}(\ov\vrt_1)$ (see Fig.~\ref{fcase0b1}) 
by $\ov\vrt_1:=\vrt_j$ and $c\in\{0,T_1,T_2\}$ is such that
$$
\vr_{c}(\ov\vrt_1):=
\begin{cases}
 \vrt_1\quad &=\vr_{\ov a}(\vrt_j)\qquad
 \text{if } \ov a = 0 \hskip 5pt\text{ and }\vrt_j=\vrt_1,\\
 \vr_{T_1}(\vrt_1)  &=\vr_{\ov a}(\vrt_j)\qquad 
 \text{if } \ov a = T_1 \text{ and }\vrt_j=\vrt_1,\\
  \vr_{T_1}(\vrt_2) &=\vr_{\ov a}(\vrt_j)\qquad
  \text{if } \ov a = T_1 \text{ and }\vrt_j=\vrt_2,\\
\vr_{T_2}(\vrt_2)   &\phantom{=\vr_{\ov a}(\vrt_j)\qquad}
 \text{if } \ov a = 0 \hskip 7pt\text{ and }\vrt_j=\vrt_2,\\
\end{cases}
$$

Using~\eqref{doaob} have that
\begin{align}
d(\vr_{c}(\ov\vrt_1),\vr_{\ov b}(\vrt_j))
&\le d(\vr_{c}(\ov\vrt_1),\vr_{\ov a}(\vrt_1))+
d(\vr_{\ov a}(\vrt_1),\vr_{\ov b}(\vrt_j))
\notag\\
&\le \max\big\{d(\vr_{T_2}(\vrt_2),\vrt_1),d(\vr_{T_1}(\vrt_1),\vr_{T_1}(\vrt_2))\big\}
+
d(\vr_{\ov a}(\vrt_1),\vr_{\ov b}(\vrt_j))
\notag\\
&\le Q^{-B} +\tfrac 12 Q^{-B+1} < Q^{-B+1}.
\label{appab}
\end{align}
So that $\big(\vr_{c}(\ov\vrt_1),\vr_{\ov b}(\vrt_j)\big)$
is an approach in $\A_{B-1}$.
Moreover, we have chosen the approach $\vr_{c}(\ov\vrt_1)$ 
so that it is in the same orbit segment as $\vr_{\ov b}(\vrt_j)$ because $\ov\vrt_1=\vrt_j$.
Also
\begin{equation*}\label{obcge1}
|\ov b-c|=
\begin{cases}
|\ov b - \;\ov a\,|=|\ov b-\ov a|_{\text{mod }T_2}\ge 1 &\text{when $c=\ov a$},
\\
|\ov b-T_2|=|\ov b -\ov a|_{\text{mod }T_2}\ge 1
&\text{when }\ov a = 0\text{ and } c=T_2.
\end{cases}
\end{equation*}
Then we have proved
\medskip
\begin{Statement}\label{0appr}\quad
\newline
The segment $\vr_{[c,\ov b]}(\vrt_j)$ if $c<\ov b$,
(or $\vr_{[\ov b, c]}(\vrt_j)$ if $\ov  b<c$),
is a periodic $Q^{-B+1}$-possible 1-specification with only one jump.
\end{Statement}

\bigskip

Define
$$
\vrt^*_1:=\begin{cases}
\vr_{T_2}(\vrt_2) &\text{if } \vr_c(\ov\vrt_1)=\vrt_1,
\\
\vr_{T_1}(\vrt_2) &\text{if } \vr_c(\ov\vrt_1)=\vr_{T_1}(\vrt_1),
\\
\vr_{T_1}(\vrt_1) &\text{if } \vr_c(\ov\vrt_1)=\vr_{T_1}(\vrt_2),
\\
\vrt_1 &\text{if } \vr_c(\ov\vrt_1)=\vr_{T_2}(\vrt_2).
\end{cases}
$$	

  \begin{figure}[h]
   \psfrag{A12}{$vr_{\ov b}(\vrt_j)$}
  \resizebox*{14cm}{9.8cm}{
  \includegraphics{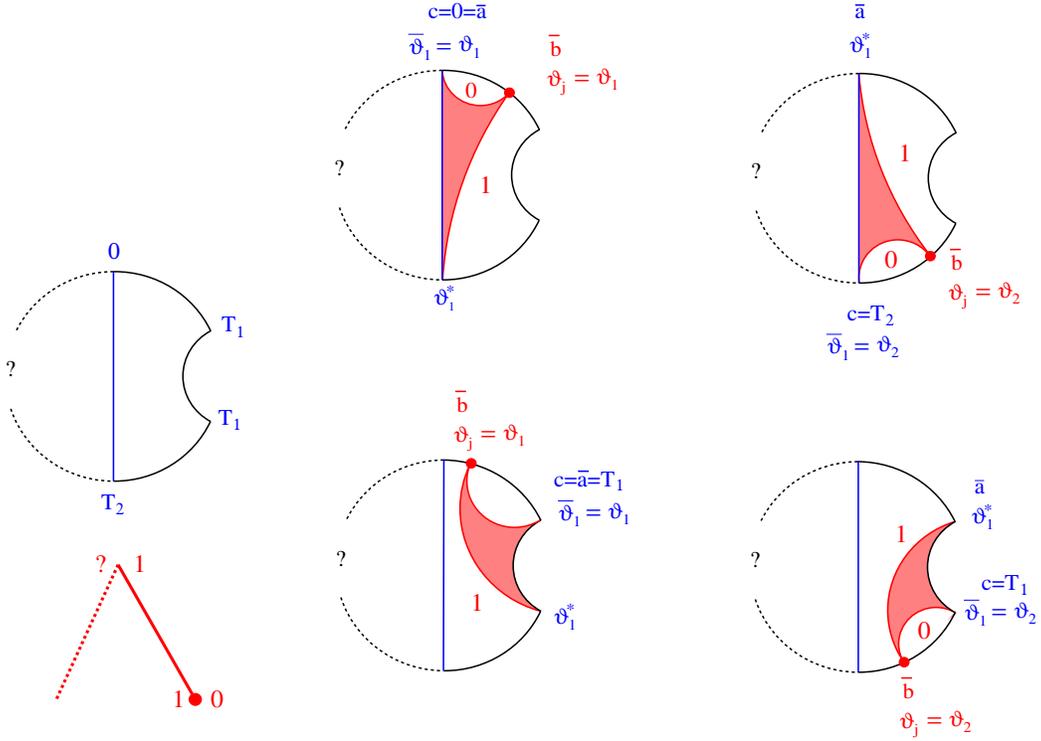}
  } 
   \caption{Case $0\bullet 1$. \small 
   In this case the red dot $\ov b$ is at time distance
   $d(\ov b,\{0,T_1,T_2\})\ge 1$ from the endpoints of its orbit segment. By \eqref{appab} 
   and~\eqref{v*b} every pair of vertices in the shaded triangles are $\A_{B-1}$ approaches.
   By Statement~\ref{0appr}, one side of the triangles define a $Q^{-B+1}$-possible 1-specification
   with only one jump. By the condition  $d(\ov b,\{0,T_1,T_2\})>1$, the other two pieces of the mother specification
   define a periodic $Q^{-B+1}$-possible 1-specification with two jumps. In the tree we write a black node with
   label $0\bullet 1$. The node will have two branches corresponding to the labels $0$ and $1$.
   }
   \label{fcase0b1} 
   \end{figure}

If $\vrt_1^*=\vr_{\ov a}(\vrt_1)$, then using~\eqref{doaob},
\begin{equation}
d(\vrt_1^*,\vr_{\ov b}(\vrt_j))=d(\vr_{\ov a}(\vrt_1),\vr_{\ov b}(\vrt_j))
\le \tfrac 12 Q^{-B+1}.
\label{v1*1}
\end{equation}
If $\vrt_1^*\ne\vr_{\ov a}(\vrt_1)$
then
\begin{align}
d(\vrt_1^*,\vr_{\ov b}(\vrt_j))
&\le d(\vrt_1^*,\vr_{\ov a}(\vrt_1))
+d(\vr_{\ov a}(\vrt_1),\vr_{\ov b}(\vrt_j))
\notag\\
&\le \max\{d(\vr_{T_2}(\vrt_2),\vrt_1),d(\vr_{T_1}(\vrt_1),\vr_{T_1}(\vrt_2))\}
+d(\vr_{\ov a}(\vrt_1),\vr_{\ov b}(\vrt_j))
\notag\\
&\le Q^{-B} +\tfrac 12 Q^{-B+1} < Q^{-B+1}.
\label{v1*2}
\end{align}
From~\eqref{v1*1} and~\eqref{v1*2}
we get that 
\begin{equation}
(\vrt_1^*,\vr_{\ov b}(\vrt_j))\in\A_{B-1}.
\label{v*b}
\end{equation}

\bigskip

\subsubsection{Case $0\bullet 1$.} \label{0B1}
{\it When $d({\ov b},\{0,T_1,T_2\})\ge 1$.}

Say that $ \vr_{\ov a}(\vrt_1)=\vrt_1$, the other case ($\vr_{\ov a}(\vrt_1)=\vr_{T_1}(\vrt_1)$) is similar.
This case $ \vr_{\ov a}(\vrt_1)=\vrt_1$ is described in the two upper pictures in Figure~\ref{fcase0b1}.
In this case we count $\vr_{\ov b}(\vrt_j)$ as a new interval of length 1
in $[t^N_\ell(\tt),t^N_{\ell+1}(\tt)]$ and we write a black dot $\bullet$ in the tree.
 If $0<\ov b <T_1$ (and $ \vrt_j=\vrt_1$, $\ov a=0$)
then by~\eqref{doaob}, $\vr_{[0,\ov b]}(\vrt_1)$ is a periodic $Q^{-B+1}$-possible 1-specification
with only 1 jump, and we write  $0\bullet$ in the  node in the tree. In this case $0<\ov b <T_1$
we also have that $\{\vr_{[\ov b,T_1]}(\vrt_1),\; \vr_{[T_1,T_2]}(\vrt_2)\}$ is
a  $Q^{-B+1}$-possible 1-specification with 2 jumps. We complete the information
writing $0\bullet 1$ in the node in the tree. In the disc we draw a line  
$\ov{\vrt_1\,\vr_{\ov b}(\vrt_1)}$ and another line $\ov{\vr_{\ov b}(\vrt_1)\,\vr_{T_2}(\vrt_2)}$, which 
is also $\ov{\vr_{\ov b}(\vrt_1)\vrt_1^*}$.
By \eqref{doaob} and \eqref{v*b} respectively, both lines are approaches in $\A_{B-1}$.
These lines together with the previously drawn line $\ov{\vr_{T_2}(\vrt_2)\,\vrt_1}$ bound 
a triangle that we shadow. The shadow means that we will not consider new approaches 
whose lines enter in the shadow and also that the three pairs corresponding to the edges of the
triangles are approaches in $\A_{B-1}$. The shadow gives a node in the tree with two sides,
left and right, chosen at the readers will, with numbers which are the amount of jumps seen
in the region of the disc separated by the shadow, in this case $0\bullet 1$.
Each number which is less than 2 in turn issues a new branch in the tree.

\begin{figure}[h]
   \psfrag{A12}{$vr_{\ov b}(\vrt_j)$}
  \resizebox*{14cm}{9.8cm}{
  \includegraphics{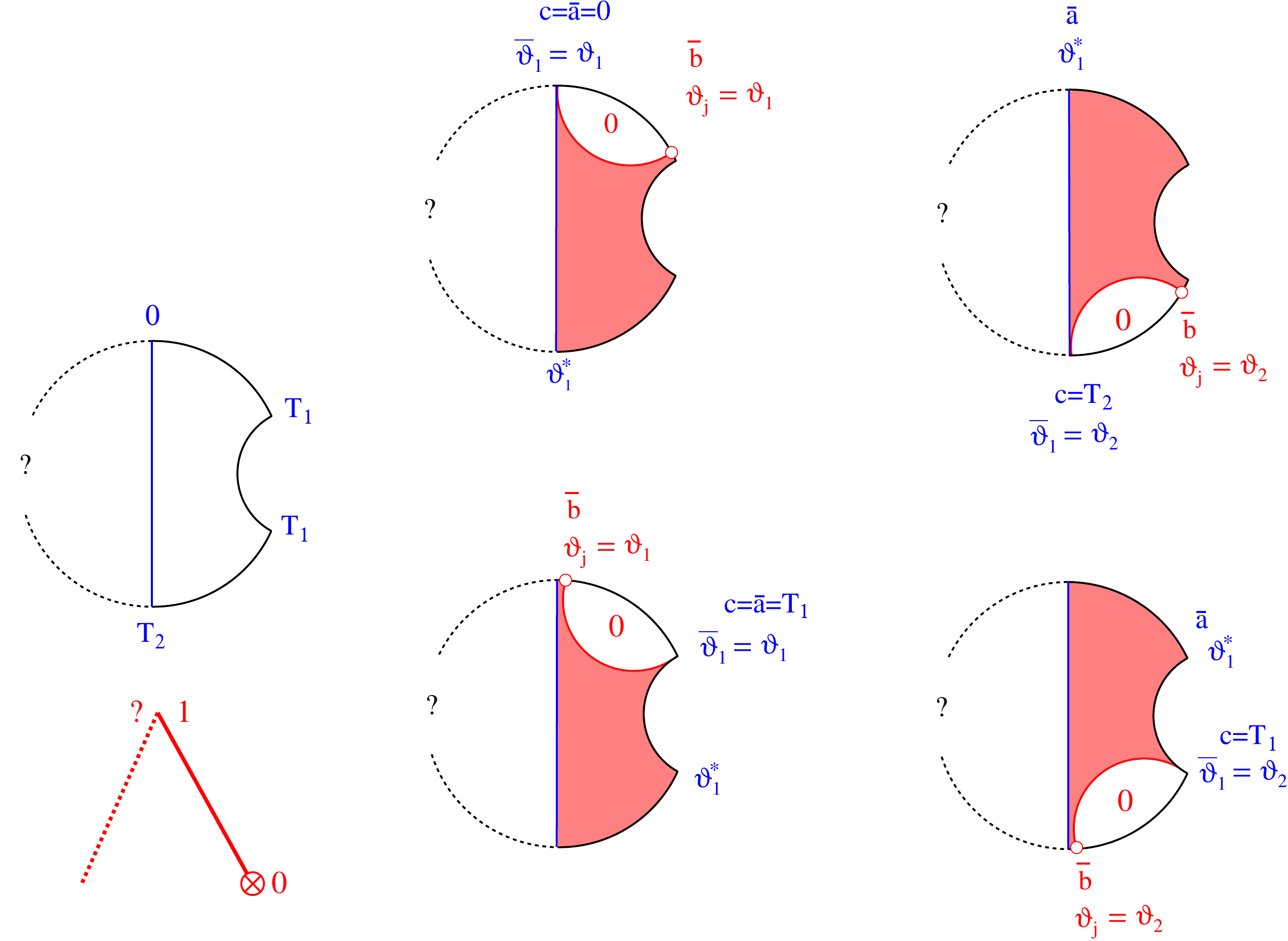}
  } 
   \caption{Case $0\otimes$. \small 
   In this case the white dot $\ov b$ is at time distance
   $d(\ov b,\{0,T_1,T_2\})\le 1$ from the endpoints of its orbit segment.
   It could even be that $\{a,b\}\subset\{0,T_1,T_2\}$.
   By Statement~\ref{0appr} the approach $(\vr_{c}(\vrt_j),\vr_{\ov b}(\vrt_j))\in \A_{B-1}$
   always determines a new periodic $Q^{-B+1}$-possible 1-specification with 
   {\it only one jump } which was
   not considered before.
   For simplicity we forget about the rest of the mother specification and shadow the 
   region delimited by it. In the tree we label the node with $0\otimes$. This node will issue
   a branch with label $0$. From subsection~\ref{br0} and  Figure~\ref{case0} such branch with label $0$
   (i.e. a 1-specification with only one jump) will restart the duplication process because it always 
   have two childs
   because its node is either $0\bullet 0$ or $0\bullet 1$.
     }
   \label{fcase0w} 
   \end{figure}

If $T_1<\ov b<T_2$ we have that $\{\vr_{[0,T_1]}(\vrt_1),\; \vr_{[T_1,\ov b]}(\vrt_2)\}$ 
is  a  $Q^{-B+1}$-possible 1-specification with 2 jumps and 
$\vr_{[\ov b, T_2]}(\vrt_2)$ is a periodic $Q^{-B+1}$-possible 1-specification
with only 1 jump. In this case we also write a node with $1\bullet 0$ in the tree.
The black node $\bullet$ will have two branches in the tree corresponding to 
numbers 0 and 1. We similarly also shadow the triangle bounded by the lines
 $\ov{\vrt_1\,\vr_{\ov b}(\vrt_2)}$,   $\ov{\vr_{\ov b}(\vrt_2)\,\vr_{T_2}(\vrt_2)}$, 
  $\ov{\vr_{T_2}(\vrt_2)\,\vrt_1}$. And we notice that by \eqref{doaob}, \eqref{appab}
  and \eqref{exqn} respectively,
  the three lines are approaches
  in $\A_{B-1}$

\subsubsection{Case $0\otimes$}
{\it When $d({\ov b},\{0,T_1,T_2\})\le 1$.}

Say that $ \vr_{\ov a}(\vrt_1)=\vrt_1$, the other case  is similar.
Since by~\eqref{bage1} 
$|\ov b-\ov a|_{\text{mod }T_2}\ge 1$, in this case $|\ov b-T_1|<1$  and then neither of the two
points  in the approach is implying a new interval of length one.
We write a white (or empty) node $\otimes$ in the tree.
By Statement~\ref{0appr} this situation implies at least one child specification with only 
one jump corresponding to the approach $(\vr_c(\vrt_j),\vr_{\ov b}(\vrt_j))\in\A_{B-1}$.
For simplicity we will forget about the other
parts of the mother specification. This case is shown in Figure~\ref{fcase0w}.
We draw a line $\ov{\vr_c(\vrt_j) \,\vr_{\ov b}(\vrt_j)}$ and shadow the complement
in the mother specification
of the region bounded by the segment $\vr_{[c,\ov b]}(\vrt_j)$ 
[or $\vr_{[\ov b,c]}(\vrt_j)$] and the line  $\ov{\vr_c(\vrt_j) \,\vr_{\ov b}(\vrt_j)}$,
as shown in Figure~\ref{fcase0w}.
On the tree we write a label $0\otimes$ on the node.
This node will issue only one branch corresponding to the label $0$: 
i.e. a child periodic $Q^{-N+1}$-possible 1-specification with only one jump.
Indeed, Statement~\ref{0appr} says it is a $Q^{-N+1}$-possible 1-specification.
\bigskip

        \begin{figure}[h]
     \resizebox*{14cm}{8cm}{\includegraphics{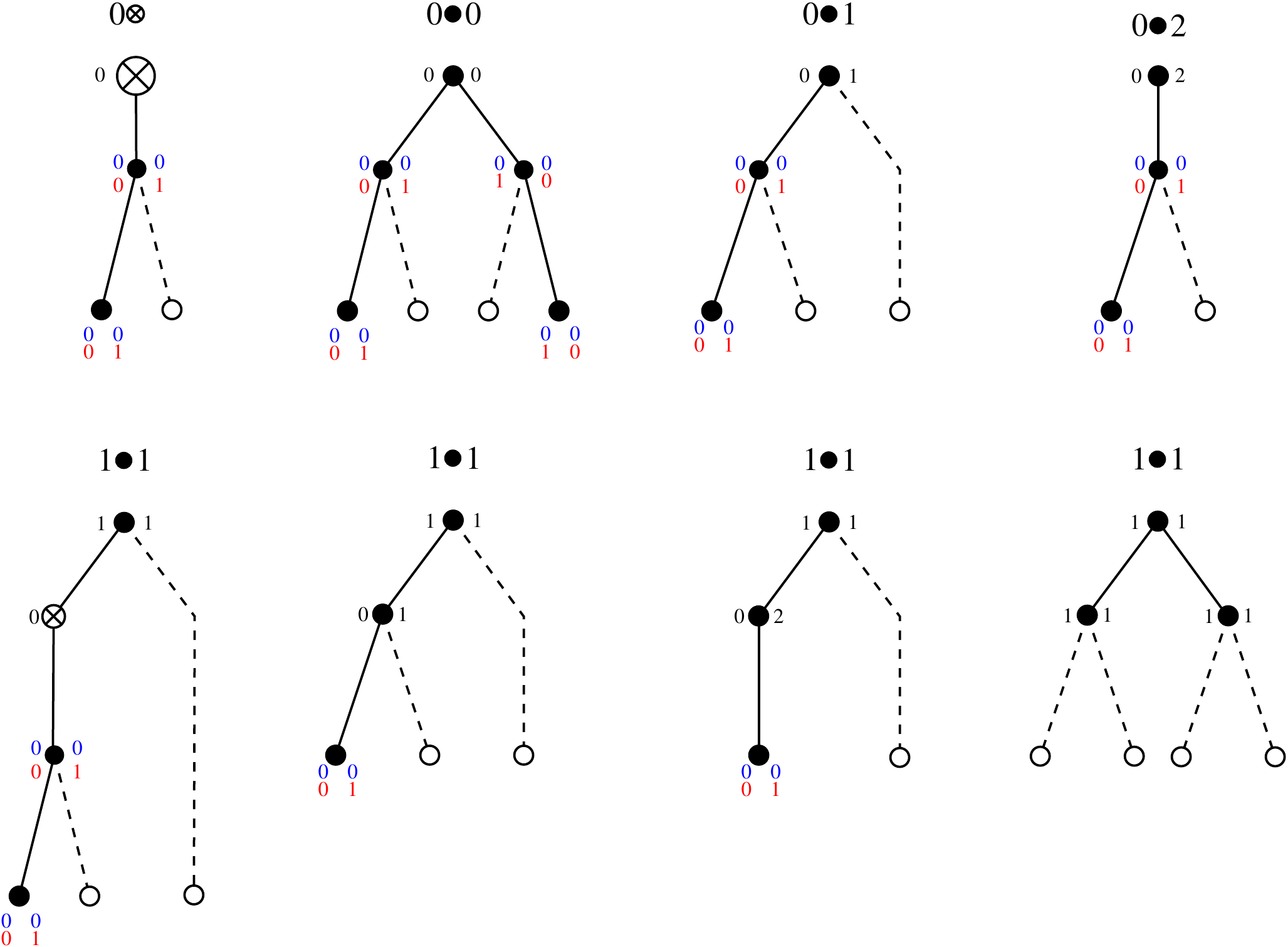}}
      \caption{\small From the childhood possibilities given in Figures~\ref{case0} and~\ref{case1} we construct the possible
      sub-trees of two and three levels to show that they satisfy Claim~\ref{c111}. Claim~\ref{c111} asks for at least
       two black nodes in the sub-tree with three levels and at least two ending nodes in the subtree with three levels.
       The dotted lines 
       denote that from this node we have at least one branch  that continues to levels 2 and 3. White dots
      $\circ$ denote that we don't know if the ending node is black $\bullet$ or empty $\otimes$.
      Since by Figures~\ref{case0} and ~\ref{case1} each node issues at least one branch, 
      we have that  if the sub-tree at
      level 2 satisfies Claim~\ref{c111}, then the sub-tree at level 3 also satisfies Claim~\ref{c111}. 
      In Figure~\ref{case1100} we develop the 2 level sub-tree for the case $1\bullet1\to 0\otimes$ and show 
      that in fact  we don't need the three level sub-tree because all the sub-trees with two levels already satisfy Claim~\ref{c111}.}
      \label{case0102}
     \end{figure}

Our tree has black nodes and white nodes. Each black node is at time interval at least one from the
    other nodes. So we have that 
     \begin{equation}\label{Bnodes}
      t^N_{\ell+1}-t^N_\ell \ge \#\{\text{black nodes}\}.
     \end{equation}
     In order to have the estimate in Proposition~\ref{PertEnt} it is enough to prove that
     \begin{claim} \label{c111}Given a
     node $P$ at level $B\ge 3$ the next three levels in the tree, $B-1$, $B-2$ and $B-3$, contain at least two
     black nodes below the node $P$ (not counting the node $P$); and also  the level $B-3$ below
     the node $P$ has at least two nodes, black or white. 
     \end{claim}
     This is because each node, black or white,
     has at least one successor in the next level. So at any two levels the number of black dots is
     duplicated and the number of successor branches is duplicated. This gives an exponential 
     rate of growth of $\sqrt[3]{2}$ new black dots per level.

          \begin{figure}[h]
     \resizebox*{14cm}{5cm}{\includegraphics{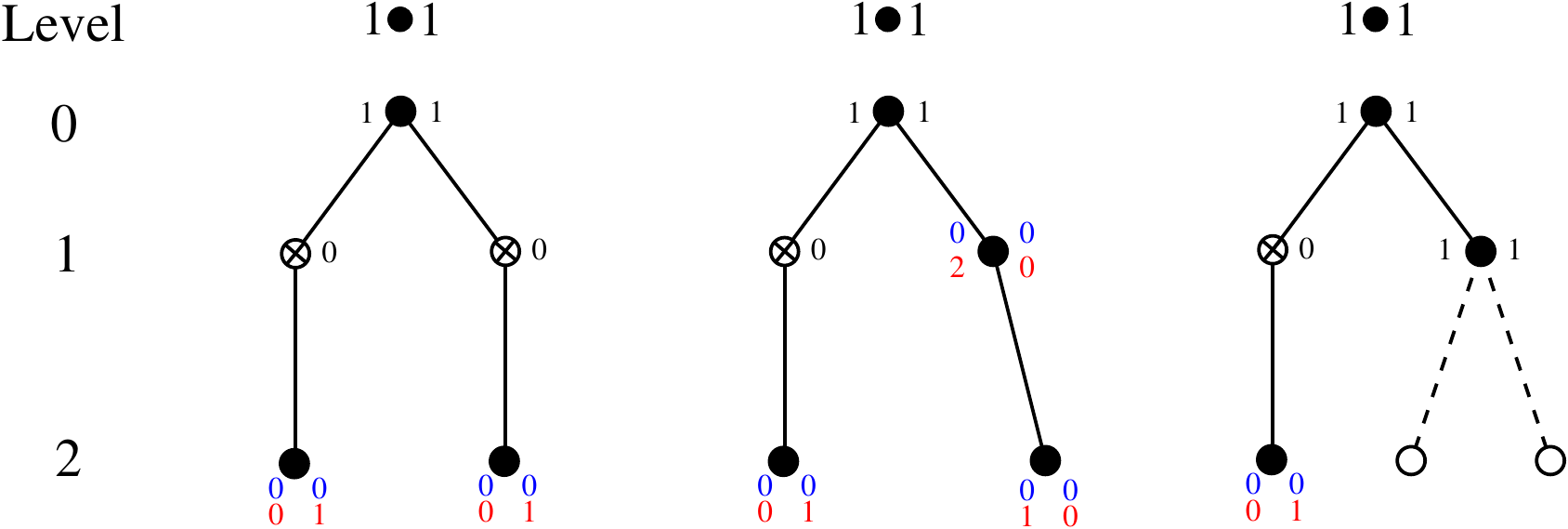}} 
      \caption{\small In this figure we develop  the possible sub-trees with original node $1\bullet 1$
      and one child node $0\otimes$. In the second sub-tree the right node could be followed by one
      or two branches depending on wether the right node is $2\bullet 0$ or $1\bullet 0$ respectivelly.
       We see that all these trees satisfy Claim~\ref{c111}: at least two black nodes in levels
      1 and 2 together,  and at least two  nodes at level 2. }
      \label{case1100}
     \end{figure}

   In subsection \S~\ref{br0} and   Figure~\ref{case0}  we showed all the possible 
   nodes ending a branch with label 0, namely $0\bullet 0$ and $0\bullet 1$.
    Also subsection \S~\ref{br1} and Figure~\ref{case1} 
   show the possible nodes ending a branch with label 1. We get that all the possible
   labels of a node are $0\otimes$, $0\bullet 0$, $0\bullet 1$, $0\bullet 2$, $1\bullet 1$ 
   and their symmetric labels $1\bullet 0$, $2\bullet 0$. All these nodes issue at least 
   one branch. In Figure~\ref{case0102} we show the  possibilities of two or three labels 
   issued by any node. The dotted branches and white dots $\circ$  denote that there is
   at least one branch issued and we don't know if it ends in a black dot $\bullet$ or empty 
   node $\otimes$. Figure~\ref{case0102} shows that Claim~\ref{c111} is satisfied.
   In most cases in Figure~\ref{case0102} we develop the tree only two levels because with 
   two levels we already have at least two ending nodes and two black dots in the sub-tree.
   Since each ending node will issue at least one branch the sub-tree with three labels will
   have at least the same amount of ending nodes as the sub-tree with two labels. Therefore 
   if a sub-tree with two labels already satisfies Claim~\ref{c111} its continuation to three labels
   also satisfies Claim~\ref{c111}. The only sub-tree with three labels on Figure~\ref{case0102}
   is a node $1\bullet 1$ followed with a node $0\otimes$. In this case the node $0\otimes$ issues
   at least two descendant black nodes in two labels. This implies that the three level sub-tree satisfies 
   Claim~\ref{c111}.  In fact this case already satisfies Claim~\ref{c111} in two labels as seen in
   Figure~\ref{case1100} where the right branch has been developed. So in fact we have a better
   exponential rate $\sqrt[2]{2}$ in Proposition~\ref{PertEnt}.
   
   The diagrams in Figure~\ref{case0102} exploit the fact that  as in Figure~\ref{case0}, a branch with
   label $0$ ends in a node with label $0\bullet 0$ or $0\bullet 1$. In both cases this node has two branches
   and at least one new branch with label $0$. The diagrams are ordered first by the label of the upper node,
   then by the existence in the next level of a given label.  Then in the last diagram we have an upper node
   $1\bullet 1$ issuing a node with label $1\bullet 1$, for the other node at level 1 the other possibilities
   $0\otimes$, $0\bullet 1$, $0\bullet 2$, given from Figure~\ref{case1} have already been studied in the 
   previous diagrams, it only remains to study a second node with label $1\bullet 1$.

 This finishes the proof of Proposition~\ref{PertEnt}.
 
 \end{proof}

\bigskip

\appendix

\section{Entropy.}\label{AE}

 Let $\phi_t$ be a continuous flow without fixed points 
 on a compact metric space $X$.
 
 For $T>0$ define the distance $d_T$ on $X$ by
 $$
 d_T(x,y) := \max_{s\in[0,T]} d(\phi_s(x),\phi_s(y)). 
 $$
 For $\de, \, T>0$, $x\in X$ the {\it dynamic ball } $V(x,T,\de)$ is defined 
 as the closed ball of radius $\de$ for the distance $d_T$ centered at $x$, equivalently
 $$
 V(x,T,\de):=\{\,y\in X\;|\; \forall s\in[0,T]\;\; d(\phi_s(x),\phi_s(y))\le \de\;\}.
 $$

 Given $E,\,F\subset X$ we say that $E$ is a $(T,\de)${\it -spanning set} for $F$
 (or that it $(T,\de)${\it -spans} $F$)
 iff
$$
F \subset \bigcup_{e\in E} V(e,T,\de).
$$
 Let 
 $$
 r(F,T,\de):=\min \{\#E \;|\; E\text{ $(T,\de)$-spans } F \,\}.
 $$
 If $F$ is compact, the continuity of $\phi_t$ implies that $r(F,T,\de)<\infty$.
 Let
 $$
 \ov r_\phi(F,\de) := \limsup_{T\to\infty} \frac 1T \log r(F,T,\de).
 $$
 
 We say that $E\subset X$ is $(T,\de)${\it -separated } if
 $$
 x,\,y\in E, \quad x\ne y \quad \then\quad 
 V(x,T,\de)\cap V(y,T,\de) =\emptyset.
 $$
 Given $F\subset X$ let 
 \begin{align*}
 s(F,T,\de):= \max\{\,\#E\,|\, E\subset F,\; E\text{ is $(T,\de)$-separated }\}.
 \end{align*}
 If $F$ is compact,
 Theorem 6.4 in \cite{Walters2}
 shows that $s(F,T,\de)<\infty$.
 Let
 $$
 \ov s_\phi(F,\de):= \limsup_{T\to+\infty}
 \frac 1T \log s(F,T,\de).
 $$
 Define the {\it topological entropy} by
 $$
 h_{top}(F,\phi):=\lim_{\de\to 0}\ov r_\phi(F,\de)
 =\lim_{\de\to 0}\ov s_\phi(F,\de).
 $$
 By Lemma~1 in \cite{Bowen10} these limits exist and are equal.
 
 In fact the topological entropy of a flow $\phi$ equals the topological
 entropy of the homeomorphism $\phi_1$, and more generally
 $h(\phi_t)= |t|\, h(\phi)$, see Proposition~21 in \cite{Bowen10}.
 
  Denote by $\mathfrak B(X)$ the Borel $\si$-algebra of $X$.
  Let $f:X\to X$ be a measurable map.
  Denote by $\cM(f)$ the set of $f$-invariant Borel probabilities on $X$.
 Given a finite Borel partition  $\A\subset\mathfrak B(X)$ 
 of $X$ and an $f$-invariant Borel
  probability $\mu\in\cM(f)$ define its {\it entropy}
  by
  \begin{align*}
  H_\mu(\A):&=-\sum_{A\in\A}\mu(A)\,\log\mu(A).
  \end{align*}
  Given two finite Borel partitions $\A,\,\B$ of $X$ define $\A\vee \B$ as
  \footnote{More generally $\A\vee \B=\si(\A\cup \B)$ is the $\si$-algebra generated by $\A\cup \B$. 
  This is the definition used for an infinite refinement $\bigvee_{i\in\na} \A_i$. } 
  $$
  \A\vee \B:=\{\, A \cap B\;|\; A\in\A,\, B\in\B\,\}.
  $$
  From Walters~\cite[Theorem~4.10]{Walters}, the map
  \begin{equation}\label{HfNdec}
  \na^+\ni N \mapsto \tfrac 1N \, 
  H_\mu\Big(\textstyle\bigvee_{n=0}^{N-1}f^{-n}(\A)\Big)
  \qquad
  \text{is decreasing.}
  \end{equation}
  Let 
  $$
  h_\mu(f,\A):= \lim_{N} \tfrac 1N\, 
  H_\mu\Big(\textstyle\bigvee_{n=0}^{N-1} f^{-n}(\A)\Big).
  $$
  The (metric) {\it entropy} of $\mu$ under $f$ is defined as
  $$
  h_\mu(f):=\sup\{\, h_\mu(f,\A)\;|\; \A \text{ is a Borel partition of $X$} \,\}.
  $$

  \begin{Theorem}[Variational Principle] (cf. Walters~\cite[Theorem~8.6]{Walters})
  \label{VarPrin}

  Let $f:X\to X$ be a continuous map of a compact metric space $X$. Then
  $$
  h_{top}(X,f)=\sup\{\, h_\mu(f)\;|\; \mu\in\cM(f)\,\}.
  $$
  \end{Theorem}

    \begin{Definition}\label{hexpan}\quad
    
  Let $f:X\to X$ be a homeomorphism. For $\e>0$ and $x\in X$ define
  $$
  \Ga_\e(x,f):=\{\,y\in X\;|\;\forall n\in \Z\quad  d(f^n(y),f^n(x))\le \e \,\}.
  $$
  We say that $f$ is {\it entropy expansive} or {\it h-expansive} if there
  is $\e>0$ such that
  $$
  \forall x\in X\qquad h_{\text{top}}(\Ga_\e(x,f),f)=0.
  $$
  Such an $\e$ is called an h-expansive constant for $f$.
    \end{Definition}
 
  The following theorems use the definition from Hurewicz and Wallman \cite{HuWall} of a finite dimensional 
  metric space. They mainly use the property from the Corollary to Theorem~V.1, page~55 in~\cite{HuWall},
  that if the metric space $X$ has dimension $\le m$, then for any $\ga>0$,  $X$ has a covering 
  $\mathfrak B(\ga)$ of diameter $<\ga$, such that  no point in $X$ is in more than $m+1$  
  elements of $\mathfrak B(\ga)$.

  \medskip

  \begin{Theorem}[Bowen~{\cite[Theorem~3.5]{Bowen9}}]\label{TBhE}\quad
  
  Let $X$ be a finite dimensional metric space and $f:X\to X$ a 
  uniformly continuous homeomorphism. Suppose that $\e>0$ is
  an h-expansive constant for $f$. 
  If $\A$ is a finite Borel partition of $X$ with $\diam\A<\e$
  then $h_\mu(f)=h_\mu(f,\A)$.
  \end{Theorem}

\medskip
  
  \begin{Definition}\label{dexpans}\quad
  
  A continuous flow $\phi:\re\times X\to X$ is said {\it flow expansive}
  (cf.~\cite{BoWa})
  if for every $\eta>0$ there exists $\de>0$ such that if
  $x,\,y\in X$ and $\a:\re\to\re$ is continuous with $\a(0)=0$
  such that $\forall t\in\re$ \;\; $d(\phi_t(x),\phi_{\a(t)}(y))\le \de$,
  then $y=\phi_v(x)$ for some $|v|\le\eta$.
  Observe that if a flow $\phi$ is flow expansive then its time~$1$ map $\phi_1$
  is entropy expansive.
  \end{Definition}
 
    For a continuous flow $\phi:\re\times X\to X$, the {\it metric entropy} 
  $h_\mu(\phi)$ of a $\phi$-invariant Borel probability measure $\mu$
  is defined as the metric entropy of its time one map
  $h_\mu(\phi):=h_\mu(\phi_1)$.

   \begin{Definition}\label{duhe}\quad
   
   Let $\cU$ be a topological subspace of $C^0(X,X)\supset\cU$ and 
   $Y\subseteq X$ compact.
   We say that $\cU$ is {\it uniformly h-expansive} on $Y$
   if there is $\e>0$ such that 
   $$
   \forall f\in\cU\quad \forall y\in Y
   \qquad
   h_{\text{top}}(\Ga_\e(y,f),f)=0.
   $$
   In our applications $\cU$ will be a $C^1$ neighbourhood of a diffeomorphism
   endowed with the $C^0$ topology. An h-expansive homeomorphism corresponds
   to $\cU=\{ f\}$.
   \end{Definition}

  Let $\cP(X)$ be the set of Borel probability measures on $X$ endowed 
  with the weak* topology. Let $\cM(f)\subset \cP(X)$ the subspace of $f$-invariant
  probabilities.

 \medskip
 
 \begin{Theorem}\quad\label{ESS}
 
Let $X$ be a finite dimensional compact metric space, $Y\subseteq X$ compact  and 
let 
\linebreak 
$\cU\subset C^0(X,X)$ be a uniformly h-expansive set on $Y$.
 Then the entropy map
$(\mu,f)\mapsto h_\mu(f)$ is upper semicontinuous on $\cU$,
i.e.

if $f\in\cU$, $\mu\in\cM(f)$ and $\supp\mu\subset Y$ then
given $\e>0$ there are open sets $\cV$, $U$,
\linebreak
$f\in\cV\subset C^0(X,X)$ and $\mu\in U\subset \cP(Y)$ such that 
$$
\forall g\in\cU\cap\cV,\quad
\forall\nu\in\cM(g)\cap U\cap\cP(Y)
\qquad 
h_\nu(g)\le h_\mu(f)+\e.
$$

In particular, this applies to time-one maps of uniformly expansive flows
as in Remark~\ref{rue}, giving
$$
\limsup_{(\psi^n_1,\nu_n)\to(\phi_1,\mu)}h_{\nu_n}(\psi^n)\le h_\mu(\phi).
$$
\end{Theorem}

\begin{proof}
\quad

Let $\e>0$ be a uniform entropy h-expansivity constant on $Y$ for all $f\in\cU$.
Fix $f\in\cU$ and $\mu\in\cM(f)$ and let $\de>0$. 

Let $\C=\{\, C_1,\ldots, C_n\,\}$ be a finite partition of $Y$
by Borel sets with $\diam C_i<\e$. By Theorem~\ref{TBhE},
$h_\nu(g)=h_\nu(g,\C)$ for all $g\in\cU$ and $\nu\in\cM(g)$.
Let $N$ be such that 
\begin{equation}\label{fixN}
\frac 1N\, H_\mu\Big(\textstyle\bigvee\limits_{k=0}^{N-1} f^{-k}\C\Big)
< h_\mu(f)+\tfrac 12\de.
\end{equation}
Since $\mu$ is regular, there are compact sets
$$
K_{i_0\ldots i_{N-1}}\subset\textstyle\bigcap\limits_{k=0}^{N-1} f^{-k}C_{i_k}
$$
such that
\begin{equation}\label{muck}
\mu\Big(\bigcap_{k=0}^{N-1} f^{-k}C_{i_k}\setminus K_{i_0\cdots i_{N-1}}\Big)<\e_1.
\end{equation}
Then
$$
L_j:=\textstyle \bigcup\limits_{k=0}^{N-1}\bigcup\limits_{i_k=j}
f^k K_{i_0\ldots i_{N-1}}
\subset C_j.
$$
The sets $L_1,\ldots, L_n$ are compact and disjoint so there is 
a partition $\D:=\{ D_1,\ldots , D_n \}$ with $\diam(D_j)<\e$ and 
$L_j\subset \intt(D_j)$.
We have that 
$$
K_{i_0\ldots i_{N-1}}\subset \intt\Big(\textstyle\bigcap\limits_{k=0}^{N-1}
f^{-k} D_{i_k}\Big).
$$
Choose open subsets $W_{i_0\ldots i_{N-1}}$ such that
$$
K_{i_0\ldots i_{N-1}}\subset 
W_{i_0\ldots i_{N-1}}\subset
\ov{W_{i_0\ldots i_{N-1}}} \subset
\intt\Big(\textstyle\bigcap\limits_{k=0}^{N-1}
f^{-k} D_{i_k}\Big).
$$
We have that 
$$
f^k\big(\ov{W_{i_0\ldots i_{N-1}}} \big)
\subset \intt D_{i_k}.
$$
Choose a relatively open subset $f\in\cU_1\subset \cU$ such that 
$$
\forall g\in\cU_1\quad \forall (i_0,\ldots, i_{N-1})
\quad
\forall k\in\{0,\ldots,N-1\}
\qquad
g^k\big(\ov{W_{i_0\ldots i_{N-1}}} \big)
\subset \intt D_{i_k}.
$$
So that 
$$
\forall g\in\cU_1\qquad
K_{i_0\ldots i_{N-1}}\subset 
W_{i_0\ldots i_{N-1}}\subset
\intt\Big(\textstyle\bigcap\limits_{k=0}^{N-1}
g^{-k} D_{i_k}\Big).
$$
By Urysohn's Lemma there exist $\psi_{i_0\ldots i_N}\in C^0(X,\re)$ such that 
\begin{itemize}
\item $0\le \psi_{i_0\ldots i_N}\le 1$.
\item equals $1$ on $K_{i_0\ldots i_{N-1}}$.
\item vanishes on $X\setminus W_{i_0\ldots i_{N-1}}$.
\end{itemize}
Define
$$
U_{i_0\ldots i_{N-1}}:=\Big\{ m\in\cP(Y)\;:\;
\Big|\int \psi_{i_0\ldots i_{N-1}}\,dm-\int \psi_{i_0\ldots i_{N-1}}\,d\mu\Big|
<\e_1 \Big\}.
$$
The set $U_{i_0\ldots i_{N-1}}$ is open in $\cP(Y)$
and if $m\in U_{i_0\ldots i_{N-1}}$ and $g\in\cU_1$ then
\begin{align}\label{mDk}
m\Big({\textstyle\bigcap\limits_{k=0}^{N-1}} g^{-k} D_{i_k}\Big)
\ge \int \psi_{i_0\ldots i_{N-1}}\,dm
>\int \psi_{i_0\ldots i_{N-1}}\,d\mu-\e_1
\ge \mu\big(K_{i_0\ldots i_{N-1}}\big) -\e_1.
\end{align}
From~\eqref{muck} and~\eqref{mDk} we get that
$$
\mu\Big(\textstyle\bigcap\limits_{k=0}^{N-1}f^{-k}C_{i_k}\Big)
-m\Big(\textstyle\bigcap\limits_{k=0}^{N-1}g^{-k}D_{i_k}\Big)
< 2 \e_1.
$$
If $U:=\bigcap_{i_0\ldots i_{N-1}}U_{i_0\ldots i_{N-1}}$ and $m\in U$  and
$g\in\cU_1$
 then:
$$
\Big|\mu\Big(\textstyle\bigcap\limits_{k=0}^{N-1} f^{-k}C_{i_k}\Big)
-m\Big(\bigcap\limits_{k=0}^{N-1}g^{-k}D_{i_k}\Big)\Big|
< 2\e_1 n^N,
$$
because if $\sum_{i=1}^q a_i=1=\sum_{i=1}^qb_i$ and there exists
$c>0$ such that $a_i-b_i<c$ for all $i$, then $\forall i$ $|a_i-b_i|<c\,q$,
because $b_i-a_i=\sum_{j\ne i}(a_j-b_j)<c\,q$.

So if $m\in U$, $g\in\cU_1$ and $\e_1$ is small enough, for $N$ fixed in \eqref{fixN}, the continuity of 
$x\,\log x$ gives:
$$
\frac 1N\,H_m\Big({\textstyle\bigvee\limits_{k=0}^{N-1}} g^{-k}\D\Big)
<
\frac 1N\,H_\mu\Big({\textstyle\bigvee\limits_{k=0}^{N-1}} f^{-k}\C\Big)
+\frac\de 2.
$$
Hence, for $g\in\cU_1$, $m\in U\cap\cM(g)$  and $\e_1$ small enough,
using  Theorem~\ref{TBhE} and \eqref{HfNdec}  we have that
\begin{align*}
h_m(g)&=h_m(g,\D)
\le \frac 1N\,H_m\Big(\textstyle\bigvee\limits_{k=0}^{N-1} g^{-k}\D\Big)
\\
&\le \frac 1N\,H_\mu\Big({\textstyle\bigvee\limits_{k=0}^{N-1}} f^{-k}\C\Big)
+\frac\de 2
< h_\mu(f)+\de.
\end{align*}
\end{proof}

\bigskip

 Let $W(x,N,\e)$ be the dynamic ball for the time 1 map $\phi_1$\,:
 $$
 W(x,N,\e) := \{\, y\in X\,|\, d(\phi_n(x),\phi_n(y))\le \e\quad
 \forall n=0,\ldots,N-1\,\}.
 $$
  Given $\e>0$ there is $\de>0$ such that 
 \begin{equation}\label{ed}
 x,\,y\in X,\quad d(x,y)<\de \quad\then\quad
 \forall t\in[0,1]\quad d(\phi_t(x),\phi_t(y))<\e.
 \end{equation}
 If $\e$, $\de$ are as in~\eqref{ed}, $N\in\na$ and $N\le T\le  N+1$, we have that 
 \begin{equation}\label{VB}
 W(x,N+1,\de)\subset V(x,T,\de)\subset V(x,N,\de)\subset W(x,N,\e).
 \end{equation}

 Using~\eqref{VB} we can use the Brin-Katok theorem for maps 
 (cf. Brin-Katok~\cite{BK})
 to obtain
 
 \medskip
 
 \begin{Theorem}[Brin-Katok \cite{BK}]\label{BK}\quad
 
 If $\mu$ is an ergodic $\phi$-invariant Borel probability, then
 for $\mu$-almost every $x\in X$ we have 
 \begin{align*}
  h_\mu(\phi)
  &=\lim_{\e\to 0}\limsup_{T\to+\infty}-\frac 1T \log \mu(V(x,T,\e)) 
  \\&=
  \lim_{\e\to 0}\liminf_{N\to+\infty}-\frac 1N \log \mu(W(x,N,\e)). 
 \end{align*}
 \end{Theorem}

\bigskip
 
  \section{Shadowing.}\label{asha}
  
  Let $\phi$ be the flow of a $C^1$ vector field on a compact manifold $M$.
  A compact $\phi$-invariant subset  $\La\subset M$
  is {\it  hyperbolic } for $\phi$ if the tangent bundle 
  restricted to $\La$ is decomposed as the Whitney sum
  $T_\La M = E^s \oplus E\oplus E^u$, where $E$ is the 1-dimensional 
  vector bundle tangent to the flow and there are constants $C,\la>0$ such that
  \begin{enumerate}[(a)]
  \item $D\phi_t(E^s) = E^s$, $D\phi_t(E^u)=E^u$ for all $t\in\re$
  \item\label{hipb} $| D\phi_t(v)| \le C\,\ee^{-\la t} |v|$ for all $v\in E^s$, $t\ge 0$.
  \item\label{hipc} $|D\phi_{-t}(u)|\le C\, \ee^{-\la t} |u|$ for all $u\in E^u$, $t\ge 0$.
  \end{enumerate} 
  It follows from the definition that the hyperbolic splittig 
  $E^s \oplus E\oplus E^u$ over $\La$ is continuous.

  From now on we shall assume that $\La$ does not contain fixed points
  for $\phi$.
  For $x\in \La$ define the following   stable and unstable sets:
  \begin{align}
  W^{ss}(x):&=\{\,y\in M\;|\; d(\phi_t(x),\phi_t(y))\to 0 \text{ as }t\to +\infty\,\},
  \notag\\
  W^{ss}_\e(x):&=\{\,y\in W^{ss}(x)\;|\;d(\phi_t(x),\phi_t(y))\le\e\;\;\forall t\ge 0\,\},
  \notag\\
  W^{uu}(x):&=\{\,y\in M\;|\; d(\phi_{-t}(x),\phi_{-t}(y))\to 0\text{ as } t\to +\infty\,\},
  \notag\\
  W^{uu}_\e(x):&=\{\,y\in W^{uu}(x)\;|\;d(\phi_{-t}(x),\phi_{-t}(y))\le \e\;\;\forall t\ge 0\,\},
  \label{wuue}\\
  \intertext\quad
  W^s_\e(x):&=\{\,y\in M\;|\;d(\phi_t(x),\phi_t(y))\le\e\;\;\forall t\ge 0\,\},
  \notag\\
   W^{u}_\e(x):&=\{\,y\in M\;|\;d(\phi_{-t}(x),\phi_{-t}(y))\le \e\;\;\forall t\ge 0\,\}.
   \notag
  \end{align}
  
  \medskip 
  
  Conditions~\eqref{hipb} and ~\eqref{hipc} are equivalent to
  \begin{enumerate}[(a)]
  \addtocounter{enumi}{3}
  \item\label{hipd} There exists $T>0$ such that
   $\lV D\phi_T|_{E^s}\rV <\tfrac 12$ \quad and \quad
   $\lV D\phi_{-T}|_{E^u}\rV<\tfrac 12$.
  \end{enumerate}

  Let $\fX^k(M)$ be the Banach manifold of the $C^k$ vector fields on $M$, $k\ge 1$.
  Let $X=\partial_t\phi_t$ the vector field of $\phi_t$.

  \begin{Proposition}\label{unifhip}\quad
  
  There are open sets $X\in\cU\subset\fX^1(M)$ and $\La\subset U\subset M$ 
  such that for every $Y\in\cU$ the set $\La_Y:=\bigcap_{t\in\re}\psi^Y_t(\ov U)$
  is hyperbolic for the flow $\psi^Y_t$ of $Y$, with uniform constants $C$, $\la$, $T$
  on~\eqref{hipb}, \eqref{hipc} and \eqref{hipd}.
  \end{Proposition}
  
   Proposition~\ref{unifhip} can be proven by a characterization of hyperbolicity using
   cones (cf. Hasselblatt-Katok~\cite[Proposition 17.4.4]{HK}) and obtaining uniform 
   contraction (expansion) for a fixed iterate in $\La_Y$.

  \medskip

  \begin{Proposition}[Hirsch, Pugh, Shub {\cite[Corollary 5.6, p.~63]{HPS}}, 
  Bowen~{\cite[Prop. 1.3]{Bowen6}}]
  \label{pHPS}\quad

  There are constants $C,\,\la>0$ such that, for small $\e$,
  \begin{enumerate}[(a)]
  \item $d\big(\phi_t(x),\phi_t(y)\big)\le C\, \ee^{-\la t}\, d(x,y)$
  when $x\in \La$, $y\in W_\e^{ss}(x)$, $t\ge 0$.
  \item $d\big(\phi_{-t}(x),\phi_{-t}(y)\big)\le C\, \ee^{-\la t}\, d(x,y)$
  when $x\in \La$, $y\in W_\e^{uu}(x)$, $t\ge 0$.

  \end{enumerate} 
  \end{Proposition}
  
  \medskip
  \begin{mysec}{\bf Canonical Coordinates \cite[(3.1)]{PS}, \cite[(4.1)]{HPS}, \cite[(7.4)]{Smale},
  \cite[(1.4)]{Bowen6}, \cite[(1.2)]{Bowen3}:}
  \label{caco}
  \end{mysec}\vskip -5pt
  {\it There are $\de, \,\ga>0$ for which the following is true:
  If $x, y\in \La$ and $d(x,y)\le \de$ then there is a unique 
  $v=v(x,y)\in \re$ with $|v|\le \ga$ such that 
  \begin{equation}\label{ecaco}
  \langle x,y\rangle:=W_\ga^{ss}(\phi_v(x))\cap W^{uu}_\ga(y) \ne \emptyset.
  \end{equation}
  This set consists of a single point, which we denote 
  $\langle x,y\rangle\in M$. The maps $v$ and 
  $\langle\;,\;\rangle$ are continuous on the set
  $\{\,(x,y)\;|\; d(x,y)\le \de\,\}\subset \La\times\La$.
  }
  
  \medskip
  
     We will take a small neighborhoods $\La\subset U\subset M$ and  $\cU\subset\fX^1(M)$
     and  take uniform constants from \ref{pHPS} and \ref{caco} which hold for every $Y\in\cU$ and
     all points in the maximal invariant set $\La^Y_U:=\bigcap_{t\in\re}\ov U$.
    The following proposition is a modification of 
    Bowen {\cite[Prop.~1.6, p.~4]{Bowen6}} which we prove below.
  
    \begin{Proposition}
    \label{B16}
  \quad
  
     There are open sets $X\in\cU\subset\fX^1(M)$
     and $\La\subset U\subset M$ and $\eta_0>0$, $B>1$ such that
     $$
     \forall \eta>0 \qquad \exists
      \be=\be(\eta)=\tfrac 1B\,\min\{\eta,\eta_0\} \qquad \forall Y\in\cU
     $$
   if $\psi_t=\psi^Y_t$  is the flow of $Y$, 
  $x,\, y \in \Om^Y_U:=\bigcap_{t\in\re}\psi_t(\ov U)$ and $s:\re\to\re$  continuous with $s(0)=0$ satisfy
  \begin{equation}\label{csh}
  d(\psi_{t+s(t)}(y),\psi_t(x))\le\be\quad\text{ for }|t|\le L,
  \end{equation}
  then 
  \begin{equation}\label{setav}
  |s(t)|\le3\eta \quad \text{ for all }|t|\le L,
  \qquad
   |v(x,y)|\le\eta \quad \text{ and }
  \end{equation}
    \begin{align}
  \forall |s|\le L, \qquad
  &d(\psi_s(y),\psi_{s+v}(x))\le C\,\ee^{-\la(L-|s|)}\,
  \big[d(\psi_L(w),\psi_L(y))+d(\psi_{-L}(w),\psi_{-L+v}(x))\big],
  \notag\\
  &\text{where }\qquad w:=\langle x,y\rangle 
  =W^{ss}_\ga(\psi_v(x))\cap W^{uu}_\ga(y).
  \label{d<egdd1}
  \end{align}
  also
  \begin{equation}\label{dsh}
  \forall |s|\le L, \qquad
   d(\psi_s(y),\psi_s\psi_v(x))\le C\, \ga\,e^{-\la (L-|s|)}.
  \end{equation}
   In particular 
  $$
  d(y,\psi_v(x))\le C\, \ga\, e^{-\la L}.
  $$
  \end{Proposition}
  
  \bigskip
  
  For the proof of Proposition~\ref{B16} we need the following
  
    \begin{Lemma}\label{B4}

  There is $\eta_0>0$ and $B>1$ such that  
  \newline
  if $d(x,y)\le\eta_0$,  
  $Y\in\cU$, $x,\,y\in\La^Y_U$ and $\eta = B\, d(x,y)$  then
  \begin{gather}
  \langle x,y\rangle\in W^{ss}_\eta(\psi^Y_v(x))\cap W^{uu}_\eta(y)
  \qquad \text{with }\quad
  |v(x,y)| \le \eta \qquad 
  \label{vdxy}\\
  \text{and } \qquad d(x,\psi^Y_v(x))\le \eta.
  \label{Bdxy}
  \end{gather}
  \end{Lemma}

 \begin{proof}\quad

We have that $\langle x,x\rangle =x$ and $v(x,x)=0$.
By uniform continuity, given $\de>0$, for $d(x,y)$ small enough
\begin{equation}\label{dewx}
d(\langle x,y\rangle,x)\le \de,
\qquad
d(\langle x, y\rangle,y)\le \de,
\end{equation}
and $v=v(x,y)$ is so small that
\begin{equation}\label{dpsvd}
d(\psi_v(x),x)\le \de.
\end{equation}

 The continuity of the hyperbolic splitting implies that the angles
  $\measuredangle(E^s,E^u)$, $\measuredangle(Y,E^s)$ and 
  $\measuredangle(E^s\oplus\re Y,E^u)$
   are bounded away from zero, 
  uniformly on $\La^Y_V$, for some $V\supset U$
   and all $Y$ in an open set $\cU_0\subset\fX^1(M)$
  with $X\in\cU_0$.   
    There is $\be_1>0$ such that if $x,\,y\in\La^Y_U$ and $d(x,y)<\be_1$
  then 
  $$
  \langle x, y\rangle =W^s_\ga(x)\cap W^{uu}_\ga(y)\in V.
  $$

   The strong local invariant manifolds $W^{ss}_\ga$, $W^{uu}_\ga$ 
  are tangent to $E^s$, $E^u$ at $\La^Y_V$ and for a fixed $\ga$ as $C^1$ submanifolds
  they vary continuously on the base point $x\in M$ and on the vector field in $C^1$
  topology (cf.  \cite[Thm. 4.3]{CroPo}\cite[Thm. 4.1]{HPS}). 
  There is a family of small cones $E^u_X(x)\subset C^u(x)\subset T_xM$, $E^s_X(x)\subset C^s(x)\subset T_xM$ 
  defined on a neighbourhood $W$ of $\La$ such that $\exp_x(C^u(x)\cap B_\de(0))$, 
  $\exp_x (C^s(x)\cap B_\de(0))$ are invariant under $\psi^Y_1$ and $\psi^Y_1$ respectively, for
  $Y$ in a $C^1$ neighborhood $\cW$ of $X$. These cones contain $W^{uu}_{\ga}(x)$ and $W^{ss}_{\ga}(x)$ for 
  $x\in\La^Y_W$ and $Y\in\cW$. The angles between these cones are uniformly bounded 
  away form zero, so for example if $z^u\in W^{uu}(x)$, $z^s\in W^{ss}(x)$ and $d(z^u,x)$, $d(z^s,x)$ are small,
  then $d(z^u,x)+d(z^s,x)< A_0\,d(z^u, z^s)$ for some $A_0>0$.
  We can construct similiar cones separating $E^u$ from $E^s\oplus\re X$.
  
  Shrinking $U$ and $\cU$ if necessary there are $0<\be_2<\be_1$  and $A_1,\,A_2,\,A_3>0$
  such that if $Y\in\cU$, $x,\,y\in \La^Y_U$ and $d(x,y)<\be_2$,
  taking 
  $w:=\langle x,y\rangle\in W^s_\ga(x)\cap W^{uu}_\ga(y)$
  and $v$  such  that $w\in W^{ss}_\ga(\psi^Y_v(x))$,  i.e. 
$\psi_v(x)\in\psi^Y_{[-1,1]}(x)\cap W^{ss}_\ga(w)$, then
  \begin{align}
d(x,w)+d(w,y)&\le A_1\, d(x,y),
 \label{a1xwy}
\\ 
d(x,\psi^Y_v(x))+d(\psi^Y_v(x),w)
&\le A_2\, d(x,w) \le A_2 A_1 \,d(x,y),
\label{a2xpw}
\\
|v|\le A_3\, d(x,\psi^Y_v(x))&\le A_3 A_2 A_1 \, d(x,y).
\notag
  \end{align}

  We can assume that $\cU_0$ and $U$ are so small that the constants
  $C$, $\la$, $\e$ in Proposition~\ref{pHPS} can be taken uniform for all 
  $Y\in\cU_0$ and in $\La^Y_U$.
    By Proposition~\ref{pHPS},
since $\langle x,y\rangle \in W^{ss}_\ga(\psi_v(x))$, 
 we have that
\begin{align*}
\forall t\ge 0 \qquad
d\big(\psi^Y_t(\langle x,y\rangle),\psi^Y_{t}(\psi^Y_v(x))\big)
&\le C\,\ee^{-\la t}\,d(w,\psi_v^Y(x)) 
\\
&\le A_2 A_1 C\, \ee^{-\la t} \,d(x,y)
\qquad\text{ using~\eqref{a2xpw}.} 
\end{align*} 
Take $B_1:=(1+A_2)A_1 C$. Then if $d(x,y)<\be_2$ and 
$\eta= B_1\,d(x,y)$
 we obtain that
$\langle x,y\rangle\in W^{ss}_\eta(\psi^Y_v(x))$.

Since $\langle x,y\rangle\in W^{uu}_\ga(y)$ we have that 
\begin{align*}
\forall t\ge 0 \qquad
d(\psi^Y_{-t}(\langle x,y\rangle),\psi^Y_{-t}(y))
&\le C\, \ee^{-\la t} \,d(w,y) 
\\
&\le A_1C\, \ee^{-\la t} d(x,y)
\qquad\text{using }\eqref{a1xwy}.
\end{align*}
Thus if $\eta = B_1\, d(x,y)$ then $\langle x, y\rangle \in W^{uu}_\eta(y)$.

By~\eqref{dewx} and~\eqref{dpsvd}
there is $0<\be_0<\be_2$ such that if $d(x,y)<\be_0$ then
$d(w,x)$, $d(w,y)$ and $d(\psi_v(x),x)$ are small enough
to satisfy the above inequalities.
Now let 
$$
B:=\max\{ 1,\, B_1,\, A_3 A_2 A_1,\, A_2 A_1\}.
$$

 \end{proof}

  \noindent{\bf Proof of Proposition~\ref{B16}:}
  
  Let $\ga$ be 
  from~\ref{caco}.
  We may assume that $\eta$ is so small that
  \begin{gather}
  \eta < \tfrac{\ga}8,
  \label{etaga8}
  \\
  \sup\{\,d(\psi_u(x),x)\;:\; x\in M,\,|u|\le 4\eta\,\}\le \tfrac \ga 8.
  \label{u4eta}
  \end{gather}
  Let 
  \begin{equation}\label{beeta}
  \be=\be(\eta) := \tfrac 1B\, \min\{\eta,\eta_0\},
  \end{equation}
  where $B>1$ and $\eta_0$  are
  from Lemma~\ref{B4}. Consider $x$, $y$ and $s(t)$ as
  in the hypothesis. Since $s(0)=0$ we have that  $d(x,y)\le \be$.
  Using Lemma~\ref{B4} we can define
   \begin{equation}\label{wxyeta}
   w:=\langle x,y\rangle =W^{ss}_\eta(\psi_v(x))\cap W^{uu}_\eta(y)\ne\emptyset,
   \end{equation}
   we also have 
   \begin{equation}\label{vxyeta}
   |v|=|v(x,y)|\le \eta.
    \end{equation}
       Define the sets
   \begin{alignat*}{2}
   A&:=\{\,t\in[0,L]\;:\;|s(t)|\ge 3\eta\;&&\text{ or }
   \;d(\psi_t(y),\psi_t(w))\ge \tfrac 12{\ga}\,\},
   \\
   B&:=\{\,t\in[0,L]\;:\; |s(-t)|\ge 3\eta\;&&\text{ or }
    \;d(\psi_{-t+v}(x),\psi_{-t}(w))\ge \tfrac12\ga\,\}.
   \end{alignat*}
   
   Suppose that $A\ne \emptyset$. Let $t_1:=\inf A$. 
   Then 
   $d(\psi_t(y),\psi_t(w))\le \tfrac 12 \,\ga$, $\forall t\in[0,t_1]$.
   Since $w\in W^{uu}_\eta(y)$ and by~\eqref{etaga8}, $\eta<\tfrac 1{C}\ga$;
   from~\eqref{wuue} 
   we have that $d(\psi_t(y),\psi_t(w))\le \tfrac 18 \ga$, $\forall t\le 0$.
   Therefore
   \begin{equation}\label{wuut1}
   d(\psi_{t_1-r}(y),\psi_{t_1-r}(w))\le \tfrac 12 \,\ga, \qquad \forall r\ge 0.
   \end{equation}

   Since $s$ is continuous, $s(0)=0$ and $t_1\in\partial A$, we have that 
   $|s(t_1)|\le 3\eta$.
   Using~\eqref{u4eta} twice with $u=|s(t_1)|$ and the triangle inequality
   we obtain
   $$
   d(\psi_{t_1+s(t_1)-r}(y),\psi_{t_1+s(t_1)-r}(w))\le  \tfrac 34 \ga,
   \qquad \forall r\ge 0.
   $$
   Hence $\psi_{t_1+s(t_1)}(w)\in W^{uu}_\ga(\psi_{t_1+s(t_1)}(y))$. 
   From~\eqref{wxyeta},
   $w\in W^{ss}_\eta(\psi_v(x))$, and then
   \begin{equation}\label{uvga8}
   d(\psi_r(w),\psi_{r+v}(x))\le \eta<\tfrac \ga 8,
   \qquad \forall r\ge0.
   \end{equation}
   Since $|s(t_1)|\le 3\eta$, using~\eqref{u4eta} twice with $u=s(t_1)$, 
   and~\eqref{uvga8} with $r=t_1+p\ge 0$, and the triangle inequality,
   we get
   $$
   d(\psi_{t_1+s(t_1)+p}(w),\psi_{t_1+s(t_1)+v+p}(x))\
   \le\tfrac {3\ga}8, \qquad \forall p\ge 0.
   $$
   Hence $\psi_{t_1+s(t_1)}(w)\in W^{ss}_\ga(\psi_{s(t_1)+v}(\psi_{t_1}(x)))$.
   We have shown that
   \begin{equation}\label{pstst}
   \psi_{t_1+s(t_1)}(w)\in W^{ss}_\ga(\psi_{s(t_1)+v}(\psi_{t_1}(x)))
   \cap W^{uu}_\ga(\psi_{t_1+s(t_1)}(y)).
   \end{equation}
   Since $|s(t_1)+v|\le|s(t_1)|+|v|\le 4\eta<\ga$
   and by~\eqref{csh}, 
   \begin{equation}\label{lbe}
   d(\psi_{t_1+s(t_1)}(y),\psi_{t_1}(x))\le\be,
   \end{equation}
   equation~\eqref{pstst} implies that 
   \begin{gather*}
   v(\psi_{t_1}(x),\psi_{t_1+s(t_1)}(y))=s(t_1)+v(x,y),
   \\
   \psi_{t_1+s(t_1)}(w)=\langle \psi_{t_1}(x),\psi_{t_1+s(t_1)}(y)\rangle.
   \end{gather*}
   By Lemma~\ref{B4}, \eqref{lbe} and \eqref{beeta}, 
   \begin{gather}
   |s(t_1)+v|\le\eta \qquad \text{ and }
   \label{st1v}
   \\
   \psi_{t_1+s(t_1)}(w)\in W^{uu}_\eta(\psi_{t_1+s(t_1)}(y)),
   \;\text{in particular}
   \notag
   \\
   d(\psi_{t_1+s(t_1)}(w),\psi_{t_1+s(t_1)}(y))\le\eta.
   \label{dpstst}
   \end{gather}
   Since $|s(t_1)|\le 3\eta$, from~\eqref{u4eta},~\eqref{dpstst}
   and~\eqref{etaga8}, we get that
   $$
   d(\psi_{t_1}(w),\psi_{t_1}(y))\le \eta+2\left(\tfrac\ga 8\right)\le 
   \tfrac {3\ga}8.
   $$
   From~\eqref{st1v} and~\eqref{vxyeta} we have that
   $$
   |s(t_1)|\le|s(t_1)+v|+|v|\le 2\eta.
   $$
   These statements contradict $t_1\in A$.
   Hence $A=\emptyset$.
   
   Similarly one shows that $B=\emptyset$.
   Since $A=\emptyset$, inequality~\eqref{ywtga} holds for all $t\in[0,L]$.
   From~\eqref{wxyeta}, $w\in W^{uu}_\eta(y)$ and by~\eqref{etaga8}, 
   $\eta<\tfrac\ga 8$; thus inequality~\eqref{ywtga} also holds for $t\le 0$.
   \begin{equation}\label{ywtga}
   \forall t\le L\qquad d(\psi_t(y),\psi_t(w))< \tfrac 12{\ga}.
   \end{equation}
   Therefore
   \begin{equation}\label{psiLWuu}
    \psi_L(w)\in W^{uu}_{\frac 12\ga}(\psi_L(y)).
   \end{equation} 
   From Proposition~\ref{pHPS} we get
  \begin{equation*}
  \forall |s|\le L \qquad 
  d(\psi_s(w),\psi_s(y))\le C\,\ee^{-\la(L-|s|)}
  \,d(\psi_L(w),\psi_L(y)).
  \end{equation*}
  Similarly, $B=\emptyset$ imples that 
  \begin{equation}\label{psi-Lwss}
  \psi_{-L}(w)\in W^{ss}_{\frac 12\ga}(\psi_{-L+v}(x))
  \qquad \text{ and }
  \end{equation}
  $$
  \forall |s|\le L
  \qquad
  d(\psi_s(w),\psi_{s+v}(x))\le C\,\ee^{-\la(L-|s|)}\,
  d(\psi_{-L}(w),\psi_{-L+v}(x)).
  $$
   Adding these inequalities we obtain
  \begin{align}
  \forall |s|\le L \qquad
  &d(\psi_s(y),\psi_{s+v}(x))\le C\,\ee^{-\la(L-|s|)}\,
  \big[d(\psi_L(w),\psi_L(y))+d(\psi_{-L}(w),\psi_{-L+v}(x))\big],
  \notag\\
  &\text{where }\qquad w:=\langle x,y\rangle 
  =W^{ss}_\ga(\psi_v(x))\cap W^{uu}_\ga(y).
  \label{d<egdd}
  \end{align}
  This proves inequality~\eqref{d<egdd1}.
  
    From~\eqref{vxyeta}, $|v(x,y)|\le\eta$.
  The fact $A\cup B=\emptyset$ also gives $|s(t)|\le 3\eta$
  for $t\in[-L,L]$. This proves~\eqref{setav}.
  From~\eqref{psiLWuu}, \eqref{psi-Lwss} and~\eqref{d<egdd}
  we get inequality~\eqref{dsh}.
  
  \qed

   \begin{Proposition}\label{B5}\quad
   
   Let $\be(\eta)$ be from Proposition~\ref{B16}.
   \begin{enumerate}[(a)]
   \item\label{B5a}
   If $x,\,y\in\La$ and $s:[0,+\infty[\to\re$ continuous
   with $s(0)=0$
   satisfy 
   $$
   d(\phi_{t+s(t)}(y),\phi_t(x))\le \be \qquad \forall t\ge0,
   $$
   then $|s(t)|\le 3\eta$ for all $t\ge 0$ and there is $|v(x,y)|\le\eta$
   such that $y\in W^{ss}_\ga(\phi_v(x))$.
   
   \item\label{B5b}
   Similarly, if $x,\,y\in\La$, $s:]-\!\infty,0]\to\re$ is continuous with $s(0)=0$
   and
   $$
   d(\phi_{t+s(t)}(y),\phi_t(x))\le\be \qquad \forall t\le 0,
   $$
    then $|s(t)|\le 3 \eta$ for all $t\le 0$ and there is $|v(x,y)|\le\eta$
   such that $y\in W^{uu}_\ga(\phi_v(x))$.
   \end{enumerate}
   \end{Proposition}

  \begin{proof}\quad
  
  We only prove item~\eqref{B5a}.
  The same proof as in Proposition~\ref{B16} shows that 
  taking
  $$
  w:=\langle x,y\rangle=W^{ss}_\eta(\phi_v(x))\cap W^{uu}_\eta(y)\ne \emptyset,
  $$
  we have that
  $|v|=|v(x,y)|\le\eta$ and 
  $$
   \emptyset=A:=\{\,t\in[0,+\infty[\;:\;|s(t)|\ge 3\eta\;\text{ or }
   \;d(\phi_t(y),\phi_t(w))\ge \tfrac 12 {\ga}\,\}.
  $$
  Therefore $|s(t)|\le 3\eta$ for all $t\ge 0$ and 
  $w\in W^{ss}_{\frac 12 \ga }(y)\cap W^{ss}_\eta(\phi_v(x))$.
  Since $\tfrac 12 \ga+\eta<\ga$ we get that
  $y\in W^{ss}_\ga(\phi_v(x))$.

  \end{proof}

  \medskip
  
   \begin{Proposition}\label{B71}\quad

         There are $D>0$, $\be_0>0$ and  open sets $X\in\cU\subset\fX^1(M)$,
     $\La\subset U\subset M$,  such that
     $$
     \forall \be\in]0,\be_0] \qquad \forall Y\in\cU,
     $$
   if $Y\in\cU$, $\psi_t=\psi^Y_t$  is the flow of $Y$, 
  $x,\, y \in \La^Y_U:=\bigcap_{t\in\re}\psi_t(\ov U)$ 
  and $s:\re\to\re$  continuous with $s(0)=0$ satisfy
  \begin{equation}\label{csh2}
  d(\psi_{t+s(t)}(y),\psi_t(x))\le\be\quad\text{ for }|t|\le L,
  \end{equation}
  then $|s(t)|\le D\be$ for all $|t|\le L$ and there is $|v|=|v(x,y)|\le D\be$ 
  such that
     \begin{align*}
  \forall |s|\le L, \qquad
  d(\psi_s(y),\psi_{s+v}(x))\le D\,\be\,\ee^{-\la(L-|s|)}.
  \end{align*}
  
  Moreover for all $|s|\le L$,
   \begin{equation}\label{dextxy}
  d(\psi_s(y),\psi_{s+v}(x))\le
   D\,\ee^{-\la(L-|s|)}\,
  \big[ d(\psi_L(y),\psi_{L+v}(x)) + d(\psi_{-L}(y),\psi_{-L+v}(x)) \big],
  \end{equation}
  and $v$ is determined by
  $$
  \langle x,y\rangle =W^{ss}_\ga(\psi_v(x))\cap W^{uu}_\ga(y)\ne \emptyset.
  $$

  \end{Proposition}

    \begin{proof}\quad
  
  Let $C$, $\cU$, $U$ $\eta_0>0$ and $B$ be from Proposition~\ref{B16}.
  The continuity of the hyperbolic splitting implies that the angle
  $\measuredangle(E^s,E^u)$ is bounded away from zero.
  As in the argument after \eqref{dpsvd}, there are invariant families of cones separating
  $E^s$ from $E^u$ whose image under the exponential map contain the local invariant 
  manifolds $W^{ss}_\ga$, $W^{uu}_\ga$.
  And hence as in~\eqref{a1xwy}
  there are $A,\,\be_1>0$ such that if $x,\, y\in \La^Y_U$,
  $d(x,y)<\be_1$ and 
  $$
  w=\langle x,y\rangle =W^{ss}_\ga(\psi_v(x))\cap W^{uu}_\ga(y),
  $$
  then 
  \begin{equation}\label{wpsiv}
  d(w,\psi_v(x))+d(w,y) \le A\, d(\psi_v(x),y).
  \end{equation}
  Suppose that $0<\be<\min\{\tfrac 1B\eta_0,\,\be_1\}$
  and 
  $x$, $y$, $s(t)$, $\psi^Y_t$, $L$ satisfy~\eqref{csh2}.
  Apply Proposition~\ref{B16} with $\eta:= B\be$.
  
  Then $|s(L)|\le 3\eta$, 
  and 
  \begin{align*}
  d(\psi_L(y),\psi_L(x))
  &\le d(\psi_{L+s(L)}(y),\psi_L(x))+|s(L)| \cdot \Vert Y\Vert_{\sup}
  \\
  &\le \be + 3\eta \lV Y\rV_{\sup} < \de,
  \end{align*}
  if $\be$ is small enough.
  So that $\langle \psi_L(x),\psi_L(y)\rangle$ is well defined.
  Similarly $|s(-L)|\le 3\eta$ and
  $d(\psi_{-L}(y),\psi_{-L}(x))<\de$.
  Since the time $t$ map $\psi_t$ preserves the family of
  strong invariant manifolds,
  in equation~\eqref{d<egdd1} we have that 
  \begin{align*}
  \psi_L(w) &=\langle \psi_{L}(x),\psi_L(y)\rangle =
  W^{ss}_\ga(\psi_{L+v}(x))\cap W^{uu}_\ga(\psi_L(y)),
  \\
    \psi_{-L}(w) &=\langle \psi_{-L}(x),\psi_{-L}(y)\rangle =
  W^{ss}_\ga(\psi_{-L+v}(x))\cap W^{uu}_\ga(\psi_{-L}(y)).
  \end{align*}
  Therefore, using~\eqref{wpsiv},
  \begin{align}
  d(\psi_L(w),\psi_L(y))+d&(\psi_{-L}(w),\psi_{-L+v}(x))
  \notag\\
  &\le  A \big[ d(\psi_{L+v}(x),\psi_L(y)) + d(\psi_{-L+v}(x),\psi_{-L}(y)) \big],
   \label{wxyd}
   \\
  d(\psi_{L+v}(x),\psi_{L}(y)) &\le
  d(\psi_{L+v}(x),\psi_{L}(x))+d(\psi_{L}(x),\psi_{L+s(L)}(y))
  +d(\psi_{L+s(L)}(y),\psi_{L}(y))
  \notag \\
  &\le |v| \lV Y\rV_{\sup}+\be+ |s(L)|\,\lV Y\rV_{\sup}
  \notag\\
  &\le B_1 \be,
  \notag
  \end{align}
  for some $B_1=B_1(\cU)>0$, because by Proposition~\ref{B16},
  $|v|\le \eta$, $|s(t)|\le 3\eta$ and $\eta = B \be$, so that
  $$
  |v| \le B\be,\qquad |s(t)|\le 3 B \be.
  $$
  A similar estimate holds for $d(\psi_{-L+v}(x),\psi_{-L}(y))$
  and hence from~\eqref{wxyd},
  $$
  d(\psi_L(w),\psi_L(y))+d(\psi_{-L}(w),\psi_{-L+v}(x))
  \le  2AB_1 \,\be.
  $$
  Replacing this  in~\eqref{d<egdd1} we have that
  $$
  \forall |s|\le L,\qquad
  d(\psi_s(y),\psi_{s+v}(x))\le D_1\,\be \,\ee^{-\la(L-|s|)},
  $$
  where $D_1=2 A B_1 C$.
  
  By~\eqref{wxyd} and~\eqref{d<egdd1} we also have that 
  $$
  d(\psi_s(y),\psi_{s+v}(x))\le
  AC\,\ee^{-\la(L-|s|)}\,
  \big[ d(\psi_L(y),\psi_{L+v}(x)) + d(\psi_{-L}(y),\psi_{-L+v}(x)) \big].
  $$
  Now take $D:=\max\{ D_1,\,B,\,3B,\,AC\,\}$.
  
  \end{proof}

  \begin{Definition}\label{dfe}\quad
  
  We say that $\psi|_\La$ is {\it flow expansive} if for every 
  $\eta>0$ there is $\ov\a=\ov\a(\eta)>0$ such that 
  if $x\in\La$, $y\in M$ and  there is $s:\re\to\re$ continuous
  with $s(0)=0$ and $d(\psi_{s(t)}(y),\psi_t(x))\le\ov \a$ for all $t\in\re$,
  then
  $y=\psi_v(x)$ for some 
  $|v|\le \eta$.
  \end{Definition}

  \begin{Remark}\label{rue}\quad
  
  Observe that Proposition~\ref{B16} implies 
  uniform expansivity in a neighbourhood of $(X,\La)$, namely
  there are neighbourhoods $X\in\cU\subset\fX^1(M)$ and $\La\subset U\subset M$
  such that for every $\eta>0$ there is $\a=\a(\eta,\cU,U)>0$ such that if
  $x\in\La^Y_U:=\cap_{t\in\re}\psi^Y_t(\ov U)$, $y\in M$, $s:(\re,0)\to (\re,0)$ 
  continuous and $\forall t\in \re$, $d(\psi^Y_{s(t)}\big(y),\psi^Y_t(x)\big)<\a$;
  then $y=\psi^Y_v(x)$ for some $|v|<\eta$.
  
  This also implies uniform h-expansivity as in Definition~\ref{duhe}.
  \end{Remark}
  
  \medskip
  
  \begin{Corollary}\label{Rfe}\quad
  
  There are open sets $X\in\cU\subset\fX^1(M)$, $\La\subset U\subset M$
  such that 
  
  \noindent 
  for all $Y\in\cU$ the set $\bigcap_{t\in\re}\psi^Y_t(\ov U)$ is hyperbolic for 
  $\psi^Y_t$ with uniform hyperbolic constants $C,\,\la$ on $\cU$ and
  $$
  \forall \eta>0 \qquad \exists \a=\a(\eta)>0 \qquad \forall Y\in\cU
  $$
  if $\psi^Y_t$ is the flow of $Y$, \;$x,y\in\bigcap_{t\in\re}\psi^Y_t(\ov U)$, 
  $s:\re\to\re$ is continuous, $s(0)=0$ and
  $$
  \forall t\in\re \qquad 
  d\big(\psi^Y_{s(t)}(y),\psi^Y_t(x)\big)\le \a,
  $$
  then $y=\psi_v(x)$ for some $|v|\le \eta$.
  \end{Corollary}
  
  \bigskip
  
  \begin{Definition}\label{B8}\quad
  
  Let $L>0$, we say that $(T,\Ga)$ is an $L$-specification if
  \begin{enumerate}[(a)]
  \item $\Ga=\{x_i\}_{i\in\Z}\subset \La$.
  \item $T=\{t_i\}_{i\in\Z}\subset\re$\quad and\quad  $t_{i+1}-t_i\ge L$\; $\forall i\in\Z$.
  \end{enumerate}
  We say that the specification $(T,\Ga)$ is $\de$-possible if
  $$
  \forall i\in\Z\qquad d(\psi_{t_i}(x_i),x_{i+1})\le \de.
  $$ 
  \end{Definition}
  
  If $s:\re\to\re$ we denote
  \begin{align*}
  U_\e(s,T,\Ga):&=\big\{\,y\in M\,\big|\, d(\psi_{t+s(t)}(y),\psi_t(x_i))\le\e\quad\text{for }
  t\in]t_i,t_{i+1}[\,\big\};
  \\
  STEP_\e(T):&= \big\{\, s \;\big|\; s|_{]t_i,t_{i+1}[} \text{ is constant,}\;
  s(t_i)\in\{ s(t_i-),\,s(t_i+)\}, 
  \\
  &\hskip 2cm |s(t_0)|\le \e \text{ and } |s(t_i+)-s(t_i-)|\le \e\;\big\};
  \\
  U^*_\e(T,\Ga):&=\bigcup \big\{\,U_\e(s,T,\Ga)\;|\;s\in STEP_\e(T)\,\big\}.
  \end{align*}
  
  If $y\in U^*_\e(T,\Ga)$ we say that  the point $y$ $\e$-{\it shadows} the specification
  $(T,\Ga)$. 
  
  \bigskip

  \begin{Remark}\label{RSHL}\quad
  \begin{enumerate}[(a)]
  \item
  Observe that a function $s\in STEP_\e(T)$ is possibly discontinuous. But
  from the conditions in $STEP_\e(T)$ and $U_\e(s,T,\Ga)$ it is easy to
  replace $s$ by a continuous function satisfying~\eqref{csh} with $\be= K \e$.
  Indeed, replace $s$ by
  \begin{equation}\label{sie1}
  \si(t) = s(t_i-\e)+ \frac{s(t_i+\e)-s(t_i-\e)}{2\e}\big( t-(t_i-\e)\big)
  \quad\text{if}\quad t\in[t_i-\e,t_i+\e],
  \end{equation}
  and $\si(t)=s(t)$ otherwise. Then
  \begin{align}
  d(\psi_{t+\si(t)}(y),\psi_{t+s(t)}(y))&\le \lV\partial_t\psi\rV_{\sup} |\si(t)-s(t)|
  \le  \lV\partial_t\psi\rV_{\sup} \e,
  \notag\\
   d(\psi_{t+\si(t)}(y),\psi_t(x_i))&\le \e +  \lV\partial_t\psi\rV_{\sup} \e
   = K\,\e \qquad \text{ when }\quad |t-t_i|\le\e.
   \label{Ke1}
   \end{align}
   \item  In~\eqref{sie1} the continuous function $t+\si(t)$ is strictly increasing.
   Indeed, $s(t)$ is constant on each interval $]t_i,t_{i+1}[$ and
   $|s(t_{i}+)-s(t_i-)|\le \e$, therefore
   $$
   |\si'(t)| =\lv\tfrac{s(t_i+)-s(t_i-)}{2\e}\rv\le\tfrac 12 
   \quad\text{ on } t\in[t_i-\e,t_i+\e], 
   \quad \si'(t)=0 \text{ otherwise}.
   $$
   And thus
   $$
   \tfrac{d\,}{dt}(t+\si(t))\ge 1-\tfrac 12 >0, \qquad t\ne t_i.
   $$
   \item\label{B7c}
    Similarly we can modify the function s by a function $\si$ which 
   is continuous, strictly increasing, satisfying \eqref{Ke1} for some $K$ 
   independent of $\e$ and also $\si(t_0)=0$.
   Indeed, define
  $$
  \si_2(t) =
  \begin{cases}
   \frac1{3\e}(t-t_0) \,s(t_0+3\e) &\text{if} \quad t-t_0\in[0,3\e],
   \\
   \frac 1{3\e}(t_0-t) \, s(t_0-3\e) &\text{if} \quad t-t_0\in[-3\e,0],
   \\
   \si(t) &\text{if} \quad t-t_0\notin[-3\e,3\e].
  \end{cases}
  $$
  Then $\si_2(t_0)=0$.
  Since $|s(t_0)|\le \e$ and $|s(t_0+)-s(t_0-)|\le\e$, if $\e<L/3$ we have that
 $|s(t_0+3\e)|=|s(t_0+)|\le 2\e$ and  $|s(t_0-3\e)|=|s(t_0-)|\le 2\e$. Therefore
 $|\si_2'(t)|\le \frac23$ and then
 $$
 \tfrac{d\,}{dt}(t+\si_2(t))\ge 1 -\tfrac 23 >0
 \qquad \text{if }\quad 0<|t-t_0|\le 3\e.
 $$
 Also $|\si_2(t)-s(t)|\le \max\{|s(t_0+)|,|s(t_0-)|\}\le 2\e$
  and the argument in \eqref{Ke1} gives
  $$
     d(\psi_{t+\si(t)}(y),\psi_t(x_i))\le \e +  \lV\partial_t\psi\rV_{\sup}  2\e
   = :K_2\,\e \qquad \text{ when }\quad |t-t_0|\le3\e.
  $$

   \end{enumerate}
   \end{Remark}
   \bigskip
  
  \begin{Theorem}[Bowen~\cite{Bowen6} Thm. (2.2) p. 6]\label{SHL}
  \quad
  
  Given $L>0$ there are $\de_0, \,Q>0$ such that if $0<\de<\de_0$ and
  $(T,\Ga)$ is a $\de$-possible $L$-specification on $\La$ then
  $U^*_\e(T,\Ga)\ne \emptyset$ with $\e= Q\de$.   \end{Theorem}
  
  \medskip 
  
  This Theorem is proven in Bowen~\cite{Bowen6} with a similarly
  presented statement without the estimate on $\e$. 
  In the context of Bowen~\cite{Bowen6} the
  set $\La$ is locally maximal but this is not 
  needed for Theorem~\ref{SHL}. 
  A proof of this theorem for flows without the local maximality hypothesis
  and with the explicit estimate on $\e$ appears in Palmer \cite{Palmer}
  Theorem~9.3, p. 188. In \cite{Palmer}, \cite{Palmer2009} the theorem 
  requires an upper bound on the lengths of the intervals in $T$. 
  This is because there the theorem is proven also for perturbations 
  of the flow. Indeed by Proposition~\ref{B16} longer intervals in 
  $T$  {\sl improve} the estimate on $\e$.
  
  \bigskip
  
  \begin{Remark}\quad
  \begin{enumerate}[(a)]
  \item Theorem~\ref{SHL} does not require the local maximality of $\La$. 
  \item Without the local maximality the shadowing orbit may not be in $\La$.
  \item In Palmer~\cite{Palmer} Theorem~9.3, p.~188 there is a proof for this
           Theorem where a specification for $\phi$ in $\La$ is shadowed by
           a perturbation $\psi$ of the flow. It requires an upper bound in the 
           lengths of the intervals in $T$ and the estimate is $\e=M(\de+\si)$,
           where $\si$ is the $C^1$ distance of their vector fields.
  \item  It is possible to shadow specifications which are in a neighbourhood
            of $\La$. Namely, given $\e>0$, $L>0$ there is $\de>0$ and a 
            neighbourhood $U(\La)$ of $\La$ such that  if $(T,\Ga)$ is a 
            $\de$-possible $L$-specification on $U(\La)$  then 
            $U^*_\e(T,\Ga)\ne\emptyset$.
  \item If $y\in U_\e(s,T,\Ga)$, by Remark~\ref{RSHL},  $s(t)$ can be replaced 
           by a continuous function
           satisfying~\eqref{csh2}  with $\be=\e K_2$ and such that $t\mapsto t+s(t)$ 
           is strictly increasing and $s(t_0)=0$. By Corollary~\ref{B71}, $|s(t)|\le \e K_3$ for some 
           $K_3>0$. 
  \item  If the specification is periodic  with period $T$ and $y\in U_\e(s,T,\Ga)$, 
           with $\si(t):=t+s(t)$ a homeomorphism we have that 
           $d(\psi_{\si(t)}(y),\psi_{\si(t+T)}(y))\le 2\e$, $\forall t\in\re$.
           By the flow expansivity of $\psi$ in $\La$ (Remark~\ref{rue}),
           if $\e$ is small enough then there is $\tau\in\re$ with $\psi_\tau(y)=y$.
           Then $y$ is a periodic point.
  
          \end{enumerate}
  \end{Remark}
  
   Therefore we get           

  \begin{Corollary}\label{CSH}\quad
  
    Given $\ell>0$ there are $\de_0=\de_0(\ell)>0$ and $Q=Q(\ell)>0$ 
    such that if $0<\de<\de_0$ and
  $(T,\Ga)=(\{t_i\},\{x_i\})_{i\in\Z}$ is a $\de$-possible $\ell$-specification on $\La$ then
  there exist $y\in M$ and $\si:\re\to\re$ continuous, piecewise linear,  strictly increasing with 
  $\si(t_0)=t_0$ and $|\si(t)-t| < Q\,\de$  such that
  $$
  \forall i\in \Z \quad
   \forall t\in]t_i,t_{i+1}[
   \qquad
  d\big(\psi_{\si(t)}(y),\psi_t(x_i)\big)< Q\,\de.
  $$
  Moreover, if the specification is periodic then $y$ is a periodic point for $\phi$.
   \end{Corollary}
   
   \bigskip

  \section{Symbolic Dynamics.}\label{ASD}
  
  Let $\cA:=\{T_1,\ldots,T_M\}$ be a finite set, 
  called the set of symbols, and let
  $$
  \cA^\Z =\textstyle\prod\limits_{n\in\Z}\cA
  =\big\{\,(x_i)_{i\in\Z}\;\big|\;\forall i\in\Z\;\;\; x_i\in \cA\,\big\}.
  $$
  We denote $\ox=(x_i)_{i\in\Z}$ and $\ox_i =x_i$.
  Endow $\cA$ with the discrete topology and $\cA^\Z$ with the 
  product topology. By Tychonoff Theorem $\cA^\Z$ is compact. 
  Given $a>1$ the metric 
  \begin{equation}\label{da}
  d_a(\ox,\oy)= a^{-n}, 
  \qquad 
  n=\max\{\, k\in\na \; |\; \forall|i|\le k,\quad x_i=y_i\,\}  
  \end{equation}
  induces the same topology.
  
  The {\it shift} homeomorphism $\si:\cA^\Z\to\cA^\Z$ is defined by
  $\si(\ox)_i=x_{i+1}$. A subset $\Om\subset\cA^\Z$ is called a {\it subshift}
  if $\Om$ is closed and $\si(\Om)=\Om$. 
  We call $\Om$ a {\it subshift of finite type} iff  
  there is a function
  $A:\cA\times\cA\to\{0,1\}$ (or equivalently a matrix $A\in\{0,1\}^{M\times M}$)
  such that
  $$
  \Om=\Si(A):=\{\,\ox\in \cA^\Z\;|\; \forall i\in\Z \quad A(x_i,x_{i+1})=1\;\}.
  $$
  
  Suppose that $\Om$ is a subshift and $\tau:\Om\to\re^+$ is a positive continuous function. 
  The {\it suspension} $S(\Om,\si,\tau)$ is defined as the topological quotient
  space $S(\Om,\si,\tau)=\Om\times\re/_\equiv$, where
  \begin{equation}\label{som}
  \forall (x,s)\in\Om\times\re\qquad \big(x, \,s+\tau(x)\big) \equiv \big(\si(x), s\big).
  \end{equation}
  Equivalently,
  $$
  \forall (x,s)\in\Om\times\re \quad \forall n\in\Z \qquad
  \Big(x,\; s+\tsum_{i=0}^{n-1}\tau(\si^i(x))\Big)
  \equiv\big(\si^n(x),s\big).
  $$
  Then $S(\Om,\si,\tau)$ is a compact metrizable space. A metric appears in \cite{BoWa}.
  
  We obtain the  {\it suspension} flow $S_t=sus_t(\Om,\si,\tau)$ by ``flowing vertically'' and remembering identifications, i.e.
  \begin{equation}\label{sust}
  S_t(P(x,s)) = P(x,s+t),
   \end{equation}
  where $P:\Om\times\re \to S(\Om,\si,\tau)$ is the canonical projection. 
  
  \begin{Definition}\label{DHSF}\quad
  
  A {\it hyperbolic symbolic flow} is a suspension flow $sus_t(\Si_A,\si,\tau)$
   on $S(\Si_A,\si,\tau)$,
  where $\Si_A$ is a subshift of finite type and $\tau:\Si_A\to\re$ is positive and
  Lipschitz with respect to the metric $d_a$ for some $a>1$.
  \end{Definition}
   
  Given a $\si$-invariant Borel probability $\nu$ on $\Si_A$, construct the
  $S_t$-invariant Borel probability $\mu$ on $S(\Si_A,\si,\tau)$ as  
  \begin{equation}\label{amunu}
  \int_{S(\Si_A,\si,\tau)} f \; d\mu:=
  \left(\int_{\Si_A} \tau\,d\nu\right)^{-1}
  \int_{\Si_A}\int_0^{\tau(x)} f(x,t)\;dt\; d\nu(x).
  \end{equation}
  
  \medskip
   
   \begin{Theorem}[{\bf Abramov Formula  \cite{abramov}}]\quad
   
   If $\nu$ is a $\si$-invariant probability on $\Si_A$ and $\mu$ is from
   \eqref{amunu} then their entropies satisfy
   $$
   h(S_1,\mu) = \frac{h(\si,\nu)}{\textstyle\int \tau\,d\nu}.
   $$
   \end{Theorem}

  \section{Markov Partitions.}\label{AMP}
  Markov partitions were constructed by Ratner~\cite{Rat0} for Anosov flows
  and then by Bowen~\cite{Bowen3} for locally maximal hyperbolic sets.
  In this paper we need to consider hyperbolic sets which may be non locally
  maximal. This appendix covers the definition and application of Markov partitions
  and Appendix~\ref{ALM} covers their construction.
  
  There are flows without a global cross section~\cite{Fried}.
  For example geodesic flows never have a global cross section
  because a closed geodesic in opposite directions give a counterexample
  to Fried's criterion~\cite[Theorem D]{Fried}.  In negative curvature 
  geodesic flows are hyperbolic (Anosov). Similarly, any energy level 
  of an autonomous Tonelli Hamiltonian has no global cross section 
  because the Liouville measure has zero asymptotic cycle (see Section~3
  in \cite{Maslov}).  Thus Markov partitions must be constructed from 
  {\sl local} transversal sections.
  
   In this section we follow Bowen~\cite{Bowen3} so that we can quote Theorem~\ref{TB22}
   refering its proof to~\cite{Bowen3}.
   Let $\phi_t:M\hookleftarrow$ be the flow of a $C^1$ vector field on a compact manifold $M$ and
   let $\La$ be a hyperbolic compact $\phi$-invariant set.
   This appendix deals mainly with the definition of a Markov
   partition and its symbolic dynamics. It will be applied to a larger
   hyperbolic set $X\supset \La$ which we construct on Appendix~\ref{ALM} 
   that will have a Markov partition.
   
   Suppose that $D\subset M$ is a a differentiable closed disk containing
   $x\in X$ of $\dim D=\dim M-1$ and transverse to the flow $\phi_t$. 
   Then $D$ is a local transverse section to $\phi_t$, i.e. there is $\xi>0$
   such that $(x,t)\to \phi_t(x)$ is a diffeomorphism of $D\times[-\xi,\xi]$ 
   onto a neighbourhood $U_\xi(D)$ of $x$. The projection map 
   $\pr_D:U_\xi(D)\to D$ defined by $\pr_D(\phi_t(y))=y$ for $|t|\le \xi$ 
   is differentiable.
  
   We use the canonical coordinates $\langle\cdot,\cdot\rangle$ 
   from~\ref{caco}. They are defined in    $X\times X$
   with values on $M$.
   For a closed set $T\subset X\cap D$, disjoint  from $\partial D$ of small diameter
   (depending on $d(T,\partial D)$) we have that 
    $$
   \langle x,y\rangle_D:  T\times T\to D,\qquad
   \langle x,y\rangle_D:=\pr_D\langle x,y\rangle
   =D\cap W^s_\ga(x)\cap W^u_\ga(y).
   $$
   is well defined and continuous.

   \begin{Definition}\quad
   
   We say that $T$ is a {\it rectangle} if $T$ is closed and $\langle x,y\rangle_D\in T$ for all
   $x,\,y\in T$; 
   in this case we write $\langle x,y\rangle_T$ for 
   $\langle x, y\rangle_D$ and notice that it does not actually depend on $D$.

   We say that a rectangle $T\subset D$ is {\it proper} if $T=\ov{T^*}$, 
   where $T^*$ is the
   interior of $T$ as subset of $D\cap X$.
   \end{Definition}
   
   \bigskip
   
      Observe that even if $X\cap D$ were a union of rectangles, it does not
   mean that $\langle\cdot,\cdot\rangle_D$ was closed on $X\cap D$.
   It may happen that $x$ and $y$ are close but in different rectangles 
   in $X\cap D$ and $\langle x,y\rangle_D\notin X$.

   For $x\in T$ with $T$ a small rectangle set
   \begin{align*}
   W^s(x,T)&:= \{\,\langle x,y\rangle_T\;|\; y\in T\,\}
   = T\cap \pr_D\big(U_\xi(D)\cap W^s_\xi(x)\big),
   \\
   W^u(x,T)&:=\{\,\langle z,x\rangle_T\;|\;z\in T\,\}
   = T\cap \pr_D\big(U_\xi(D)\cap W^u_\xi(x)\big),
   \end{align*}
   where $\xi$ is large compared to $\diam T$.
   The map $(u,v)\mapsto \langle u,v\rangle_T$
   defines a homeomorphism 
   $G_x:W^u(x,T)\times W^s(x,T)\to T$.
   
   \begin{Lemma}\label{LD2}\quad 
   
   The boundary $\partial T$ (as subset of $D\cap X$) consists of
   two parts $\partial T =\partial^s T\cup \partial^u T$, where
   \begin{align*}
   \partial^s T&=G_x\big(\partial W^u(x,T)\times W^s(x,T)\big),
   \\
   \partial^u T&=G_x\big(W^u(x,T)\times \partial W^s(x,T)\big).
   \end{align*}
   Here $\partial W^u(x,T)$ and $\partial W^s(x,T)$ denote 
   the boundaries of these sets as subsets of 
   \linebreak
   $W^u_\xi(x)\cap(D\cap X)$ and $W^s_\xi(x)\cap(D\cap X)$ respectively.

   \end{Lemma}
   
   \begin{proof}
   We have to prove that
   $$
   x\in\intt T \quad\Longleftrightarrow\qquad
   x\in\intt W^s(x,T)\; \cap \;\intt W^u(x,T).
   $$
   Suppose that $x\in\intt T$, then $W^u(x,T)=T\cap (W^u_\ga(x)\cap D\cap X)$
   is a neighbourhood of $x$ in $W^u_\ga(x)\cap D\cap X$ because 
   $T$ is a neighbourhood of $x$ in $D\cap X$. Therefore $x\in \intt W^u(x,T)$.
   Similarly $x\in\intt W^s(x,T)$.
   
   Now suppose that $x\in\intt W^s(x,T)\cap\intt W^u(x,T)$. Let $B_1$ be a 
   small neighbourhood of $x$ in $X$. Since $x\in \intt W^s(x,T)$, 
   if $B_1$ is small enough,
   $B_1\cap  W^{s}_\ga(x)\cap (D\cap  X)\subset W^s(x,T)$. 
   By the continuity of $\langle\cdot,\cdot\rangle_D$, there is an open set
   $x\in B_2\subset B_1$ such that if 
   $y\in B_2$ then $\langle x, y\rangle_D\in B_1$ and therefore 
   $\langle x, y\rangle_D\in W^s(x,T)$.   
   Similarly there is an open set $x\in B_3\subset B_2$ such that 
   if $y\in B_3$ then
   $\langle y, x\rangle_D\in W^u(x,T)$. By the rectangle property
   we have that if $y\in B_3\cap D$ then
   $$
   z= \big\langle \langle y,x\rangle_D,\langle x,y\rangle_D\big\rangle_D \in T.
   $$
   Since $z\in W^s_\ga(y)\cap D$ and $z\in W^u_\ga(y)\cap D$,
   by the expansivity property \ref{dfe}, $y=z\in T$.   
   We have proved that $T\supset B_3\cap D\ni x$.
   Therefore $x\in \intt T$.
   
   \end{proof}

      For $x\in T$ one can check that for small $\e>0$ one has
   \parskip -1pt
   \begin{enumerate}[(a)]
   \item $x\in\partial^s T$ iff  $x=\lim_n x_n$ for some sequence of 
   $x_n\in W^u_\e (x)\cap X$ with $x_n\notin\phi_{[-\e,\e]}T$.
   \item $x\in\partial^u T$ iff $x=\lim_n x_n$ for some sequence of 
   $x_n\in W^s_\e (x)\cap X$ with $x_n\notin\phi_{[-\e,\e]}T$.
   \end{enumerate}

   \begin{Definition}\label{dft}\quad
   
   A {\it proper family of size $\a$}\/ is a set $\fT=\{T_1,\ldots,T_n\}$ such that 
   \begin{enumerate}[(i)]
   \item each $T_1$ is a closed subset of $X$,
   \item \label{fT2} $X = \phi_{[-\a,0[}(\cup\fT)$, where $\cup\fT=T_1\cup\cdots\cup T_n$,
   \end{enumerate}
   and there are differentiable closed disks $D_1,\ldots D_n$ in $M$ 
   transverse to the flow $\phi_t$ so that 
   \begin{enumerate}[(i)]
   \setcounter{enumi}{2}
   \item $\dim D_i=\dim M-1$,
   \item $\diam D_i<\a$,
   \item $T_i\subset\interior D_i$ and $T_i=\ov{T_i^*}$, where $T_i^*$ is the 
            interior of  $T_i$ as a subset of the metric space $D_i\cap X$,
    \item\label{fT6}
             for $i\ne j$, at least one of the sets $D_i\cap\phi_{[0,\a]}D_j$ 
             and $D_j\cap\phi_{[0,\a]}D_i$ is empty; in particular
             $D_i\cap D_j=\emptyset$ and also $T_i\cap T_j=\emptyset$.     
   \end{enumerate}
   \end{Definition}
   
   Suppose that $\fT$ is as above and $\a$ is small. 
   From \eqref{fT2} it follows that for any $x\in\cup\fT$ there is a first 
   positive time $0<\tau(x)\le\a$ such that $\phi_{\tau(x)}(x)\in\cup\fT$. Since the
   $D_i$ are compact and pairwise disjoint and each $D_i$ is a local cross-section to 
   the flow, there is $\be>0$ such that $\tau(x)\ge\be$ for all $x\in\cup\fT$. 
   The {\it first return map} $F=F_\fT:\cup\fT\to\cup\fT$ defined by 
   $F_\fT(x)=\phi_{\tau(x)}(x)$ is a bijection: it is onto because by \eqref{fT2},
   $\cup\fT\subset X\subset\phi_{[-\a,0[}(\cup\fT)$.
  
   The functions $\tau(x)$ and $F_\fT$ are not continuous on $\cup\fT$.
   However they are continuous on 
   $$
   \cup'\fT:=\left\{\,x\in\cup\fT\;\Big|\;
   \forall k\in\Z\quad
   (F_\fT)^k(x)\in \textstyle \bigcup_{i=1}^n T_i^*\,\right\}.
   $$
   Since the rectangles are proper, the complement of $\cup' \fT$ is a countable union of local
   invariant manifolds $T_i\cap F^{-k}(\partial T_j)$
   which are nowhere dense in $\cup\fT$.
   By the Baire category theorem $\cup'\fT$ is dense in $\cup \fT$ 
   and
   $$
   \phi_\re(\cup'\fT) =\Big\{\, x\in X\;\Big|\;
   \phi_\re(x)\cap \cup\fT\subset \textstyle\bigcup_{i=1}^n T_i^*\,\Big\}
   $$
   is dense in $X$.

    For $x\in \cup'\fT$ let $q(x)$ be the unique $T_i$ containing $x$.
    Since $q:\cup'\fT\to\fT$ is continuous, the {\it itinerary} map 
    $Q:\cup'\fT\to\fT^\Z=\prod_\Z\fT$ given by
    $$
    Q(x) =\big(q(F_\fT ^i(x)\big)_{i\in\Z}
    $$
    is continuous. Since $\phi_t:X\to X$ is  flow expansive 
    (cf. Corollary~\ref{Rfe}) it follows that $Q$ is injective 
    (remember that $\a$ is small and use Theorem~3(iv) in \cite{BoWa}). 
    Thus the map
    $$
    Q^{-1}:Q(\cup'\fT)\to\cup'\fT
    $$
    is well-defined. The following lemma says that 
    $Q^{-1}$ extends to a continuous function 
    $\pi:\Om\to\cup\fT$, where
    $$
    \Om=\Om(\fT) =\ov{\{\,Q(x)\,|\,x\in\cup'\fT\,\}}
    \subset\textstyle \prod_\Z \fT.
    $$
    Observe that for the shift map $\si$, we have that $\si(\Om)=\Om$
    because $\si (Q(x)) = Q(F_\fT(x))$.
    
    \medskip 
    
    \begin{Lemma}[Bowen~{\cite[Lemma~2.2]{Bowen3}}]\label{LD4}\quad
    
    There is a continuous surjective map $\pi:\Om\to\cup\fT$ such that
    \begin{enumerate}[(i)]
    \item\label{LD4i} $\pi$ is Lipschitz with respect to the metric $d_a$ on $X$
             for some $a>1$.
    \item $\pi(\ov S)\in S_0$ for $\ov S\in \Om$.
    \item $\pi^{-1}\{x\}=\{ Q(x)\}$ for $x\in\cup'\fT$.         
    \end{enumerate}
    \end{Lemma}
    
    For $\ov S\in Q(\cup'\fT)$ consider  $\tau(\pi(\ov S))$. On the set
    \begin{equation}\label{cilT}
    \{\, \ov S=Q(x)\;|\; x\in T_i,\, F_\fT(x)\in T_j\,\}
    \end{equation}
    the function $\tau(\pi(\ov S))$ is just the time it takes for the point 
    $\pi(\ov S)$ to go from $D_i$ to $D_j$. As $D_i$, $D_j$ are differentiable
    local cross-sections, this time depends differentiably upon $x$, 
    hence  Lipschitz. On the metric $d_a$ the (cylindrical) sets \eqref{cilT}
    are open, {\sl closed} and {\sl disjoint}. Thus $\tau(\pi(\ov S))$ is Lipschitz 
    on $\ov S$ and extends to a function $f:\Om\to\re$ Lipschitz with respect 
    to the metric $d_a$. For $\ov S=Q(x)$ we have that 
    $$
    \phi_{f(\ov S)}\pi(\ov S) = \pi(\si(\ov S)),
    $$
    (both sides are $F_\fT(x)$); so by continuity this formula holds for all 
    $\ov S\in\Om$.
    
    Construct the suspension $S_t=sus_t(\Om,\si,f)$ as in appendix~\ref{ASD}.
    Define the map 
    \linebreak
    $\rho:S(\Om,\si,f)\to X$ by
    $$
    \rho(S_t(\ov S))= \phi_t(\pi(\ov S)).
    $$
    This map is well defined, and continuous, because 
    $\phi_{f(\ov S)}\pi(\ov S)=\pi(\si(\ov S))$.     
    We want $\Om$ to be a subshift of finite type, as in definition~\ref{DHSF}.
    For this we need the following Markov property:
    
    \begin{Definition}\quad\label{dfm}
    
    Let $\fT$ be a proper family for $\phi_t:X\to X$ of small size $\a$.
    We say that $\fT$ is {\it Markov} if
    \begin{enumerate}
    \item each $T_i$ is a rectangle,
    \item $x\in U(T_i,T_j):=\ov{\{\,w\in \cup'\fT\;|\; w\in T_i,\, F_\fT(w)\in T_j\,\}}$
             implies $W^s(x,T_i)\subset U(T_i,T_j)$.
    \item $y\in  V(T_k,T_i):=\ov{\{\,w\in \cup'\fT\;|\; F_\fT^{-1}(w)\in T_k,\, w\in T_i\,\}}$
             implies $W^u(y,T_i)\subset V(T_k,T_i)$.       
    \end{enumerate}
    \end{Definition}
    
    Define $A_\fT:\fT\times\fT\to\{0,1\}$ by $A(T_i,T_j)=1$ iff 
    $\exists x\in\cup'\fT$ such that $x\in T_i$ and $F_\fT(x)\in T_j$.
    Then $\Om\subset \Si(A_\fT)$ with the notation of appendix~\ref{ASD}.
    
    \bigskip
    
    \begin{Theorem}[{Bowen~\cite{Bowen3} Theorem 2.4, p. 437}]
    \label{TB22}\quad
    $\Om=\Si(A_\fT)$ \, iff \; $\fT$ is Markov.
    \end{Theorem}

    \bigskip
    
    \section{Local Maximality.}\label{ALM}
    
   A (compact) hyperbolic set $\La$ for the flow $\phi_t$ on $M$ 
   is {\it locally maximal} (or {\it isolated})
   if there is an open set $\La\subset U\subset M$ such that 
   \begin{equation}\label{dlm}
   \La = \bigcap_{t\in\re}\phi_t(U).
   \end{equation}
   
   We say that a hyperbolic set has {\it local product structure}
   if there are $\de,\,\ga>0$ as in the canonical coordinates~\ref{caco}
   such that
   \begin{equation*}\label{dlps}
   x,\,y\in \La,\quad d(x,y)<\de
    \quad\then \quad
    \langle x,y\rangle \in \La,
   \end{equation*}  
   where $\langle\cdot,\cdot\rangle$ is from~\eqref{ecaco}.
   
   We show that this definition is invariant under time reversal of the flow.
   Indeed,
   it is symmetric on $x$ and $y$, but the definition of 
   \begin{equation}\label{bra}
   \langle x,y\rangle = W^s_\ga(x)\cap W^{uu}_\ga(y)
   \end{equation}
   is not invariant under time reversal. Nevertheless
   the symmetric definitions 
   $$
   W^s_\ga(x)\cap W^{uu}_\ga(y)\quad\text{ and }\quad
   W^{ss}_\ga(x)\cap W^{u}_\ga(y)
   $$
   are in the same orbit because
   $$
   W^s_\ga(x) = \phi_{[-\ga,\ga]}(W^{ss}_\ga(x))
   \quad\text{ and }\quad
   W^u_\ga(x) = \phi_{[-\ga,\ga]}(W^{uu}_\ga(x)).
   $$
   Since $\La$ is $\phi_t$-invariant we get that the property of having 
   local product structure is invariant under time reversal.
   
    \parskip+5pt
    
  \begin{Remark}\quad\label{LaU}
  
    Given a hyperbolic set $\La$ there is a neighbourhood $V$ of $\La$ 
    such that for any open set $\La\subset U\subset V$ the set
    $$
    \hat\La :=\bigcap_{t\in\re}\phi_t(\ov U)
    $$
    is also a hyperbolic set. Indeed, it is enough to consider the 
    time 1 map $\phi_1$ as a partially hyperbolic diffeomorphism
    and obtain the extension to a neighbourhood of the dominated 
    splitting for $\phi_1$ as in~\cite{BoDiVi} appendix~B, p.~289 
    and pp.~292-293.  
     \end{Remark}
  
   \medskip
   \begin{Proposition}\label{LMlps}\quad
   
   A hyperbolic set is locally maximal if and only if it has local product structure.
   
   \end{Proposition}
   
   Here we adapt the proof  for diffeomorphisms 
   given in Theorem 18.4.1 in Hasselblatt-Katok \cite{HK}.
   
   \begin{proof}\quad
   
   Suppose that $\La$ is locally maximal with $U$ as in~\eqref{dlm}.
    Take $\de,\,\ga$ as in the canonical coordinates~\ref{caco} with
    \begin{equation*}
    B_\ga(\La):=\{\,x\in M\;|\; d(x,\La)\le \ga\,\}\subset U.
    \end{equation*}
    Then if $x,\,y\in\La$ and $d(x,y)<\de$ we have that
    \begin{align*}
    \langle x,y\rangle &=W^{ss}_\ga(\phi_\nu(x))\cap W^{uu}_\ga(y)
    \\
    &=W^s_\ga(x)\cap W^{uu}_\ga(y)
    \subset \Big[\, \textstyle\bigcap\limits_{t\le 0}\phi_t(U) \Big]
    \cap 
    \Big[\,\bigcap\limits_{t\ge 0}\phi_t(U)\Big]
    =\La.
    \end{align*}
    Therefore $\La$ has local product structure.
   
    For the converse we need the following 
    \begin{Lemma}\label{Le2}
    Let $\La$ be a hyperbolic set with local product structure.
    \newline
    There exist $\de_1,\,\de_2>0$ such that if $x\in\La$, 
    $y\in W^{uu}_{\de_1}(x)$ and 
    $d(\phi_t(y),\La)<\de_2$ for all $t\ge 0$, then $y\in\La$.
    \end{Lemma}

    \begin{proof}\quad
    
    Using Proposition~\ref{pHPS} choose $T>0$ such that for some $\mu<1$
    \begin{equation}\label{emu}
    d(\phi_{-T}(x),\phi_{-T}(y))\le \mu\, d(x,y)\qquad
    \text{ when } x\in \La,\; y\in W^{uu}_\e(x).
    \end{equation}
    Let $K:=\sup\lV D\phi_T\rV$ taken over a neighbourhood of $\La$.
    Let $\de,\ga>0$ be from the definition of local product structure for $\La$.
    Observe that if $\de_3:=\max\{\de_1,\de_2\}$ is small enough, we have that
    both $x$ and $y$ are in the larger hyperbolic set
    \begin{equation}\label{hLa}
    \hat \La =\bigcap_{t\in\re} \phi_t\big(B_{\de_3}(\La)\big).
    \end{equation}
    Let $d^u$ be the distance along the strong unstable leaves $W^{uu}$.
    The continuity of the hyperbolic splitting implies that the angles among 
    its subspaces are bounded below. Then there exists $M>0$ such that 
    \begin{equation}\label{pqp}
    d^u (p, \langle q,p\rangle)
    =d^u\big(p,W^s_\ga(q)\cap W^{uu}_\ga(p)\big)
    \le M\, d(p,q),
    \end{equation}
    whenever $d(p,q)$ is small enough.
    For any $\de_1<\min\{\frac 1{MK},\frac 12,\frac\de 2,\frac\ga 2\}$
    take $\de_2\le \min\{\de_1 /MK,\,\de_1\}$ such that 
    \begin{equation}\label{usede2}
    p,q\in\hat\La,\quad d(p,q)<\de_2\quad\then\quad W^s_{\de_1}(p)\cap W^{uu}_{\de_1}(q)\ne\emptyset.
    \end{equation}
    \begin{Claim}\label{c4}
    If $x\in\La$, $y\in W^{uu}_{\de_1}(x)$ and 
    $d(\phi_t(y),\La)<\de_2$ for all $t\ge 0$, then
    $$
    a_n:=\min_{z\in\La\cap W^{uu}(\phi_{nT}(y))}
    d^u(\phi_{nT}(y),z) = 0
    \qquad \text{for some } 
    n\in\na.
    $$
    \end{Claim}
    Then by~\eqref{emu} for $z$ attaining the minimum  we would have
    $$
    d(y,\La)\le  d(y,\phi_{-nT}(z))\le  \mu^n\,a_n=0.
    $$
    
    Since $\La$ is closed and $\phi_t$ invariant, for Lemma~\ref{Le2} it is enough 
    to prove Claim~\ref{c4}.
    
    \end{proof}
    
    To prove  Claim~\ref{c4} we first show inductively that 
    \begin{equation}\label{ande}
    a_n<\de_1\qquad\text{  for all }
    n\in\na.
    \end{equation}
    Since $y\in W^{uu}_{\de_1}(x)$, we have that 
    $$
    a_0\le d^u(y,x)<\de_1.
    $$
    Suppose that $a_n<\de_1$ then there is $w_n\in\La\cap W^{uu}(\phi_{nT}(y))$
    such that $d^u(\phi_{nT}(y),w_n)<\de_1$. Take $z_n\in \La$ such that 
    $d(\phi_{nT}(y),z_n)<\de_2$.  Using \eqref{usede2} observe that 
    \begin{align*}
    p_n:=\langle z_n,\phi_{nT}(y)\rangle
    &=W^s_{\de_1}(z_n)\cap W^{uu}_{\de_1}(\phi_{nT}(y))
    \\
    &=W^s_{\de_1}(z_n)\cap W^{uu}_{2\de_1}(w_n)
    =\langle z_n,w_n\rangle\in \La
    \end{align*}
    by the local product structure.
    Therefore, using \eqref{pqp},
    $$
    a_n\le d^u(\phi_{nT}(y),p_n)\le M\; d(\phi_{nT}(y),z_n)
    < M \de_2 \le \frac{\de_1}K.
    $$
    By the choice of $K$ we  have that
    $a_{n+1}\le K \,a_n<\de_1$. 
    This proves~\eqref{ande}.

    The definition of $a_n$ implies that
    $$
    a_n \le \mu\,a_{n+1}
    $$
    when $a_{n+1}<\e$ and $\mu<1$ is from~\eqref{emu}.
    Therefore, if $\de_1$ is small enough we have that 
    $a_n<\de_1\Rightarrow a_{n+1}\ge \mu^{-1} a_n$.
    Combining this with~\eqref{ande} we get that $a_n=0$
    for all $n\in\na$.
     This proves Claim~\ref{c4}.
     
      \hfill$\triangle$
     
     Using Lemma~\ref{Le2} we finish the proof of Proposition~\ref{LMlps}.
     Let $\rho>0$ be such that 
     $$
     z,\,w\in\hat\La,\quad 
     d(z,w)<\rho\quad  \then \quad
     W^s_{{\de_2}/2}(z)\cap W^{uu}_{{\de_2}/2}(w)\ne \emptyset.
     $$
      Take $\de_4:=\min\{\rho,\de_2/2\}$.
     Suppose now that 
     $d(\phi_t(y),\La)<\de_4$ for all $t\in\re$. From~\eqref{hLa}
     we have that 
     $y\in\hat\La$.
     Let $x\in\La\subset\hat\La$ be such that $d(y,x)<\frac{\de_2}2$ and let
     $$
     p=\langle y,x\rangle 
     = W^s_{{\de_2}/2}(y)\cap W^{uu}_{{\de_2}/2}(x).
     $$
     We have that 
     $$
     \forall t\ge 0 \qquad
     d(\phi_t(p),\La) \le d(\phi_t(p),\phi_t(y))+d(\phi_t(y),\La)
     \le  \frac{\de_2}2+\frac{\de_2}2\le \de_2,
     $$
     and $p\in W^{uu}_{\de_1}(x)$, $x\in \La$.
     By Lemma~\ref{Le2}, $p\in\La$.
     Applying the same arguments to the reverse flow  $\phi_{-t}$
     (recall the discussion in~\eqref{bra})
     we get 
     $$
     q=W^u_{\de_2/2}(y)\cap W^{ss}_{\de_2/2}(x)\in \La.
     $$
     Thus $p,\,q\in\La$ and
     $$
     y\in W^s_{\de_2/2}(p)\cap W^u_{\de_2/2}(q)
     \subset \phi_{[-\ga,\ga]}
     \big(W^s_{\ga}(p)\cap W^{uu}_\ga(q)\big)
     =\phi_{[-\ga,\ga]}(\langle p,q\rangle).
     $$
     By the local product structure $\langle p,q\rangle\in\La$.
     Therefore by the invariance of $\La$, $y\in\La$.
     This finishes the proof of Proposition~\ref{LMlps}.
     
      \end{proof} 
   
   Crovisier~\cite{crovisier} has shown an example of a hyperbolic set $\La$ 
   for a diffeomorphism which is not contained in a locally maximal
   hyperbolic set. Nevertheless Fisher~\cite{Fisher} shows that for 
   diffeomorphisms every hyperbolic set has an extension with a Markov
   partition.   Here we extend Fisher theorem to flows.
   See also Bowen~\cite{Bowen0}.

  \bigskip
  
  \begin{Theorem}\label{E4}\quad
  
  Let $\phi_t$ be a flow on a compact manifold $M$.
  If $\La$ is a hyperbolic  set without fixed points for $\phi_t$
  and $U\supset\La$ is an open neighbourhood of $\La$.
  Then there is a hyperbolic set $\La\subset\tilde{\La}\subset U$ 
  which has a Markov partition.
  \end{Theorem}
  
    \medskip

    The remaining of the section is dedicated to the
    
    \noindent{\bf Proof of Theorem~\ref{E4}:}

                Let $V\subset U_0$ be  a neighbourhood of $\La$ such that the 
            Shadowing Theorem~\ref{SHL} holds for specifications in $V$
            with jumps of size at most $\de_0$.
      Using Remark~\ref{LaU} obtain an
    open set $U$ such that $\La\subset U\subset \ov U\subset V$ and 
   that
   \begin{equation}\label{defu}
   \La_U:=\bigcap\limits_{t\in\re}\phi_t(\ov U)
   \end{equation}
   is hyperbolic.

   Let $0<\a_0\ll  1$ be such that

   \begin{gather}
   \a_0\,\lV\partial_t\phi\rV_{\sup}
     + 2 \,\Big|\sup_{|t|\le 1} \Lip(\phi_t)\Big| \,\a_0
     < \tfrac 12 (\text{\smallskip
     flow expansivity constant on $\La_U$}),
     \label{alph0}\\
       \forall x\in \ov U\qquad \diam \phi_{[-\a_0,\a_0]}(B_{\a_0}(x))<\tfrac 14\ga,
    \label{alfa1}
   \end{gather}
   where $\ga$ is from the definition of canonical coordinates~\ref{caco}.
   
   Given $0<\a<\a_0$    we first construct a proper family of size $\a$ made with rectangles
   as in definitions~\ref{dft} and~\ref{dfm}.
    Choose a finite family of smooth disjoint  open discs $D_1,\cdots, D_m$
   transversal to the flow such that 
   \refstepcounter{equation}\label{edix}
   \begin{enumerate}[(\ref{edix}.a)]
   \item $\dim D_i=\dim M-1$. \label{edixa}
   \item $D_i$ is open.
   \item $\diam D_i<\a$.
   \item For $i\ne j$ at least one of the  sets
            $D_i\cap \phi_{[0,4\a]}( D_j)$  or $\phi_{[0,4\a]}(D_i)\cap D_j$
            is empty.
   \item $\La \subset \bigcup_{i=1}^m \phi_{]-\a,0[}(\vD_i)$,
   where  \label{edixe}
   \begin{equation}\label{brev}
   \vD_i=\big\{\,x\in D_i\;|\; d(x,\partial D_i)>\tfrac{\a}{10}\big\}.
   \end{equation}
   \item\label{edixf} The sets $D_i\cap\phi_{]-\a,0]}(D_j)\cap \phi_{]-\a,0]}(D_k)$ and $D_i\cap\phi_{]-\a,0]}(D_j)$ 
   when non-empty are connected.
     \end{enumerate}
     
      Let $\D:=\bigcup_{i=1}^m D_i$ and 
   \begin{equation}\label{mbeta}
   2\be :=\inf\{\,t>0\;|\; x\in\D,\;\phi_t(x)\in\D\,\}>0.
   \end{equation}
   By~(\ref{edix}.\ref{edixe}) $0<2\be<\a$.

   Let $A>1$ be such that  
      \begin{equation}\label{Alip}
       \exists i \quad y,\,z\in\phi_{[-2\a,2\a]}(D_i)
      \quad 
   \then \quad d(P_{D_i}(y),P_{D_i}(z)) \le A\, d(y,z),
   \end{equation}
   where $P_{D_i}:\phi_{[-2\a,2\a]}(D_i)\to D_i$ is the projection 
   \begin{equation}\label{projd}
   P_{D_i}(\phi_t(y))=y,  \quad\text{ when }\quad y\in D_i.
   \end{equation}
   Denote the projection time by
  $ \tau_{D_i}:\phi_{[-2\a,2\a]}(D_i)\to[2\be,2\a]$
   where
   $$
   \tau_{D_i}(x):=\min\{\,t>0\;|\; \phi_t(x)\in D_i\,\}.
   $$

   Let 
   \begin{equation}\label{me0}
   0<\e_0<\tfrac{\a}{40} \cdot \tfrac 1{A+1},
   \end{equation}
    be a Lebesgue number for the open cover $\{\phi_{]-\a,0[}(D_i)\}_{i=1}^m$
   of $\La$. 
   From~\eqref{alph0} and~\eqref{me0} we have that 
   \begin{equation}   \label{aphisup1}
        \a_0\,\lV\partial_t\phi\rV_{\sup}
     + 2 \,\Big|\sup_{|t|\le \a_0} \Lip(\phi_t)\Big| \,\e_0
     < \tfrac 12 (\text{\small flow expansivity constant on $\La_U$}).
       \end{equation}
       
   Let $0<\e_1<\e_0$ be such that 
   \refstepcounter{equation}\label{defe1}
   \begin{enumerate}[(\ref{defe1}.i)]
   \item \quad $0<4\e_1<\min_{i\ne j}d(D_i,D_j)$.
   \label{me11}
   \item \quad
   $\exists j \quad w\in \vD_j  \quad \then \quad
   B_{\e_1}(w)\subset \phi_{[-\a,\a]}(D_j)$,
    where $\vD_j$ is from \eqref{brev}.
    \label{me115}
   \item \quad
   $B_{4\e_1}(\La)\subset\textstyle  \bigcup_{i=1}^m \phi_{]-\a,0[}(D_i)$.
   \label{me12}
   \item \quad
   $\La_1:=\bigcap\nolimits_{t\in\re} \phi_t(\ov{B_{3\e_1}(\La)})
   \qquad\text{is hyperbolic}$.
   \label{me13}
   \end{enumerate}
   Using $A$ from~\eqref{Alip}, let $0<\e_2<\e_1$ be such that
   \refstepcounter{equation}\label{defe2}
   \begin{enumerate}[(\ref{defe2}.i)]
  \item\label{me2}
   $0<\e_2<\frac {\e_1}A<\e_1$
   \item\label{me21}
   $\tfrac{(1+A)\,\e_2}{\lV\partial_t\phi\rV_{\sup}}
   < \tfrac 14 \,\a$.
   \item\label{me22}
    $5 \e_2<\ov\a(\tfrac14\be)$ is a flow expansivity constant for the hyperbolic set  $\La_1$
   for $\eta=\tfrac 14\be$ in Definition~\ref{dfe}.
       \item\label{me225} 
    $ \e_2 < \tfrac 12 \,\be(\tfrac 14 \a)$,
     where $\be(\eta)$ is from Proposition~\ref{B16} for $\La_1$.
     \end{enumerate}
 
     Using $A>1$ from~\eqref{Alip} and Corollary~\ref{CSH}, let
   \begin{equation}\label{me3}
   0<\e_3< \e_2
   \end{equation}
   be such that 
    any $2\e_3$-possible $\be$-specification on $\La_1$
    is $\tfrac 12 \e_2$-shadowed
    (by an orbit which is possibly not in $\La_1$).
    Let
    \begin{equation}\label{me4}
    0<\e_4<\frac{\e_3}{A+1}.
    \end{equation}

    \begin{Lemma}\label{40Ae}
    If $z\in D_i$, $d(z,\vD_i)<4\e_1$,		 
    $w=\phi_{a}(z)\in D_j$, $0<a<\a$,
    $d(w,\vD_j)<4\e_1$,
    then $B_{4\e_1}(z)\cap D_i\subset\phi_{]-\a,0[}(D_j)$.
    \end{Lemma}
    
   \begin{proof}
   Since $\e_1<\e_0$, from \eqref{me0} we have that 
   \begin{equation}\label{dwd}
   4\e_1\, A < \tfrac{\a}{10}-4\e_1.
   \end{equation}
       By \eqref{Alip} the projection $P_{D_j}:\phi_{]-\a,0[}(D_j)\to D_j$
   has Lipschitz constant $A$.
 Since $w=P_{D_j}(z)$ and $d(w,\vD_j)<4\e_1$, we have that 
   \begin{align*}
   d(P_{D_j}(z),\partial D_j)=
   d(w,\partial D_j)&\ge d(\vD_j,\partial D_j) - d(w,\vD_j)
   \\
   &\ge \tfrac \a{10}-4\e_1.
   \end{align*}

   If $x\in B_{4\e_1}(z)\cap D_i\cap\phi_{]-\a,0[}(D_j)=:E$, then
   \begin{align*}
   d(P_{D_j}(x),\partial D_j)&\ge
   d(P_{D_j}(z),\partial D_j) -d(P_{D_j}(x),P_{D_j}(z))
   \\
   &>\tfrac \a{10}-4\e_1-  4\e_1\, A>0.
   \end{align*}
   This implies that the set $E$ contains $B_{4\e_1}(z)\cap D_i$ and hence
   $B_{4\e_1}(z)\cap D_i\subset \phi_{]-\a,0[}(D_j)$.

   \end{proof}
   
   Let
   $$
   \D :=\cup_{i=1}^m D_i.
   $$
   For $x\in\D$ let $D(x):=D_i$ where $x\in D_i$.
    
    Let $W$ be a finite $\e_4$-dense set in $\La\cap \bigcup_{i=1}^m \vD_i$.
       For $w\in W$, by (\ref{defe1}.\eqref{me12})  and the definition of $\e_0$ after
       \eqref{me0}, there is $j\in\{1,\ldots, m\}$
   such that
   \begin{equation}\label{me1234}
   B_{3\e_1}(w)\subset \phi_{]-\a,0[}(D_j). 
   \end{equation}
   For $w\in W$ let $E(w):=B_{3\e_1}(w)\cap D(w)$.
   From Lemma~\ref{40Ae} we get
   \begin{Corollary}\label{core6}\quad
  
    If $w_1,\,w_2\in W$ and
    $P_{D(w_2)}(E(w_1))\cap E(w_2)\ne \emptyset$
    then $E(w_1)\subset \phi_{]-\a,0[}(D(w_2))$.
    In particular $P_{D(w_2)}:E(w_1)\to D(w_2)$ is well defined
    and smooth.
     \end{Corollary}

        Let $\Om\subset W^\Z$ be the set of sequences $\ov w=(w_k)_{k\in\Z}$
   such that $\forall k\in \Z$, $w_k\in \phi_{]-\a,0[}(D(w_{k+1}))$ and
   \begin{equation}
   \forall k\in\Z\qquad d(P_{D(w_{k+1})}(w_k),w_{k+1})<\e_3.
   \label{defOm1}
   \end{equation}
     Then $\Om$ is closed and invariant under the shift map $\si:W^\Z\to W^\Z$,
   $\si(\ow)_i=w_{i+1}$. In fact it is a subshift of finite type.
   
   Observe that Corollary~\ref{core6} and \eqref{mbeta} imply that if $\ow=(w_k)_{k\in\Z}\in\Om$
   then
   \begin{equation}\label{mtau1}
   \be < \tau_{D(w_{k+1})}|_{E(w_{k})}\le\a.
   \end{equation}

     If $\ow\in\Om$, 
       write
       \begin{equation}\label{snt}
       S_n\tau(\ow):=
       \begin{cases}
       \textstyle
      \phantom{-} \sum\limits_{k=0}^{n-1}\tau_{D(w_{k+1})}(w_k) &\text{if }\quad n\ge 1,
       \\
       \qquad 0 &\text{if }\quad n=0,
       \\
       -\sum\limits_{k=-n}^{-1} \tau_{D(w_{k+1})}(w_k) &\text{if } \quad n \le -1.
       \end{cases}
       \end{equation}
       Let $f_{\ow}$ be the  $\e_3$-possible 
       $\be$-specification
       \begin{align}
      f_\ow\big(S_n\tau(\ow)+t\big) = \phi_t(w_n),\qquad t\in[0,\tau_{D(w_{k+1})}(w_n)[,
      \quad n\in\Z.
      \label{fwt}
       \end{align}
       By \eqref{me3} and  the Shadowing Corollary~\ref{CSH}
   there is $y\in M$ and $s:(\re,0)\to (\re,0)$ strictly increasing 
   piecewise linear  with $s(0)=0$
   such  that 
   \begin{equation}\label{shy}
   d(\phi_{s(t)}(y),f_\ow(t))<  \e_2, 
   \qquad
   |s(t)-t|<\e_2.
   \end{equation}
   Since $f_{\ov w}(0)=w_0$ and $s(0)=0$, we have that $d(y,w_0)<\e_2<\e_1$.
   Also
    by 
   (\ref{defe1}.\ref{me115}),
   $B_{\e_1}(w_0)\subset\phi_{[-\a,\a]}(D(w_0))$
   and by \eqref{Alip} and (\ref{defe2}.\ref{me2})
   \begin{equation}\label{dwpy2}
   d(w_0,P_{D(w_0)}(y))\le A \,d(w_0,y)
   \le A\, \e_2 <\e_1.
   \end{equation}
    Thus $P_{D(w_0)}(y)\in E(w_0) = B_{3\e_1}(w_0)\cap D(w_0)$.
  
   Define $ \Pii :\Om\to \D$ by
   \begin{equation}\label{defpi1}
  \Pii (\ov w) := P_{D(w_0)}(y)\in E(w_0),
   \end{equation}
   where $y$ is from~\eqref{shy}.
   Since $W\subset\La$ and in \eqref{shy}  $s:\re\to\re $ is a homeomorphism,
   we have that $d(\phi_s(y),\La)<\e_2$ for all $s\in\re$.
   From the definition in (\ref{defe1}.\ref{me13})
   we have that 
   \begin{equation}\label{PiLa12}
   \Pii (\ov w)\in \La_1.
   \end{equation}

     From~\eqref{shy} and ~\eqref{dwpy2} we have that 
   $$
   d(y,P_{D(w_0)}(y))
   \le d(y,w_0)+d(w_0,P_{D(w_0)}(y))
   < (1+A)\,\e_2.
   $$
   Using~(\ref{defe2}.\ref{me21}) we have that $\Pii (\ov w)=P_{D(w_0)}(y)=\phi_b(y)$ with
   \begin{equation}\label{bndna1}
   |b|<\tfrac{(1+A)\,\e_2}{\lV\partial_t\phi\rV}
   < \tfrac 14 \,\a.
   \end{equation}

      For the record, using~\eqref{shy}, \eqref{bndna1} and $y=\phi_{-b}(\Pii (\ow))$, we have that 
   \begin{equation}\label{shy2}
   \exists |b|\le \tfrac{\a}4  \qquad \forall t\in\re
   \qquad
   d(\phi_{(s(t)+b)}(\Pii (\ow)),f_\ow(t))<\e_2,
   \end{equation}
   where $s:(\re,0)\to(\re,0)$ is continuous, strictly increasing and $s(0)=0$.

      On $\Om$ we use the restriction of the metric $d_a$  in $W^\Z$ 
    defined in~\eqref{da}.
    \begin{Lemma}\label{le6}
    \qquad 
    
    There is  $a\in]0,1[$ such that the map $\Pii $ is Lipschitz for $d_a$.
    
    If $w\in W$ and $T_w:=\Pii (\{\ow\;|\; w_0=w\,\})$ then
    \begin{equation}\label{twe212}
    \diam T_w < \e_2.
    \end{equation}
    In particular, from $\e_2<\e_3$ and~\eqref{mtau1},
    \begin{equation}\label{mtau2}
    \be < \tau_{D(w_{k+1})}|_{T_{w_{k}}}\le\a.
    \end{equation}

    \end{Lemma}

       \begin{proof}
    Observe that  from~\eqref{shy}, 
    if $\ow,\,\oz\in\Om$ and $d(\ow,\oz)<a^N$, then
    there are strictly increasing piecewise linear  functions $s_1,\,s_2:(\re,0)\to(\re,0)$ 
    such that 
    \begin{align}
    \forall t\in[S_{-N}&\tau(\ow),S_N\tau(\ow)]
    \notag
    \\
    &d(\phi_{s_1(t)}(\Pii (\ow)),f_\ow(t))<\e_2
    \quad\text{and}\quad
    d(\phi_{s_2(t)}(\Pii (\oz)),f_\ow(t))<\e_2,
    \label{s1s212}
    \end{align}
    because $f_\ow(t)=f_\oz(t)$ for all 
    $S_{-N}\tau(\ow)\le t\le S_N\tau(\ow)$.
    Thus
    $$
    d\big(\phi_{s_1(t)}(\Pii \ow),\phi_{s_2(t)}(\Pii \oz)\big)< 2 \e_2
    \qquad \text{ for all } \quad  t\in[S_{-N}\tau(\ow),S_N\tau(\ow)].
    $$
    By \eqref{mbeta} we have that $S_{-N}\tau<-N\be$ and $S_N\tau>N\be$.
    Since $s_2$ is a homeomorphism, we have that $s_1\circ s_2^{-1}:\re\to\re$
    is continuous, $s_1\circ s_2^{-1}(0)=0$ and 
    $$
    d\big(\phi_{s_1\circ s_2^{-1}(t)}(\Pii \ow),\phi_t(\Pii \oz)\big)\le 2\e_2
    \qquad \text{ for all } \quad |t|\le N\be.
    $$
    From Proposition~\ref{B16} and (\ref{defe2}.\ref{me225})
    there is 
    $|v|\le \a$
    such that
    $d(\Pii \ow,\phi_v(\Pii \oz))< C\,\ga\,e^{-\la\,N\be}$.
    Since $w_0=z_0$, by~\eqref{defpi1} 
    we have that $\Pii (\ov z),\,\Pii (\ov w)\in D(w_0)$.
    By Property~\eqref{Alip},
    \begin{align*}
    d(\Pii \ow,\Pii \oz)=d(P_{D(w_0)}(\Pii \ow),P_{D(w_0)}(\phi_v(\Pii \oz)))
    &\le A \,d(\Pii \ow,\phi_v(\Pii \oz))
    \\
    &\le A C\ga\, e^{-\la N\be}
    \le A C \ga\, d_a(\ow,\oz)
   \end{align*}
    if  $a=e^{-\la \be}$.
    Therefore $\Pii $ is Lipschitz.
    
    We now prove~\eqref{twe212}. If $\ow,\,\oz\in\Pii ^{-1}( T_w)$, applying~\eqref{s1s212}
    to $N=0$, $t=0$, $s_2(0)=0$ we get
    $$
    d(\Pii \oz,\Pii \ow)=d(\phi_{s_2(0)}(\Pii \oz),f_\ow(0))<\e_2.
    $$
    \end{proof}

             Since $\Om$ is a subshift of finite type, on $\Om$ we have the local product structure     
             \linebreak
    $[\cdot,\cdot]:\{ (\ow,\oz)\in \Om\times\Om\;|\; w_0=z_0\,\}\mapsto \Om$
    given by
    $$
    [\ow,\oz]_n :=
    \begin{cases}
    w_n &\text{if}\quad n\ge 0,\\
    z_n &\text{if}\quad n\le 0.
    \end{cases}
    $$

    \begin{Lemma}
    If $w_0=z_0$, 
    \begin{align}
    \pi(\Om)\ni
    \Pii ([\ow,\oz])&=\langle \Pii \ow,\Pii \oz\rangle_{D(w_0)}
    =P_{D(w_0)}(W^s_\ga(\Pii \ow)\cap W^{uu}_\ga(\Pii \oz)).
      \label{presbra12}
    \end{align}
    \end{Lemma}
    
    \begin{proof}
    Let $\ou:=[\ow,\oz]$.  By~\eqref{shy} there are $s_1,s_2:(\re,0)\to(\re,0)$ such that
    $$
    \forall t\ge 0\qquad 
    d(\phi_{s_1(t)}(\pi\ow),f_{\ow}(t))<\e_2,
    \quad
     d(\phi_{s_2(t)}(\pi\ou),f_{\ou}(t))<\e_2.
    $$
      Since $u_n=w_n$ for all $n\in \na$,
      $f_\ow(t)=f_\ou(t)$ for all $t\ge 0$.
    Therefore
        $$
    d\big(\phi_{s_2\circ s_1^{-1}(t)}(\Pii \ou),\phi_t(\Pii \ow)\big)\le 2\e_2
    \qquad \text{ for all } \quad t\ge 0.
    $$
    Using~(\ref{defe2}.\ref{me225}) and Proposition~\ref{B5} we have that
    $\pi\ou\in W^{ss}_\ga(\phi_v(\pi\ow))$ with $|v|<\frac 14 \a$, then
    $\pi\ou\in W^{s}_\ga(\phi_{v}(\pi\ow))$.
    Similarly $\pi\ou\in W^{u}_\ga(\phi_{v_2}(\pi\oz))$.
    Therefore
    \begin{align*}
    \pi\ou\in P_{D(w_0)}(W^s_\ga(\pi\ow)\cap W^u_\ga(\pi z))
    &= P_{D(w_0)}(W^s_\ga(\pi\ow)\cap W^{uu}_\ga(\pi \oz))
    \\
    &=\langle \pi\ow,\pi\oz\rangle_{D(w_0)}.
    \end{align*}

      \end{proof}
    
    For each $w\in W$ define
    $$
    T_w := \Pii (\{\oz\in\Om\;|\; z_0=w\,\}).
    $$
    Observe that by Lemma~\ref{le6}, $T_w$ is the continuous image of a closed cylinder in $\Om$, 
    thus it is closed. Equation~\eqref{presbra12}  implies that 
    $T_w$ is a rectangle in $D(w)$.
    
    Define 
    \begin{align*}
    \bLa := \phi_\re\big(\Pii (\Om)\big)=\phi_{[0,\a]}\big(\Pii (\Om)\big),
    \end{align*}
     where the second equality follows from~\eqref{defpi1} and~\eqref{mtau2}.
       We have that $\bLa$ is closed because $\Pii $ is continuous.
     The set $\bLa$ is  invariant  by construction.  By \eqref{PiLa12} 
     we have that
     $\bLa\subset \La_1$
     and hence by the choice of $\e_1$ in~(\ref{defe1}.\ref{me13}),
      $\bLa$ is a hyperbolic set.

             \begin{Lemma} \label{lala}
             \quad $\La\subset\bLa$. 
    \end{Lemma}
    \begin{proof}
    Let $x\in \La$. By (\ref{edix}.\ref{edixe}) there is $i\in\{1,\ldots,m\}$ such that $x\in\phi_{[-\a,0[}(\vD_i)$.
    Then $\exists P_{D_i}(x)=:y\in \vD_i$. Since 
    $\bLa$ is invariant it is enough to prove $y\in\bLa$.

    By   (\ref{edix}.\ref{edixe})  $\La\subset\bigcup_{i=1}^m\phi_{[-\a,0[}(\vD_i)$.
    Define inductively the sequence of returns $\{y_n\}_{n\in\Z}$ of $y$ to 
    $\breve\D:=\cup_i\vD_i$ by
    $y_0:=y$, 
    $$
    \tau_{n+1}=\min\{\,t>\tau_n\;|\; \phi_t(y)\in\cup_i \vD_i \,\},
    \qquad n\in\Z;
    $$
    $y_{n}=\phi_{\tau_n}(y)\in \vD_{i_n}$, $n\in\Z$.
    By the definition of $W$ before \eqref{me1234}
    there are $w_n\in W$ such that 
    \begin{equation}\label{dynwn}
    d(y_n,w_n)<\e_4.
    \end{equation}
    Let $E(w_n):=B_{3\e_3}(w_n)\cap D(w_n)$.
    Since $y_n\in E(w_n)$ and $y_{n+1}\in E(w_{n+1})$ we have that
    $y_{n+1}\in P_{D(w_{n+1})}(E(w_n))\cap E(w_{n+1})\ne\emptyset$.
    Then by Corollary~\ref{core6},
    $E(w_n)\subset \phi_{]-\a,0[}(D(w_{n+1}))$.
    In particular $\{y_n,\, w_n\}\subset \phi_{]-\a,0[}(D(w_{n+1}))\subset \text{Domain of  } P_{D(w_{n+1})}$.
    Using~\eqref{Alip} and ~\eqref{me4} we have that 
        \begin{align*}
    d(P_{D(w_{n+1})}(w_n),w_{n+1})
    &\le d(P_{D(w_{n+1})}(w_n),P_{D(w_{n+1})}(y_n)) + d(P_{D(w_{n+1})}(y_n),w_{n+1})
    \\
    &\le A\; d(w_n,y_n) + d(y_{n+1},w_{n+1})
    \\
    &\le A\,\e_4+ \e_4 < \e_3.
    \end{align*}
     Therefore by \eqref{defOm1}, $\ow:=(w_n)_{n\in\Z}\in \Om$.
   
     Define $S_n\tau(\ov w)$ by \eqref{snt} and $f_{\ov w}(t)$ by \eqref{fwt}.
     Let $\si:(\re,0)\to(\re,0)$ be the continuous function which is
    affine on the intervals $[S_n\tau(\ow),S_{n+1}\tau(\ow)]$, $n\in\Z$
    and such that
    $$
    \si(S_n\tau(\ow)) =
    \begin{cases}
    \phantom{-}\sum\limits_{k=0}^{n-1}\tau_{D({w_{k+1}})}(y_k) &n\ge 1,
    \\
    \hskip .6cm 0 & n=0,
    \\
    -\sum\limits_{k=n}^{-1}\tau_{D(w_{k+1})}(y_k) & n\le -1.
    \end{cases}
    $$
    If $s_n := S_n\tau(\ow)$ and $0\le t\le \tau_{D(w_{n+1})}(w_n)$ we have that 
    \begin{align}
    d(\phi_{\si(s_n + t)}(y)&, f_\ow(s_n+t))
    = d\big(\phi_{(\si(s_n+t)-\si(s_n))}(y_n),\phi_t(w_n)\big),
    \notag\\
    &=d\big(\phi_{bt}(y_n),\phi_t(w_n)\big),
    \hskip 2.15cm b= \tfrac{\tau_{D(w_{n+1})}(y_n)}{\tau_{D(w_{n+1})}(w_n)},
    \notag\\
    &\le d\big(\phi_{bt}(y_n),\phi_t(y_n)\big)
    + d(\phi_t(y_n), \phi_t(w_n)),
    \notag\\
    &\le \lv \tau_{D(w_{n+1})}(w_n)-\tau_{D(w_{n+1})}(y_n)\rv \; \lV \partial_t\phi\rV_{\sup}
    + \Big| \sup_{|t|\le \a}\Lip(\phi_t)\Big| \; d(y_n,w_n)
    \notag\\
    &\le \a \,\lV\partial_t\phi\rV_{\sup}
    + \Big| \sup_{|t|\le \a}\Lip(\phi_t)\Big| \; \e_4.
    \label{dphig1}
    \end{align}
    Thus, by \eqref{aphisup1},
    $$
     \forall t\in\re \qquad
    d(\phi_{\si(t)}(y),f_\ow(t)) <
    \tfrac 12 (\text{\small flow expansivity constant for $\La_U$}).
    $$
    From~\eqref{shy2} and~(\ref{defe2}.\ref{me22})  we have that
    $$
    \forall t\in\re \qquad
    d(f_\ow(t), \phi_{b+s(t)}(\Pii \ow))
    <\tfrac 15 \, (\text{\small flow expansivity constant for $\La_U$}).
    $$
    Since $t\mapsto b+s(t)$ is a	 homeomorphism of $\re$ the last two
    inequalities and the definition~\ref{dfe} of flow expansivity imply that 
    $y$ and $\Pii \ow$ are in the same orbit. Since $y=P_{D_i}(x)$ we have 
    that $y$ and $x$ are in the same orbit of $\Pii \ow$. Therefore 
    $$
    x\in \phi_\re( \Pii (\Om))=\bLa.
    $$    
    \end{proof}

              Observe that $\bLa$ may fail to be locally maximal.
    We have covered $\bLa$ with  `flowed rectangles' $\phi_{[0,\a]}(T_w)$.
    Each rectangle $T_w$ is closed under the bracket $\langle\cdot,\cdot\rangle_{D_i}$,
    $w\in D_i$.
    But it does not imply that $\bLa$ has a local product structure.
    There could be points $x,\,y\in\bLa$ with $d(x,y)$
      small which lie in different flowed rectangles  for which
    $\langle x,y\rangle\notin\bLa$. In the symbolic dynamics $\Om$
    the pre-images of these points lie in different cylinders, and thus
    they are far in $\Om$. In $\Om$ the return time is continuous.

    \begin{Lemma}\label{le9}
  Write $T^*:=\intt T$.
  
    Let $W_1:=\{\,w\in W\;|\; T_w^*\ne \emptyset\,\}$,
    then 
    $
    \Pii (\Om)=\textstyle\bigcup_{w\in W_1}  \ov{T_w^*}.
    $
       \end{Lemma}
    \begin{proof}

    The rectangles $T_w$ are images of symbolic cylinders under a continuous map, hence 
    they are closed. 
    If $w\in W$ then $\intt(T_w\setminus T_w^*)=\emptyset$ and hence its complement
    $T_w^c \cup T^*_w$ is open and dense in $\pi(\Om)$.
    By Baire Theorem $Y:=\bigcap_{w\in W}(T_w^c\cup T_w^*)$ is dense in $\pi(\Om)$.
    Since $W$ is finite, $Y$ is also open in $\pi(\Om)$.
    If $x\in Y$ and $x\in T_u$ then $x\in T_u^*$.
    Therefore $Y\subset\cO:=\bigcup_{w\in W} \ov{T^*_w}=\bigcup_{w\in W_1}\ov{T^*_w}$.
    Since $W$ is finite we have that $\cO$ is closed and dense in $\pi(\Om)$.
    Therefore $\cO=\pi(\Om)$.
    \end{proof}

       From now on we replace $W$ by $W_1$ and the rectangles $T_w$ by $\ov{T_w^*}$, so that they are proper.
       And replace correspondingly $\Om$ by $\Om\cap (W_1)^\Z$.
    Then
    \begin{equation*}\label{dfT}
    \fT:=\{\, T_w\;|\;w\in W_1\,\}
    \end{equation*}
     satisfies almost 
   all the requirements for  a proper family
    of rectangles for $\bLa$ as in definitions~\ref{dft}. and~\ref{dfm}.
    It may only not satisfy condition~\eqref{fT6} in Definition~\ref{dft}
    for $\phi_{t=0}$, namely, the rectangles in $\fT$ may intersect
    and also any discs which contain them may intersect.
    For simplicity we will only solve this problem at the end of the construction.
        We now refine the family $\fT$ to obtain a family with
    the Markov Property.

     For $\ow\in \Om$ write
     \begin{align*}
     W^s_{\text{loc}}(\ow) 
     &:= \{\,\oz\in\Om\;|\;   \forall k\ge 0 \;\; \oz_k=\ow_k\,\},
     \\
     W^u_{\text{loc}}(\ow) 
     &:= \{\,\oz\in\Om\;|\;   \forall k\le 0 \;\; \oz_k=\ow_k\,\}.
     \end{align*}

    \begin{Lemma}\label{Lpw1}\quad
    
    For $\ow=(w_k)_{k\in\Z}\in\Om$ we have that
    \begin{align*}
    \Pii (W^s_{\text{loc}}(\ow))&\subset W^s(\Pii (\ow),T_{w_0}),
    \qquad \\
     \Pii (W^u_{\text{loc}}(\ow))&\subset W^u(\Pii (\ow),T_{w_0}).
    \end{align*}
    \end{Lemma}
    
    \begin{proof} We only prove the inclusion for the stable manifolds.
    
    Let $\oz=(z_k)_{k\in\Z}\in W^s_{\text{loc}}(\ow)$ then from~\eqref{fwt} $f_\oz(t)=f_\ow(t)$ for all $t\ge 0$.
    Also $\Pii \oz\in T_{w_0}$.
    From~\eqref{shy2} we have that there are 
    $s_1,\,s_2:(\re,0)\to(\re,0)$ continuous and 
    strictly increasing and $|a_1|,|a_2|< \frac\a 4$ such that 
    \begin{align*}
    d(\phi_{s_1(t)+a_1}(\Pii \oz),f_\ow(t)) &< \e_2, \quad \forall t\ge 0,
    \\
    d(\phi_{s_2(t)+a_2}(\Pii \ow),f_\ow(t)) &< \e_2, \quad \forall t\ge 0.
    \end{align*}
    Then $\si:=s_1\circ s_2^{-1}$ is also continuous, striclty increasing,
    $\si(0)=0$ and 
    $$
    \forall t\ge a_2, \qquad
    d(\phi_{a_1+\si(t-a_2)}(\Pii \oz),\phi_t(\Pii \ow))
    < 2\,\e_2.
    $$
    Equivalently, setting $s=t-a_2$,
    $$
    \forall s\ge 0,
    \qquad
    d(\phi_{\si(s)}(\phi_{a_1}(\Pii \oz)),\phi_s(\phi_{a_2}(\Pii \ow)))
     < 2\,\e_2.
    $$
    By Proposition~\ref{B5} and~(\ref{defe2}.\ref{me225}) we have that 
    $\phi_{a_1}(\Pii \oz)\in W^{ss}_\ga(\phi_v(\phi_{a_2}(\pi\ow)))$
    with $|v|<\tfrac 14 \a$. Since $|v+a_2-a_1|<\a$ and $z_0=w_0$,
    we get that
    $\Pii \oz\in P_{D(w_0)}(W^s_\ga(\Pii \ow))\cap T_{w_0}=W^s(\Pii \ow,T_{w_0})$.
   
      \end{proof}

       By Lemma~\ref{le9}, $\cup\fT=\pi(\Om)$.
       For $w_0\in W_1$
       define $\tau:T_{w_0}\to ]\be,\a]$ and 
       $G_{w_0}:T_{w_0}\to\pi(\Om)$ as
       \begin{align}
       \tau_{w_0}(x)&=\min \{\,\tau_{D(w_1)}(x)\,|\,\exists\ow\in\Om\quad x=\pi(\ow), \ow_0=w_0,\, \ow_1=w_1\in W_1\,\};
       \label{deftau0}\\
       G_{w_0}(x) &=\phi_{\tau_{w_0}(x)}(x).
       \notag
       \end{align}
       The function $G_w$ may not be the first return map of $T_w$ to $\pi(\Om)$.
       On Lemma~\ref{GFla} we will prove that $G_w$ is the first return 
       map to $\pi(\Om)$ on rectangles which intersect $\La$.

          \begin{Lemma}\label{Lfw12}
    For $x\in \pi(\Om)$, if $w_0,w_1\in W_1$, $x\in T_{w_0}$ and $G_{w_0}(x)\in T_{w_1}$, then
    \begin{align*}
    G_{w_0}(W^s(x,T_{w_0})) &\subset W^s(G_{w_0}(x),T_{w_1}),
    \\
    G_{w_0}(W^u(x,T_{w_0})) &\supset W^u(G_{w_0}(x),T_{w_1}),
    \end{align*}
        
    \end{Lemma}  
      
    \begin{proof}
    Let $x\in T_{w_0}$ and let $w_1\in W_1$ be such that $G_{w_0}(x)\in T_{w_1}$.
       Let $\ow\in\Om$,  be such that $\Pii (\ow)=x$, $\ow_0=w_0$
       and $\ow_1=w_1$.
        Let $y\in W^s(x,T_{w_0})$, then $y=\Pii (\ou)$ for some
    $\ou\in\Om$ with  $\ou_0=w_0$  and $G_{w_0}(y)\in T_{\ou_1}$.
     Let $\ou^*:=[\ow,\ou]\in W^s_{\text{loc}}(\ow)$, then $\ou^*_1=w_1$.
    By~\eqref{presbra12} we have that 
    $\Pii \ou^* =\langle \Pii \ow,\Pii \ou\rangle_{D(w_0)}=\langle x,y\rangle_{D(w_0)}=y$.
    Since $\ou^*_1=w_1$, we have that 
    \begin{equation}\label{yw12}
    y=\pi(\ou^*)\in\phi_{[- \a,0[}(T_{w_1})\subset \phi_{[- \a,0[}(D(w_1)).
    \end{equation}

    Now let $\ow^*:=[\ou,\ow]$. Similarly $\pi( \ow^*)=x$,  $\ow^*_1=\ou_1$
    and $x\in\phi_{[-\a,0[}(D(\ou_1))$. Since  $G_{w_0}(x)\in T_{w_1}\subset D(w_1)$,
    we have that $\tau_{D(w_1)}(x)\le \tau_{D(\ou_1)}(x)$.
    Since the discs $D_i$ are disjoint and by~(\ref{edix}.\ref{edixf}), either $\ou_1=w_1$ 
    or 
    \begin{equation}\label{or2}
     \tau_{D(w_1)}(z) < \tau_{D(\ou_1)}(z)
     \text{ for all }
     z\in D(w_0)\cap\phi_{[-\a,0[}(D(w_1))\cap \phi_{[-\a,0[}(D(\ou_1)).
     \end{equation}
    But since for $y$ we have that $G_{w_0}(y)\in T_{\ou_1}$, then 
    $\tau_{D(w_1)}(y)\ge \tau_{D(\ou_1)}(y)$.
    By \eqref{yw12} we have that 
    $y\in D(w_0)\cap \phi_{[-\a,0[}(D(w_1))\cap\phi_{[-\a,0[}(D(\ou_1))$.
    Therefore \eqref{or2} does not hold, and hence $\ou_1=w_1$.
    
    Then we have that $G_{w_0}(y)=P_{D(w_1)}(y)=\pi(\si(\ou))\in T_{w_1}\cap W^{s}(G_{w_0}(x))$.
       Since $d(y,x)<\diam T_{w_0}<\a<\a_0$,
   by~\eqref{alfa1} we have that 
   $d(G_{w_0}(x),G_{w_0}(y))=d(P_{D(w_1)}(x),P_{D(w_1)}(y))<\ga$, and thus
   $G_{w_0}(y)\in W^s_\ga(G_{w_0}(x))\cap T_{w_1}=W^s(G_{w_0}(x),T_{w_1})$.
   
   The  inclusion for the unstable manifolds is proved similarly.

   \end{proof}

            For $u,v\in W_1$  define
       \begin{align*}
       T^1_{uv}&:= T_u\cap T_v,
       \\
       T^2_{uv}&:=\{\, x\in T_u\;|\; W^s(x,T_u)\cap T_v= \emptyset,\;
                                                 W^u(x,T_u)\cap T_v\ne\emptyset\,\},
       \\
         T^3_{uv}&:=\{\, x\in T_u\;|\; W^s(x,T_u)\cap T_v\ne \emptyset,\;
                                                 W^u(x,T_u)\cap T_v=\emptyset\,\},
       \\
         T^4_{uv}&:=\{\, x\in T_u\;|\; W^s(x,T_u)\cap T_v= \emptyset,\;
                                                 W^u(x,T_u)\cap T_v= \emptyset\,\}.                                        
       \end{align*}

        By Lemma~\ref{LD2} the boundary of each rectangle $T_w$ is 
        $\partial T_w =\partial^s T_w \cup \partial^u T_w$, where
        \begin{align*}
        \partial^s T_w:&= \{\,x\in T_w\;|\; x\notin\intt W^u(x,T_w)\,\},
        \\
         \partial^u T_w:&= \{\,x\in T_w\;|\; x\notin\intt W^s(x,T_w)\,\}.
        \end{align*}
             
       Define
       \begin{equation}\label{defx}
       Y:=\{\,x\in\pi(\Om)\;|\; W^s_\ga(x)\cap \partial^sT_w=\emptyset,\;
       W^u_\ga(x)\cap \partial^uT_w=\emptyset,\;\forall w\in W_1\,\},
       \end{equation}
       $$
       R(x):=\cap\{\intt T^a_{uv} \;|\; x\in T^a_{uv},\;  T_u\cap T_v\ne \emptyset\,\}.
       $$
       Comparing with condition~\eqref{alfa1} we see that the invariant manifolds $W^s_\ga(x)$, $W^u_\ga(x)$
       cross entirely rectangles in $D(x)$.
       It follows that if $x,\,y\in Y$ then either $R(x)=R(y)$ or
       $R(x)\cap R(y)=\emptyset$.
               By construction the sets $\ov{R(x)}$ for $x\in Y$ are proper and rectangles.

                    \begin{Lemma}\label{le100}
          If $w,\,u\in W_1$ then
          $$
          T^1_{wu}=T_w\cap T_u =\{\, x\in T_w\;|\;
          W^s(x,T_w)\cap T_u\ne \emptyset,\;
          W^u(x,T_w)\cap T_u\ne \emptyset\,\}.
          $$
          \end{Lemma}
          \begin{proof}
          We only prove
          $$
          \{\, x\in T_w\;|\;
          W^s(x,T_w)\cap T_u\ne \emptyset,\;
          W^u(x,T_w)\cap T_u\ne \emptyset\,\}
          \subseteq T_u,
          $$
          the other inclusions are easy.
          Suppose that
          $$
          \exists y\in W^s(x,T_w)\cap T_u\ne \emptyset
          \quad\text{and}\quad
          \exists z\in W^u(x,T_w)\cap T_u\ne \emptyset.
          $$
          Since $T_u$ is a rectangle we have that
          $$
          x=\langle y,z\rangle_{T_u}\in T_u.
          $$
          \end{proof}
 
     \begin{Lemma}\label{le1111}\quad
     
     If $w_0,w_1\in W_1$, $x,\,y\in T_{w_0}$, $G_{w_0}(x),\,G_{w_0}(y)\in T_{w_1}$ then
     $$
     G_{w_0}(\langle x,y\rangle_{T_{w_0}})
     =\langle G_{w_0}(x),G_{w_0}(y)\rangle_{T_{w_1}}.
     $$
     \end{Lemma}
     \begin{proof}
     \quad
     
     Since $\diam T_{w_0}<\a_0$ we have that
     $$
     \langle x,y\rangle_{T_{w_0}} = T_{w_0} \cap W^s_{\a_0}(x)\cap W^u_{\a_0}(y).
     $$
     By~\eqref{alfa1} 
     $$
      G_{w_0}(\langle x,y\rangle_{T_{w_0}})\in W^s_\ga(G_{w_0}(x))\cap W^u_\ga(G_{w_0}(y)). 
     $$
     By Lemma~\ref{Lfw12} we have that
     $$
     G_{w_0}(\langle x,y\rangle_{T_{w_0}})\in G_{w_0}\big(W^s(x,T_{w_0})\big)
     \subset W^s(G_{w_0}(x),T_{w_1})\subset T_{w_1}.
     $$
     Therefore
     $$
      G_{w_0}(\langle x,y\rangle_{T_{w_0}})\in T_{w_1}\cap W^s_\ga(G_{w_0}(x))\cap W^u_\ga(G_{w_0}(y))
      = \langle G_{w_0}(x),G_{w_0}(y)\rangle_{T_{w_1}}.
     $$
         \end{proof}

                 Let  $F:\pi(\Om)\to\pi(\Om)$ be the first return map to $\pi(\Om)$.     
       For the Markov property we will prove that for $x\in\La$, 
       \begin{align}
       &F\big(W^s(x,R(x)\big)\subset W^s\big(F(x),R(F(x))\big) \qquad\text{and}
       \notag\\
       &W^u\big(F(x),R(F(x))\big) \subset F\big(W^u(x,R(x))\big).
       \label{mwu}
       \end{align}
       We only show the proof for the stable manifolds.

  \begin{Lemma}\label{GFla}
  If $x\in\La$ there is $\ow\in\Om$ such that $\pi(\ow)=x$ and $G_{\ow_0}(x)=F(x)\in T_{\ow_1}$.
  \end{Lemma}
  
  \begin{proof}
  For $n\in\Z$ let $w_n\in W_1$ be such that $d(w_n,F^n(x))<\e_4$ and $F^n(x)\in D(w_n)$.
  We have that 
  \begin{align*}
  d(P_{D(w_{n+1})}(w_n),w_{w+1})
  &\le d(P_{D(w_{n+1})}(w_n),F^{n+1}(x))+d(F^{n+1}(x),w_{n+1})
  \\
  &\le d(P_{D(w_{n+1})}(w_n),P_{D(w_{n+1})}(F^n(x)) )+\e_4
  \\
  &\le A\, d(w_n, F^n(x)) +\e_4
  \le (A+1)\,\e_4
  < \e_3 \qquad\text{ using~\eqref{me4}.}
  \end{align*}
       Then by~\eqref{defOm1}, $\ow=(w_n)_{n\in\Z}\in\Om\cap (W_1)^\Z$.
     Also $x=\pi(\ow)$ by the same argument as in the end of Lemma~\ref{lala}:
     Lemma~\ref{lala} concludes that $x$ and $\pi(\ow)$ are in the same orbit, but
     in this case $x=P_{D(w_0)}(x)=P_{D(w_0)}(\pi\ow)=\pi\ow$.
     Since $F(x)\in D(w_1)$ we have that the minimum in \eqref{deftau0} is attained in
     $D(w_1)$:  $\tau_{w_0}(x) =\tau_{D(w_1)}(x)$.
     Therefore
     $F(x)=P_{D(w_1)}(x)=G_{w_0}(x)$.
     
  \end{proof}

     \begin{Lemma}\label{le10}\quad
     
     If $x,\,y\in\La\cap Y$, $R(x)=R(y)$ and $y\in W^s(x,R(x))$, then
     $R(F(x))=R(F(y))$.
     \end{Lemma}
     
     \begin{proof}\quad
     
     Let $w_1\in W_1$ be such that $F(x)\in  T_{w_1}$.
         Since $x\in\La$ 
       and $F(x)\in T_{w_1}$, by Lemma~\ref{GFla} 
       there is $\ow\in\Om$ such that
      $\pi(\ow)=x$ and $\ow_1=w_1$.
      Let $w_0:=\ow_0$, 
      then $F(x)=G_{w_0}(x)$. 
     
     We have that $y\in W^s(x,R(x))\subset W^s(x,T_{w_0})$.
     Since $G_{w_0}(x)=F(x)\in T_{w_1}$, by Lemma~\ref{Lfw12},
     \begin{equation}\label{hzti1}
     G_{w_0}(y)\in G_{w_0}(W^s(x,T_{w_0}))\subset W^s(G_{w_0}(x),T_{w_1}).
     \end{equation}
     Therefore $W^s(G_{w_0}(x),T_{w_1})=W^s(G_{w_0}(y),T_{w_1})$.
     Thus
     \begin{equation*}\label{e10a1}
     G_{w_0}(x)\in T_{w_1}
     \quad\then\quad
     \forall z\in W_1 \quad 
     W^s(G_{w_0}(x),T_{w_1})\cap T_z = W^s(G_{w_0}(y),T_{w_1})\cap T_z.
      \end{equation*}

      Since $G_{w_0}(y)\in T_{w_1}$ we also get that $G_{w_0}(y)=F(y)$.
      Therefore
           \begin{equation}\label{e10a2}
     F(x)\in T_{w_1}
     \quad\then\quad
     \forall z\in W_1 \quad 
     W^s(F(x),T_{w_1})\cap T_z = W^s(F(y),T_{w_1})\cap T_z.
      \end{equation}

     Now suppose that $z\in W_1$ and $\exists q_1\in W^u(G_{w_0}(x),T_{w_1})\cap T_z\ne \emptyset$.
     By Lemma~\ref{Lfw12}, 
     $$
     q_1\in W^u(G_{w_0}(x), T_{w_1})\subset G_{w_0}(W^u(x,T_{w_0})).
     $$
     Therefore there exists $q_0\in W^u(x,T_{w_0})$ such that 
     \begin{equation}\label{Gq0q1}
     G_{w_0}(q_0)=q_1\in T_{w_1}\cap T_z.
     \end{equation}
     Since $G_{w_0}(q_0)=q_1\in T_z$, using Lemma~\ref{Lfw12} we have that
     \begin{align*}
     G_{w_0}(\langle q_0,y\rangle_{T_{w_0}})
     \in G_{w_0}(W^s(q_0,T_{w_0})) \subset W^s(G_{w_0}(q_0),T_{z})\subset T_{z}.
     \end{align*}
     By construction $G_{w_0}(q_0)=q_1\in T_{w_1}$.
     From~\eqref{hzti1}, also $G_{w_0}(y)\in T_{w_1}$ and
      using Lemma~\ref{le1111}
     $$
     G_{w_0}(\langle q_0,y\rangle_{T_{w_0}})=\langle G_{w_0}(q_0),G_{w_0}(y)\rangle_{T_{w_1}}
     \in W^u(G_{w_0}(y),T_{w_1}).
     $$
     We have proven that
     $$
     W^u(G_{w_0}(x),T_{w_1})\cap T_z\ne \emptyset
     \qquad\then\qquad
      W^u(G_{w_0}(y),T_{w_1})\cap T_z\ne \emptyset.
     $$
     Since the hypothesis on $x$ and $y$ are symmetric
     we obtain
     \begin{equation}\label{e10b1}
     W^u(G_{w_0}(x),T_{w_1})\cap T_z\ne \emptyset
     \qquad\Longleftrightarrow\qquad
      W^u(G_{w_0}(y),T_{w_1})\cap T_z\ne \emptyset.
      \end{equation}
      Since $G_{w_0}(x)=F(x)$ and $G_{w_0}(y)=F(y)$, we have that 
     \begin{equation}\label{e10b2}
     W^u(F(x),T_{w_1})\cap T_z\ne \emptyset
     \qquad\Longleftrightarrow\qquad
      W^u(F(y),T_{w_1})\cap T_z\ne \emptyset.
      \end{equation}
      
      The statements~\eqref{e10a2} for $W^s$ and~\eqref{e10b2} for $W^u$, and Lemma~\ref{le100}
      imply Lemma~\ref{le10}.
     
          \end{proof}

       \begin{Lemma}\label{le122}
       If $x\in \La$ then
      \begin{equation}\label{mws}
       F\big(W^s(x,R(x))\big)
       \subset W^s\big(F(x),R(F(x))\big).
       \end{equation}
       \end{Lemma}
       
       \begin{proof}
       Given $y\in W^s(x,R(x))$ by Lemma~\ref{le10} we 
       have that $F(y)\in R(F(y))=R(F(x))$.
       Since $\diam R(x)<\a_0$, by~\eqref{alfa1} we have that $d(F(y),F(x))<\ga$
       and hence 
       $$
       F(y)\in W^s_\ga(F(x))\cap R(F(x))=W^s\big(F(x),R(F(x))\big).
       $$
       \end{proof}
       
       The inclusion~\eqref{mwu} is proven similarly, it can also be proved applying Lemma~\ref{le122}
       to the reverse flow $\phi_{-t}$.
       
       Then the set of rectangles $\{ R(x)\;|\;x\in\La\,\}$ satisfies the Markov property.
       Let 
       \begin{align*}
        \cR:&=\cup\{R(x)\;|\;x\in\La\,\}, \qquad
        \Si_1:=\{ y \in \cR\;|\;\forall k\in\Z\quad F^k(y)\in\cR\,\}.
        \end{align*}
       Then $\Si_1$ is closed and $F$-invariant.
       \begin{Lemma}
       If $p\in\La\cap Y$ then $R(p)\cap \Si_1$ is a rectangle in $\Si_1$.
       \end{Lemma}
       \begin{proof}
        Let $x,y\in R(p)\cap \Si_1$ and $z:=\langle x,y\rangle_{R(p)}$.
        Since $x,y\in\Si_1$ we have that $\forall n\in\Z$ \quad
        $R(F^n(x))$, $R(F^n(y))\in\{\,R(q)\;|\;q\in\La\}$.
        Since $z\in W^s(x,R(p))$ by the Markov property \eqref{mws},
        for $n\ge 0$, $F^n(z)\in W^s(F^n(x),R(F^n(x)))\subset R(F^n(x))\subset\cR$.
        Since $z\in W^u(y,R(p))$,  by~\eqref{mwu}
        $F^{-n}(z)\in W^u(F^{-n}(y),R(F^{-n}(y))\subset \cR$.
       Therefore $z\in\Si_1$.
       
       \end{proof}

        The rectangles in $\{ R(p)\;|\;p\in\La\,\}$ will give a Markov partition in 
        the hyperbolic set
        $\bLa_1:=\phi_\re(\Si_1)$,
       with $\La\subset\bLa_1\subseteq \bLa$.
       
       Finally, in order to fulfill the requirement  $D_i\cap D_j=\emptyset$  
       in Definition~\ref{dft}.\eqref{fT6} for $t=0$ we slightly modify the disks
       $D(w)$, $w\in W_1$ in the following way.
       Take $u_1, \ldots, u_L$, $L=\# W_1$ very small and distinct.
       Enumerate $W_1=\{\,w_1,\ldots,w_L\,\}$.
       Then redefine the disks $D(w_i)$ and the rectangles $\ov{R(w_i)}$
       by
       $$
       D_1(w_i):=\phi_{u_i}(D(w_i)),
       \qquad
       \ov{R_1(w_i)}:=\phi_{u_i}(\ov{R(w_i)}).
       $$
       This finishes the proof of Theorem~\ref{E4}.

       \qed

    \bigskip

    \color{black}

    \bigskip
    
    \section{Structural Stability.}\label{asst}

    As we shall see the structural stability result presented here
    does not need the local maximality of the hyperbolic set.
    
    Let $M$ be a $C^\infty$ compact manifold     and $\phi$ a $C^k$ flow on
    $M$. Let $\La$ be a hyperbolic set for $\phi$.
    Define
    \begin{align*}
    C^\a(\La,M) &:=\big\{\, u:\La\to M\;\big|\;
    u \text{ is $\a$-H\"older continuous }\}.
    \end{align*}
    This space has the structure of a Banach manifold
    modelled by the Banach space $C^\a(\La,\re^n)$ 
    with the norm $\lV f\rV:=\lV f\rV_0+\lV f\rV_\a$,
    where for $f\in C^\a(\La,\re^n)$,
    \begin{align*}
    \lV f\rV_0:=\sup_{x\in\La}|f(x)|,
    \qquad
    \lV f\rV_\a:=\sup_{x\ne y}\frac{|f(x)-f(y)|}{d(x,y)^\a}.
    \end{align*}

   Let
    \begin{align*}
    C^0_\phi(\La,M)&:=\big\{\,u\in C^0(\La, M)\;\big|\;
    D_\phi u(x):=\tfrac{d\,}{dt}u(\phi_t(x))\big\vert_{t=0}
    \text{ exists }\big\},
    \\
    C^\a_\phi(\La,M) &:=\big\{\,u\in C^\a(\La,M)\;\big|\;
    D_\phi u(x):=\tfrac{d\,}{dt}u(\phi_t(x))\big\vert_{t=0}
    \text{ exists and is $\a$-H\"older }\}.
    \end{align*}
    
    Let $X$ be the vector field of $\phi$. 
    The structural stability of the hyperbolic set $\La$ can be
    written as a solution $(u,\ga)\in C^0_\phi(\La,M)\times C^0(\La,\re^+)$ 
    to the equation
    $$
    Y\circ u = \ga \;D_\phi u
    $$
    for a vector field $Y$ nearby $X$.
    Here $u$ is the topological equivalence and $\ga$ encodes
    the reparametrization of the flow. The following theorem
    says that such solutions can be obtained as implicit
    functions of $Y$.
    
    \medskip
    
    \begin{Theorem}\label{Fss}\quad
    
    Let $M$ be a $C^{k+1}$ compact manifold and $\phi$ a flow of a $C^k$ vector field $X$ on  $M$.
    Let $\fX^k$ be the Banach manifold  of $C^k$ vector fields on $M$.
    Suppose that $\La$ is a hyperbolic set for $\phi_t$. 
    Then 
    \begin{enumerate}[(a)]
    \item\label{Fssa}
     There exist $0<\be<1$, a neighbourhood $\cU\subset \fX^k(M)$ of $X$
    and $C^{k-1}$ maps 
    \linebreak
    $\cU\to C_\phi^\be(\La,M):$
    $Y\mapsto u_Y$ and $\cU\to C^\be(\La,\re^+):$ $Y\mapsto \ga_Y$
    such that 
    \begin{equation}\label{Fssa1}
    Y\circ u_Y = \ga_Y\,D_\phi u_Y.
    \end{equation}
    \item\label{Fssb} The maps $\cU\to C_\phi^0(\La,M):$
    $Y\mapsto u_Y$ and $\cU\to C^0(\La,\re^+):$ $Y\mapsto \ga_Y$
    are $C^k$.
     \end{enumerate}
    \end{Theorem}
     
     The version for H\"older maps in item~\eqref{Fssa} is useful for 
     proving smooth dependence of equilibrium states, entropies and
     SBR measures, see~\cite{regu}. 
     For $Y$ near $X$ the topological equivalence $u_Y$ is uniquely
     determined if we require that
     $u_Y(x)\in\exp_x(X(x)^\perp)$
     and $u_Y$ near the identity. 
     One can change $\Ga(x)=\exp_x(X(x)^\perp\cap B(0,\de))$
     by any other smooth family of local transversal sections to the flow.
     
     The first version of Theorem~\ref{Fss}
     appears in de la Llave, Marco, Moriy\'on~\cite{LlMM}.
     Item~\eqref{Fssa} is proven in~\cite[p.~591]{KKPW2} and item~\eqref{Fssb}
     is proven in~\cite[p.~23ff.]{KKW} in the case $k=1$,
      but the proof can be immediately
     generalized to and arbitrary positive integer $k$.

     \begin{Corollary}\label{Css}\quad
     
     There is a neighbourhood $\cU\subset\fX^k(M)$ of $X$
     such that for every $Y\in\cU$ the map
      $u_Y$
     is a homeomorphism $u_Y:\La\to u_Y(\La)$.
     \end{Corollary}
     \begin{proof}\quad

     Since $\La$ is compact and $u_Y$ is continuous, it is enough
     (cf. Rudin~\cite{rudin2} Theorem~4.17)
     to prove that $u_Y$ is injective for $Y$ near $X$.

     Let $\eta>0$ be such that every periodic orbit in $\La$ has period
     larger than $4\eta$.
     Let $\a=\ov\a(\eta)<\eta$ be a flow expansivity constant for $(\La,\phi_t)$ as in 
     Definition~\ref{dfe}.  There is a neighbourhood $\cU_0$ of $X$ such that
     for all $Y\in\cU_0$ every periodic orbit
     in $B(\La,\a):=$ $\{\, y\in M\,|\, d(y,\La)<\a\,\}$ 
     has 
     \begin{equation}\label{pereta}
     \text{ period larger than $2\eta$.}
     \end{equation}
     Since $u_X=id_\La$ there is a neighbourhood 
     $\cU_1\subset\cU_0$  of $X$ such that 
     \begin{equation}\label{Yga}
      \forall Y\in\cU_1\quad \forall x\in \La \qquad 
     d\big(u_Y(x),x\big)<\tfrac 12 \a 
     \quad \text{ and }\quad
     \tfrac 12 < \ga_Y(x)< 2.
     \end{equation}

      Denote by $\psi^Y_s$  the flow of $Y\in\cU_1$.
      From equation~\eqref{Fssa1} we get that for $Y\in\cU_1$, $x\in\La$ and $t\in\re$ 
      we have that
      $$
      Y(\phi_t(x)) = \ga_Y(\phi_t(x))\, \tfrac{d\,}{ds} u_Y(\phi_s(x))\big\vert_{s=t}\,.
      $$
      This implies that the equation
      \begin{equation}\label{uypsiy}
      u_Y(\phi_t(x)) = \psi^Y_{s(t)}(u_Y(x))
      \end{equation}
      has a solution $s(t)$ with $s(0)=0$ satisfying
      \begin{equation}\label{sdga}
      \tfrac{ds}{dt}(t)= \ga_Y(\phi_t(x))^{-1}.
      \end{equation}
      Since $\tfrac{ds}{dt}>0$ we have that $s(t)$ has a continuous inverse 
      $t(s):\re\hookleftarrow$
      satisfying $t(0)=0$. Also
      $$
      \forall s\in\re\qquad
      u_Y(\phi_{t(s)}(x))=\psi^Y_s(u_Y(x)),
      $$
      and by~\eqref{Yga},
      $$
      \forall s\in\re \qquad
      d\big(\psi^Y_s(u_Y(x)),\phi_{t(s)}(x)\big)<\tfrac 12\,\a. 
      $$

     Suppose that $Y\in\cU_1$ and  $x,\,y\in \La$ are such that $u_Y(x)=u_Y(y)$.
     There are increasing homeomorphisms $t_1,\,t_2:(\re,0)\hookleftarrow$
     such that
     $$
      d\big(\psi^Y_s(u_Y(x)),\phi_{t_1(s)}(x)\big)<\tfrac 12\,\a
      \quad\text{ and }\quad
       d\big(\psi^Y_s(u_Y(y)),\phi_{t_2(s)}(y)\big)<\tfrac 12\,\a.
     $$
     Since $u_Y(x)=u_Y(y)$ we get that
     $$
     \forall \tau\in\re \qquad
     d(\phi_\tau(x),\phi_{t_2\circ t_1^{-1}(\tau)}(y))<\a.
     $$
     Since $t_2\circ t_1^{-1}$ is continuous and $t_2\circ t_1^{-1}(0)=0$,
     by the flow expansivity of $\phi_\tau$ we have that 
     $y=\phi_v(x)$ with $|v|<\eta$. Suppose that $v\ne 0$.
     Since by~\eqref{Yga}, $\ga_Y^{-1}<2$,
     the orbit segment from $x$ to $\phi_v(x)$ is sent by $u_Y$ to a closed
     orbit $\psi^Y_s(y)$ with a period smaller than $2 \eta$.
     This contradicts the choice of $\cU_1$ in~\eqref{pereta} and~\eqref{Yga}.
     Therefore $v=0$ and hence $y=x$.

     \end{proof}

     \begin{Proposition}\label{PSSLM}\quad
     
     If $k\in\na^+$, $\La\subset M$ is a hyperbolic set for the flow $\phi$ on $M$ 
     with vector field $X$
     and $V$ is an open neighbourhood of $\La$, 
     then there is an open set $U$ 
     such that $\La\subset U\subset \ov{U}\subset V$, 
     an open set  $X\in\cU\subset \fX^k(M)$,
     a subshift of finite type $\si:\Om\to\Om$,
     $0<\be<1$ and $C^{k-1}$ maps
     $\tau:\cU\to C^\be(\Om,\re^+)$, $Y\mapsto \tau_Y$
     and
     $\pi:\cU\to C^\be(\Om,M)$, $Y\mapsto \pi_Y$
     such that the natural extension
     of $\pi_Y$ to $\pi_Y:S(\Om,\tau_Y)\to M$
     is a well defined time preserving semiconjugacy
     \begin{equation}\label{SOmtauy}
     \begin{CD}
     S(\Om,\tau_Y) @> S_t >> S(\Om,\tau_Y)
     \\
     @V \pi_Y VV @VV \pi_Y V
     \\
     \bLa_Y @> \psi^Y_t >> \bLa_Y
     \end{CD}
     \end{equation}
     between the suspended flow $S_t$ of $\si$ in $S(\Om,\tau_Y)$
     and a hyperbolic set $\bLa_Y$ for the flow $\psi^Y_t$ of $Y$
     which satisfies
     \begin{equation}\label{maxinu}
     \forall Y\in\cU\qquad 
     \textstyle\bigcap\limits_{t\in\re}\psi^Y_t(\ov U)\subset\bLa_Y\subset V.
     \end{equation}
     In particular $\La\subset \bLa_X$.

     \end{Proposition}

        \begin{proof}\quad
     
     Let $\e_0>0$ be such that 
     $$
      B(\La,\e_0):=\{\,y\in M\;|\; d(y,x)<\e_0\,\}\subset V.
     $$
     
     Let $\cU_0\subset \fX^k(M)$ be the neighbourhood of $X$ given by 
     Theorem~\ref{Fss} and Corollary~\ref{Css}.
     
     Using Corollary~\ref{Rfe} with $\eta=1$ there exist an open set
     $\cU_1\subset\fX^k(M)$ with
      $X\subset\cU_1\subset\cU_0$ and $\e_1,\,\a\in\re$ such that 
      \begin{equation}\label{a<e1}
      0<\a<\e_1<\e_0
      \end{equation}
      $\La_1^Y:=  \bigcap_{t\in\re}\psi^Y_t(\ov{B(\La,\e_1)})$ is hyperbolic
      and
     if $Y\in\cU_1$, $z,w\in \La_1^Y$, 
     $\be\in C^0(\re,\re)$, $\be(0)=0$ 
     satisfy
     \begin{equation}\label{cfa}
     \forall t\in\re\qquad d(\psi^Y_{\be(t)}(w),\phi^Y_t(z))<\a,
     \end{equation}
     then $w=\psi^Y_\xi(z)$ for some $|\xi|<1$.

     Let 
     \begin{equation}\label{defde0}
      0<\de_0<\a
     \end{equation}
      be such that any $\de_0$ possible $1$-specification for
     $(\La,\phi_t)$ is $\tfrac 13 \a$-shadowed as in the Shadowing Corollary~\ref{CSH}.
      For a metric space $(B,d)$ and $f,\,g:A\to B$ write
     $$
     d_0(f,g):=\sup_{a\in A}d(f(a),f(b)).
     $$
     
     Let $\e_2$ be such that 
     \begin{equation}\label{e2a}
     0<\a<\e_2<\e_1.
     \end{equation}

     Let $0<\e_3<\e_2$ and $X\in \cU_2\subset\cU_1$ be such that
     \begin{gather}
     \Big(\sup_{t\in[0,1]}\Lip(\phi_t)\Big) \,\e_3 +
     \sup_{Y\in\cU_2}\sup_{t\in[0,1]}d_{0}(\phi_t,\psi^Y_t)
     <\frac{\de_0}3< \frac \a 3
     \label{lipesup}
     \qquad\text{and}
     \\
     \tfrac1 3 \de_0 + \e_3 < \tfrac 23 \de_0.
     \label{de0e1de0}
     \end{gather}
     
     Observe that by the choice of $\e_1$ in~\eqref{a<e1} and Corollary~\ref{Rfe} the set
     $$
     \La_1:=\textstyle
     \bigcap_{t\in\re}\phi_t(\ov{B(\La,\e_1)})
     $$
     is hyperbolic.
     Let $X\in\cU_3\subset\cU_2$ be a neighbourhood of $X$
     from Theorem~\ref{Fss} applied to $\La_1$  so that $u_Y$
     and $\ga_Y$ are defined on $\La_1$ for all $Y\in\cU_3$
     and that $\cU_3$ is small enough so that
      \begin{equation}\label{d0ui}
     \forall Y\in\cU_3\qquad 
     d_0(u_Y|_{\La_1},id|_{\La_1})< \tfrac 13 \a.
     \end{equation}

     Define
     \begin{equation}\label{defla2}
     \La_2:=\textstyle\bigcap_{t\in\re}\phi_t(\ov{B(\La,\e_2)}).
     \end{equation}

     By Theorem~\ref{E4} there is a hyperbolic set $\bLa$ for $\phi_t$ such that 
     $\La\subset\La_2\subset\bLa\subset B(\La,\e_1)\subset V$ and with
     a Markov partition which is the image by a time preserving Lipschitz
     semiconjugacy $\pi$ of a suspension of a subshift of finite type 
     $\si:\Om\to\Om$
     with H\"older ceiling function $\tau:\Om\to\re^+$.
     \begin{equation*}
     \begin{CD}
     S(\Om,\tau) @> S_t >> S(\Om,\tau)
     \\
     @V \pi VV @VV \pi V
     \\
     \bLa @> \phi_t >> \bLa
     \end{CD}
     \end{equation*}
     
     Since $\bLa$ is invariant by $\phi_t$ and $\bLa\subset B(\La,\e_1)$
     we have that $\bLa\subset\La_1$. In particular
     for all $Y\in\cU_3$ the functions $u_Y$ and $\ga_Y$
     are defined on $\La_1\supset\bLa$.
     For  $Y\in\cU_3$ let $\bLa_Y:=u_Y(\bLa)$ and 
     $\pi_Y:=u_Y\circ \pi|_{\Om\times\{0\}}\in C^\be(\Om\times\{0\},\bLa_Y)$.
     Since $u_Y:\bLa\to \bLa_Y$ is a (non time preserving) topological equivalence
     among the flows $(\bLa,\phi_t)$ and $(\bLa_Y,\psi^Y_s)$ we have that the
     following diagram commutes:
     $$
     \begin{CD}
     \Om\times\{0\} @> \si>> \Om\times\{0\}
     \\
     @V \pi VV @VV \pi V
     \\
     \bLa @> F >> \bLa
     \\
     @V u_Y VV @VV u_Y V
       \\
       \bLa_Y @> F_Y >> \bLa_Y
     \end{CD}
     $$
     where $F_Y:= u_Y\circ F\circ u_Y^{-1}$ and $F$ is the first return map
     to the Markov partition for $\bLa$, $F(\pi\ow)=\phi_{\tau(\ow)}(\pi\ow)$.
     
     Since $u_Y$ satisfies the equalities~\eqref{uypsiy} and~\eqref{sdga} 
     we have that
     \begin{align*}
     F_Y(\pi_Y\ow) &= \psi^Y_{\tau_Y(\ow)}(\pi_Y(\ow)), \quad \text{where}
     \\
     \tau_Y(\ow) :&=\int_0^{\tau(\ow)}\ga_Y(\phi_t(\pi \ow))^{-1}\,dt.
     \end{align*}
     Then the diagram~\eqref{SOmtauy} commutes.

     It remains to prove the inclusions in~\eqref{maxinu}.
     Since $Y\mapsto u_Y\in C^0(\La,M)$ is continuous,
     and $u_X(\bLa)=\bLa\subset V$,
     there is a neighbourhood $X\in \cU_4\subset\cU_3$ such that 
     $$
     \forall Y\in\cU_4\qquad \bLa_Y=u_Y(\bLa)\subset V.
     $$
     
     Let $U:=B(\La,\e_3)$.
     Given
     $
     z\in\textstyle \bigcap_{s\in\re}\psi^Y_s(\ov U)
     $
     let $z_n:=\psi^Y_n(z)$, $n\in\Z$.
     Since 
     \linebreak
     $z_n\in \ov U=\ov{B(\La,\e_3)}$ there is
     $y_n\in \La$ such that $d(y_n,z_n)\le \e_3$. 
      Define a $1$-specification $f(t)$ 
      (cf. definition~\ref{B8}) for $(\La,\phi_t)$
      by 
      $$
      f(n+t):=\phi_t(y_n), \quad n\in\Z, \quad t\in[0,1[.
      $$
      Using~\eqref{lipesup}
      we have that 
      \begin{align}\label{fphit}
      d(f(n+t),\psi^Y_t(z_n))&=d(\phi_t(y_n),\psi^Y_t(z_n))
      \notag\\
      &\le d(\phi_t(y_n),\phi_t(z_n))+d(\phi_t(z_n),\psi^Y_t(z_n))
      \notag\\
      &\le  \Big(\sup_{t\in[0,1]}\Lip(\phi_t) \Big) \,d(z_n,y_n)+
      \sup_{t\in[0,1]}d_{0}(\phi_t,\psi^Y_t)
      \notag\\
      &\le \tfrac 13 \de_0.
      \end{align}
      Observe  that $f$ is $\de_0$-possible because,
      using that $\psi^Y_1(z_n)=z_{n+1}$ and~\eqref{de0e1de0}, we have
      that
      $$
      d(f(n+1^-),y_{n+1})\le d(f(n+1^-),\psi^Y_1(z_n))
      +d(z_{n+1},y_{n+1})
      \le \tfrac13 \de_0 +\e_3
      < \tfrac 23 \de_0.
      $$
      By the Shadowing Corollary~\ref{CSH} and the choice of $\de_0$ in~\eqref{defde0},
       there is  $x\in M$ and 
       an increasing homeomorphism $\be:(\re,0)\to(\re,0)$ 
        such that  
      \begin{equation}\label{dbef}
      \forall t\in\re\qquad d(\phi_{\be(t)}(x),f(t^\pm))<\tfrac 13 \a.
      \end{equation}
      Since by definition $f(t)\in\La$ and $\be(t)$ is a homeomorphism,
      and using~\eqref{e2a} and \eqref{defla2}
       we obtain  that
      \begin{equation}\label{xinsatu}
      x\in \textstyle\bigcap_{\be\in\re}\phi_\be\big(B(\La,\tfrac 13 \a)\big)
      \subset\bigcap_{t\in\re}\phi_t(\ov{B(\La,\e_2)})=\La_2\subset\bLa.
      \end{equation}
      
      Then by~\eqref{dbef}, ~\eqref{fphit} and \eqref{defde0},
      \begin{align}
      d(\phi_{\be(n+t)}(x),\psi^Y_{n+t}(z) )
      &=
      d(\phi_{\be(n+t)}(x),\psi^Y_t(z_n))
      \notag\\
      &\le d(\phi_{\be(n+t)}(x),f(n+t))+d(f(n+t),\psi^Y_t(z_n))
      \notag\\
      &\le \tfrac 13 \a+ \tfrac 13 \de_0\le \tfrac 23 \a.
      \label{le23a}
       \end{align}

      There is a homeomorphism $\si:(\re,0)\to(\re,0)$ such that
      \begin{equation}\label{uYsi}
      	\psi^Y_{\si(t)}(u_Y(x))=u_Y(\phi_{\be(t)}(x)).
      \end{equation}
      Indeed, comparing~\eqref{uYsi} and~\eqref{Fssa1} we get that
      $$
      \si(0)=0 \quad\text{and}\quad \dot\si(t)=\dot\be(t)\,\ga_Y(\phi_{\be(t)}(x))^{-1}.
      $$
      Therefore, using~\eqref{d0ui} and \eqref{le23a},
      \begin{align*}
      d(\psi^Y_{\si(t)}(u_Y(x)),\psi^Y_t(z))
      &= d(u_Y(\phi_{\be(t)}(x)),\psi^Y_t(z))
      \\
      &\le 
      d(u_Y(\phi_{\be(t)}(x)),\phi_{\be(t)}(x))
      +d(\phi_{\be(t)}(x),\psi^Y_t(z))
      \\
      &\le d_0(u_Y,id)+  \tfrac 23 \a <\a.
      \end{align*}
      
      By the choice of $\a$ in~\eqref{cfa}
       we have that there is $\xi\in\re$ such that
      $z=\psi^Y_\xi(u_Y(x))$.
      Since $\bLa_Y=u_Y(\bLa)$ is $\psi^Y_t$-invariant 
      and by~\eqref{xinsatu}, $x\in\bLa$,
      we obtain that $z\in\bLa_Y$. Therefore
      $$
      \textstyle\bigcap_{s\in\re}\psi^Y_s(\ov U)\subset\bLa_Y.
      $$

        \end{proof}

     \begin{Proposition}\label{liftmuSSLM}\quad
     
     Let $\phi\in\cU$ and $\La\subset U\subset V$ be from Proposition~\ref{PSSLM}.
     If $Y\in\cU$ and $\mu$ is a $\psi_t^Y$-invariant Borel probability with 
     $\supp\mu\subset \ov U$ then there is and $S_t$-invariant Borel probability
     on $S(\Om,\tau_Y)$ such that $(\pi_Y)_*(\nu)=\mu$.
     \end{Proposition}
     
     \begin{proof}
     By Proposition~\ref{PSSLM} $\supp(\mu)\subset \La_Y$.
     For $f\in C^0(\La_Y,\re)$ let  $G(f\circ\pi_Y)=\int f\,d\mu$.
     Then $G$ defines a positive linear functional on a subspace $W$
     of $C^0(S(\Om,\tau_Y),\re)$. By the Riesz extension theorem, 
     $G$ extends to a positive linear functional on $C^0(S(\Om,\tau_Y),\re)$.
     Since $G(1)=G(1\circ \pi)=1$, the extension $G$ corresponds to a Borel probability 
     $\be$ 
     on $S(\Om,\tau_Y)$. 
     By the construction $(\pi_Y)_*\be=\mu$.
     By compactness we can choose a sequence 
     $T_k\to\infty$ such that 
     $ \nu=\lim_k\tfrac 1{T_k}\int_0^{T_k}(S_t)_*\be\, dt$
     exists. We have that $\nu$ is $S_t$-invariant and using that
     $$
     (\pi_Y)_* (S_t)_*\be =(\psi^Y_t)_* (\pi_Y)_*\be=(\psi^Y_t)_* \mu=\mu,
     $$
     we get that $(\pi_Y)_*\nu=\mu$.
     
     \end{proof}

     \section{Stability of hyperbolic Ma\~n\'e sets.}\label{ashms}
   
      Observe that the definition~\ref{defhipL} of a hyperbolic set 
      for an autonomous lagrangian corresponds to hyperbolicity
      for the flow restricted to an energy level. Nearby energy levels 
      can be considered as perturbations of the flow.
      
      Given Tonelli lagrangian $L:TM\to\re$ let $E_L:=v\cdot L_v-L$
      be its energy function and define
      $$
      e_0(L):= \inf\{\, k\in\re\;|\; \pi(E_L^{-1}\{k\})=M\,\}.
      $$
      Observe that
      $$
      \tfrac{d\,}{dt}E_L(x,tv)=v\cdot L_{vv}(x,tv)\cdot v>0 \qquad\text{if }v\ne 0.
      $$
      Then $f(t):=E_L(x,tv)$ is increasing on $t>0$.
      This implies that for $k>e_0(L)$ the radial projection is a diffeomorphism
      between the unit tangent bundle and the energy level $E^{-1}_L\{k\}$.

      Denote by  $O_M$ the zero section of $TM$.
      
      \begin{Lemma}\label{LSH1}\quad
      
      If $E^{-1}\{c(L)\}\cap O_M\ne \emptyset$ then 
      $c(L)=e_0(L)$ and $E^{-1}_L(e_0)\cap O_M\subset \cA(L)\subset\mN(L)$.
      \end{Lemma}
      \begin{proof}\quad
      
      Let $$
      \Si^+(L):=\big\{v\in TM\;\big|\; \pi\circ \phi_t(v)|_{[0,+\infty[}\text{ is semi-static }\big\}.
      $$
      We have that $E(\Si^+(L))=\{c(L)\}$ (Ma\~n\'e~\cite[p.~146]{Ma7}).
      By the Covering Property (Ma\~n\'e~\cite[Theorem~VII]{Ma7} 
      also~\cite[Th.~VII]{CDI}), $\pi(\Si^+(L))=M$. Therefore $c(L)\ge e_0(L)$ and
      $$
      E^{-1}\{c(L)\}\cap O_M\ne \emptyset \qquad\then\qquad
      c(L)=e_0(L).
      $$
      Let 
      $$
      F(x,v):=L(x,v)+E_L(x,v)=v\cdot L_v(x,v).
      $$
      If $\ga:[0,T]\to M$ is a closed curve with energy $c(L)$ we have that
      \begin{align}\label{intFcA}
      \int_0^T F(\ga,\dga)\,dt =
      \int_0^T \big(L(\ga,\dga)+c(L)\big)\,dt
     \ge A_{L+c}(\ga|_{[0,t]})+\Phi_c(\ga(t),\ga(0))\ge 0
     \quad \forall t\in]0,T[.
      \end{align}
      Suppose that $(x_0,0)\in E^{-1}_L\{c(L)\}\cap O_M\ne \emptyset$.
      Let $\ga(t)\equiv x_0$ for $t\in[0,2]$.
      Then $E_L(\ga,\dga)=E_L(x_0,0)=c(L)$,
      $F(x_0,0)=0$ and $\int_0^2F(\ga,\dga)\, dt =0$.
      From~\eqref{intFcA} we obtain that $\ga$ is static and hence
      $(\ga,\dga)=(x_0,0)\in\cA(L)$.

     \end{proof}
     
     \begin{Corollary}\label{CSH1}\quad 
     
     If $\mN(L)$ does not contain fixed points of the lagrangian flow
     then $c(L)>e_0(L)$ and the energy level 
     $E_L^{-1}\{c(L)\}$ is diffeomorphic to the unit 
     tangent bundle $SM$ under the radial projection. 
     \end{Corollary}
   
    \medskip

     \begin{Theorem}\label{SDL}\quad
     
     Suppose that $L:TM\to\re$ is a Tonelli lagrangian and that 
     its Ma\~n\'e set $\mN(L)$ is hyperbolic without fixed points. 
     Let $V$ be an open set with $\mN(L)\subset V$.
     Then there is a subshift
     of finite type $\si:\Om\to\Om$ and there are
     open sets $\mN(L)\subset U\subset V$ and $0\in\cU\subset C^2(M,\re)$
     and continuous maps 
     $C^2(M,\re)\supset\cU\ni \phi\mapsto \tau_\phi\in C^0(\Om,\re^+)$ and
     $C^2(M,\re)\supset\cU\ni \phi\mapsto \pi_\phi\in C^0(S(\Om,\tau_\phi),TM)$, 
     where $\big(S(\Om,\tau_\phi),S_t\big)$ is the suspension flow of $\si$ 
     with ceiling function $\tau_\phi$,
     and there are hyperbolic sets $\La_\phi=\pi_\phi(S(\Om,\tau_\phi))$ for the flow
     $\vr^{L+\phi}_t$ of $L+\phi$ restricted to the energy level 
     $E_{L+\phi}^{-1}\{c(L+\phi)\}$ such that
     $$
     \forall \phi\in\cU\qquad
     \mN(L+\phi)\subset 
     \textstyle\bigcap_{t\in\re}\vr^{L+\phi}_t(U)\subset \La_\phi \subset V
     $$
     and the following diagram commutes for all $t\in\re$:
     $$
     \begin{CD}
     S(\Om,\tau_\phi) @> S_t>> S(\Om,\tau_\phi)
     \\
     @V \pi_\phi VV @VV \pi_\phi V
     \\
     \La_\phi @> \vr^{L+\phi}_t >> \La_\phi
     \end{CD}
     $$
     
     Moreover any invariant measure $\mu$ for $L+\phi$ with $\supp(\mu)\subset U$
     lifts to an invariant measure $\nu$ on $S(\Om,\tau_\phi)$ 
     with $(\pi_\phi)_*\,\nu=\mu$. In particular $\mu$ is a 
     minimizing measure for 
     $L+\phi$ iff it is the projection of an invariant probability $\nu$ on 
     $S(\Om,\tau_\phi)$ which minimizes the integral of the function 
     $A_\phi:=(L+\phi)\circ\pi_\phi$.
    
     \end{Theorem}
   
   \pagebreak
   
      \begin{proof}\quad
   
     By Lemma~5.1 in~\cite{CP} for all $\ell\ge 2$ the map 
     $C^\ell(M,\re)\ni\phi\mapsto c(L+\phi)\in\re$ is continuous.
     And by Lemma~5.2 in~\cite{CP} the map
     $C^\ell(M,\re)\ni\phi\mapsto \mN(L+\phi)$ is upper semicontinuous.
          By Corollary~\ref{CSH1} for $\phi\in\cU$ small enough we can 
     identify the energy levels 
     $E^{-1}_{L+\phi}\{c(L+\phi)\}\approx E^{-1}_{L}\{c(L)\}\approx SM$
     with the unit tangent bundle $SM$ 
     and consider their lagrangian flows as perturbations of the
     flow of $L$ on the same manifold $SM$.
     
     Let $\I:=]c(L)-\e,c(L)+\e[$ with $\e>0$ small.
     Let $P_{\phi,k}:SM \to E^{-1}_{L+\phi}\{k\}$
     be the radial projection and let $X_\phi$ be the Lagrangian 
     vector field for $L+\phi$. Let $\fX^1(SM)$ be the
     vector space of  $C^1$ vector fields on $SM$.
     The map $\X:C^\ell(M,\re)\times \I\to \fX^1(SM)$,
     $(\phi,k)\mapsto (dP_{\phi,k})^{-1}\circ X_{\phi}\circ P_{\phi,k}$
     is $C^{\ell-2}$ in a neighbourhood of $(\phi,k)=(0,c(L))$. 
     If we compose this map with the continuous function 
     $\phi\mapsto k=c(L+\phi)$ we obtain a continuous map 
     $C^2(M,\re) \to \fX^1(SM)$,
     $\phi\mapsto \X(\phi,c(L+\phi))$.
     The flow of 
     this vector field is   
     \linebreak
     $\psi^\phi_t:=P_\phi^{-1}\circ \vr^{L+\phi}_t\circ P_\phi$,
     which is smoothly conjugate to the lagrangian flow of $L+\phi$ on
     $E^{-1}_{L+\phi}\{c(L+\phi)\}$, and $\phi\mapsto \psi^\phi_t$ is a
     continuous family of $C^1$ flows on $SM$.
     Then there are neighbourhoods $\cU$ of $0$ and $U\subset V$ of $\mN(L)$
     in $SM$ such that for any $\phi\in\cU$ the set 
     $\bigcap_{t\in\re}\psi^{\phi}_t(\ov U)$ is hyperbolic and
     $P_\phi^{-1}(\mN(L+\phi))\subset U$, using the 
     upper semicontinuity of $\phi\mapsto \mN(L+\phi)$.

     Applying Proposition~\ref{PSSLM} and Proposition~\ref{liftmuSSLM},
     shrinking $\cU$ and $U$ if necessary, we obtain Proposition~\ref{SDL}.
     
     \end{proof}

    \begin{Remark}\label{remell}
    \quad

    In the proof of Theorem~\ref{SDL} the map 
    $\X:C^\ell(M,\re)\times \I\to \fX^{\ell-1}(SM)$
    is conitnuous and then the map 
    $C^\ell(M,\re)\to \fX^{\ell-1}(SM)$, 
    $\phi\mapsto \X(\phi,c(L+\phi))$ is 
    continuous in the $C^{\ell-1}$ topology
    for $\fX^{\ell-1}(SM)$.

    \end{Remark}


\def\cprime{$'$} \def\cprime{$'$} \def\cprime{$'$} \def\cprime{$'$}
\providecommand{\bysame}{\leavevmode\hbox to3em{\hrulefill}\thinspace}
\providecommand{\MR}{\relax\ifhmode\unskip\space\fi MR }
\providecommand{\MRhref}[2]{%
  \href{http://www.ams.org/mathscinet-getitem?mr=#1}{#2}
}
\providecommand{\href}[2]{#2}

\end{document}